\documentclass[10pt]{amsart}
\usepackage[latin2]{inputenc}
\usepackage{amsmath}
\usepackage{graphicx}
\usepackage{amssymb}
\usepackage{esint}
\usepackage[dvipsnames]{xcolor}
\usepackage{tikz}
\usepackage{xxcolor}
\usepackage{floatrow}
\usepackage{color}
\usepackage{amsthm}
\usepackage{epsfig}
\usepackage[english]{babel}
\usepackage{enumerate}
\newtheorem{theorem}{Theorem}
\newtheorem{proposition}[theorem]{Proposition}
\newtheorem{lemma}[theorem]{Lemma}
\newtheorem{corollary}[theorem]{Corollary}
\newtheorem{definition}[theorem]{Definition}
\newtheorem{remark}[theorem]{Remark}

\usetikzlibrary{shapes.multipart}
\usepackage[linktocpage=true,colorlinks=true,linkcolor=Blue,citecolor=Green]{hyperref}

\newtheorem*{theorem*}{Theorem}

\def\XXint#1#2#3{{\setbox0=\hbox{$#1{#2#3}{\int}$ }
\vcenter{\hbox{$#2#3$ }}\kern-.6\wd0}}

\def\e{\varepsilon}
\def\dashint{\fint}

\allowdisplaybreaks[2]

\newcommand{\supp}{\operatorname{supp}}

\newcommand{\dist}{\operatorname{dist}}

\newcommand{\RVE}{{\operatorname{RVE}}}

\newcommand{\Cov}{\operatorname{Cov}}

\newcommand{\R}{{\mathbb{R}}}
\newcommand{\Rd}{{\mathbb{R}^d}}
\newcommand{\Rmd}{{\mathbb{R}^{m\times d}}}
\newcommand{\Rmdd}{{\mathbb{R}^{m\times d\times d}}}
\newcommand{\Rddskew}{{\mathbb{R}^{d\times d}_{\rm skew}}}
\newcommand{\Rm}{{\mathbb{R}^{m}}}
\newcommand{\Hilbert}{{\operatorname{H}}}

\newcommand{\shom}{{\mathsf{hom}}}

\newcommand{\domain}{{\mathcal{O}}}

\newcommand{\Ahom}{A_{\operatorname{hom}}}
\newcommand{\ubar}{\bar u}

\newcommand{\SO}{\operatorname{SO}}

\newcommand{\Huloc}{H^1_{\mathrm{uloc}}}
\newcommand{\loc}{{\mathrm{loc}}}

\newcommand{\eps}{\varepsilon}

\definecolor{Yellow}{rgb}{0.95,0.9,0.0} 
\definecolor{Red}{rgb}{0.8,0.1,0.1}
\definecolor{Green}{rgb}{0.1,0.65,0.2}
\definecolor{Blue}{rgb}{0.1,0.1,0.8}
\definecolor{Purple}{rgb}{0.7,0.1,0.7}
\definecolor{Grey}{rgb}{0.6,0.6,0.6}

\begin{document}

\title[Optimal error estimates in nonlinear stochastic homogenization]{Optimal homogenization rates in stochastic homogenization of nonlinear uniformly elliptic equations and systems}

\author{Julian Fischer}
\address{(J. Fischer) IST Austria, Am Campus 1, 3400 Klosterneuburg, Austria}
\author{Stefan Neukamm}
\address{(S. Neukamm) TU Dresden, Faculty of Mathematics, 01062 Dresden}

\thanks{SN acknowledges partial support by the Deutsche Forschungsgemeinschaft (DFG, German Research Foundation) -- project number 405009441.}

\begin{abstract}
We derive optimal-order homogenization rates for random nonlinear elliptic PDEs with monotone nonlinearity in the uniformly elliptic case. More precisely, for a random monotone operator on $\mathbb{R}^d$ with
stationary law
(i.\,e.\ spatially homogeneous statistics)
and fast decay of correlations on scales larger than the microscale $\varepsilon>0$, we establish homogenization error estimates of the order $\varepsilon$ in case
$d\geq 3$, respectively of the order $\varepsilon |\log \varepsilon|^{1/2}$ in case $d=2$.
Previous results in nonlinear stochastic homogenization have been limited to a small algebraic rate of convergence $\varepsilon^\delta$.
We also establish error estimates for the approximation of the homogenized operator by the method of representative volumes of the order $(L/\varepsilon)^{-d/2}$ for a representative volume of size $L$.
Our results also hold in the case of systems for which a (small-scale) $C^{1,\alpha}$ regularity theory is available.
\end{abstract}


\maketitle

\tableofcontents

\section{Introduction}

In the present work, we establish quantitative homogenization results with optimal rates for nonlinear elliptic PDEs of the form
\begin{align}
\label{Equation}
-\nabla \cdot (A_\varepsilon(x,\nabla u_\varepsilon)\big)=f
\quad\quad
\text{in }\Rd
,
\end{align}
where $A_\varepsilon$ is a random monotone operator whose correlations decay quickly on scales larger than a microscopic scale $\varepsilon$. For scalar problems and also certain systems, we obtain the optimal convergence rate $O(\varepsilon)$ of the solutions $u_\varepsilon$ towards the solution $u_\shom$ of the homogenized problem
\begin{align}
\label{EffectiveEquation}
-\nabla \cdot (A_\shom(\nabla u_\shom))=f
\quad\quad
\text{in }\Rd
\end{align}
in three or more spatial dimensions $d\geq 3$. In two spatial dimensions $d=2$, we obtain the optimal convergence rate $O(\varepsilon |\log \varepsilon|^{1/2})$ upon including a lower-order term in the PDEs \eqref{Equation} and \eqref{EffectiveEquation}.

Our results may be seen as the optimal quantitative counterpart in the case of $2$-growth to the qualitative stochastic homogenization theory for monotone systems developed by Dal~Maso and Modica \cite{DalMasoModica2,DalMasoModica}, respectively as the nonlinear counterpart of the optimal-order stochastic homogenization theory for linear elliptic equations developed by Gloria and Otto \cite{GloriaOtto,GloriaOtto2} and Gloria, Otto, and the second author \cite{GloriaNeukammOtto,GloriaNeukammOtto2}. Just like for \cite{DalMasoModica2,DalMasoModica}, a key motivation for our work is the homogenization of nonlinear materials.

In the context of random materials, the first -- and to date also the only -- homogenization rates for elliptic PDEs with monotone nonlinearity were obtained by Armstrong and Smart \cite{ArmstrongSmart}, Armstrong and Mourrat \cite{ArmstrongMourrat}, and Armstrong, Ferguson, and Kuusi~\cite{ArmstrongFergusonKuusi} in the form of a small algebraic convergence rate $\varepsilon^\delta$ for some $\delta>0$. The optimal convergence rates derived in the present work improve substantially upon their rate:
{We derive an error estimate of the form
\begin{align}
\label{ErrorEstimate}
||u-u_\shom||_{L^p(\Rd)}
\leq
\begin{cases}
\mathcal{C}(f) \, \varepsilon^{1/2} &\text{for }d=1,
\\
\mathcal{C}(f) \, \varepsilon |\log \varepsilon|^{1/2} &\text{for }d=2,
\\
\mathcal{C}(f) \, \varepsilon &\text{for }d\geq 3,
\end{cases}
\end{align}
with $p=\frac{2d}{d-2}$ for $d\geq 3$, and $p=2$ for $d=1,2$.
However, in contrast to the works of Armstrong et al.\ we make no attempt to reach optimal stochastic integrability:
While in our homogenization error estimate the random constant $\mathcal{C}(f)$ in \eqref{ErrorEstimate} has bounded stretched exponential moments in the sense
\begin{align*}
\mathbb{E}\bigg[\exp\bigg(\Big(\frac{\mathcal{C}(f)}{C(f)}\Big)^{\bar\nu} \bigg)\bigg]\leq 2
\end{align*}
for some universal constant $\bar \nu>0$ (which is in particular independent of the right-hand side $f$) and for some constant $C(f)=C(||f||_{L^1},||f||_{L^{d+1}})$, the homogenization error estimates for linear elliptic PDEs with optimal rate in \cite{ArmstrongKuusiMourrat,GloriaOttoNew} establish (essentially) Gaussian stochastic moments $\mathbb{E}[\exp(|\mathcal{C}_f/C(f,\mu)|^{2-\mu})]\leq 2$ for any $\mu>0$.
Likewise, the homogenization error estimates for monotone operators with non-optimal rate $\varepsilon^\delta$ of \cite{ArmstrongFergusonKuusi,ArmstrongMourrat,ArmstrongSmart} include optimal stochastic moment bounds.
}

Before providing a more detailed summary of our results, let us give a brief overview of the previous quantitative results in nonlinear stochastic homogenization. The first -- logarithmic -- rates of convergence in the stochastic homogenization of a nonlinear second-order elliptic PDE were obtained by Caffarelli and Souganidis \cite{CaffarelliSouganidis} in the setting of non-divergence form equations. Subsequently, a rate of convergence $\varepsilon^\delta$ has been derived both for equations in divergence form and non-divergence form by Armstrong and Smart \cite{ArmstrongSmart,ArmstrongSmart2} and Armstrong and Mourrat \cite{ArmstrongMourrat}.
In the homogenization of Hamilton-Jacobi equations, a rate of convergence of the order $\varepsilon^{1/8}$ has been obtained by Armstrong, Cardaliaguet, and Souganidis \cite{ArmstrongCardaliaguetSouganidisErrorEstimates}. For forced mean curvature flow, Armstrong and Cardaliaguet \cite{ArmstrongCardaliaguet} have derived a convergence rate of order $\varepsilon^{1/90}$.
These rates of convergence are all expected to be non-optimal (compare, for instance, the result for Hamilton-Jacobi equations to the rate of convergence $\varepsilon$ known in the periodic homogenization setting \cite{MitakeTranYu}).

To the best of our knowledge, the present work constitutes the first optimal-order convergence results for any nonlinear stochastic homogenization problem. However, we are aware of an independent work in preparation by Armstrong, Ferguson, and Kuusi~\cite{ArmstrongFergusonKuusi2}, which aims to address the same question.
In contrast to our work -- which is inspired by the approach to quantitative stochastic homogenization via spectral gap inequalities of \cite{GloriaOtto,GloriaOtto2,GloriaNeukammOtto,GloriaNeukammOtto2} -- the upcoming work \cite{ArmstrongFergusonKuusi2} relies on the approach of sub- and superadditive quantities of \cite{ArmstrongSmart,ArmstrongKuusiMourrat,ArmstrongFergusonKuusi}. Nevertheless, both our present work and the approach of \cite{ArmstrongFergusonKuusi2,ArmstrongFergusonKuusi} use the concept of correctors for the linearized PDE, see Section~\ref{SectionStrategy} for details.

Before turning to a more detailed description of our results, let us briefly comment on the theory of periodic homogenization of nonlinear elliptic equations.
A quantitative theory for the periodic homogenization of nonlinear monotone operators has recently been derived by Wang, Xu, and Zhao \cite{WangXuZhao}. A corresponding result for degenerate elliptic equations of $p$-Laplacian type may be found in \cite{WangXuZhao2}.
In the periodic homogenization of polyconvex integral functionals, the single-cell formula for the effective material law (which determines the effective material law by a variational problem on a single periodicity cell) may fail \cite{BarchiesiGloria,GeymonatMuellerTriantafyllidis}, a phenomenon associated with possible ``buckling'' of the microstructure. 
A related phenomenon of loss of ellipticity may occur in the periodic homogenization of linear elasticity \cite{BrianeFrancfort,FrancfortGloria,GloriaRuf,Gutierrez}. Note that polyconvex integral functionals occur naturally in the framework of nonlinear elasticity \cite{BallPolyconvex}; however, their Euler-Lagrange equations in general lack a monotone structure.
Nevertheless, in periodic homogenization of nonlinear elasticity the single-cell formula is valid for small deformations \cite{NeukammSchaeffner,NeukammSchaeffner2}, and rates of convergence may be derived.
 
\subsection{Summary of results}
\label{SectionSummary}

To summarize our results in a continuum mechanical language, we consider the effective macroscopic behavior of a nonlinear and microscopically heterogeneous material. We assume that the behavior of the nonlinear material is described by the solution $u_\varepsilon:\Rd\rightarrow \Rm$ of a second-order nonlinear elliptic system of the form
\begin{align*}
-\nabla \cdot \big(A_\varepsilon(x,\nabla u_\varepsilon)\big)=f
\end{align*}
for some random monotone operator $A_\varepsilon:\Rd\times \Rmd\rightarrow \Rmd$ with correlation length $\varepsilon$ and some right-hand side $f\in L^\frac{2d}{d+2}(\Rd;\Rm)$. We further assume that the random monotone operator $A_\varepsilon$ is of the form $A_\varepsilon(x,\xi):=A(\omega_\varepsilon(x),\xi)$, where $\omega_\varepsilon$ is a random field representing the random heterogeneities in the material; for each realization of the random material (i.\,e.\ each realization of the probability distribution), $\omega_\varepsilon$ selects at each point $x\in \Rd$ a local material law $A(\omega_\varepsilon(x),\cdot):\Rmd\rightarrow\Rmd$ from a family $A$ of potential material laws. Under some suitable additional conditions, the theory of stochastic homogenization shows that for small correlation lengths $\varepsilon\ll 1$ the above nonlinear elliptic system is well-approximated by a homogenized effective equation. The effective equation again takes the form of a nonlinear elliptic system, however now with a spatially homogeneous effective material law $A_\shom:\Rmd\rightarrow \Rmd$. It is our goal to provide an optimal-order estimate for the difference of the solution $u_\varepsilon$ to the solution $u_\shom$ of the effective equation
\begin{align*}
-\nabla \cdot \big(A_\shom(\nabla u_\shom)\big) = f,
\end{align*}
as well as to give an optimal-order error bound for the approximation of the effective material law $A_\shom$ by the method of representative volumes.

To be mathematically more precise, we consider a random field taking values in the unit ball of a Hilbert space $\omega_\varepsilon:\Omega\times \Rd \rightarrow \Hilbert\cap B_1$ and a family of monotone operators $A:(\Hilbert \cap B_1) \times \Rmd\rightarrow \Rmd$ indexed by the unit ball of this Hilbert space. We then define a random monotone operator $A_\varepsilon(x,\cdot):= A(\omega_\varepsilon(x),\cdot)$ by selecting a monotone operator from the family $A$ at each point $x\in \Rd$ according to the value of the random field $\omega_\varepsilon(x)$. Note that the property of $\omega_\varepsilon$ being Hilbert-space valued is not an essential point and just included for generality: Even the homogenization for a scalar-valued random field (and correspondingly a single-parameter family of monotone operators $A:(\mathbb{R}\cap B_1)\times \Rmd\rightarrow \Rmd$) would be highly relevant and just as difficult, as it could describe e.\,g.\ composite materials.

The conditions on the random field $\omega_\varepsilon$ and the family of monotone operators $A$ are as follows:
\begin{itemize}
\item We assume \emph{spatial statistical homogeneity of the material}: The statistics of the random material should not depend on the position in space. In terms of a mathematical formulation, this assumption corresponds to stationarity of the probability distribution of $\omega_\varepsilon$ under spatial translations.
\item We assume sufficiently fast \emph{decorrelation of the material properties on scales larger than a correlation length $\varepsilon$}. In terms of a mathematical formulation, we make this notion rigorous by assuming that a spectral gap inequality holds. More precisely, we shall assume that $\omega_\varepsilon$ itself is a Hilbert-space valued random field on $\Rd$ which satisfies a spectral gap inequality and on which the random monotone operator $A_\varepsilon$ depends in a pointwise way as $A_\varepsilon(x,\xi):=A(\omega_\varepsilon(x),\xi)$, where the map $A:(\Hilbert\cap B_1)\times \Rmd\rightarrow \Rmd$ is continuously differentiable and Lipschitz and where $\partial_\xi A$ is continuously differentiable and Lipschitz in its first variable.
\item We assume \emph{uniform coercivity and boundedness} of the monotone operator in the sense that $(A(\omega,\xi_2)-A(\omega,\xi_1)):(\xi_2-\xi_1)\geq \lambda |\xi_2-\xi_1|^2$ as well as $|A(\omega,\xi_2)-A(\omega,\xi_1)|\leq \Lambda |\xi_2-\xi_1|$ hold for all $\xi_1,\xi_2\in \Rmd$ and every $\omega\in \Hilbert\cap B_1$ for suitably chosen constants $0<\lambda<\Lambda<\infty$.
\item
For some of our results, we shall impose an additional condition, which essentially entails a $C^{1,\alpha}$ regularity theory for the equation \eqref{Equation} on the microscopic scale $\varepsilon$. Namely, we shall assume Lipschitz continuity of the random field $\omega_\varepsilon$ on the $\varepsilon$-scale with suitable stochastic moment bounds on the local Lipschitz norm and a uniform bound on the second derivative $\partial_\xi^2 A$, along with one of the following three conditions:
\begin{itemize}
\item Our problem consists of a single nonlinear monotone PDE, i.\,e.\ $m=1$.
\item We are in the two-dimensional case $d=2$.
\item Our system has Uhlenbeck structure, i.\,e.\ the nonlinearity has the structure $A(\omega,\xi)=\rho(\omega,|\xi|^2)\xi$ for some scalar function $\rho$, and the same is true for the homogenized operator.
\end{itemize}
\vspace{2mm}
\end{itemize}
Under these assumptions, we establish the following quantitative stochastic homogenization results with optimal rates for the nonlinear elliptic PDE \eqref{Equation}.
\begin{itemize}
\item 
The solution $u_\varepsilon$ to the nonlinear PDE with fluctuating random material law \eqref{Equation}
can be approximated by the solution $u_\shom$ to a homogenized effective PDE of the form \eqref{EffectiveEquation}. In case $d=2$ or $d=1$, we include a lower-order term in the PDEs, see Theorem~\ref{TheoremErrorEstimate2d}.
The homogenized effective material law is given by a monotone operator $A_\shom:\Rmd\rightarrow \Rmd$ which is independent of the spatial variable $x\in \Rd$ and satisfies analogous uniform ellipticity and boundedness properties.  {The error $u-u_\shom$ is estimated by
\begin{align*}
||u-u_\shom||_{L^p(\Rd)}
\leq
\begin{cases}
\mathcal{C}(f) \, \varepsilon &\text{for }d\geq 3,
\\
\mathcal{C}(f) \, \varepsilon |\log \varepsilon|^{1/2} &\text{for }d=2,
\\
\mathcal{C}(f) \, \varepsilon^{1/2} &\text{for }d=1,
\end{cases}
\end{align*}
with $p$ and $\mathcal{C}(f)$ as in \eqref{ErrorEstimate}.
Without the additional small-scale regularity assumption, we still achieve half of the rate of convergence $\varepsilon^{1/2}$ for $d\geq 3$, $\varepsilon^{1/2} |\log \varepsilon|^{1/4}$ for $d=2$, and $\varepsilon^{\frac13}$ for $d=1$, respectively ---} a result that we also establish for the Dirichlet problem in bounded domains.
\item The homogenized effective operator $A_\shom$ may be approximated by the method of representative volumes, and this approximation is subject to the following a priori error estimate: If a box of size $L\geq \varepsilon$ is chosen as the representative volume, the error estimate
\begin{align*}
\big|A_\shom^\RVE(\xi)-A_\shom(\xi)\big|\leq \mathcal{C}(L,\xi) (1+|\xi|)^C |\xi|
\bigg(\frac{L}{\varepsilon}\bigg)^{-d/2}
\end{align*}
holds true for every $\xi\in \Rmd$, where $A_\shom^\RVE$ denotes the approximation of $A_\shom$ by the method of representative volumes and where again $\mathcal{C}(L,\xi)$ denotes a random constant with bounded stretched exponential moments (independent of $L$, $\xi$, and $\varepsilon$).
The systematic error is of higher order
\begin{align*}
\big|\mathbb{E}\big[A_\shom^\RVE(\xi)\big]-A_\shom(\xi)\big| \leq C (1+|\xi|)^C |\xi| \bigg(\frac{L}{\varepsilon}\bigg)^{-d} \bigg|\log \frac{L}{\varepsilon} \bigg|^{d+2},
\end{align*}
at least in case $d\leq 4$ (which includes the physically relevant cases $d=2$ and $d=3$).
Without the additional small-scale regularity assumption, we achieve almost the same overall estimate $|A_\shom^\RVE(\xi)-A_\shom(\xi)|\leq \mathcal{C}(L,\xi) |\xi| (L/\varepsilon)^{-d/2} (\log \frac{L}{\varepsilon})^C$, but not the improved bound for the systematic error.
\end{itemize}
Note that the rates of convergence $||u_\varepsilon-u_\shom||_{L^{2d/(d-2)}}\leq \mathcal{C} \varepsilon$ in case $d\geq 3$ respectively $||u_\varepsilon-u_\shom||_{L^2}\leq \mathcal{C} \varepsilon |\log \varepsilon\smash{|^{1/2}}$ in case $d=2$ coincide with the optimal rate of convergence in the homogenization of linear elliptic PDEs, see e.\,g.\ \cite{ArmstrongKuusiMourrat,GloriaNeukammOtto,GloriaNeukammOtto2,GloriaOttoNew}. Similarly, the rate of convergence for the representative volume approximation $|A_\shom^\RVE(\xi)-A_\shom(\xi)|\leq \mathcal{C} (L/\varepsilon)^{-d/2}$ coincides with the corresponding optimal rate for linear elliptic PDEs, as does (essentially) the higher-order convergence rate for the systematic error. As linear elliptic PDEs may be regarded as a particular case of our nonlinear PDE \eqref{Equation}, our rates of convergence are optimal.

{Beyond the scope of the present paper -- but subject of current ongoing work by various authors, and building in parts on the results of the present work -- are problems like describing the fluctuations in solutions to random nonlinear elliptic PDEs or the quantitative homogenization of nonlinear elliptic PDEs with $p$ growth in the case $p\neq 2$.
}

\subsection{Examples}

To illustrate our results, let us mention two examples of random nonlinear elliptic PDEs and systems to which our theorems apply, as well as an important class of random fields $\omega_\varepsilon$ which satisfy our assumptions.

We first give an example for the random field $\omega_\varepsilon$. Let $\theta:\mathbb{R}^k\rightarrow \mathbb{R}^k\cap B_1$ be any Lipschitz map taking values only in the unit ball. Let $Y_\varepsilon:\Rd\rightarrow \mathbb{R}^k$ be any stationary Gaussian random field whose correlations decay sufficiently quickly in the sense
\begin{align*}
\left|\Cov\big[Y_\varepsilon(x),Y_\varepsilon(y)\big]\right| \lesssim \frac{1}{1+\big(\frac{|x-y|}{\varepsilon}\big)^{d+\delta}}
\end{align*}
for some $\delta>0$. Set $\Hilbert:=\mathbb{R}^k$. Then the random field $\omega_\varepsilon:\Rd\rightarrow \Hilbert$ defined by
\begin{align*}
\omega_\varepsilon(x):=\theta(Y_\varepsilon(x))
\end{align*}
satisfies a spectral gap inequality with correlation length $\varepsilon$ in the sense of Definition~\ref{DefinitionSpectralGap}; for a proof see e.\,g.\  \cite{DuerinckxGloriaFunctionalInequality}. As stationarity is immediate, any such $\omega_\varepsilon$ satisfies our key assumptions on the random field \hyperlink{P1}{(P1)} and \hyperlink{P2}{(P2)} stated in Section~\ref{SectionAssumptions} below.
Note in particular that the spectral gap assumption allows for the presence of (sufficiently quickly decaying, namely integrable) long-range correlations. Typical realizations for two such random fields are depicted in Figure~\ref{FigureRandomField}.

To state the first example of a random monotone operator satisfying our assumptions, consider any two deterministic spatially homogeneous monotone operators $A_1:\Rmd\rightarrow \Rmd$ and $A_2:\Rmd\rightarrow \Rmd$ subject to the ellipticity and Lipschitz continuity assumptions \hyperlink{A1}{(A1)} and \hyperlink{A2}{(A2)}. Furthermore, consider any random field $\omega_\varepsilon:\mathbb{R}^d\rightarrow [0,1]$.
Then the operator
\begin{align*}
A_\varepsilon(x,\xi):=\omega_\varepsilon(x) A_1(\xi) + (1-\omega_\varepsilon(x)) A_2(\xi) 
\end{align*}
satisfies our assumptions \hyperlink{A1}{(A1)}--\hyperlink{A3}{(A3)}. Note that this operator corresponds to the PDE
\begin{align*}
-\nabla \cdot \Big(\omega_\varepsilon(x) A_1(\nabla u)+(1-\omega_\varepsilon(x)) A_2(\nabla u)\Big)=f.
\end{align*}
The additional small-scale regularity assumption \hyperlink{R}{(R)} is satisfied whenever the operators $A_1$ and $A_2$ have uniformly bounded second derivatives, the random field $\omega_\varepsilon$ is regular enough, and one of the three following conditions holds: The equation is scalar ($m=1$), the spatial dimension is at most two ($d\leq 2$), or both $A_1$ and $A_2$ as well as the homogenized operator $A_\shom$ are of Uhlenbeck structure.

As a second simple example of a monotone operator, consider for any random field $\omega_\varepsilon:\Rd\rightarrow [0,1]$ the operator
\begin{align*}
A_\varepsilon(x,\xi):=\frac{1+|\xi|^2}{1+(1+\omega_\varepsilon(x)) |\xi|^2} \xi.
\end{align*}
It satisfies our assumptions, possibly with the exception of the additional regularity condition \hyperlink{R}{(R)}. Note that this operator corresponds to the PDE
\begin{align*}
-\nabla \cdot \bigg(\frac{1+|\nabla u|^2}{1+\omega_\varepsilon(x) |\nabla u|^2} \nabla u\bigg)=f.
\end{align*}
The additional small-scale regularity assumption -- stated in \hyperlink{R}{(R)} below -- is satisfied in the scalar case $m=1$ as well as in the low-dimensional case $d\leq 2$, provided that the random field $\omega_\varepsilon$ is sufficiently smooth on the microscopic scale $\varepsilon$.

\begin{figure}
\includegraphics[scale=1.2]{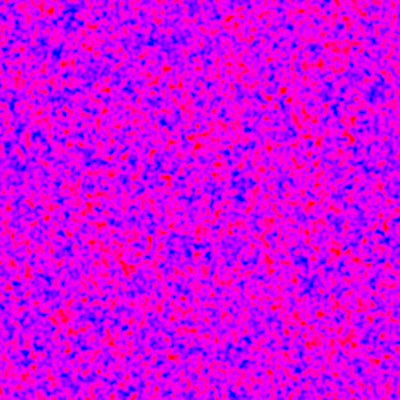}
~~~~~
\includegraphics[scale=1.2]{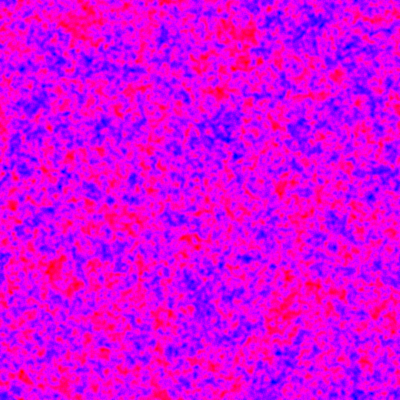}
\caption{Typical realizations of random fields obtained by applying a nonlinear map pointwise to a stationary Gaussian random field with only short-range correlations (left), respectively to a stationary Gaussian random field with barely integrable correlations (right).\label{FigureRandomField}}
\end{figure}

\subsection{Notation}

The number of spatial dimensions will be denoted by $d\in \mathbb{N}$.
For a measurable function $u$, we denote by $\nabla u$ its (weak) spatial derivative. For a function of two variables $A(\omega,\xi)$, we denote its partial derivatives by $\partial_\omega A$ and $\partial_\xi A$. For a function $f:\Rd\rightarrow\R$, we denote by $\partial_i f$ its partial derivative with respect to the coordinate $i$. For a matrix-valued function $M:\Rd\rightarrow \Rmd$, we denote by $\nabla \cdot M$ its divergence with respect to the second index, i.\,e.\ $(\nabla \cdot M)_i = \sum_{j=1}^d \partial_j M_{ij}$.

Throughout the paper, we use standard notation for Sobolev spaces. In particular, we denote by $H^1(\Rd)$ the space of all measurable functions $u:\Rd\rightarrow \mathbb{R}$ whose weak spatial derivative $\nabla u$ exists and which satisfy $||u||_{H^1}:=(\int_\Rd |u|^2 + |\nabla u|^2 \,dx)^{1/2}<\infty$. Similarly, we denote by $H^1(\Rd;\Rm)$ the space of $\Rm$-valued vector fields with the analogous properties and the analogous norm. For $d\geq 3$, we denote by $\dot H^1(\Rd;\Rm)$ the space of all measurable functions $u$ with $||u||_{\dot H^1}:=(\int_\Rd |\nabla u|^2 \,dx)^{1/2} + ||u||_{L^{2d/(d-2)}}<\infty$.
By $H^1_{\rm loc}(\Rd)$ we denote the space of all measurable functions $u:\Rd\rightarrow \mathbb{R}$ for which all restrictions $u|_{B_r}$ to finite balls ($0<r<\infty$) belong to $H^1(B_r)$. For a box $[0,L]^d$, we denote by $H^1_{\rm per}([0,L]^d)$ the closure in the $H^1([0,L]^d)$ norm of the smooth $L$-periodic functions.
By $\Huloc(\Rd)$, we denote the space of measurable functions $u$ whose weak derivative $\nabla u$ exists and which satisfy the bound $||u||_{\Huloc}:=\sup_{x\in \Rd} (\int_{B_1(x)} |\nabla u|^2 + |u|^2 \,d\tilde x)^{1/2}<\infty$.

In order not to overburden notation, we shall frequently suppress the dependence on the spatial variable $x$ in many expressions, for instance we will write $A(\omega_\varepsilon,\xi)$ or $A(\omega_\varepsilon,\nabla u)$ instead of $A(\omega_\varepsilon(x),\xi)$ respectively $A(\omega_\varepsilon(x),\nabla u(x))$. By an expression like $\partial_\xi A(\omega,\xi) \Xi$, we denote the derivative of $A$ with respect to the second variable at the point $(\omega,\xi)$ evaluated in direction $\Xi$. Similarly, by $\partial_\omega \partial_\xi A(\omega,\xi)\delta \omega \Xi$ we denote the second mixed derivative of $A$ with respect to its two variables at the point $(\omega,\xi)$ evaluated in directions $\delta\omega$ and $\Xi$. We use notation like $\delta F$ or $\delta\omega_\varepsilon$ to indicate infinitesimal changes (i.\,e.\ differentials) of various quantities and functions.

For two numbers $a,b\in \R$, we denote by $a\wedge b$ the minimum of $a$ and $b$. We write $a\sim b$ to indicate that two constants $a,b\in (0,\infty)$ are of a similar order of magnitude.
For a matrix $M\in \Rmd$, we denote by $|M|:=\sum_{i,j}|M_{ij}|^2$ its Frobenius norm. By $\Rddskew$ we denote the set of skew-symmetric matrices of dimension $d\times d$.

By $B_r(x)$ we denote the ball of radius $r$ centered at $x$. By $B_r$ we will denote the ball of radius $r$ around the origin. For two sets $A,B\subset \Rd$ and a point $x\in \Rd$, we denote by $A+B:=\{a+b:A\in A,b\in B\}$ their Minkowski sum, respectively by $x+A$ the translation of $A$ by $x$.
By $\Hilbert$ we will denote a Hilbert space; we will denote its unit ball by $\Hilbert\cap B_1$.

By $C$ and $c$ we will denote -- typically large respectively typically small -- nonnegative constants, whose precise value may change from occurrence to occurrence but which only depend on a certain set of parameters.
For a set $M$, we denote by $\sharp M$ the number of its elements.

We write $\omega_\varepsilon \sim \mathbb{P}$ to indicate that a random field $\omega_\varepsilon$ is distributed according to the probability distribution $\mathbb{P}$. For two random variables or random fields $X$ and $Y$, we write $X\sim Y$ to indicate that their laws coincide.

Whenever we use the terms ``coefficient field'' or ``monotone operator'', we shall implicitly assume measurability.

\section{Main results}

Before stating our main results and the precise setting, let us introduce the key objects in the homogenization of nonlinear elliptic PDEs and systems with monotone nonlinearity. To fix a physical setting, we will here give an outline of the meaning of the objects in the context of electric fields and associated currents. However, a major motivation for the present work -- and in particular for the choice to include the case of nonlinear elliptic systems -- stems from the homogenization of nonlinear elastic materials. While in this context the monotone structure is lost \cite{BallPolyconvex}, it may be retained for small deformations \cite{ContiEtAlCauchyBorn,FrieseckeTheilCauchyBorn,NeukammSchaeffner,
NeukammSchaeffner2,ZhangMatchingConvex}. A corresponding result in the context of stochastic homogenization will be established in \cite{DNRS}. The homogenization of monotone operators may also be viewed as a simple but necessary first step towards a possible quantitative homogenization theory for nonlinear elastic materials for larger deformations.

In the context of electric fields and currents, we are concerned with a scalar equation (i.\,e.\ $m=1$); the functions $u$ and $u_\shom$ in \eqref{Equation} and \eqref{EffectiveEquation} correspond (up to a sign) to the electric potential in the heterogeneous material respectively in the homogenized picture. Their gradients $\nabla u$ respectively $\nabla u_\shom$ are the associated electric fields. The monotone functions $A(\omega_\varepsilon(x),\xi)$ respectively $A_\shom(\xi)$ are the material law and describe the electric current created by a given electric field $\xi$. Note that in contrast to existing optimal results in stochastic homogenization, the material law may be nonlinear in $\xi$. Finally, the PDEs \eqref{Equation} and \eqref{EffectiveEquation} correspond to prescribing the sources and sinks of the electric current.

The central object in the quantitative homogenization of elliptic PDEs is the \emph{homogenization corrector} $\phi_\xi$, which in our context is an $\Rm$-valued random field on $\Rd$. It provides a bridge between the microscopic (heterogeneous) and the macroscopic (homogenized) picture: For a given constant macroscopic field gradient $\xi\in \Rmd$, the corrector $\phi_\xi$ provides the correction yielding the associated microscopic field gradient $\xi+\nabla \phi_\xi$. The corrector $\phi_\xi$ is defined as the unique (up to a constant) sublinearly growing distributional solution to the PDE
\begin{align}
\label{EquationCorrector}
-\nabla \cdot (A(\omega_\varepsilon,\xi+\nabla \phi_\xi))=0,
\end{align}
see Definition~\ref{D:corr} for the precise definition.  Note that \emph{a priori}, similarly to the linear elliptic case, the existence and uniqueness of such a solution is unclear. {In the setting of our main results, by our choice of scaling the corrector $\phi_\xi$ is expected to display fluctuations on the length scale $\varepsilon$ with typical gradients of the order $|\nabla \phi_\xi|\sim |\xi|$; furthermore, in case $d\geq 3$ the typical magnitude of the corrector is of order $|\phi_\xi|\sim |\xi| \varepsilon$.}

The effective (homogenized) material law $A_\shom$ (see Definition~\ref{D:corr} for the precise definition) may be determined in terms of the homogenization corrector: In principle, at each point $x\in \Rd$ the microscopic material law $A(\omega_\varepsilon(x),\cdot):\Rmd\rightarrow \Rmd$ associates a current $A(\omega_\varepsilon(x),\xi_\varepsilon)$ to a given electric field $\xi_\varepsilon$; likewise, the effective macroscopic material law $A_\shom(\cdot):\Rmd\rightarrow \Rmd$ associates a current $A_\shom(\xi)$ to a given macroscopic electric field $\xi$. As the macroscopic current corresponds to an ``averaged'' microscopic current, the macroscopic material law should be obtained by averaging the microscopic flux. More precisely, the homogenization corrector $\phi_\xi$ associates a microscopic electric field $\xi+\nabla \phi_\xi$ to a given macroscopic electric field $\xi$; therefore the macroscopic current $A_\shom(\xi)$ should be given by the ``average'' of the microscopic current $A(\omega_\varepsilon,\xi+\nabla \phi_\xi)$. In our setting, due to stationarity and ergodicity ``averaging'' corresponds to taking the expected value at an arbitrary point $x\in \Rd$ (and we will suppress the point $x\in \Rd$ in the notation). In other words, we have
\begin{align*}
A_\shom(\xi)=\mathbb{E}[A(\omega_\varepsilon,\xi+\nabla \phi_\xi)]
\stackrel{a.\,s.}{=} \lim_{r\rightarrow \infty} \fint_{B_r} A(\omega_\varepsilon(x),\xi+\nabla \phi_\xi(x)) \,dx.
\end{align*}

Our main results on the quantitative approximation of the solution $u_\varepsilon$ to the nonlinear elliptic PDE with randomly fluctuating material law
\begin{align*}
-\nabla \cdot (A(\omega_\varepsilon,\nabla u_\varepsilon))=f
\end{align*}
by the solution $u_\shom$ to the homogenized equation
\begin{align*}
-\nabla \cdot (A_\shom(\nabla u_\shom))=f
\end{align*}
are stated in Theorem~\ref{TheoremErrorEstimate} and Theorem~\ref{TheoremErrorEstimate2d} in the case of the full space $\mathbb{R}^d$. The case of a bounded domain -- however with a lower rate of convergence -- is considered in Theorem~\ref{TheoremErrorEstimateDomains}.

Our second main result -- the error estimates for the approximation of the effective material law $A_\shom$ by the method of representative volumes -- is stated in Theorem~\ref{T:RVE} and Corollary~\ref{CorollaryTotalRVE}.

\subsection{Assumptions and setting}
\label{SectionAssumptions}
We denote by $d\in \mathbb{N}$ the spatial dimension and by $m\in \mathbb{N}$ the system size; in particular, the case $m=1$ corresponds to a scalar PDE.
Let $\lambda$ and $\Lambda$, $0<\lambda\leq \Lambda<\infty$, denote ellipticity and boundedness constants. Let $H$ and $\Hilbert \cap B_1$ denote a Hilbert space and the open unit ball in $H$, respectively. We denote by
\begin{align*}
A:\Hilbert \times \Rmd\rightarrow \Rmd
\end{align*}
a family of operators, indexed by a parameter in $\Hilbert$. We require $A$ to satisfy the following conditions:
\begin{itemize}
\item[\hypertarget{A1}{(A1)}] Each operator $A(\omega,\cdot)$ in the family is monotone in the second variable in the sense
  \begin{align*}
    \big(A(\omega,\xi_2)-A(\omega,\xi_1)\big)\cdot (\xi_2-\xi_1) \geq \lambda |\xi_2-\xi_1|^2
  \end{align*}
  for every parameter $\omega \in \Hilbert$ and all $\xi_1,\xi_2\in \Rmd$.
\item[\hypertarget{A2}{(A2)}] Each operator $A(\omega,\cdot)$ is continuously differentiable in the second variable and Lipschitz in the sense
  \begin{align*}
    |A(\omega,\xi_2)-A(\omega,\xi_1)| \leq \Lambda |\xi_2-\xi_1|
  \end{align*}
for every parameter $\omega \in \Hilbert$ and all $\xi_1,\xi_2\in \Rmd$. Furthermore, we have $A(\omega,0)=0$ for every parameter $\omega\in \Hilbert$.
\item[\hypertarget{A3}{(A3)}] The operator $A(\omega,\xi)$ and its derivative $\partial_\xi A(\omega,\xi)$ are continuously differentiable in the parameter $\omega$ with bounded derivative in the sense
\begin{equation*}
|\partial_\omega A(\omega,\xi)| \leq \Lambda |\xi|,\qquad|\partial_\omega \partial_\xi A(\omega,\xi)| \leq \Lambda,
\end{equation*}
for every $\omega \in\Hilbert$ and all $\xi\in \Rmd$. Here, $\partial_\omega$ and $\partial_\xi$ denote the Fr\'echet derivative with respect to the first variable and the partial derivative with respect to the second variable, respectively. Furthermore, $|\cdot|$ denotes the operator norm on $\Hilbert\rightarrow \Rmd$ and $\Hilbert\times \Rmd \rightarrow \Rmd$, respectively. 
\end{itemize}

Throughout our paper, we will reserve the term \textit{parameter field} for a measurable function $\widetilde\omega:\R^d\to\Hilbert\cap \smash{B_1}$. With help of the operator family $A$, we may associate to each parameter field $\smash{\widetilde\omega}$  a space-dependent monotone material law $\smash{\Rd}\ni x\mapsto A(\widetilde\omega(x),\cdot)$. We denote the space of all parameter fields by $\Omega$ and equip $\Omega$ with the $L^1_{loc}(\Rd;\Hilbert)$ topology.
We then equip the space of parameter fields $\Omega$ with a probability measure $\mathbb P$ and write $\omega_\varepsilon:\R^d\to\Hilbert\cap B_1$ to denote a random parameter field sampled with $\mathbb P$.

It will be our second key assumption that the probability measure $\mathbb P$ describes a stationary random field with correlation length $\varepsilon$ (which is also the reason why we include the index ``$\varepsilon$'' in our notation). To be precise, we impose the following conditions on $\mathbb P$:
\begin{itemize}
\item[\hypertarget{P1}{(P1)}] $\mathbb{P}$ is \textit{stationary} in the sense that the probability distribution of $\omega_\varepsilon(\cdot+y)$ coincides with the probability distribution of $\omega_\varepsilon(\cdot)$ for all $y\in \mathbb{R}^d$. From a physical viewpoint this corresponds to the assumption of statistical spatial homogeneity of the random material: While each sample of the random material is typically spatially heterogeneous, the underlying probability distribution is spatially homogeneous. 
\item[\hypertarget{P2}{(P2)}] $\mathbb{P}$ features \textit{fast decorrelation on scales $\geq\varepsilon$} in the sense of the spectral gap assumption of Definition~\ref{DefinitionSpectralGap} below. Here, and throughout the paper $0<\varepsilon\leq 1$ is fixed and denotes the \textit{correlation length} of the material. Note that this corresponds to a quantitative assumption of ergodicity by assuming a decorrelation in the coefficient field $\omega_\varepsilon$ on scales $\geq \varepsilon$. 
\end{itemize}
Under the previous conditions homogenization occurs (in fact \hyperlink{P2}{(P2)} can be weakend to qualitative assumption of ergodicity). In particular, we may introduce the corrector of stochastic homogenization and define a homogenized monotone operator (i.\,e.\ a homogenized material law) $\Ahom$ as follows.
{Note that we suppress the (implicit) dependence of quantities like the corrector $\phi_\xi$ on the correlation length $\varepsilon$ in order to not overburden notation.
}
\begin{definition}[Corrector and homogenized operator]\label{D:corr}
  Let the assumptions \hyperlink{A1}{(A1)}--\hyperlink{A3}{(A3)} and \hyperlink{P1}{(P1)}--\hyperlink{P2}{(P2)} be in place. Then for all $\xi\in\Rmd$ there exists a unique random field $\phi_\xi:\Omega\times \Rd \rightarrow \Rm$, called the \emph{corrector} associated with $\xi$, with the following properties:
  \begin{enumerate}[(a)]
  \item For $\mathbb P$-almost every realization of the random field $\omega_\varepsilon$ the corrector $\phi_\xi(\omega_\e,\cdot)$ has the regularity $\phi_\xi(\omega_\e,\cdot)\in H^1_{\rm loc}(\Rd;\Rm)$, satisfies $\fint_{B_1}\phi_\xi(\omega_\varepsilon,\cdot)\,dx=0$, and solves the corrector equation \eqref{EquationCorrector} in the sense of distributions.
  \item The gradient of the corrector $\nabla\phi_\xi$ is stationary in the sense that 
    \begin{equation*}
      \nabla\phi_\xi(\omega_\e,\cdot+y)=\nabla\phi_\xi(\omega_\e(\cdot+y),\cdot)\quad\text{a.e. in }\Rd
    \end{equation*}
    holds for $\mathbb P$-a.e.~$\omega_\varepsilon$ and all $y\in\Rd$.
  \item The gradient of the corrector $\nabla\phi_\xi$ has finite second moments and vanishing expectation, that is
    \begin{equation*}
      \mathbb E\big[\nabla\phi_\xi\big]=0,\qquad       \mathbb E\big[|\nabla\phi_\xi|^2\big]<\infty.
    \end{equation*}
  \item The corrector $\mathbb P$-almost surely grows sublinearly at infinity in the sense
  \begin{align*}
  \lim_{R\rightarrow \infty } \frac{1}{R^2} \fint_{B_R} |\phi_\xi(\omega_\e,x)|^2 \,dx =0.
  \end{align*}
  \end{enumerate}
  Moreover, for each $\xi\in\Rmd$ we may define
  \begin{equation}\label{eq:homogop}
    A_{\shom}(\xi):=\mathbb E\big[A(\omega_\varepsilon,\xi+\nabla\phi_\xi)\big],
  \end{equation}
  where the right-hand side in this definition is independent of the spatial coordinate $x$.
  The map $A_{\shom}:\Rmd\to\Rmd$ is called the \emph{effective operator} or the \emph{effective material law}.
\end{definition}
We shall see that the homogenized material law $A_\shom:\Rmd\rightarrow \Rmd$  defined by \eqref{eq:homogop} inherits the monotone structure from the heterogeneous material law $A(\omega_\varepsilon,\cdot)$, see Theorem~\ref{TheoremStructureProperties} below. 
\smallskip

For \emph{some} of our results we will assume that
the following microscopic regularity condition is satisfied. Note that the condition essentially implies a small-scale $\smash{C^{1,\alpha}}$ regularity theory (i.\,e.\ a $\smash{C^{1,\alpha}}$ theory on the $\varepsilon$ scale) for the heterogeneous equation and a global $C^{1,\alpha}$ regularity theory for the homogenized (effective) equation.
\begin{itemize}
\item[\hypertarget{R}{(R)}] Suppose that at least one of the following three conditions is satisfied:
\begin{itemize}
\item The equation is scalar (i.\,e.\ $m=1$).
\item The number of spatial dimension is at most two (i.\,e.\ $d\leq 2$).
\item The system is of Uhlenbeck structure in the sense that there exists a function $\rho:\Hilbert\cap B_1\times \mathbb{R}_0^+ \rightarrow \mathbb{R}_0^+$ with $A(\omega,\xi)=\rho(\omega,|\xi|^2)\xi$ for all $\omega\in \Hilbert\cap B_1$ and all $\xi\in \Rmd$; furthermore, the effective operator given by \eqref{eq:homogop} is also of Uhlenbeck structure.
\end{itemize}
Suppose in addition that the second derivative of $A$ with respect to the second variable exists and satisfies the bound $|\partial_\xi\partial_\xi A(\omega,\xi)|\leq \Lambda$ for all $\omega \in \Hilbert\cap B_1$ and all $\xi\in \Rmd$.

Suppose furthermore that the random field $\omega_\varepsilon$ is Lipschitz regular on small scales in the following sense: There exists a random field $\mathcal{C}$ with uniformly bounded stretched exponential moments $\mathbb{E}[\exp(\nu |\mathcal{C}(x)|^\nu)]\leq 2$ for all $x\in \Rd$ for some $\nu>0$ such that
\begin{align*}
\sup_{y\in B_\varepsilon(x)} |\nabla \omega_\varepsilon(y)| \leq \mathcal{C}(x) \varepsilon^{-1}
\end{align*}
holds for all $x\in \Rd$.
\end{itemize}

\subsection{Optimal-order homogenization error estimates}
Our first main result is an optimal-order estimate on the homogenization error in the stochastic homogenization of nonlinear uniformly elliptic PDEs (and systems) with monotone nonlinearity.
Note that our rate of convergence coincides with the optimal rate of convergence for linear elliptic PDEs and systems \cite{ArmstrongKuusiMourrat,GloriaNeukammOtto,GloriaNeukammOtto2,GloriaOtto,
GloriaOttoNew}, which form a subclass of the class of elliptic PDEs with monotone nonlinearity. In the theorems below, we present our homogenization errors in ``a posteriori form'', i.e.~with norms of $\nabla u_{\shom}$ exlicitly appearing on the right-hand side of the error estimates. These estimates can be combined with classical regularity results for uniformly elliptic monotone systems with constant coefficients, see~Remarks~\ref{R:T2},~\ref{R:T3},~\ref{R:T4}.

\begin{theorem}[Optimal-order estimates for the homogenization error for $d\geq 3$]
\label{TheoremErrorEstimate}
Let $d\geq 3$. Let the assumptions \hyperlink{A1}{(A1)}--\hyperlink{A3}{(A3)} and \hyperlink{P1}{(P1)}--\hyperlink{P2}{(P2)} be in place. Suppose furthermore that the small-scale regularity condition \hyperlink{R}{(R)} holds.
Let the effective (homogenized) monotone operator $A_\shom:\Rmd\rightarrow \Rmd$ be given by the defining formula \eqref{eq:homogop}.
Let $u_{\shom}\in H^2(\Rd;\Rm)\cap W^{1,\infty}(\Rd;\Rm)$ and let $u_\e\in H^1(\Rd;\Rm)$ be the unique weak solution to
\begin{align*}
  -\nabla \cdot (A(\omega_\varepsilon,\nabla u_\varepsilon))=  -\nabla \cdot (A_\shom(\nabla u_\shom))\qquad\text{in }\R^d
\end{align*}
in a distributional sense.  Then the estimate
\begin{align*}
  ||u_\varepsilon-u_\shom||_{L^{2d/(d-2)}(\Rd)}
  \leq\,
  \mathcal C\,\widehat C(\nabla u_{\shom})\,\e
\end{align*}
holds, where
\begin{equation*}
  \widehat C(\nabla u_{\shom})=(1+\sup_{x\in\Rd}|\nabla u_{\shom}|)^C\|\nabla u_{\shom}\|_{H^1(\Rd)}
\end{equation*}
and where $\mathcal C=\mathcal{C}(\omega_\varepsilon)$ is a random constant whose values may depend on $\omega_\varepsilon$ and $\nabla u_{\shom}$, but whose stretched exponential stochastic moments are uniformly estimated by
\begin{align*}
  \mathbb{E}\bigg[\exp\bigg(\frac{\mathcal{C}^{\bar\nu}}{C}\bigg)\bigg]\leq 2.
\end{align*}
Above, $\bar\nu,C>0$ depend only on $d$, $m$, $\lambda$, $\Lambda$, $\rho$, and on $\nu$ from assumption \hyperlink{R}{(R)}.
\end{theorem}
Using classical regularity theory to estimate $\widehat C(\nabla u_{\shom})$, the theorem implies the following homogenization error estimate formulated just in terms of the data.
\begin{remark}\label{R:T2}
  In the situation of Theorem~\ref{TheoremErrorEstimate}, additionally suppose that $A_\shom$ satisfies the regularity condition \hyperlink{R}{(R)}, and let $u_{\shom}\in H^1(\Rd;\Rm)$ be the unique weak solution to
  \begin{equation*}
    -\nabla\cdot (A_{\shom}(\nabla u_{\shom}))=\nabla\cdot g\qquad\text{in }\Rd,
  \end{equation*}
  where  $g\in H^1(\Rd;\Rmd)$ with $\nabla g\in L^p(\Rd;\Rmdd)$ for some $p>d$. With help of the regularity condition \hyperlink{R}{(R)} we can show that
  \begin{equation*}
    \widehat C(\nabla u_{\shom})\leq C\big(1+\|g\|_{L^2(\Rd)}+\|\nabla g\|_{L^p(\Rd)}\big)^C\|g\|_{H^1(\Rd)},
  \end{equation*}
  (see the end of Appendix~\ref{SectionRegularity} for details). Thus, the estimate of Theorem~\ref{TheoremErrorEstimate} can be upgraded to
  \begin{equation*}
    ||u_\varepsilon-u_\shom||_{L^{2d/(d-2)}(\Rd)}
    \leq\,\mathcal C\,C\,\big(1+\|g\|_{L^2(\Rd)}+\|\nabla g\|_{L^p(\Rd)}\big)^C\|g\|_{H^1(\Rd)}\,\e.
  \end{equation*}
  In particular, for all $1\leq \theta<\infty$ we obtain
  \begin{equation*}
    \mathbb E\big[||u_\varepsilon-u_\shom||_{L^{2d/(d-2)}(\Rd)}^\theta\big]^{\frac1\theta}
    \leq\, \theta^\frac{1}{\bar\nu} C\big(1+\|g\|_{L^2(\Rd)}+\|\nabla g\|_{L^p(\Rd)}\big)^C\|g\|_{H^1(\Rd)}\,\e.
  \end{equation*}
  Above, $\bar\nu,C>0$ only depend on $d,\lambda,\Lambda,\rho,p$ and $\nu$.  
\end{remark}

In the case of low dimension $d\leq 2$ the rate of convergence becomes limited by the central limit theorem scaling. In particular, as for linear elliptic equations, the case $d=2$ is critical, leading to a logarithmic correction. Furthermore, even for the Poisson equation -- which may be regarded as a very particular case of our PDEs \eqref{Equation} or \eqref{EffectiveEquation} -- the gradient of solutions of the whole-space problem may fail to be square-integrable. For this reason we include a lower order term in our PDEs.

\begin{theorem}[Optimal-order estimates for the homogenization error for $d=2$ and $d=1$]
  \label{TheoremErrorEstimate2d}
  Let $d=2$ or $d=1$ and let otherwise the assumptions of Theorem~\ref{TheoremErrorEstimate} be in place.
  Let $u_{\shom}\in H^2(\Rd;\Rm)\cap W^{1,\infty}(\Rd;\Rm)$ and let $u_\e\in H^1(\Rd;\Rm)$ be the unique weak solution to 
  \begin{align*}
    -\nabla \cdot (A(\omega_\varepsilon,\nabla u_\varepsilon)) +u_\e=  -\nabla \cdot (A_\shom(\nabla u_\shom))+u_\shom\qquad\text{in }\R^d
  \end{align*}
  in a distributional sense.  Then the estimate
  \begin{align*}
    ||u_\varepsilon-u_\shom||_{L^2(\Rd)}
    \leq
    \mathcal C\,\widehat C(\nabla u_{\shom})\,\left\{\begin{aligned}
        &\varepsilon^{1/2} &\text{for }d=1,
      \\
      &\varepsilon |\log \varepsilon|^{1/2} &\text{for }d=2
    \end{aligned}\right.
  \end{align*}
  holds, where  $\widehat C(\nabla u_{\shom})$ and $\mathcal C$ are defined as in Theorem~\ref{TheoremErrorEstimate}.
\end{theorem}
Again, we may make use of classical regularity results to obtain a homogenization error estimate in terms of only the data.
\begin{remark}\label{R:T3}
  In the situation of Theorem~\ref{TheoremErrorEstimate2d} suppose that $u_{\shom}\in H^1(\Rd;\Rm)$ is the unique weak solution to
  \begin{equation*}
   -\nabla\cdot (A_{\shom}(\nabla u_{\shom}))+u_{\shom}=\nabla\cdot g\qquad\text{in }\Rd,
  \end{equation*}
  where  $g\in H^1(\Rd;\Rmd)$ with $\nabla g\in L^p(\Rd;\Rmdd)$ for some $p>d$. Since $d\leq 2$, by appealing to Meyers' estimate we can show that
  \begin{equation*}
    \widehat C(\nabla u_{\shom})\leq C\big(1+\|g\|_{H^1(\Rd)}+\|\nabla g\|_{L^p(\Rd)}\big)^C\|g\|_{H^1(\Rd)},
  \end{equation*}
  (see the end of Appendix~\ref{SectionRegularity} for details). Thus, the estimate of Theorem~\ref{TheoremErrorEstimate2d} can be upgraded to
  \begin{align*}
    &||u_\varepsilon-u_\shom||_{L^{2}(\Rd)}\\
    &\leq\,\mathcal C\,C\,\big(1+\|g\|_{H^1(\Rd)}+\|\nabla g\|_{L^p(\Rd)}\big)^C\|g\|_{H^1(\Rd)}\,\left\{\begin{aligned}
        &\varepsilon^{1/2} &\text{for }d=1,
        \\
        &\varepsilon |\log \varepsilon|^{1/2} &\text{for }d=2.
      \end{aligned}\right.
  \end{align*}
  In particular, for all $1\leq \theta<\infty$ we obtain
  \begin{align*}
    &\mathbb E\big[||u_\varepsilon-u_\shom||_{L^{2}(\Rd)}^\theta\big]^{\frac1\theta}\\
    &\leq\, \theta^\frac{1}{\bar\nu} C\big(1+\|g\|_{H^1(\Rd)}+\|\nabla g\|_{L^p(\Rd)}\big)^C\|g\|_{H^1(\Rd)}\,\left\{\begin{aligned}
        &\varepsilon^{1/2} &\text{for }d=1,
      \\
      &\varepsilon |\log \varepsilon|^{1/2} &\text{for }d=2.
    \end{aligned}\right.
  \end{align*}
  Above, $\bar\nu,C>0$ only depend on $d,\lambda,\Lambda,\rho,p$ and $\nu$.  
\end{remark}

\begin{remark}
    Instead of the right-hand side  $\nabla\cdot g$, we may consider in Theorem~\ref{TheoremErrorEstimate} and Theorem~\ref{TheoremErrorEstimate2d} a right-hand side $f\in L^1(\Rd;\Rm)\cap L^{d+1}(\Rd;\Rm)$ with $q>d$ by solving for $-\Delta v=f$ and setting $g:=-\nabla v$. In this way we recover the homogenization error estimate in the form of \eqref{ErrorEstimate} as claimed in the introduction.
\end{remark}

In the absence of the small-scale regularity condition \hyperlink{R}{(R)}, we still obtain half of the optimal rate of convergence. However, we also directly obtain this result in bounded domains. Note that in order to recover the optimal rates of convergence from Theorem~\ref{TheoremErrorEstimate} and Theorem~\ref{TheoremErrorEstimate2d} also for the Dirichlet problem on bounded domains, one would need to construct boundary correctors, as done for linear elliptic PDEs for instance in \cite{AvellanedaLin} in the setting of periodic homogenization or in \cite{ArmstrongKuusiMourratBook,FischerRaithel} in the setting of stochastic homogenization.
\begin{theorem}[Estimates for the homogenization error on bounded domains]
  \label{TheoremErrorEstimateDomains}
  Let $d\geq 1$, let $\domain\subset \Rd$ be a bounded $C^1$-domain or a convex Lipschitz domain, and let the assumptions \hyperlink{A1}{(A1)}--\hyperlink{A3}{(A3)} and \hyperlink{P1}{(P1)}--\hyperlink{P2}{(P2)} be in place. Define $A_\shom$ as in formula \eqref{eq:homogop}. Let $u_{\shom}\in H^2(\domain;\Rm)$ and let $u_\e\in u_{\shom}+H^1_0(\domain;\Rm)$ be the unique weak solution to
\begin{align*}
  -\nabla \cdot (A(\omega_\varepsilon,\nabla u_\varepsilon)) =  -\nabla \cdot (A_\shom(\nabla u_\shom))\qquad\text{in }\domain
\end{align*}
in a distributional sense. Then 
\begin{align*}
&||u_\varepsilon-u_\shom||_{L^2(\domain)}\leq\mathcal C\,C(\domain)\,
\,\|\nabla u_{\shom}\|_{H^1(\domain)}\,                 \left\{\begin{aligned}
                     &\e^{\frac13}&\text{for }d=1,\\
                     &\e^\frac12|\log\e|^{\frac14}&\text{for }d=2,\\
                     &\e^\frac12&\text{for }d\geq 3.
                 \end{aligned}\right.
\end{align*}
Here, $\mathcal C=\mathcal{C}(\omega_\varepsilon)$ denotes a random constant whose values may depend on $\omega_\varepsilon$, $\domain$, and $\nabla u_{\shom}$, but whose (stretched exponential) stochastic moments are uniformly estimated by
\begin{align*}
  \mathbb{E}\bigg[\exp\bigg(\frac{\mathcal{C}^{\bar\nu}}{C}\bigg)\bigg]\leq 2.
\end{align*}
Above, $\bar\nu,C>0$ depend only on $d$, $m$, $\lambda$, $\Lambda$, and $\rho$, and $C(\domain)$ additionally depends on $\domain$.
\end{theorem}
\begin{remark}\label{R:T4}
  In the situation of Theorem~\ref{TheoremErrorEstimateDomains} additionally suppose that $\domain$ is of class $C^{1,1}$. Let $u_{\rm Dir}\in H^{2}(\domain;\Rm)$, $f\in L^2(\domain;\Rm)$, $g\in H^1(\domain;\Rmd)$, and let $u_{\hom}\in u_{\rm Dir}+ H^1_0(\domain;\Rm)$ be the unique weak solution to the boundary value problem
  \begin{equation*}
    -\nabla\cdot (A_{\shom}(\nabla u_{\shom}))=f+\nabla\cdot g\qquad\text{in }\domain
  \end{equation*}
  in a distributional sense. Then $u_{\shom}\in H^2(\domain;\R^m)$ holds with the estimate
  \begin{equation*}
    \|u_{\shom}\|_{H^2(\domain)}\leq C\big(\|f\|_{L^2(\domain)}+\|g\|_{H^1(\domain)}+\|u_{\rm Dir}\|_{H^2(\domain)}\big),
  \end{equation*}
  where $C$ only depends on $d,m,\lambda,\Lambda$ and $\domain$, see \cite[Theorem 5.2]{NeffKnees08}. We thus can upgrade the estimate of Theorem~\ref{TheoremErrorEstimateDomains} to
  \begin{align*}
    &||u_\varepsilon-u_\shom||_{L^2(\domain)}
    \\&~~~~~~
    \leq\mathcal C\,C\,\big(\|f\|_{L^2(\domain)}+\|g\|_{H^1(\domain)}+\|u_{\rm Dir}\|_{H^2(\domain)}\big)
      \left\{\begin{aligned}
          &\e^{\frac13}&\text{for }d=1,\\
          &\e^\frac12|\log\e|^{\frac14}&\text{for }d=2,\\
          &\e^\frac12&\text{for }d\geq 3.
        \end{aligned}\right.
  \end{align*}
\end{remark}
\begin{remark}
  An error estimate analogous to Theorem~\ref{TheoremErrorEstimateDomains} also holds on the full space $\mathbb{R}^d$, provided that in case $d\geq 3$ we measure the error $u_\varepsilon-u_\shom$ in the $L^{2d/(d-2)}$ norm as in Theorem~\ref{TheoremErrorEstimate} and assume additionally $f\in L^1\cap L^{\frac{2d}{d+2}}$, and provided that in case $d=1,2$ we include a massive term in the equation as in Theorem~\ref{TheoremErrorEstimate2d}.
\end{remark}
\begin{remark}[Error estimate for a piecewise affine two-scale expansion]\label{R:two-scale}
  In the proof of these theorems we also establish an $H^1$-error estimate for a two-scale expansion of $\nabla u_{\shom}$. More precisely, we consider a two-scale expansion of the form
  \begin{align}
    \hat u_\e=u_\shom+\sum_{k\in K}\eta_k\phi_k,
    \label{TwoScaleDef}
  \end{align}
  where $\{\eta_k\}_{k\in K}$ is a smooth partition of unity that only depends on the domain and a suitably chosen discretization scale $\delta$ (with $\e\leq\delta\ll 1$), and $\phi_k$ is the corrector associated with the local approximation $\xi_k=\frac{1}{\int_{\domain}\eta_k\,dx}\int_{\domain}\nabla u_{\shom}\eta_k\,dx$, see Proposition~\ref{PropositionTwoScaleExpansion} below. (We remark that there is no $\e$ in front of the $\phi_k$ in \eqref{TwoScaleDef}, since we define the correctors with respect to the original coefficients that oscillate on scale $\e$ and not --- as it is often done --- for rescaled coefficients.) In the situation of Theorem~\ref{TheoremErrorEstimate} and Theorem~\ref{TheoremErrorEstimate2d} we show that 
  \begin{equation*}
    \|\nabla u_\e-\nabla\hat u_\e\|_{L^2(\R^d)}\leq \mathcal C\widehat C(\nabla u_{\shom})
    \left\{\begin{aligned}
        &\e^{\frac12}&\text{for }d=1,\\
        &\e|\log\e|^{\frac12}&\text{for }d=2,\\
        &\e&\text{for }d\geq 3,
      \end{aligned}\right.
  \end{equation*}
  see Step~2 in the proof of  Theorem~\ref{TheoremErrorEstimate} and Theorem~\ref{TheoremErrorEstimate2d}. For bounded domains, we also obtain an estimate for the two-scale expansion with suboptimal scaling and away from the boundary.
\end{remark}

We next establish several structural properties of the homogenized operator $A_\shom$, including in particular the statement that the homogenized operator $A_\shom$ inherits the monotone structure from $A_\varepsilon$.
Note that these properties are actually true even in the context of qualitative stochastic homogenization, and their proof is fairly elementary.
The monotonicity of the effective operator $A_\shom$ has (essentially) already been established by Dal~Maso and Modica \cite{DalMasoModica,DalMasoModica2}, at least in the setting of convex integral functionals. We also give sufficient criteria for frame-indifference and isotropy of the homogenized operator $A_\shom$.
\begin{theorem}[Structure properties of the homogenized equation]
\label{TheoremStructureProperties}
Let the assumptions \hyperlink{A1}{(A1)}--\hyperlink{A3}{(A3)} and \hyperlink{P1}{(P1)}--\hyperlink{P2}{(P2)} be in place.
Define the effective (homogenized) material law $A_\shom:\Rmd\rightarrow \Rmd$ by \eqref{eq:homogop}. Then $A_\shom$ has the following properties:
\begin{itemize}
\item[a)] The map $A_\shom$ is monotone and Lipschitz continuous in the sense that
\begin{align*}
\big(A_\shom(\xi_2)-A_\shom(\xi_1)\big) \cdot (\xi_2-\xi_1)
&\geq \lambda |\xi_2-\xi_1|^2
\end{align*}
and
\begin{align*}
|A_\shom(\xi_2)-A_\shom(\xi_1)|\leq C(d,m)\frac{\Lambda^2}{\lambda} |\xi_2-\xi_1|
\end{align*}
hold for all $\xi_1,\xi_2\in \Rmd$. Furthermore, if in addition the condition \hyperlink{R}{(R)} holds for $A$, $A_\shom$ is continuously differentiable and satisfies $A_\shom(0)=0$.
\item[b)] If the operator $A$ is frame-indifferent in the sense $A(\omega,O\xi)=O A(\omega,\xi)$ for all $O\in \operatorname{SO}(m)$, all $\xi\in \Rmd$, and all $\omega\in \Hilbert$, the operator $A_\shom$ inherits the frame-indifference in the sense $A_\shom(O\xi)=O A_\shom(\xi)$ for all $O\in \operatorname{SO}(m)$ and all $\xi\in \Rmd$.
\item[c)] If the law of the operator $A_\varepsilon(x,\xi):=A(\omega_\varepsilon(x),\xi)$ is isotropic in the sense that the law of the rotated operator $A_\varepsilon^O(x,\xi):=A_\varepsilon(x,\xi O)O$ coincides with the law of $A_\varepsilon$ for all $O\in \operatorname{SO}(d)$, the operator $A_\shom$ inherits the isotropy in the sense $A_\shom(\xi O)=A_\shom(\xi)O$ for all $O\in \operatorname{SO}(d)$ and all $\xi\in \Rmd$.
\end{itemize}
\end{theorem}

\subsection{Optimal-order error estimates for the approximation of the homogenized operator by periodic RVEs}

To perform numerical simulations based on the homogenized PDE \eqref{EffectiveEquation}, the effective material law $A_\shom$ must be determined. The theoretical expression \eqref{eq:homogop} for the effective material law $A_\shom$ is essentially an average over all possible realizations of the random material; it is therefore a quantity that is not directly computable. 
To numerically determine the effective (homogenized) material law $A_\shom$ in practice, the method of representative volumes is typically employed: A finite sample of the random medium is chosen, say, a cube $[0,L]^d$ of side length $L\gg \varepsilon$, and the cell formula from homogenization theory is evaluated on this sample.

In most cases, a representative volume of mesoscopic size $\varepsilon\ll L\ll 1$ is sufficient to approximate the effective material law $A_\shom$ well. For this reason, at least in the case of a clear separation of scales $\varepsilon\ll 1$ numerical simulations based on the homogenized equation \eqref{EffectiveEquation} and the RVE method are several to many orders of magnitude faster than numerical simulations based on the original PDE \eqref{Equation}.

In our next theorem, we establish a~priori error estimates for the approximation of the effective material law $A_\shom$ by the method of representative volumes. Our a~priori error estimates are of optimal order in the physically relevant cases $d\leq 4$, at least if the issue of boundary layers on the RVE is addressed appropriately.
More precisely, we show that the homogenized coefficients $A_\shom$ may be approximated by the representative volume element method on an RVE of size $L$ up to an error of the order of the natural fluctuations $\smash{(\frac{L}{\e})^{-d/2}}$. This rate of convergence coincides with the optimal rate in the case of linear elliptic PDEs and systems, see \cite{GloriaOtto,GloriaOtto2}.

There exist various options to define the RVE approximation, see e.\,g.\ \cite[Section~1.4]{FischerVarianceReduction} for a discussion in the linear elliptic setting. In the present work, we shall consider the case of periodic RVEs: One considers an \emph{$L$-periodic approximation} $\mathbb P_{L}$ of the probability distribution $\mathbb{P}$ of the random field $\omega_\varepsilon$, that is an $L$-periodic variant $\omega_{\varepsilon,L}$ of the random field $\omega_\varepsilon$ (see below for a precise definition), and imposes periodic boundary conditions on the RVE boundary $\partial [0,L]^d$.

Before rigorously defining the notion of periodization, we first note that to any periodic parameter field $\widetilde \omega_L$ one can unambiguously (and deterministically) associate a homogenized coefficient via the classical periodic homogenization formula.
\begin{definition}[Periodic RVE approximation]
\label{DefRVEApproximation}
Let $A:\Hilbert\times \Rmd\rightarrow \Rmd$ be a family of monotone operators satisfying \hyperlink{A1}{(A1)}--\hyperlink{A3}{(A3)}. To any $L$-periodic parameter field $\widetilde\omega_L$ and any $\xi\in\Rmd$ we associate the RVE approximation $A^{\RVE,L}(\widetilde \omega_L,\xi)$ for the effective coefficient $A_\shom(\xi)$ given by
\begin{equation*}
  A^{\RVE,L}(\widetilde\omega_L,\xi):=\fint_{[0,L]^d} A(\widetilde\omega_L,\xi+\nabla\phi_\xi)\,dx,
\end{equation*}
where the (periodic) corrector $\phi_\xi=\phi_\xi(\widetilde\omega_L,\cdot)$ is defined as the unique solution in $H^1_{\operatorname{per}}([0,L]^d;\R^m)$ with vanishing mean $\fint_{[0,L]^d}\phi_\xi \,dx=0$ to the corrector equation
\begin{equation*}
  -\nabla\cdot (A(\widetilde\omega_L(x),\xi+\nabla\phi_\xi))=0\qquad\text{in }\R^d.
\end{equation*}
\end{definition}
We next define our notion of $L$-periodic approximation of the coefficient field $\omega_\varepsilon$. The main condition will be that the statistics of $\omega_\varepsilon$ and $\omega_{\varepsilon,L}$ must coincide on balls of the form $B_{\frac L 4}(x_0)$ around any $x_0\in \Rd$.
\begin{definition}[$L$-periodic approximation of $\mathbb P$]\label{D:Lperproxy}
Let $\mathbb{P}$ satisfy the assumptions \hyperlink{P1}{(P1)} and \hyperlink{P2}{(P2)}. Let $L\geq \varepsilon$. We say that $\mathbb P_{L}$ is an $L$-periodic approximation to $\mathbb P$, if $\mathbb P_{L}$ is stationary in the sense of \hyperlink{P1}{(P1)}, concentrates on  $L$-periodic parameter fields, satisfies the periodic spectral gap of Definition~\ref{DefinitionSpectralGap}b, and the following property holds: For $\mathbb P$-a.e. random field $\omega_\varepsilon$ and $\mathbb P_{L}$-a.e. random field $\omega_{\varepsilon,L}$ we have
    \begin{equation*}
      \omega_\varepsilon\sim\mathbb P\text{ and }\omega_{\varepsilon,L}\sim\mathbb P_{L}\qquad\Rightarrow\qquad \omega_\varepsilon\vert_{B_{\frac L4}}\sim \omega_{\varepsilon,L}\vert_{B_{\frac L4}},
    \end{equation*}
    i.e., if $\omega_\varepsilon$ is distributed with $\mathbb P$, and $\omega_{\varepsilon,L}$ is distributed with $\mathbb P_{L}$, then the restricted random fields $\omega_\varepsilon\vert_{B_{\frac L4}}$ and $\omega_{\varepsilon,L}\vert_{B_{\frac L4}}$ have the same distribution.
    
For such an $L$-periodic approximation $\mathbb{P}_L$, we abbreviate the representative volume approximation as introduced in Definition~\ref{DefRVEApproximation} by $A^{\RVE,L}(\xi):=A^{\RVE,L}(\omega_{\varepsilon,L},\xi)$.
\end{definition}

Note, however, that the existence of such an $L$-periodic approximation $\smash{\mathbb{P}_L}$ of $\mathbb{P}$ has to be proven on a case-by-case basis, depending on the probability distribution. To give one example, for a random field $\omega_\varepsilon$ of the form $\omega_\varepsilon(x):=\theta(Y_\varepsilon(x))$, where $Y_\varepsilon$ is a stationary Gaussian random field arising as the convolution $Y_\varepsilon(x):=(\beta_\varepsilon \ast W)(x)$ of white noise $W$ with a kernel $\beta_\varepsilon$ with $\supp \beta_\varepsilon\subset B_\varepsilon$, one may construct an $L$-periodic approximation simply by replacing the white noise $W$ by $L$-periodic white noise $W_L$, i.\,e.\ by defining $\omega_{\varepsilon,L}(x):=\theta((\beta_\varepsilon\ast W_L)(x))$.
{We remark that for more complex probability distributions it may become necessary to generalize the notion of periodization of Definition~\ref{D:Lperproxy}, e.\,g.\ by relaxing the condition of identical distributions $\smash{\omega_\varepsilon\vert_{B_{L/4}}\sim \omega_{\varepsilon,L}\vert_{B_{L/4}}}$ to the two distributions just being sufficiently similar (in a suitable metric). In such a case, our proofs for Theorem~\ref{T:RVE} below would need to be adapted in a suitable way.}

For the approximation of the effective material law $A_\shom$ by an $L$-periodic representative volume in the sense of Definition~\ref{DefRVEApproximation}--\ref{D:Lperproxy}, we establish the following a~priori error estimate. 
Again, our rates of convergence coincide with the optimal rates of convergence for linear elliptic PDEs and systems \cite{GloriaNeukammOtto,GloriaNeukammOtto2,GloriaOtto}.
\begin{theorem}[Error estimate for periodic RVEs]\label{T:RVE}
  Let $A:\Hilbert\times\Rmd\rightarrow \Rmd$ satisfy the assumptions \hyperlink{A1}{(A1)} -- \hyperlink{A3}{(A3)} and let $\mathbb{P}$ satisfy the assumptions \hyperlink{P1}{(P1)}--\hyperlink{P2}{(P2)}. Let $L\geq 2\varepsilon$, let $\mathbb{P}_L$ be an $L$-periodic approximation of $\mathbb{P}$ in the sense of Definition~\ref{D:Lperproxy}, and denote by $A^{\RVE,L}(\xi)$ the corresponding representative volume approximation for the homogenized material law $A_\shom(\xi)$.
  \begin{enumerate}[(a)]
  \item (Estimate on random fluctuations). For all $\xi\in\Rmd$ we have the estimate on random fluctuations
    \begin{equation*}
      \big|A^{\RVE, L}(\xi)-\mathbb{E}_L\left[A^{\RVE,L}(\xi)\right]\big|\leq\mathcal C|\xi|\left(\frac L\varepsilon\right)^{-\frac d2},
    \end{equation*}
    where $\mathcal C$ denotes a random variable that satisfies a stretched exponential moment bound uniformly in $L$, i.\,e.\ there exists $C=C(d,m,\lambda,\Lambda,\rho)$ such that
    \begin{equation*}
      \mathbb E_L\left[\exp\bigg(\frac{\mathcal C^{1/C}}{C}\bigg)\right]\leq 2.
    \end{equation*}
  \item (Estimate for the systematic error). Suppose that $\mathbb{P}$ and $\mathbb{P}_L$ additionally satisfy the small-scale regularity condition \hyperlink{R}{(R)}. Then for any $\xi\in \Rmd$ the systematic error of the representative volume method is of higher order in the sense
    \begin{equation*}
      \big|\mathbb{E}_L\big[A^{\RVE,L}(\xi)\big]-\Ahom(\xi)\big|
      \leq C(1+|\xi|)^C|\xi|\left(\frac{L}{\e}\right)^{-(d\wedge 4)}\bigg|\log\left(\frac{L}{\e}\right)\bigg|^{\alpha_d}.
    \end{equation*}
    Here, $\alpha_d$ is given by
    \begin{equation*}
      \alpha_d:=d\wedge 4+
      \begin{cases}
        2 &\text{for }d\in\{2,4\},\\
        0 &\text{for }d=3,\ d=1,\text{ and }d\geq 5
      \end{cases}
    \end{equation*}
    for some $C=C(d,m,\lambda,\Lambda,\nu,\rho)$.
    In dimension $d\leq 4$ the estimate is optimal (up to the logarithmic factor). Without the small-scale regularity condition \hyperlink{R}{(R)}, the suboptimal estimate
    \begin{equation*}
      \big|\mathbb{E}_L\big[A^{\RVE,L}(\xi)\big]-\Ahom(\xi)\big| \leq C|\xi|\left(\frac{L}{\e}\right)^{-\frac{d\wedge 4}{2}}
    \bigg|\log\left(\frac L\e\right)\bigg|^{\alpha_d}
    \end{equation*}
    holds, where
    \begin{equation*}
      \alpha_d:=\frac{d\wedge 4}{2}+
      \begin{cases}
        \frac{1}{2} &\text{for }d\in\{2,4\},\\
        0 &\text{for }d=3,\ d=1,\text{ and }d\geq 5.
      \end{cases}
    \end{equation*}
  \end{enumerate}
\end{theorem}
In particular, we derive the following overall a~priori error estimate for the method of representative volumes.
\begin{corollary}[Total $L^2$-error for periodic RVEs]
\label{CorollaryTotalRVE}
Let assumptions \hyperlink{A1}{(A1)} -- \hyperlink{A3}{(A3)} and \hyperlink{P1}{(P1)} -- \hyperlink{P2}{(P2)} be in place. Let $L\geq\varepsilon$. Let $\mathbb P_{L}$ be a $L$-periodic approximation to $\mathbb P$ in the sense of Definition~\ref{D:Lperproxy}.
  \begin{enumerate}[(a)]
  \item If the small-scale regularity condition \hyperlink{R}{(R)} is satisfied, then for $2\leq d\leq 7$ we obtain the optimal estimate
    \begin{equation*}
      \mathbb E_L\left[|A^{\RVE,L}(\xi)-\Ahom(\xi)|^2\right]^\frac12\leq C(1+|\xi|)^C|\xi|\left(\frac L\varepsilon\right)^{-\frac d2}.
    \end{equation*}
    Moreover, for $d>7$ we obtain the suboptimal estimate
    \begin{equation*}
      \mathbb E_L\left[|A^{\RVE,L}(\xi)-\Ahom(\xi)|^2\right]^\frac12\leq C(1+|\xi|)^C|\xi|\left(\frac L\varepsilon\right)^{-4}\bigg|\log\left(\frac{L}{\e}\right)\bigg|^{C(d)}.
    \end{equation*}
  \item If \hyperlink{R}{(R)} is not satisfied, then for any $d\geq 2$ we obtain the subotimal estimate
    \begin{equation*}
      \mathbb E_L\left[|A^{\RVE,L}(\xi)-\Ahom(\xi)|^2\right]^\frac12\leq C|\xi|\left(\frac L\varepsilon\right)^{-\frac {d\wedge 4}{2}}\bigg|\log\left(\frac{L}{\e}\right)\bigg|^{C(d)}.
    \end{equation*}
    For $d\leq 4$ this estimate is optimal except for the logarithmic factor.
  \end{enumerate}
  In the above estimates, we have $C=C(d,m,\lambda,\Lambda,\rho)$.
\end{corollary}

{Let us briefly comment on the prefactor $(1+|\xi|)^C$ in the error estimates of Corollary~\ref{CorollaryTotalRVE}. It originates from the corresponding prefactor in the estimates on the linearized correctors in Proposition~\ref{PropositionLinearizedCorrectorEstimate}. On a mathematical level, this prefactor reflects the fact that a large macroscopic slope $\xi$ may increase the sensitivity of the correctors $\phi_\xi$ and $\sigma_\xi$ with respect to perturbations in the coefficient field (at least given just our assumptions \hyperlink{A1}{(A1)} -- \hyperlink{A3}{(A3)} and \hyperlink{P1}{(P1)} -- \hyperlink{P2}{(P2)}), as an inspection of \eqref{FunctionalSensitivityEquation} and \eqref{EquationDual2Linearized} in the proof of Lemma~\ref{LemmaEstimateLinearFunctionalsLinearized} reveals. It may be possible to mitigate this prefactor by imposing a sharper (but still natural) condition like $|\partial_\omega \partial_\xi A(\omega,\xi)|\leq \frac{\Lambda}{1+|\xi|}$ (as opposed to just $|\partial_\omega \partial_\xi A|\leq \Lambda$ in \hyperlink{A3}{(A3)}); however, this would require to rigorously establish a nontrivial nonlocal cancellation in \eqref{FunctionalSensitivityEquation} and \eqref{EquationDual2Linearized}. For this reason, it is beyond the scope of our paper.
}
\subsection{The decorrelation assumption: spectral gap inequality}
We finally state and discuss the quantitative assumption on the decorrelation of the random field $\omega_\varepsilon$. The majority of results in the present work require the correlations of the random parameter field $\omega_\varepsilon$ to decay on scales larger than $\varepsilon$ in a quantified way. This quantified decay is enforced by assuming that the probability distribution $\mathbb P$ satisfies the following spectral gap inequality. As already mentioned previously, this spectral gap inequality holds for example for random fields of the form $\omega_\varepsilon(x)=\beta(\widetilde\omega_\varepsilon(x))$, where $\beta:\mathbb{R}\rightarrow (-1,1)$ is a $1$-Lipschitz function and where $\widetilde\omega_\varepsilon$ is a stationary Gaussian random field subject to a covariance estimate of the form
\begin{align*}
\big|\Cov\big[\widetilde\omega_\varepsilon(x),\widetilde\omega_\varepsilon(y)\big]\big|\leq C~ \bigg(\frac{\varepsilon}{\varepsilon+|x-y|}\bigg)^{d+\kappa}
\end{align*}
for all $x,y\in \Rd$ for some $\kappa>0$ and some $C<\infty$. The derivation is standard and may be found e.\,g.\ in \cite{DuerinckxGloriaFunctionalInequality}.
\begin{definition}[Spectral gap inequality encoding fast decorrelation on scales $\geq \varepsilon$]
~\\
\vspace{-9mm}
\label{DefinitionSpectralGap}
\begin{enumerate}
\item[(a)] We say that the probability distribution $\mathbb P$ of random fields $\omega_\varepsilon$ satisfies a spectral gap inequality with correlation length $\e$ and constant $\rho>0$ if any random variable $F=F(\omega_\varepsilon)$ satisfies the estimate
\begin{align}
  \label{SpectralGapInequality}
  \mathbb{E}\Big[\big|F-\mathbb{E}\big[F\big]\big|^2\Big]
  \leq \frac{\varepsilon^d}{\rho^2}
  &\mathbb{E}
    \Bigg[\int_\Rd \bigg(\fint_{B_\varepsilon(x)} \bigg|\frac{\partial F}{\partial \omega_\varepsilon}\bigg| \,d\tilde x \bigg)^2 \,dx\Bigg].
\end{align}
Here, $\fint_{B_\varepsilon(x)} \big|\frac{\partial F}{\partial \omega_\varepsilon}(\omega_\varepsilon)\big|\,d\tilde x$ stands short for
\begin{equation*}
  \sup_{\delta\omega_\varepsilon}\limsup\limits_{t\to 0}\frac{|F(\omega_\varepsilon+t\delta\omega_\varepsilon)-F(\omega_\varepsilon)|}{t},
\end{equation*}
where the $\sup$ runs over all random fields $\delta\omega_\varepsilon:\Rd\to H$ supported in $B_\e(x)$ with $\|\delta\omega_\varepsilon\|_{L^\infty(\Rd)}\leq 1$.
\item[(b)] Let $L\geq \varepsilon$ and let $\mathbb P_L$ be a probability distribution of $L$-periodic random fields $\omega_{\varepsilon,L}$. We say that $\mathbb P_L$ satisfies a periodic spectral gap inequality with correlation length $\e$ and constant $\rho>0$ if any random variable $F=F(\omega_{\varepsilon,L})$ satisfies the estimate
  \begin{align}
    \label{SpectralGapInequality-periodic}
    \mathbb{E}_L\Big[\big|F-\mathbb{E}_L\big[F\big]\big|^2\Big]
    \leq \frac{\varepsilon^d}{\rho^2}
    &\mathbb{E}_L
    \Bigg[\int_{[0,L)^d} \bigg(\fint_{B_\varepsilon(x)} \bigg|\frac{\partial F}{\partial \omega_{\varepsilon,L}}\bigg|\,d\tilde x \bigg)^2 \,dx\Bigg].
    \end{align}
  Here, $\fint_{B_\varepsilon(x)} \big|\frac{\partial F}{\partial \omega_{\varepsilon,L}}(\omega_{\varepsilon,L})\big|\,d\tilde x$ stands short for
\begin{equation*}
  \sup_{\delta\omega_{\varepsilon,L}}\limsup\limits_{t\to 0}\frac{|F(\omega_{\varepsilon,L}+t\delta\omega_{\varepsilon,L})-F(\omega_{\varepsilon,L})|}{t},
\end{equation*}
where the $\sup$ runs over all $L$-periodic random fields $\delta\omega_{\varepsilon,L}:\Rd\to H$ with $\|\delta\omega_{\varepsilon,L}\|_{L^\infty([0,L)^d)}\leq 1$ and support in $B_\e(x)+L\mathbb Z^d$.
\end{enumerate}
\end{definition}
For a reader not familiar with the concept of spectral gap inequalities, it may be instructive to inserting the spatial average $\smash{F(\omega_\varepsilon):=\fint_{B_r} \omega_\varepsilon \,dx}$ into the definition \eqref{SpectralGapInequality}. For this particular choice, the Fr\'echet derivative is given by $\smash{\frac{\partial F}{\partial\omega_\varepsilon}=\frac{1}{|B_r|} \chi_{B_r}}$ and the expression on the right-hand side in \eqref{SpectralGapInequality} is of the order $(\frac{\varepsilon}{r})^d$. Thus, the fluctuations of the spatial average $\smash{F(\omega_\varepsilon)=\fint_{B_r} \omega_\varepsilon(x) \,dx}$ are (at most) of the order $\smash{(\frac{\varepsilon}{r})^{d/2}}$, just as expected when averaging $\smash{(\frac{r}{\varepsilon})^d}$ independent random variables of similar variance. However, the importance of spectral gap inequalities in stochastic homogenization \cite{GloriaNeukammOttoInventiones,NaddafSpencer,GloriaNeukammOtto} stems from the fact that they also entail fluctuation bounds for appropriate \emph{nonlinear} functionals of the random field $\omega_\varepsilon$, as the homogenization corrector $\phi_\xi$ is a nonlinear function of the random field $\omega_\varepsilon$ (even in the setting of linear elliptic PDEs). To give a simple example for fluctuation bounds for a nonlinear functional of $\omega_\varepsilon$, the reader will easily verify that the spectral gap inequality \eqref{SpectralGapInequality} also implies a fluctuation bound of the order $\smash{(\frac{\varepsilon}{r})^{d/2}}$ for a quantity like $\smash{f\big(\fint_{B_r} g(\omega_\varepsilon(x)) \,dx\big)}$ with two $1$-Lipschitz functions $f,g:\mathbb{R}\rightarrow \mathbb{R}$.

{We remark that in the present paper, we only consider the spectral gap inequality of Definition~\ref{DefinitionSpectralGap}; it is satisfied for many random fields taking continuum values. For random fields taking only a discrete set of values, a different notion of spectral gap inequality phrased in terms of finite differences (also called $\operatorname{osc}$ spectral gap) is available; see for instance \cite{DuerinckxGloriaFunctionalInequality}. Proving corrector estimates given an $\operatorname{osc}$ spectral gap inequality would in particular require a nontrivial additionabl decay result for the difference between the homogenization corrector and the corrector for a locally perturbed operator.

For strongly correlated random fields (with non-integrable tails), the spectral gap inequality of Definition~\ref{DefinitionSpectralGap} does not hold; instead, an appropriate multiscale spectral gap inequality has been derived \cite{DuerinckxGloriaFunctionalInequality}. In this setting, our strategy for obtaining corrector estimates does not apply directly, but would require an additional regularity ingredient analogous to \cite{GloriaNeukammOtto2}.
}

\medskip

\section{Strategy and intermediate results}
\label{SectionStrategy}
\subsection{Key objects: Localized correctors, the two-scale expansion, and linearized correctors}

Before turning to the description of our strategy, we introduce the central objects for our approach. Besides the homogenization corrector $\phi_\xi$ defined in Definition~\ref{D:corr}, these key objects include the flux corrector $\sigma_\xi$, localized versions $\phi_\xi^T$ and $\sigma_\xi^T$ of the corrector and the flux corrector, as well as the homogenization correctors $\phi_{\xi,\Xi}^T$ and flux correctors $\sigma_{\xi,\Xi}^T$ for an associated linearized PDE. At the end of the section we also state optimal stochastic moment bounds for these correctors, which will play a central role in the derivation of our main results.

The \emph{flux corrector} $\sigma_\xi$ is an important quantity in the quantitative homogenization theory of elliptic PDEs, see \cite{GloriaNeukammOtto} for a first reference in the context of quantitative stochastic homogenization. Recall that the corrector $\phi_\xi$ provides the connection between a given macroscopic constant field gradient $\xi$ and the corresponding microscopic field gradient $\xi+\nabla \phi_\xi$. Similarly, the flux corrector $\sigma_\xi:\Rd\rightarrow \Rm \otimes \Rddskew$ provides a ``vector potential'' for the difference between the ``microscopic'' and the ``macroscopic'' flux in the sense
\begin{align}
\label{EquationFluxCorrector}
\nabla \cdot \sigma_{\xi} = A(\omega_\varepsilon,\xi+\nabla \phi_\xi) - A_\shom(\xi).
\end{align}
The central importance of the flux corrector $\sigma_\xi$ is that it allows for an elementary representation of the error of the two-scale expansion, see below.
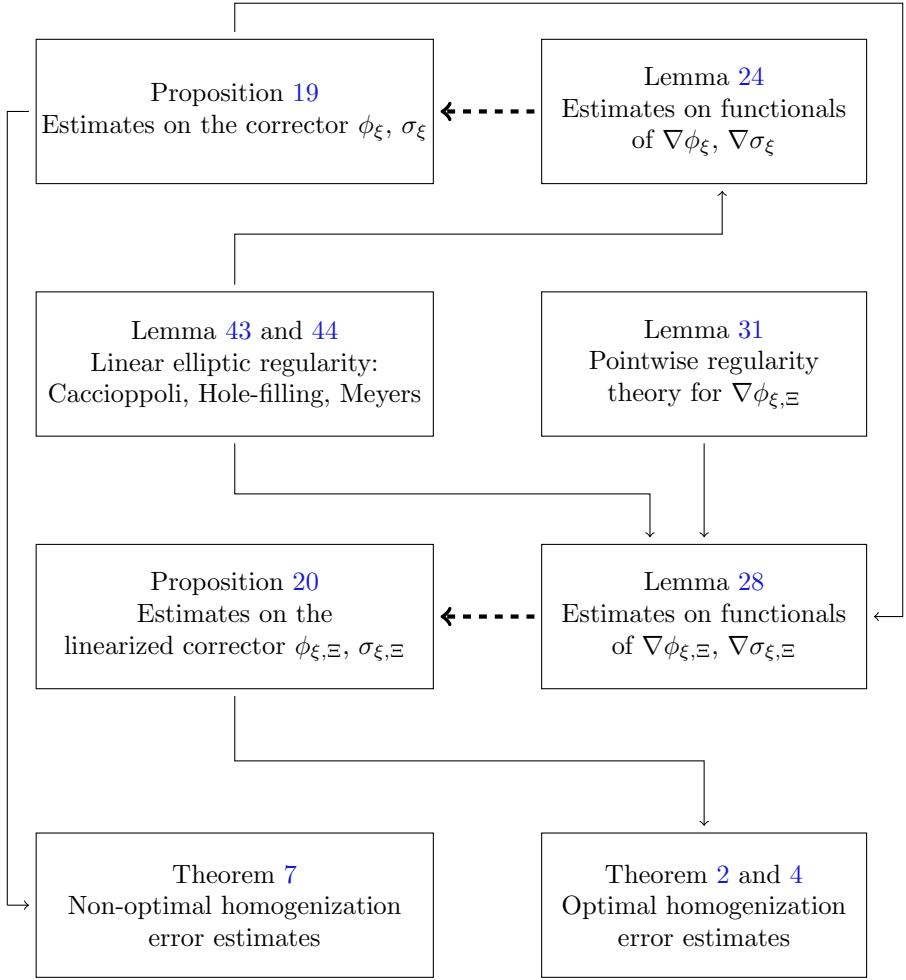
\begin{figure}
\begin{tikzpicture}[every text node part/.style={align=center},scale=0.96]
\draw (7.5,0) rectangle node{Lemma~\ref{LemmaEstimateLinearFunctionals} \\ Estimates on functionals \\ of $\nabla\phi_\xi$, $\nabla\sigma_\xi$} (12,2);
\draw (0.5,0) rectangle node{Proposition~\ref{PropositionCorrectorEstimate} \\ Estimates on the corrector $\phi_\xi$, $\sigma_\xi$} (6,2);
\draw (7.5,-3.5) rectangle node{Lemma~\ref{PointwiseRegularityLinearized} \\ Pointwise regularity\\ theory for $\nabla\phi_{\xi,\Xi}$} (12,-1.5);
\draw[->] (9.75,-3.6) -- (9.75,-4.9);
\draw[->] (3.25,-3.6) -- (3.25,-4.3) -- (9.0,-4.3) -- (9.0,-4.9);
\draw[->] (3.25,-1.4) -- (3.25,-0.7) -- (10.0,-0.7) -- (10.0,-0.1);
%
\draw[->] (3.25,2.1) -- (3.25,2.5) -- (12.5,2.5) -- (12.5,-6) -- (12.1,-6);
\draw[ultra thick,dashed,->] (7.4,1) -- (6.1,1);
\draw[ultra thick,dashed,->] (7.4,-6) -- (6.1,-6);
\draw (0.4,1) -- (0.1,1) -- (0.1,-10);
\draw[->] (0.1,-10) -- (0.4,-10);
\draw[->] (3.25,-7.1) -- (3.25,-8) -- (9.75,-8) -- (9.75,-8.9);
\draw (0.5,-3.5) rectangle node{Lemma~\ref{LemmaCaccioppoliHoleFilling} and \ref{LemmaWeightedMeyers} \\ Linear elliptic regularity:\\Caccioppoli, Hole-filling, Meyers} (6,-1.5);
\draw (0.5,-7) rectangle node{Proposition~\ref{PropositionLinearizedCorrectorEstimate} \\ Estimates on the \\ linearized corrector $\phi_{\xi,\Xi}$, $\sigma_{\xi,\Xi}$} (6,-5);
\draw (7.5,-7) rectangle node{Lemma~\ref{LemmaEstimateLinearFunctionalsLinearized}\\ Estimates on functionals \\ of $\nabla\phi_{\xi,\Xi}$, $\nabla\sigma_{\xi,\Xi}$} (12,-5);
\draw (7.5,-11) rectangle node{Theorem~\ref{TheoremErrorEstimate} and \ref{TheoremErrorEstimate2d}\\ Optimal homogenization\\error estimates} (12,-9);
\draw (0.5,-11) rectangle node{Theorem~\ref{TheoremErrorEstimateDomains}\\ Non-optimal homogenization\\error estimates} (6,-9);
\end{tikzpicture}
\caption{The structure of the proof of the main results. Note that the dashed arrows each involve an application of Lemma~\ref{LemmaMultiscaleDecomposition} and each result in an estimate for the corresponding minimal radii $r_{*,\xi}$ respectively $r_{*,\xi,\Xi}$ (see Lemma~\ref{MomentsMinimalRadius} and Lemma~\ref{MomentsMinimalRadiusLinearized}). For simplicity, we have omitted the localization parameter $T$ and the associated technicalities throughout the diagram.}
\end{figure}

The flux corrector $\sigma_\xi$ is a $d-1$-form, i.\,e.\ it satisfies the skew-symmetry condition $\sigma_\xi(x)\in \Rm \otimes \Rddskew$ for each $x\in \Rd$. As $\sigma_\xi$ is a ``vector potential'' (a $d-1$-form), the equation \eqref{EquationFluxCorrector} only defines $\sigma_\xi$ up to gauge invariance. It is standard to fix the gauge by requiring $\sigma_\xi$ to satisfy
\begin{align}
\label{EquationFluxCorrectorGauge}
-\Delta \sigma_{\xi,jk} &= \partial_j (A(\omega_\varepsilon,\xi+\nabla \phi_\xi)\cdot e_k)-\partial_k (A(\omega_\varepsilon,\xi+\nabla \phi_\xi)\cdot e_j).
\end{align}
\begin{definition}[Flux corrector]\label{D:fluxcorr}
  Let the assumptions \hyperlink{A1}{(A1)}--\hyperlink{A3}{(A3)} and \hyperlink{P1}{(P1)}--\hyperlink{P2}{(P2)} be in place. Then for all $\xi\in\Rmd$ there exists a unique random field $\sigma_\xi:\Omega\times \Rd \rightarrow \Rm \otimes \Rddskew$, called the \emph{flux corrector} associated with $\xi$, with the following properties:
  \begin{enumerate}[(a)]
  \item For $\mathbb P$-almost every realization of the random field $\omega_\varepsilon$ the flux corrector $\sigma_\xi(\omega_\e,\cdot)$ has the regularity $\sigma_\xi(\omega_\e,\cdot)\in H^1_{\rm loc}(\Rd;\Rm)$, satisfies $\fint_{B_1}\sigma_\xi(\omega_\varepsilon,\cdot)\,dx=0$, and solves the equations \eqref{EquationFluxCorrector} and \eqref{EquationFluxCorrectorGauge} in the sense of distributions.
  \item The gradient of the flux corrector $\nabla\sigma_\xi$ is stationary in the sense that 
    \begin{equation*}
      \nabla\sigma_\xi(\omega_\e,\cdot+y)=\nabla\sigma_\xi(\omega_\e(\cdot+y),\cdot)\quad\text{a.e. in }\Rd
    \end{equation*}
    holds for $\mathbb P$-a.e.~$\omega_\varepsilon$ and all $y\in\Rd$.
  \item The gradient of the flux corrector $\nabla\sigma_\xi$ has finite second moments and vanishing expectation, that is
    \begin{equation*}
      \mathbb E\big[\nabla\sigma_\xi\big]=0,\qquad       \mathbb E\big[|\nabla\sigma_\xi|^2\big]<\infty.
    \end{equation*}
  \item The flux corrector $\mathbb P$-almost surely grows sublinearly at infinity in the sense
  \begin{align*}
  \lim_{R\rightarrow \infty } \frac{1}{R^2} \fint_{B_R} |\sigma_\xi(\omega_\e,x)|^2 \,dx =0.
  \end{align*}
  \end{enumerate}
\end{definition}

\paragraph{\it Localized correctors.} It is cumbersome to work \textit{directly} with the corrector $\phi_\xi$ and flux corrector $\sigma_\xi$ in a rigorous manner; in  particular, since we need to consider derivatives of $(\phi_\xi,\sigma_\xi)$ with respect to compactly supported perturbations of the parameter field $\omega_\varepsilon$. For this reason, we shall frequently work with the \emph{localized correctors} $\phi_\xi^T$ and $\sigma_\xi^T$, which for $T>0$ and \textit{any} parameter field $\widetilde\omega:\R^d\to\Hilbert\cap B_1$ can be defined unambiguously on a purely deterministic level.
\begin{lemma}[Existence of localized correctors]\label{L:localized}
  Let assumptions \hyperlink{A1}{(A1)}--\hyperlink{A3}{(A3)} be in place.  Let  $\widetilde\omega:\R^d\to\Hilbert\cap B_1$ be a parameter field, $T\geq1$, and $\xi\in \Rmd$. There exist unique vector
fields $\phi_\xi^T:=\phi_\xi^T(\widetilde\omega,\cdot)$ and unique tensor fields $\sigma_{\xi}^T:=\{\sigma_{\xi,jk}^T(\widetilde\omega,\cdot)\}_{jk}$ with
  \begin{align*}
    \phi_\xi^T\in H^1_{\rm uloc}(\Rd;\Rm),\qquad \sigma_\xi^T\in &H^1_{\rm uloc}(\Rd;\Rm\times\R^{d\times d}_{\rm skew}),
  \end{align*}
  which solve the PDEs
  \begin{subequations}
    \begin{eqnarray}
      \label{EquationLocalizedCorrector}
      -\nabla \cdot (A(\widetilde\omega,\xi+\nabla \phi_\xi^T))+\frac{1}{T} \phi_\xi^T&=&0,\\
      q^T_\xi&:=&A(\widetilde\omega,\xi+\nabla\phi_\xi^T),\\
      \label{EquationLocalizedFluxCorrector}
      -\Delta \sigma_{\xi,jk}^T + \frac{1}{T} \sigma_{\xi,jk}^T &=& \partial_j q^T_{\xi,k}-\partial_k q^T_{\xi,j},
    \end{eqnarray}
  \end{subequations}
  in a distributional sense in $\Rd$. Furthermore, the map $\widetilde \omega \mapsto (\phi_\xi^T,\sigma_\xi^T)$ is continuous as a map $L^\infty(\Hilbert)\rightarrow\Huloc$.
\end{lemma}
Note that the additional term $\smash{\frac{1}{T} \phi_\xi^T}$ introduces an exponential localization effect. As a consequence, existence and uniqueness of $\smash{\phi_\xi^T}$ follow by standard arguments. For example, existence follows by considering the sequence of solutions to the PDEs $-\nabla \cdot (A(\widetilde\omega,\chi_{B_r(x_0)}\xi+\nabla w_r))+\frac{1}{T} w_r=0$ (which admit a unique solution $w_r\in H^1$ by standard monotone operator theory) and passing to the limit $r \rightarrow \infty$; the exponential localization of Lemma~\ref{L:exploc} below then yields the convergence and the independence of the limit from the choice of $x_0$. The argument for $\sigma_{\xi}^T$ is similar. For a detailed proof see \cite{DNRS}.

Note also that for a stationary random field $\omega_\varepsilon$, by uniqueness the associated localized homogenization corrector $\phi_\xi^T$ and the localized flux corrector $\sigma_\xi^T$ inherit the property of stationarity.

The localized correctors approximate the original correctors $(\phi_\xi,\sigma_\xi)$ in the sense that
\begin{align}
\label{limtlocalized}
(\nabla\phi_\xi^T,\nabla\sigma_\xi^T)\to(\nabla\phi_\xi,\nabla\sigma_\xi)\qquad\text{in }L^2(\Omega\times B_r)~\text{for any }r>0
\end{align}
in the limit $T\rightarrow \infty$; for $d\geq 3$, we even have the convergence of $\phi_\xi^T$ and $\sigma_\xi^T$ itself to stationary limits $\phi_\xi$ and $\sigma_\xi$ (which however may differ from the $\phi_\xi$ and $\sigma_\xi$ from Definition~\ref{D:corr} by an additive constant). A proof of these facts is provided in Lemma~\ref{LemmaCorrectorMassiveLimit}.

\smallskip

\paragraph{\it Two-scale expansion.} The correctors $(\phi_\xi,\sigma_\xi)$ provide a link between the gradient of the homogenized solution $\nabla u_\shom$ and the gradient of the solution to the random PDE  $\nabla u_\varepsilon$ to \eqref{Equation} via the \emph{two-scale expansion}: Indeed, the two-scale expansion
\begin{equation}\label{eq:two-scaleexpansion}
  \hat u_\e(x):=u_\shom(x)+\phi_{\nabla u_\shom(x)}(x)
\end{equation}
is a (formal) approximate solution to the equation with microscopically varying material law in the sense
\begin{align}
\label{TwoScaleExpansionError}
-\nabla \cdot (A(\omega_\varepsilon,\nabla \hat u_\e))
=f - \nabla \cdot R
\end{align}
with the residual $R$ given by
\begin{equation}\label{TwoScaleResidual}
R=(\partial_\xi A(\omega_\varepsilon,\xi+\nabla \phi_\xi) \partial_\xi \phi_\xi - \partial_\xi \sigma_\xi)|_{\xi=\nabla u_\shom(x)} :\nabla^2 u_\shom.
\end{equation}
If $\varepsilon$ is much smaller than the scale on which $\nabla u_\shom$ changes, the residual $R$ will be small (as we will typically have $\partial_\xi \phi_\xi\sim \varepsilon$ and $\partial_\xi \sigma_\xi \sim \varepsilon$, see Proposition~\ref{PropositionLinearizedCorrectorEstimate} for a more precise statement). In this case, taking the difference of \eqref{Equation} and \eqref{TwoScaleExpansionError} one obtains
\begin{align*}
-\nabla \cdot \big(A(\omega_\varepsilon,\nabla u_\varepsilon)-A(\omega_\varepsilon,\nabla \hat u_\e)\big)=\nabla \cdot R\qquad \text{ with $R\sim \e$,}
\end{align*}
which by virtue of the monotonicity of the material law $A$ (see \hyperlink{A1}{(A1)}--\hyperlink{A2}{(A2)}) gives rise to an estimate on $\nabla u_\varepsilon-\nabla \hat u_\e$. From this estimate, one may then derive a bound on $u_\varepsilon-u_\shom$. For a derivation of the error expression in the two-scale expansion in the form \eqref{TwoScaleExpansionError}, which leads to some subtleties regarding measurability, we refer to the forthcoming paper \cite{DNRS}. In our result, we will replace the term $\partial_\xi \phi_\xi |_{\xi=\nabla u_\shom}$ in the two-scale expansion by a piecewise constant approximation for $u_\shom$ (with interpolation); for this variant of the two-scale expansion, essentially finite differences of the form $\phi_{\xi_1}-\phi_{\xi_2}$ appear in the error expression. Details may be found in Proposition~\ref{PropositionTwoScaleExpansion} respectively in its proof.\smallskip

\paragraph{\it Linearized correctors.} In the error expression \eqref{TwoScaleResidual} for the residual of the two-scale expansion the derivatives $\partial_\xi \phi_\xi$ and $\partial_\xi \sigma_\xi$ of the corrector and the flux corrector with respect to the field $\xi$ appear. For this reason, we need optimal-order estimates on these derivatives in order to derive an optimal-order error bound for $u_\varepsilon-u_\shom$. More precisely, in our variant of the two-scale expansion we need improved estimates for finite differences $\phi_{\xi_1}-\phi_{\xi_2}$ which reflect both the magnitude of the difference $|\xi_1-\xi_2|$ and the decorrelation in order to derive optimal-order error estimates. Note, however, that this is basically an equivalent problem to estimating the derivatives $\partial_\xi \phi_\xi$.
It is interesting to observe that the derivatives $\partial_\xi \phi_\xi$ and $\partial_\xi \sigma_\xi$ are at the same time the homogenization corrector and the flux corrector for the PDE linearized around the field $\xi+\nabla \phi_\xi$: With the coefficient field
\begin{equation*}
a_\xi(x)=a_\xi(\omega_\varepsilon,x):=\partial_\xi A(\omega_\varepsilon(x),\xi+\nabla \phi_\xi(\omega_\varepsilon,x)),
\end{equation*}
we deduce by differentiating \eqref{EquationCorrector} and \eqref{EquationFluxCorrectorGauge} that the equalities $\partial_{\xi}\phi_\xi \Xi=\phi_{\xi,\Xi}$ and $\partial_{\xi}\sigma_\xi \Xi=\sigma_{\xi,\Xi}$ hold, where $\phi_{\xi,\Xi}$ and $\sigma_{\xi,\Xi}$ are defined as the solution to the PDEs
\begin{subequations}
  \begin{align}
    \label{EquationLinearizedCorrector}
    -\nabla \cdot (a_\xi(x)(\Xi+\nabla \phi_{\xi,\Xi}))=&0,\\
    q_{\xi,\Xi}=&a_\xi(\Xi+\nabla \phi_{\xi,\Xi}),\\
    \label{EquationLinearizedFluxCorrector}
    -\Delta \sigma_{\xi,\Xi,jk} = &\partial_j
    q_{\xi,\Xi,k}-\partial_kq_{\xi,\Xi,j},
  \end{align}
\end{subequations}
i.\,e.\ $\phi_{\xi,\Xi}$ and $\sigma_{\xi,\Xi}$  are the homogenization correctors for the linear elliptic PDE with random coefficient field $a_{\xi}$.

In our analysis we will first estimate the differences $\phi_{\xi_1}^T-\phi_{\xi_2}^T$ of the localized correctors and obtain the estimate on the differences $\phi_{\xi_1}-\phi_{\xi_2}$ by passing to the limit $T\rightarrow \infty$. For this reason, we introduce the linearized localized correctors $\phi_{\xi,\Xi}^T$ and $\sigma_{\xi,\Xi}^T$
as the unique solutions in $\Huloc$ to the PDEs
\begin{subequations}
  \begin{align}
    \label{EquationLocalizedCorrectorLinearized}
    -\nabla \cdot (a_{\xi}^T(\Xi+\nabla \phi_{\xi,\Xi}^T))+\frac1T\phi_{\xi,\Xi}^T=&0,\\
    q_{\xi,\Xi}^T=&a_\xi^T(\Xi+\nabla \phi^T_{\xi,\Xi}),\\
    \label{EquationLocalizedFluxCorrectorLinearized}
    -\Delta \sigma_{\xi,\Xi,jk}^T +\frac1T\sigma_{\xi,\Xi,jk}^T= &\partial_j
    q_{\xi,\Xi,k}^T-\partial_kq_{\xi,\Xi,j}^T,
  \end{align}
\end{subequations}
where we have abbreviated
\begin{equation*}
  a^T_\xi(x)=a_\xi^T(\omega_\varepsilon,x):=\partial_\xi A(\omega_\varepsilon(x),\xi+\nabla\phi^T_\xi(\omega_\varepsilon,x)).
\end{equation*}
By differentiating \eqref{EquationLocalizedCorrector} and \eqref{EquationLocalizedFluxCorrector}, we identify $\partial_\xi \phi_\xi^T \Xi=\phi_{\xi,\Xi}^T$ and $\partial_\xi \sigma_\xi^T \Xi=\sigma_{\xi,\Xi}^T$, a fact made rigorous in Lemma~\ref{LemmaCorrectorDifferentiability}.

Note that by our assumption of monotonicity and Lipschitz continuity of $A$ \hyperlink{A1}{(A1)}--\hyperlink{A2}{(A2)}, the linearized coefficient fields $a_\xi$ and $a^T_\xi$ are uniformly elliptic and bounded random coefficient fields with a stationary and ergodic distribution. Therefore, the existence (and the notion) of solution to \eqref{EquationLinearizedCorrector} -- \eqref{EquationLinearizedFluxCorrector} and \eqref{EquationLocalizedCorrectorLinearized} -- \eqref{EquationLocalizedFluxCorrectorLinearized} would follow by the linear theory of stochastic homogenization (see e.\,g.\ \cite[Lemma~1]{GloriaNeukammOtto}), a fact that we however do not need to use:
Again, the existence of solutions to the localized corrector problems \eqref{EquationLocalizedCorrectorLinearized} and \eqref{EquationLocalizedFluxCorrectorLinearized} is evident for any parameter field $\tilde \omega$; see Lemma~\ref{LemmaCorrectorDifferentiability} for a proof.

The linearized corrector problem \eqref{EquationLinearizedCorrector} has also been central for the homogenization result for the linearized equation by Armstrong, Ferguson, and Kuusi \cite{ArmstrongFergusonKuusi}, and some of our lemmas are analogues of results from \cite{ArmstrongFergusonKuusi}: For instance, our differentiability result for the corrector $\smash{\phi_\xi^T}$ in Lemma~\ref{LemmaCorrectorDifferentiability} is essentially analogous to \cite[Lemma~2.4]{ArmstrongFergusonKuusi}. Furthermore, in the proofs of \cite{ArmstrongFergusonKuusi} small-scale regularity estimates similar to our estimate \eqref{EstimateLinearCorrectorLinfty} have been employed. However, in contrast to \cite{ArmstrongFergusonKuusi} we establish optimal-order estimates on the linearized correctors $\smash{\phi_{\xi,\Xi}^T}$ (see Proposition~\ref{PropositionLinearizedCorrectorEstimate} below), a result that will be of key importance for our main theorems.

\smallskip

\paragraph{\it Corrector estimates.} In view of the form of the formula for the residual \eqref{TwoScaleResidual} of the two-scale expansion it is clear that a key ingredient in the quantification of the homogenization error are estimates on the correctors. In this section we state estimates on the localized correctors and its linearizations.

As the typical size of the correctors $\phi_\xi^T$ and $\phi_{\xi,\Xi}^T$ will be at least of order $\varepsilon$ due to small-scale fluctuations, the corrector bounds alone cannot capture the decay of fluctuations in $d\geq 3$ optimally. For $d\geq 3$, it is therefore useful to state estimates also for associated vector potentials for the correctors $\phi_\xi^T$ and $\phi_{\xi,\Xi}^T$. In case $d\geq 3$, we introduce a (vector) potential $\smash{\theta_\xi^T:\Rd\rightarrow \Rmd}$ for the corrector $\smash{\phi_\xi^T}$ as the (up to an additive constant unique) sublinearly growing solution to the equation
\begin{subequations}
\begin{align}
\label{EquationPotentialFieldGauge}
\Delta \theta_{\xi,i}^T = \partial_i \phi_\xi^T
\end{align}
which entails
\begin{align}
\label{EquationPotentialField}
\nabla \cdot \theta_\xi^T = \phi_\xi^T,
\end{align}
\end{subequations}
as well as the corresponding quantity $\theta_{\xi,\Xi}^T$ defined by
\begin{subequations}
\begin{align}
\label{EquationPotentialFieldLinearizedGauge}
\Delta \theta_{\xi,\Xi,i}^T = \partial_i \phi_{\xi,\Xi}^T
\end{align}
which yields
\begin{align}
\label{EquationPotentialFieldLinearized}
\nabla \cdot \theta_{\xi,\Xi}^T = \phi_{\xi,\Xi}^T.
\end{align}
\end{subequations}
Note that as the equations \eqref{EquationPotentialFieldGauge} and \eqref{EquationPotentialFieldLinearizedGauge} feature a non-decaying right-hand side but lack a massive term, their solvability is not guaranteed for arbitrary parameter fields $\tilde \omega$. For random fields $\omega_\varepsilon$ subject to \hyperlink{P1}{(P1)}--\hyperlink{P2}{(P2)}, in $d\geq 3$ the existence of solutions to \eqref{EquationPotentialFieldGauge} and \eqref{EquationPotentialFieldLinearizedGauge} is a consequence of standard methods in qualitative stochastic homogenization, see e.\,g.\ \cite{GloriaNeukammOtto}.

\begin{proposition}[Estimates on the homogenization corrector for the nonlinear monotone PDE]
\label{PropositionCorrectorEstimate}
Let the assumptions \hyperlink{A1}{(A1)}--\hyperlink{A3}{(A3)} and \hyperlink{P1}{(P1)}--\hyperlink{P2}{(P2)} be in place. Then the localized homogenization correctors $\phi_\xi^T$ and the localized flux correctors $\sigma_\xi^T$ -- defined as the (unique) solution in $\Huloc$ to the PDEs \eqref{EquationLocalizedCorrector} respectively \eqref{EquationLocalizedFluxCorrector} -- are subject to the estimate
\begin{align}
\label{CorrectorEstimate}
&
\bigg(\fint_{B_r(x_0)} \bigg|\phi_\xi^T-\fint_{B_\varepsilon(x_0)} \phi_\xi^T(\tilde x) \,d\tilde x\bigg|^2 \,dx\bigg)^{1/2}
\\
\nonumber
&
+\bigg(\fint_{B_r(x_0)} \bigg|\sigma_\xi^T-\fint_{B_\varepsilon(x_0)} \sigma_\xi^T(\tilde x) \,d\tilde x\bigg|^2 \,dx\bigg)^{1/2}
\leq
\mathcal{C}(x_0) |\xi| \varepsilon\begin{cases}
(r/\varepsilon)^{1/2} &\text{for }d=1,
\\
 \big|\log \frac{r}{\varepsilon}\big|^{1/2} &\text{for }d=2,
\\
1 &\text{for }d\geq 3
\end{cases}
\end{align}
for any $r\geq 2\varepsilon$, any $x_0\in \Rd$, any $T\geq 2\varepsilon^2$, and any $\xi\in \Rmd$.
Furthermore, for $d\geq 3$ we even have
\begin{subequations}
\begin{align}
\label{CorrectorEstimateWithoutAverage}
&\bigg(\fint_{B_r(x_0)} \big|\phi_\xi^T\big|^2 \,dx\bigg)^{1/2}
+\bigg(\fint_{B_r(x_0)} \big|\sigma_\xi^T\big|^2 \,dx\bigg)^{1/2}
\leq
\mathcal{C}(x_0) |\xi| \varepsilon
\end{align}
for any $r\geq 2\varepsilon$, any $x_0\in \Rd$, any $T\geq 2\varepsilon^2$, and any $\xi\in \Rmd$,
while for $d=1$ and $d=2$ we have
\begin{align}
\label{CorrectorEstimateAveragesLowd}
&\bigg|\fint_{B_r(x_0)} \phi_{\xi}^T \,dx\bigg|
+
\bigg|\fint_{B_r(x_0)} \sigma_{\xi}^T \,dx\bigg|
\leq
\mathcal{C}(x_0) |\xi|\varepsilon
\begin{cases}
(\sqrt{T}/\varepsilon)^{1/2} &\text{for }d=1,
\\
\big|\log \frac{\sqrt{T}}{\varepsilon}\big|^{1/2} &\text{for }d=2.
\end{cases}
\end{align}
\end{subequations}
Here and below, {$\mathcal C=\mathcal C(\omega_\e,x_0)$ denotes a nonnegative stationary random field that, in particular, depends on $\xi$, the localization parameter $T$ and the radius $r$,} but whose (stretched exponential) stochastic moments are uniformly estimated by
\begin{align*}
\mathbb{E}\bigg[\exp\bigg(\frac{\mathcal{C}^{\bar \nu}}{C}\bigg)\bigg]\leq 2
\end{align*}
for some $\bar \nu>0$ and $C>0$ depending only on $d$, $m$, $\lambda$, $\Lambda$, and $\rho$ (in particular, independently of $T$).

In the case of three and more spatial dimensions $d\geq 3$, the potential field $\theta_{\xi}^T$ exists and satisfies the estimate
\begin{align}
\label{PotentialFieldEstimate}
\bigg(\fint_{B_r(x_0)} \bigg|\theta_{\xi}^T-\fint_{B_r(x_0)} \theta_{\xi}^T(\tilde x) \,d\tilde x\bigg|^2 \,dx\bigg)^{1/2}
\leq \mathcal{C}(x_0) |\xi| \varepsilon^2 
\begin{cases}
(r/\varepsilon)^{1/2} &\text{for }d=3,
\\
\big|\log \frac{r}{\varepsilon}\big|^{1/2} &\text{for }d=4,
\\
1&\text{for }d\geq 5
\end{cases}
\end{align}
for any $r\geq 2\varepsilon$, any $x_0\in \Rd$, any $T\geq 2\varepsilon^2$, and any $\xi\in \Rmd$.

Furthermore, the estimates \eqref{CorrectorEstimate}, \eqref{CorrectorEstimateWithoutAverage}, and \eqref{PotentialFieldEstimate} also hold for the $L^p$ norm in place of the $L^2$ norm as long as $2\leq p\leq \frac{2d}{d-2}$ in case $d\geq 3$ and $2\leq p<\infty$ in case $d\leq 2$, up to the following modifications: The constant $\bar \nu$ may now also depend on $p$, and the factor $\smash{|\log (r/\varepsilon)|^{1/2}}$ in \eqref{CorrectorEstimate} and \eqref{PotentialFieldEstimate} is replaced by $|\log (r/\varepsilon)|$ in the cases $d=2$ respectively $d=4$. 
\end{proposition}
Under the additional small-scale regularity assumption \hyperlink{R}{(R)}, we establish the following estimates on the homogenization corrector $\phi_{\xi,\Xi}^T$ associated with the linearized operator $-\nabla \cdot \smash{\big( \partial_\xi A(\omega_\varepsilon,\xi+\nabla \phi_\xi^T) \nabla \,\cdot\,\big)}$.
Recall that this homogenization corrector $\phi_{\xi,\Xi}^T$ is equivalently given as the derivative of the homogenization corrector $\phi_\xi^T$ with respect to $\xi$ in direction $\Xi$, see Lemma~\ref{LemmaCorrectorDifferentiability}.
\begin{proposition}[Estimates on the homogenization corrector for the linearized PDE]
\label{PropositionLinearizedCorrectorEstimate}
Let the assumptions \hyperlink{A1}{(A1)}--\hyperlink{A3}{(A3)} and \hyperlink{P1}{(P1)}--\hyperlink{P2}{(P2)} be in place. Suppose furthermore that the small-scale regularity condition \hyperlink{R}{(R)} holds.
Then the homogenization corrector for the linearized equation $\phi_{\xi,\Xi}^T=\partial_\xi \phi_\xi^T \Xi$ and the corresponding flux corrector $\sigma_{\xi,\Xi}^T=\partial_\xi \sigma_{\xi}^T \Xi$ given as the solutions to the PDEs \eqref{EquationLocalizedCorrectorLinearized} and \eqref{EquationLocalizedFluxCorrectorLinearized} are subject to the estimates
\begin{align}
\label{LinearCorrectorEstimate}
&
\bigg(\fint_{B_r(x_0)} \bigg|\phi_{\xi,\Xi}^T-\fint_{B_\varepsilon(x_0)} \phi_{\xi,\Xi}^T(\tilde x) \,d\tilde x\bigg|^2 \,dx\bigg)^{1/2}
\\&~
\nonumber
+\bigg(\fint_{B_r(x_0)} \bigg|\sigma_{\xi,\Xi}^T-\fint_{B_\varepsilon(x_0)} \sigma_{\xi,\Xi}^T(\tilde x) \,d\tilde x\bigg|^2 \,dx\bigg)^{1/2}
\\&~~~~~~~~~~~~~~~~
\nonumber
\leq
\mathcal{C}(x_0) (1+|\xi|)^C |\Xi| \varepsilon
\begin{cases}
(r/\varepsilon)^{1/2} &\text{for }d=1,
\\
\big|\log \frac{r}{\varepsilon}\big|^{1/2} &\text{for }d=2,
\\
1 &\text{for }d\geq 3,
\end{cases}
\end{align}
for any $r\geq 2\varepsilon$, any $x_0\in \Rd$, any $T\geq 2\varepsilon^2$, and any $\xi,\Xi\in \Rmd$.
Furthermore, for $d\geq 3$ we even have
\begin{align}
\label{LinearCorrectorEstimateWithoutAverage}
&\bigg(\fint_{B_r(x_0)} \big|\phi_{\xi,\Xi}^T\big|^p \,dx\bigg)^{1/p}
+\bigg(\fint_{B_r(x_0)} \big|\sigma_{\xi,\Xi}^T\big|^p \,dx\bigg)^{1/p}
\leq
\mathcal{C}(x_0)  (1+|\xi|)^C |\Xi| \, \varepsilon
\end{align}
for any $r\geq 2\varepsilon$, any $x_0\in \Rd$, any $p\in [2,\frac{2d}{(d-2)_+}]\cap [2,\infty)$, any $T\geq 2\varepsilon^2$, and any $\xi,\Xi\in \Rmd$,
while for $d=1$ and $d=2$ we have
\begin{align*}
&\bigg|\fint_{B_\varepsilon(x_0)} \phi_{\xi,\Xi}^T \,dx\bigg|
+
\bigg|\fint_{B_\varepsilon(x_0)} \sigma_{\xi,\Xi}^T \,dx\bigg|
\leq
\mathcal{C}(x_0) (1+|\xi|)^C |\Xi| \varepsilon 
\begin{cases}
(\sqrt{T}/\varepsilon)^{1/2} &\text{for }d=1,
\\
\big|\log \frac{\sqrt{T}}{\varepsilon}\big|^{1/2} &\text{for }d=2.
\end{cases}
\end{align*}
Here and below, {$\mathcal C=\mathcal C(\omega_\e,x_0)$ denotes a nonnegative, stationary random field that, in particular, depends on $\xi,\,\Xi$, the localization parameter $T$, and the radius $r$, }but whose (stretched exponential) stochastic moments are uniformly estimated by
\begin{align*}
\mathbb{E}\bigg[\exp\bigg(\frac{\mathcal{C}^{\bar \nu}}{C}\bigg)\bigg]\leq 2
\end{align*}
for some $\bar \nu>0$ and some $C$. The constants $\bar \nu$ and $C$ depend only on $d$, $m$, $\lambda$, $\Lambda$, $\nu$, $p$, and $\rho$ (in particular, they are independent of $T$).

In the case of three and more spatial dimensions $d\geq 3$, the potential field $\theta_{\xi,\Xi}^T$ exists and satisfies the estimate
\begin{align}
\label{LinearPotentialFieldEstimate}
&\bigg(\fint_{B_r(x_0)} \bigg|\theta_{\xi,\Xi}^T-\fint_{B_r(x_0)} \theta_{\xi,\Xi}^T(\tilde x) \,d\tilde x\bigg|^2 \,dx\bigg)^{1/2}
\\&~~~~~~~~~~~~~~~~~~~~~~~~~~~~~~~~
\nonumber
\leq
\mathcal{C}(x_0) (1+|\xi|)^C |\Xi| \varepsilon^2
\begin{cases}
(r/\varepsilon)^{1/2} &\text{for }d=3,
\\
\big|\log \frac{r}{\varepsilon}\big|^{1/2} &\text{for }d=4,
\\
1 &\text{for }d\geq 5
\end{cases}
\end{align}
for any $r\geq 2\varepsilon$, any $x_0\in \Rd$, any $T\geq 2\varepsilon^2$, and any $\xi,\Xi\in \Rmd$.

Furthermore, all of these estimates also hold for the correctors associated with the adjoint coefficient field $a_{\xi}^{T,*}$.

Finally, the estimates \eqref{LinearCorrectorEstimate}, \eqref{LinearCorrectorEstimateWithoutAverage}, and \eqref{LinearPotentialFieldEstimate} also hold for the $L^p$ norm with $2\leq p\leq \frac{2d}{(d-2)_+}$ and $p<\infty$ in place of the $L^2$ norm, up to the following modifications: The constant $\bar \nu$ may now also depend on $p$, and the factor $\smash{|\log (r/\varepsilon)|^{1/2}}$ in \eqref{LinearCorrectorEstimate} and \eqref{LinearPotentialFieldEstimate} is replaced by $|\log (r/\varepsilon)|$ in the cases $d=2$ respectively $d=4$. 
\end{proposition}

A key consequence of our estimates on the linearized correctors $\phi_{\xi,\Xi}^T$ and $\sigma_{\xi,\Xi}^T$ is the following estimate on differences of the correctors $\phi_\xi$ and $\sigma_\xi$ for different values of $\xi$.
\begin{corollary}
\label{CorollaryImprovedCorrectorDifferenceBounds}
Let the assumptions \hyperlink{A1}{(A1)}--\hyperlink{A3}{(A3)} and \hyperlink{P1}{(P1)}--\hyperlink{P2}{(P2)} be in place. Then for $T\rightarrow \infty$
the centered correctors $\phi_\xi^T-\fint_{B_\varepsilon(0)} \phi_\xi^T \,d\tilde x$ and $\sigma_\xi^T-\fint_{B_\varepsilon(0)} \sigma_\xi^T \,d\tilde x$
converge to {solutions $\phi_\xi$ and $\sigma_\xi$ to \eqref{EquationCorrector} and \eqref{EquationFluxCorrector} in the sense of Definition~\ref{D:corr} und Definition~\ref{D:fluxcorr}, respectively.}
Furthermore:
\begin{itemize}
\item For $d\geq 3$, we have
  \begin{align}
    \label{CorrectorBound}
    \bigg(\fint_{B_r(x_0)} \big|\phi_\xi\big|^p +  \big|\sigma_\xi\big|^p \,dx\bigg)^{1/p}
    \leq
    \mathcal{C}(x_0) |\xi| \varepsilon
  \end{align}
for any $r\geq 2\varepsilon$, any $x_0\in \Rd$, any $p\in [2,\frac{2d}{d-2}]$, and any $\xi\in \Rmd$.
Here, {$\mathcal C=\mathcal C(\omega_\e,x_0)$ denotes a nonnegative random field depending on $\xi$ and the radius $r$} with bounded stretched exponential moments in the sense
  \begin{align*}
\mathbb{E}\bigg[\exp\bigg(\frac{\mathcal{C}^{\bar \nu}}{C}\bigg)\bigg]\leq 2
  \end{align*}
  for some $\bar \nu>0$ and some $C>0$ depending only on $d$, $m$, $\lambda$, $\Lambda$, $p$, and $\rho$.
\item  For $d=1,2$, we have
\begin{align}
\label{CorrectorBound2d}
&\bigg(\fint_{B_r(x_0)} \big|\phi_\xi\big|^2 + \big|\sigma_\xi\big|^2 \,dx\bigg)^{1/2}
\leq
\mathcal{C}(x_0) |\xi| \varepsilon
 \begin{cases}
\big(\frac{r+|x_0|}{\varepsilon}\big)^{1/2} &\text{for }d=1,
\\
\big|\log \frac{r+|x_0|}{\varepsilon}\big|^{1/2} &\text{for }d=2
\end{cases}
\end{align}
for any $r\geq 2\varepsilon$, any $x_0\in \Rd$, and any $\xi\in \Rmd$.
\end{itemize}
If additionally the small-scale regularity condition \hyperlink{R}{(R)} holds, then the difference of homogenization correctors $\phi_{\xi_1}-\phi_{\xi_2}$ and $\sigma_{\xi_1}-\sigma_{\xi_2}$ is estimated by
\begin{align}
\label{DifferenceBound}
&
~~~\bigg(\fint_{B_r(x_0)} \big|\phi_{\xi_1}-\phi_{\xi_2}\big|^2 + \big|\sigma_{\xi_1}-\sigma_{\xi_2}\big|^2 \,dx\bigg)^{1/2}
\\&~~~~~~~~~
\nonumber
\leq
\mathcal{C}(x_0) (1+|\xi_1|^C+|\xi_2|^C) |\xi_1-\xi_2| \varepsilon 
\begin{cases}
\big(\frac{|x_0|}{\e}+1\big)^{1/2} &\text{for }d=1,
\\
\big|\log \big(\frac{|x_0|}{\varepsilon}+2\big)\big|^{1/2} &\text{for }d=2,
\\
1&\text{for }d\geq 3
\end{cases}
\end{align}
for any $\xi_1,\xi_2\in \Rmd$, any $r\geq 2\varepsilon$, and any $x_0\in \Rd$, where $\mathcal C=\mathcal C(\omega_\e,x_0)$ denotes a nonnegative, stationary  random field {depending on $\omega_\e,\xi_1,\xi_2,x_0,r$ satisfying the stretched exponential moment bound}
\begin{align*}
\mathbb{E}\bigg[\exp\bigg(\frac{\mathcal{C}^{\bar \nu}}{C}\bigg)\bigg]\leq 2
\end{align*}
for some $\bar \nu>0$ and some $C>0$. Here, $\bar \nu$ and $C$ depend only on $d$, $m$, $\lambda$, $\Lambda$, $\rho$, and $\nu$.
\end{corollary}

\begin{remark}
  For $d\geq 3$  in the situation of Corollary~\ref{CorollaryImprovedCorrectorDifferenceBounds} the correctors $\phi_\xi^T$ and $\sigma_\xi^T$ (without renormalization) converge to the unique \textit{stationary} solution to the correction equations \eqref{EquationCorrector} and \eqref{EquationFluxCorrector} with vanishing expectation $\mathbb{E}[\phi_\xi]=0=\mathbb{E}[\sigma_\xi]$, respectively. 
\end{remark}

\subsection{The strategy of proof for corrector estimates}

A major difficulty in quantitative stochastic homogenization is the derivation of appropriate estimates on the (localized) homogenization correctors $\smash{\phi_\xi^T}$ and the (localized) flux corrector $\smash{\sigma_\xi^T}$: As previously mentioned, corrector bounds form the basis for homogenization error estimates and the analysis of the representative volume approximation.
For our purposes, we for example need to show that the homogenization corrector $\smash{\phi_\xi^T}$ and the flux corrector $\smash{\sigma_\xi^T}$ are at most of order $|\xi|\varepsilon$ (at least in case $d\geq 3$).

In periodic homogenization, that is for $\varepsilon$-periodic operators $-\nabla \cdot (A_\varepsilon(x,\nabla \cdot))$, such corrector bounds are an easy consequence of the periodicity: By the defining equation \eqref{EquationLocalizedCorrector}, the corrector $\phi_\xi^T$ is $\varepsilon$-periodic, has vanishing average on each periodicity cell $[0,\varepsilon]^d$, and is subject to an energy estimate of the form $\dashint_{[0,\varepsilon]^d} |\nabla \phi_\xi^T|^2 \,dx\leq C |\xi|^2$. By the Poincar\'e inequality on the periodicity cell, this implies a bound of order $|\xi| \varepsilon$ on the homogenization corrector $\phi_\xi^T$. The derivation of the corresponding estimate for the flux corrector $\smash{\sigma_\xi^T}$ is analogous.

In contrast, in quantitative stochastic homogenization the derivation of such estimates on the correctors is much more involved and presents one of the main challenges.
Our strategy for the derivation of bounds on $\phi_\xi^T$, which is strongly inspired by \cite{GloriaOtto,GloriaOtto2,GloriaNeukammOtto,GloriaNeukammOtto2} but streamlined by minimizing the use of elliptic regularity, proceeds as follows:
\begin{itemize}
\item As outlined in Proposition~\ref{PropositionCorrectorEstimate}, it is our goal to prove a corrector estimate of essentially the form
\begin{align*}
\inf_{b\in \Rm} \bigg(\fint_{B_R(x_0)} |\phi_{\xi}^T-b|^p \,dx\bigg)^{1/p}
\leq \mathcal{C}(x_0)
\begin{cases}
 \varepsilon(R/\varepsilon)^{1/2}  & \text{for }d=1,\\
\varepsilon  |\log (R/\varepsilon)| & \text{for }d=2,\\
\varepsilon & \text{for }d\geq 3,
\end{cases}
\end{align*}
for any $R\geq \varepsilon$, where $\mathcal{C}$ is a random field with stretched exponential moments in the sense $\mathbb{E}[\exp(\mathcal{C}^{1/C}/C)]\leq 2$.

\smallskip
\smallskip
\item In principle, the technical Lemma~\ref{LemmaMultiscaleDecomposition} (see below) reduces the derivation of such estimates on the corrector $\phi_\xi^T$ to (mostly) the derivation of bounds on stochastic moments of integral functionals of the form
\begin{align}
\label{FunctionalInStrategy}
F(\omega_\varepsilon):=\int_\Rd g\cdot \nabla \phi_\xi^T \,dx,
\end{align}
where $g$ basically takes a weighted average of $\nabla \phi_\xi^T$ on a certain scale $r$ with $\varepsilon\leq r\leq R$.

\smallskip
\smallskip
\item The random variables $F(\omega_\varepsilon)$ as defined in \eqref{FunctionalInStrategy} have vanishing expectation $\mathbb{E}[F(\omega_\varepsilon)]=0$ due to the vanishing expectation of the corrector gradient $\smash{\mathbb{E}[\nabla \phi_\xi^T]=0}$. The latter property is a consequence of stationarity of the corrector $\smash{\phi_\xi^T}$: Since the probability distribution of the random monotone operator $A(\omega_\varepsilon(x),\cdot)$ is invariant under spatial translations, the expectation of the corrector $\mathbb{E}[\phi_\xi^T(x)]$ is the same for all points $\smash{x\in \Rd}$. This yields $\smash{\mathbb{E}[\nabla \phi_\xi^T(x)]=\nabla \mathbb{E}[\phi_\xi^T(x)]}=0$.
As a consequence of the vanishing expectation, it suffices to bound the centered moments of $\smash{F(\omega_\varepsilon)}$, that is, the stochastic moments of $F(\omega_\varepsilon)-\mathbb{E}[F(\omega_\varepsilon)]$.

\smallskip
\smallskip
\item Concentration inequalities -- like the spectral gap inequality -- are one of the most widespread probabilistic tools for establishing bounds on centered moments of random variables. In our context, that is for random fields $\omega_\varepsilon$ with correlations restricted to the length scale $\varepsilon$, they will read for example
\begin{align}
\label{StrategySpectralGap}
\mathbb{E}\Big[\Big|F(\omega_\varepsilon)-\mathbb{E}\big[F(\omega_\varepsilon)\big]\Big|^2\Big]
\leq \frac{\varepsilon^d}{\rho^2}
&\mathbb{E}
\Bigg[\int_\Rd \bigg|\fint_{B_\varepsilon(x)} \bigg|\frac{\partial F}{\partial \omega_\varepsilon}\bigg| \,d\tilde x \bigg|^2 \,dx\Bigg]
\end{align}
for all random variables $F=F(\omega_\varepsilon)$. Here, $\frac{\partial F}{\partial \omega_\varepsilon}$ denotes (essentially) the Frech\'et derivative of $F$ with respect to the field $\omega_\varepsilon$ and $\rho>0$ denotes a constant.

Note that in stochastic homogenization it is an \emph{assumption} that a spectral gap inequality like \eqref{StrategySpectralGap} holds for the random field $\omega_\varepsilon$. This assumption on the probability distribution of the random field $\omega_\varepsilon$ \emph{encodes the decorrelation properties} of $\omega_\varepsilon$. Recall that \eqref{StrategySpectralGap} implies an estimate for the average $\smash{F(\omega_\varepsilon):=\dashint_{B_r} \omega_\varepsilon \,dx}$ of the order of the CLT-scaling, see discussion  below Definition~\ref{DefinitionSpectralGap}.

A spectral gap inequality of the form \eqref{StrategySpectralGap} is valid for many classes of random fields, see Figure~\ref{FigureRandomField} and the accompanying text. It also implies estimates on higher centered moments, see Lemma~\ref{LemmaLqSpectralGap} below.

\smallskip
\smallskip
\item In order to employ the spectral gap inequality \eqref{StrategySpectralGap} to estimate the stochastic moments of random variables $F(\omega_\varepsilon)=\int g \cdot \nabla \phi_\xi^T \,dx$ as defined in \eqref{StrategySpectralGap}, we need to estimate the right-hand side of \eqref{StrategySpectralGap}, that is we need to estimate the sensitivity of the functionals $F(\omega_\varepsilon)$ with respect to changes in the coefficient field $\omega_\varepsilon$. By standard computations (see the proof of Lemma~\ref{LemmaEstimateLinearFunctionals}a for details), one may show that the Frech\'et derivative of $F$ with respect to $\omega_\varepsilon$ is given by
\begin{align}
\label{StrategySensitivityEq}
\frac{\partial F}{\partial \omega_\varepsilon} = \partial_\omega A(\omega_\varepsilon,\xi+\nabla \phi_\xi^T) \otimes \nabla h
\end{align}
where $h$ is the solution to the auxiliary linear elliptic PDE
\begin{align}
\label{StrategyDefinitionh}
-\nabla \cdot (a_{\xi}^{T,*}(x)\nabla h)+\frac{1}{T} h = \nabla \cdot g
\end{align}
(the uniformly elliptic and bounded coefficient field $\smash{a_{\xi}^{T,*}}$ being given by the transpose of $a_{\xi}^{T}(x):=\partial_\xi A(\omega(x),\xi+\nabla \phi_\xi^T)$).

By the (standard) growth assumptions for $A(\omega_\varepsilon,\cdot)$ in \hyperlink{A3}{(A3)}, the representation \eqref{StrategySensitivityEq} implies that the sensitivity $\frac{\partial F}{\partial \omega_\varepsilon}$ may be estimated by $C |\xi+\nabla \phi_\xi^T| |\nabla h|$.
In conclusion, by the spectral gap inequality \eqref{StrategySpectralGap} respectively its version for higher stochastic moments in Lemma~\ref{LemmaLqSpectralGap} we see that we have
\begin{align}
\label{StrategySensitivity}
~~~~
\mathbb{E}[|F|^q]^{1/q}
\lesssim q \varepsilon^{d/2}
\mathbb{E}\bigg[\bigg|\int_\Rd \bigg|\fint_{B_\varepsilon(x)} \big|\xi+\nabla \phi_\xi^T\big| |\nabla h| \,d\tilde x \bigg|^2 \,dx\bigg|^{q/2}\bigg]^{1/q}.
\end{align}

\item Let us next discuss how to estimate the right-hand side of \eqref{StrategySensitivity}.
Recall that $g$ basically takes a weighted average on a scale $r$ (see \eqref{FunctionalInStrategy}), i.\,e.\ $g$ is supported on a ball $B_r$ and satisfies an estimate like $|g|\lesssim r^{-d}$. As a consequence, we obtain the deterministic estimate $\int_\Rd |\nabla h|^2 \,dx\lesssim r^{-d}$ by a simple energy estimate for the defining equation \eqref{StrategyDefinitionh}. If we knew that the small-scale averages of the corrector gradient were bounded in the sense $\fint_{B_\varepsilon(x)} |\xi+\nabla \phi_\xi^T|^2 \,d\tilde x \lesssim |\xi|^2$, by H\"older's inequality we would directly obtain the deterministic bound
\begin{align*}
\int_\Rd \bigg|\fint_{B_\varepsilon(x)} \big|\xi+\nabla \phi_\xi^T\big| |\nabla h| \,d\tilde x\bigg|^2 \,dx \lesssim |\xi|^2 \int_{\Rd} |\nabla h|^2 \,dx \lesssim |\xi|^2 r^{-d}.
\end{align*}
By \eqref{StrategySensitivity} this would imply an estimate of fluctuation order on the functionals \eqref{FunctionalInStrategy} of the form $\mathbb{E}[|F|^q]^{1/q}\lesssim q |\xi| (\varepsilon/r)^{d/2}$. This would be precisely the bound we seek to obtain.

However, in the context of stochastic homogenization we cannot expect a uniform bound on the locally averaged corrector gradient $\smash{\fint_{B_\varepsilon(x)}} |\xi+\nabla \smash{\phi_\xi^T|^2} \,d\tilde x$, as in general the random field $\omega_\varepsilon$ may contain geometric configurations which could cause the corrector gradient to be arbitrarily large.
As we are dealing with the case of systems, we also cannot expect more than $L^p$ integrability of the gradient $\nabla h$ of the auxiliary function for a Meyers exponent $p$ slightly larger than $2$.
We therefore need (almost) an $L^\infty$ bound on the local averages $\dashint_{B_\varepsilon(x)} |\xi+\nabla \phi_\xi^T|^2 \,d\tilde x$. It is here useful to introduce an auxiliary quantity, namely the minimal radius $r_{*,T,\xi}(x)$ above which the corrector $\phi_\xi^T$ satisfies a bound of the form
\begin{align*}
  \frac{1}{R^2}\dashint_{B_R(x)} |\phi_\xi^T|^2 \,d\tilde x \leq 1 \quad\quad\text{for all }R\geq r_{*,T,\xi}(x)
\end{align*}
(plus one additional technical condition).
{Using now the trivial estimate $\fint_{B_\varepsilon} f\,dx\leq (\frac{r}{\varepsilon})^d\fint_{B_r}f\,dx$} and the Caccioppoli inequality, this yields the bound
\begin{align*}
  ~~~~~~
  \dashint_{B_\varepsilon(x)} |\xi+\nabla \phi_\xi^T|^2 \,d\tilde x
\leq
\left(\frac{r_{*,T,\xi}(x)}{\varepsilon}\right)^d
\underbrace{\dashint_{B_{r_{*,T,\xi}(x)}(x)} |\xi+\nabla \phi_\xi^T|^2 \,d\tilde x}_{\leq C |\xi|^2}.
\end{align*}
This estimate is however again not sufficient for our purposes.
It is here that \emph{elliptic regularity theory} in form of the hole-filling estimate enters and provides a slight but crucial improvement
\begin{align*}
\dashint_{B_\varepsilon(x)} |\xi+\nabla \phi_\xi^T|^2 \,d\tilde x
\leq
C \bigg(\frac{r_{*,T,\xi}}{\varepsilon}\bigg)^{d-\delta} |\xi|^2
\end{align*}
for some $\delta>0$
(see Lemma~\ref{EstimateGradByrstar} for details).

Together with a version of the spectral gap estimate -- see Lemma~\ref{LemmaLqSpectralGap} below -- and the vanishing expectation $\mathbb{E}[\nabla \phi_\xi^T]=0$, this estimation strategy yields a bound of the form
\begin{align}
\label{EstimateFInTermsOfMoments}
\mathbb{E}\big[|F|^{2q}\big]^{1/2q}
&\leq C |\xi| q \bigg(\frac{\varepsilon}{r} \bigg)^{d/2}
\mathbb{E}\bigg[\bigg(\frac{r_{*,T,\xi}}{\varepsilon}\bigg)^{(d-\delta)q/(1-\tau)}\bigg]^{(1-\tau)/2q}
\end{align}
for any $0<\tau<1$ and any $q$ large enough (the latter of which is not a problem).

\smallskip
\smallskip
\item One then observes that one may control the moments of $r_{*,T,\xi}$ on the right-hand side of \eqref{EstimateFInTermsOfMoments} in terms of moments of functionals $F=\int g\cdot \nabla \phi_\xi^T \,dx$ (see the proof of Proposition~\ref{PropositionCorrectorEstimate}). It is here that the slight gain $\delta$ in the exponent (from hole-filling) is crucial, as it causes the estimate to buckle, yielding a bound of the form
\begin{align*}
\mathbb{E}\bigg[\bigg(\frac{r_{*,T,\xi}}{\varepsilon}\bigg)^{q}\bigg]^{1/q}
\lesssim q^{C}.
\end{align*}
The resulting moment bounds on $r_{*,T,\xi}$ then allow to deduce the corrector estimate in Proposition~\ref{PropositionCorrectorEstimate}.

\smallskip
\smallskip
\item The derivation of the estimates for the flux corrector $\sigma_\xi^T$ and the linearized correctors $\phi_{\xi,\Xi}^T$, $\sigma_{\xi,\Xi}^T$ follows a similar strategy. However, in the case of the linearized correctors $\phi_{\xi,\Xi}^T$, $\sigma_{\xi,\Xi}^T$ the derivation of the sensitivity estimates involves an additional integrability issue on small scales, making it necessary to use a $C^{1,\alpha}$ regularity theory on the microscopic scale. It is here that our additional regularity assumption \hyperlink{R}{(R)} enters. For details, see the derivation of Lemma~\ref{LemmaEstimateLinearFunctionalsLinearized}.
\end{itemize}

Let us now state the lemmas used in the course of the proof of our main results.
An immediate consequence of the spectral gap inequality of Definition~\ref{DefinitionSpectralGap} is the following version for the $q$-th centered moment; for a proof, we refer to e.\,g.\ \cite[Proposition~3.1]{DuerinckxGloriaConcentration}. {Note that the proof in \cite[Proposition~3.1]{DuerinckxGloriaConcentration} works in microscopic spatial coordinates, i.\,e.\ considers the case $\varepsilon:=1$; by rescaling, the result holds also for general correlation lengths $\varepsilon>0$.}
\begin{lemma}
\label{LemmaLqSpectralGap}
Suppose $\mathbb P$ satisfies \hyperlink{P1}{(P1)} and \hyperlink{P2}{(P2)}. Then for any $q\geq 1$ we have
\begin{align*}
&\mathbb{E}
\Big[\Big|F(\omega_\varepsilon)-\mathbb{E}\big[F(\omega_\varepsilon)\big]\Big|^{2q}\Big]^{1/q}
\leq C q^2\varepsilon^d
\mathbb{E}
\Bigg[\Bigg|\int_\Rd \bigg(\fint_{B_\varepsilon(x)} \bigg|\frac{\partial F}{\partial \omega_\varepsilon}\bigg| \bigg)^2 \,dx\Bigg|^{q}\Bigg]^{1/q}.
\end{align*}
\end{lemma}

A central step towards the proof of the corrector estimates of Proposition~\ref{PropositionCorrectorEstimate} are the following estimates on stochastic moments of linear functionals of the localized corrector and the localized flux corrector.

\begin{lemma}[Estimates for linear functionals of the corrector and the flux corrector for the monotone system]
\label{LemmaEstimateLinearFunctionals}
Let the assumptions \hyperlink{A1}{(A1)}--\hyperlink{A3}{(A3)} and \hyperlink{P1}{(P1)}--\hyperlink{P2}{(P2)} be satisfied.
Let $\smash{\xi\in \Rmd}$, $T\geq 2\varepsilon^2$, $K_{mass} \geq C(d,m,\lambda,\Lambda)$, and let $\smash{\phi_\xi^T}$, $\smash{\sigma_\xi^T}$ {be defined as in Lemma~\ref{L:localized}}. In case $d\geq 3$, let $\smash{\theta_\xi^T}$ be the (up to constants unique) sublinearly growing solution to \eqref{EquationPotentialFieldGauge}.
Define for any $x_0\in \Rd$
\begin{align}
\label{Definitionrstar}
  r_{*,T,\xi}(x_0) := \inf \bigg\{r=2^k\varepsilon: k\in\mathbb N_0
  \text{ and for all }R=2^\ell\varepsilon\geq r,\,\ell\in\mathbb N_0,\text{ we have both }
  \\
\notag 
  \frac{1}{R^2} \fint_{B_R(x_0)} \bigg|\phi_\xi^T-\fint_{B_R(x_0)}\phi_\xi^T\,d\tilde x\bigg|^2 \,dx \leq |\xi|^2~
\\
\notag
\text{ and }~
\frac{1}{\sqrt{T}} \bigg|\fint_{B_R(x_0)} \phi_\xi^T \,dx\bigg| \leq K_{mass}|\xi|~
\bigg\}.
\end{align}
Then the random field $r_{*,T,\xi}$ is well-defined and stationary. Moreover, the following statement holds: Let $F$ denote one of the following random variables
\begin{equation*}
  \begin{aligned}
    &F:=\int_{\Rd}g(x)\cdot \nabla\phi^T_\xi(x)\,dx,\\
    \text{or }\qquad&F:=\int_{\Rd}g(x)\cdot \nabla\sigma^T_{\xi,jk}(x)\,dx,
    \,~1\leq i,j\leq d,
    \\
        \text{or }\qquad&F:=\int_{\Rd}g(x)\cdot \frac{1}{r}\nabla\theta^T_{\xi,i}(x)\,dx,\,~ 1\leq i\leq d,\text{ (in the case $d\geq 3$ only)},
\end{aligned}
\end{equation*}
for a deterministic vector field $g\in L^{p}(\Rd;\Rmd)$ satisfying for some $x_0\in\R^d$, $r\geq\varepsilon$, and some exponent $2<p<2+c$  (with a constant $c=c(d,m,\lambda,\Lambda)>0$ defined in the proof below),
\begin{align*}
  \supp g\subset B_r(x_0)\qquad\text{and}\qquad\bigg(\fint_{B_r(x_0)} |g|^{p} \,dx\bigg)^{1/p}\leq r^{-d}.
\end{align*}
Then there exists an exponent $\delta=\delta(d,m,\lambda,\Lambda)>0$ (coming from hole-filling) such that the stochastic moments of the functional $F$ satsify
  \begin{align*}
    \mathbb{E}\big[|F|^{2q}\big]^{1/2q}
    &\leq C |\xi| q \bigg(\frac{\varepsilon}{r} \bigg)^{d/2}
      \mathbb{E}\bigg[\bigg(\frac{r_{*,T,\xi}}{\varepsilon}\bigg)^{(d-\delta)q/(1-\tau)}\bigg]^{(1-\tau)/2q}
  \end{align*}
  for any $0<\tau<1$ and any $q\geq C$. Here, the constant $C$ may depend on $d$, $m$, $\lambda$, $\Lambda$, $\rho$, $K_{mass}$, $p$, and $\tau$ (but not on $\varepsilon$ and $r$). In particular, all of our estimates are independent of $T\geq 2\varepsilon^2$.
\end{lemma}
We also obtain the following estimates on averages of the correctors.
\begin{lemma}
\label{LemmaCorrectorAverages}
Given the assumptions of Lemma~\ref{LemmaEstimateLinearFunctionals}, for any $0<\tau<1$, any $q\geq C$, any $x_0\in \Rd$, any $r\geq \varepsilon$, and any $R\geq r$ we have the estimate
\begin{align}
\nonumber
&\mathbb{E}\bigg[\bigg|\fint_{B_r(x_0)} \phi_\xi^T \,dx - \fint_{B_R(x_0)} \phi_\xi^T \,dx \bigg|^{2q} \bigg]^{1/2q}
\\&
\label{EstimateCorrectorAverageDifference}
+\mathbb{E}\bigg[\bigg|\fint_{B_r(x_0)} \sigma_\xi^T \,dx - \fint_{B_R(x_0)} \sigma_\xi^T \,dx \bigg|^{2q} \bigg]^{1/2q}
\\&
\nonumber
~~~~
\leq C q |\xi| \mathbb{E}\bigg[\bigg(\frac{r_{*,T,\xi}}{\varepsilon}\bigg)^{(d-\delta)q/(1-\tau)}\bigg]^{(1-\tau)/2q} \Bigg(\sum_{l=0}^{\log_2 \frac{R}{r}} (2^l r)^2 \bigg(\frac{\varepsilon}{2^l r}\bigg)^d \Bigg)^{1/2}.
\end{align}
Furthermore, for any $R\geq \sqrt{T}$ we have the bound
\begin{align}
\label{EstimateCorrectorAverage}
&\mathbb{E}\bigg[\bigg|\fint_{B_R(x_0)} \phi_\xi^T \,dx\bigg|^q \bigg]^{1/q}
+\mathbb{E}\bigg[\bigg|\fint_{B_R(x_0)} \sigma_\xi^T \,dx\bigg|^q \bigg]^{1/q}
\\&
\nonumber
~~~~
\leq C q |\xi| \sqrt{T} \bigg(\frac{\varepsilon}{R}\bigg)^{d/2} \mathbb{E}\bigg[\bigg(\frac{r_{*,T,\xi}}{\varepsilon}\bigg)^{(d-\delta)q/(1-\tau)}\bigg]^{(1-\tau)/2q}.
\end{align}
\end{lemma}

We then establish the following moment bounds on the minimal radius $r_{*,T,\xi}$, which will enable us to deduce the corrector bounds in Proposition~\ref{PropositionCorrectorEstimate}.

\begin{lemma}[Moment bound on the minimimal radius]
\label{MomentsMinimalRadius}
Given the assumptions of Lemma~\ref{LemmaEstimateLinearFunctionals}, there exists $K_{mass}=K_{mass}(d,m,\lambda,\Lambda)$ such that the minimal radius $r_{*,T,\xi}$ has stretched exponential moments in the sense
\begin{align*}
\mathbb{E}\bigg[\exp\bigg(\frac{1}{C}\Big(\frac{r_{*,T,\xi}}{\varepsilon}\Big)^{1/C}\bigg)\bigg] \leq 2
\end{align*}
for some constant $C=C(d,m,\lambda,\Lambda,\rho)$ independent of $\xi$.
\end{lemma}

It will be central for our optimal-order homogenization error estimates to obtain optimal-order estimates on \emph{differences of correctors} for different macroscopic field gradients $\xi$, like $\phi_{\xi_1}-\phi_{\xi_2}$ for $\xi_1,\xi_2\in \Rmd$. To estimate such differences, we will rely on estimates for the derivative $\phi_{\xi,\Xi}^T$ of the corrector $\phi_\xi^T$ with respect to $\xi$ in direction $\Xi$. Formally differentiating the corrector equation  \eqref{EquationLocalizedCorrector} with respect to $\xi$, we obtain for $\phi_{\xi,\Xi}^T:=\partial_\xi \phi_\xi^T \Xi$ the PDE
\begin{align*}
-\nabla \cdot \big(\partial_\xi A(\omega_\varepsilon,\xi+\nabla \phi_\xi^T)\big(\Xi+\nabla \phi_{\xi,\Xi}^T\big)\big)+\frac{1}{T}\phi_{\xi,\Xi}^T=0
\end{align*}
(see also \eqref{EquationLocalizedCorrectorLinearized} above).
It is interesting that this PDE again takes the form of a corrector equation, namely the corrector equation for the linear elliptic equation with random coefficient field $\partial_\xi A(\omega_\varepsilon,\xi+\nabla \phi_\xi^T)$.

We next show that this differentiation can be justified rigorously.
\begin{lemma}[Differentiability of the corrector with respect to $\xi$]
\label{LemmaCorrectorDifferentiability}
Let $T>0$ and assume \hyperlink{A1}{(A1)}--\hyperlink{A2}{(A2)}. For any $\xi\in \Rmd$ and any parameter field $\tilde \omega$, let $\phi_\xi^T$ denote the unique solution in $\Huloc(\Rd;\Rm)$ to the localized corrector equation \eqref{EquationLocalizedCorrector}.
Denote by $\phi_{\xi,\Xi}^T$ the unique solution in $\Huloc(\Rd;\Rm)$ to the localized linearized corrector equation \eqref{EquationLocalizedCorrectorLinearized}.
Then the map $\xi\mapsto \smash{\phi_\xi^T}$ -- as a map $\Rmd\rightarrow \Huloc(\Rd;\Rm)$ -- is differentiable with respect to $\xi$ in the Frech\'et sense with
\begin{align}
\label{DifferentiabilityCorrectorInXi}
\partial_\xi \phi_\xi^T \Xi=\phi_{\xi,\Xi}^T
\end{align}
for any $\Xi\in \Rmd$.

Similarly, denoting by $\sigma_\xi^T$ and $\sigma_{\xi,\Xi}^T$ the unique solutions in $\Huloc$ to the PDEs \eqref{EquationLocalizedFluxCorrector} and \eqref{EquationLocalizedFluxCorrectorLinearized}, the map $\xi\mapsto \smash{\sigma_\xi^T}$ -- as a map $\smash{\Rmd}\rightarrow \smash{H^1_{\rm uloc}(\Rd;\Rm\otimes \Rddskew)}$ -- is differentiable with respect to $\xi$ in the Frech\'et sense with
\begin{align*}
\partial_\xi \sigma_\xi^T \Xi=\sigma_{\xi,\Xi}^T
\end{align*}
for any $\Xi\in \Rmd$.
\end{lemma}

In order to establish appropriate estimates on the linearized corrector $\phi_{\xi,\Xi}^T$, we again start by deriving estimates for linear functionals of the corrector gradient.

\begin{lemma}[Estimates for linear functionals of the corrector and the flux corrector for the linearized equation]
\label{LemmaEstimateLinearFunctionalsLinearized}
Let the assumptions \hyperlink{A1}{(A1)}--\hyperlink{A3}{(A3)} and \hyperlink{P1}{(P1)}--\hyperlink{P2}{(P2)} be satisfied. Suppose in addition that the regularity condition \hyperlink{R}{(R)} holds.
Let $\xi,\Xi \in \Rmd$, $K_{mass}\geq C(d,m,\lambda,\Lambda)$, $T\geq 2\varepsilon^2$, and let $\phi_{\xi,\Xi}^T$, $\sigma_{\xi,\Xi}^T$ be the unique solutions in $\Huloc$ to the corrector equations for the linearized problem \eqref{EquationLocalizedCorrectorLinearized} and \eqref{EquationLocalizedFluxCorrectorLinearized}. In case $d\geq 3$, denote by $\theta_\xi^T$ the (up to constants unique) sublinearly growing solution to \eqref{EquationPotentialFieldLinearizedGauge}. Define for any $x_0\in \Rd$
\begin{align}
\label{Definitionrstar2}
r_{*,T,\xi,\Xi}(x_0) :=
\inf \bigg\{r=2^k\varepsilon: k\in\mathbb N_0
  \text{ and for all }R=2^\ell\varepsilon\geq r,\,\ell\in\mathbb N_0,\text{ we have both }
  \\
\notag 
  \frac{1}{R^2} \fint_{B_R(x_0)} \bigg|\phi_{\xi,\Xi}^T-\fint_{B_R(x_0)}\phi_{\xi,\Xi}^T\,d\tilde x\bigg|^2 \,dx \leq |\Xi|^2~
\\
\notag
\text{ and }~
\frac{1}{\sqrt{T}} \bigg|\fint_{B_R(x_0)} \phi_{\xi,\Xi}^T \,dx\bigg| \leq K_{mass} |\Xi|~
\bigg\}.
\end{align}
Then the random field $r_{*,T,\xi,\Xi}$ is well-defined and stationary. Let $F$ denote one of the following random variables
\begin{equation*}
  \begin{aligned}
    &F:=\int_{\Rd}g(x)\cdot \nabla\phi^T_{\xi,\Xi}(x)\,dx,\\
    \text{or }\qquad&F:=\int_{\Rd}g(x)\cdot \nabla\sigma^T_{\xi,\Xi,jk}(x)\,dx,~1\leq i,j\leq d,\\
        \text{or }\qquad&F:=\int_{\Rd}g(x)\cdot \frac{1}{r}\nabla\theta^T_{\xi,\Xi,i}(x)\,dx,\,~1\leq i\leq d,\text{ (in the case $d\geq 3$ only)},
\end{aligned}
\end{equation*}
for a deterministic vector field $g\in L^{p}(\Rd;\Rmd)$ satisfying for some $x_0\in\R^d$, $r\geq\varepsilon$, and some exponent $2<p<2+c$  (with a constant $c=c(d,m,\lambda,\Lambda)>0$ defined in the proof below),
\begin{align*}
  \supp g\subset B_r(x_0)\qquad\text{and}\qquad\bigg(\fint_{B_r(x_0)} |g|^{p} \,dx\bigg)^{1/p}\leq r^{-d}.
\end{align*}
Then there exists an exponent $\delta=\delta(d,m,\lambda,\Lambda)>0$ (coming from hole-filling) such that the stochastic moments of the functional $F$ satsify
\begin{align*}
~~~~~~~~~~
\mathbb{E}\big[|F|^{2q}\big]^{1/2q}
&\leq C (1+|\xi|^C) |\Xi| q^C \bigg(\frac{\varepsilon}{r} \bigg)^{d/2}
\mathbb{E}\Big[\Big(\frac{r_{*,T,\xi,\Xi}}{\varepsilon}\Big)^{(d-\delta)q/(1-\tau)}\bigg]^{(1-\tau)/2q}
\end{align*}
for any $0<\tau<1$ and any $q\geq C$. Here, the constant $C$ may depend on $d$, $m$, $\lambda$, $\Lambda$, $\rho$, $K_{mass}$, $p$, $\nu$, and $\tau$ (but not on $\varepsilon$, $\xi$, $\Xi$, $x_0$, and $r$).\\
In particular, all of our estimates are independent of $T\geq 2\varepsilon^2$.
\end{lemma}

We also derive estimates on the averages of the linearized correctors that are analogous to the ones of Lemma~\ref{LemmaCorrectorAverages}.
\begin{lemma}
\label{LemmaCorrectorAveragesLinearized}
Given the assumptions of Lemma~\ref{LemmaEstimateLinearFunctionalsLinearized}, for any $0<\tau<1$, any $x_0\in \Rd$, any $r\geq \varepsilon$, and any $R\geq r$ we have the estimate
\begin{align}
\label{EstimateCorrectorAverageDifferenceLinearized}
&\mathbb{E}\bigg[\bigg|\fint_{B_r(x_0)} \phi_{\xi,\Xi}^T \,dx - \fint_{B_R(x_0)} \phi_{\xi,\Xi}^T \,dx \bigg|^{2q} \bigg]^{1/2q}
\\&
\nonumber
+\mathbb{E}\bigg[\bigg|\fint_{B_r(x_0)} \sigma_{\xi,\Xi}^T \,dx - \fint_{B_R(x_0)} \sigma_{\xi,\Xi}^T \,dx \bigg|^{2q} \bigg]^{1/2q}
\\&
\nonumber
~
\leq C (1+|\xi|)^C |\Xi| q^C \mathbb{E}\bigg[\bigg(\frac{r_{*,T,\xi,\Xi}}{\varepsilon}\bigg)^{(d-\delta)q/(1-\tau)}\bigg]^{(1-\tau)/2q} \Bigg(\sum_{l=0}^{\log_2 \frac{R}{r}} (2^l r)^2 \bigg(\frac{\varepsilon}{2^l r}\bigg)^d \Bigg)^{1/2}.
\end{align}
Furthermore, for any $R\geq \sqrt{T}$ we have the bound
\begin{align}
\label{EstimateCorrectorAverageLinearized}
&\mathbb{E}\bigg[\bigg|\fint_{B_R(x_0)} \phi_{\xi,\Xi}^T \,dx\bigg|^q \bigg]^{1/q}
+\mathbb{E}\bigg[\bigg|\fint_{B_R(x_0)}  \sigma_{\xi,\Xi}^T \,dx\bigg|^q \bigg]^{1/q}
\\&
\nonumber
~~~~
\leq C (1+|\xi|)^C |\Xi| q^C \sqrt{T} \bigg(\frac{\varepsilon}{R}\bigg)^{d/2} \mathbb{E}\bigg[\bigg(\frac{r_{*,T,\xi,\Xi}}{\varepsilon}\bigg)^{(d-\delta)q/(1-\tau)}\bigg]^{(1-\tau)/2q}.
\end{align}
\end{lemma}
We then establish moment bounds on the minimal radius $r_{*,T,\xi,\Xi}$.
\begin{lemma}[Moment bound on the minimimal radius of the linearized equation]
\label{MomentsMinimalRadiusLinearized}
Given the assumptions of Lemma~\ref{LemmaEstimateLinearFunctionalsLinearized}, the minimal radius $r_{*,T,\xi,\Xi}$ has stretched exponential moments in the sense
\begin{align*}
\mathbb{E}\bigg[\exp\bigg(\frac{1}{C}\Big(\frac{r_{*,T,\xi,\Xi}}{(1+|\xi|^C)\varepsilon}\Big)^{1/C}\bigg)\bigg] \leq 2
\end{align*}
for some constant $C=C(d,m,\lambda,\Lambda,\rho,\nu)$.
\end{lemma}
Note that under a different decorrelation assumption -- namely, finite range of dependence as opposed to a spectral gap inequality -- stochastic moment bounds for (basically) the quantity $r_{*,\infty,\xi,\Xi}$ have already been established in \cite{ArmstrongFergusonKuusi}. The estimates in \cite{ArmstrongFergusonKuusi} even achieve optimal stochastic integrability.

A necessary ingredient for the sensitivity estimates for functionals of the linearized corrector $\phi_{\xi,\Xi}^T$ -- as derived in Lemma~\ref{LemmaEstimateLinearFunctionalsLinearized} -- is the following regularity estimate {with rather strong stochastic integrability. It is }a consequence of our corrector estimates for the nonlinear problem and the small-scale regularity condition \hyperlink{R}{(R)}.
\begin{lemma}
\label{PointwiseRegularityLinearized}
Let the assumptions \hyperlink{A1}{(A1)}--\hyperlink{A3}{(A3)} and \hyperlink{P1}{(P1)}--\hyperlink{P2}{(P2)} be satisfied. Suppose furthermore that the small-scale regularity condition \hyperlink{R}{(R)} holds. Then there exist $\delta=\delta(d,m,\lambda,\Lambda,\rho,\nu)>0$ and a stationary nonnegative random field $\mathcal{C}_{reg,\xi}=\mathcal C_{reg,\xi}(\omega_\e,x_0)$ with uniformly bounded stretched exponential moments
\begin{align*}
\mathbb{E}\Big[\exp\Big(\frac{1}{C}\mathcal{C}_{reg,\xi}^{1/C}\Big)\Big]\leq 2
\end{align*}
such that the estimate
\begin{align}
\label{EstimateLinearCorrectorLinfty}
\sup_{\tilde x\in B_{\varepsilon/2}(x_0)} |\Xi+\nabla \phi_{\xi,\Xi}^T|(\tilde x)
\leq
\mathcal{C}_{reg,\xi}(x_0) (1+|\xi|)^C |\Xi| \bigg(\frac{r_{*,T,\xi,\Xi}(x_0)}{\varepsilon}\bigg)^{(d-\delta)/2}
\end{align}
holds for any $x_0\in \Rd$. Here, $C=C(d,m,\lambda,\Lambda,\rho,\nu)$.
\end{lemma}
\begin{proof}
We use the assumption of small-scale regularity \hyperlink{R}{(R)} to yield by Proposition~\ref{PropositionRegularityLinearizedCorrector}
\begin{align*}
\sup_{\tilde x\in B_{\varepsilon/2}(x_0)} |\Xi+\nabla \phi_{\xi,\Xi}^T|(\tilde x)
&\leq \mathcal{C}_{reg,\xi}(x_0) (1+|\xi|)^C \bigg(|\Xi|^2+\fint_{B_\varepsilon(x_0)} |\Xi+\nabla \phi_{\xi,\Xi}^T|^2 \,d\tilde x\bigg)^{\tfrac{1}{2}}.
\end{align*}
Using \eqref{EstimateGradByrstar2} to estimate the right-hand side in this equation, we obtain the desired bound.
\end{proof}

The following result converts estimates on linear functionals of the gradient of a random field and estimates on the gradient of the random field into $L^p$-estimates for the random field.
\begin{lemma}[Estimate on the $L^p$ norm by a multiscale decomposition]
\label{LemmaMultiscaleDecomposition}
Let $\gamma>0$, $\varepsilon>0$, $m\geq 2$, and $K\geq 0$. Let $u=u(\omega_\varepsilon,x)$ be a random field subject to the estimates
\begin{align}
\label{IntLemma.AssumptionGradient}
\mathbb{E}\left[\left(\fint_{B_\varepsilon(x_0)} |\nabla u|^2 \,dx\right)^{m/2}\right]^{1/m}
\leq K
\quad\quad\text{for all }x_0\in \mathbb{R}^d
\end{align}
and
\begin{align}
\label{IntLemma.AssumptionGradientAverages}
\mathbb{E}\left[\left(\int_\Rd \nabla u \cdot g \,dx\right)^m \right]^{1/m}
\leq K \bigg(\frac{\varepsilon}{r}\bigg)^\gamma
\end{align}
for all $r\geq 2\varepsilon$, all $x_0\in \mathbb{R}^d$, and all vector fields $g:\mathbb{R}^d \rightarrow \mathbb{R}^d$ supported in $B_r(x_0)$ satisfying
\begin{align}\nonumber
\Big(\fint_{B_r(x_0)} |g|^{2+1/d} \,dx\Big)^{1/(2+1/d)}\,\leq \, r^{-d}.
\end{align}
Then for all $r\geq \varepsilon$ and $2\leq p<\infty$ with $p\leq \frac{2d}{(d-2)_+}$ we have
\begin{align*}
&\mathbb{E}\left[\left(\fint_{B_r(x_0)} \left|u-\fint_{B_r(x_0)}u\right|^p \,dx\right)^{m/p} \right]^{1/m}
  \\
  &\qquad \leq
C_3 K \varepsilon
\begin{cases}
(r/\varepsilon)^{1-\gamma} &\text{for }\gamma<1,
\\
\sqrt{\log(r/\varepsilon)} &\text{for }\gamma=1\text{ and }p=2,
\\
\log(r/\varepsilon) &\text{for }\gamma=1\text{ and }p>2,
\\
1 &\text{for }\gamma>1,
\end{cases}
\end{align*}
with $C_3$ depending on $\gamma$ and $p$, but being independent of $m$.
\end{lemma}
\begin{proof}
The $L^2$-version of the statement is shown e.\,g.\ in \cite[Lemma 12]{BellaFehrmanFischerOtto}. The $L^p$ version is proven analogously.
\end{proof}

Finally, note that our main results rely on estimates for the correctors $\phi_\xi$ and $\sigma_\xi$ (and not the localized approximations $\phi_\xi^T$ and $\sigma_\xi^T$). While all of our bounds on the localized correctors $\phi_\xi^T$ and $\sigma_\xi^T$ are uniform with respect to the parameter $T\geq 2\varepsilon^2$, it remains to justify the passage to the limit $T\rightarrow \infty$.
\begin{lemma}[Convergence of the localized correctors in the limit $T\rightarrow \infty$]
\label{LemmaCorrectorMassiveLimit}
Let \hyperlink{A1}{(A1)}--\hyperlink{A3}{(A3)} and \hyperlink{P1}{(P1)}--\hyperlink{P2}{(P2)} be satisfied.
Let $\phi_\xi^T$ and $\sigma_\xi^T$ be the unique solutions to the localized corrector equations \eqref{EquationLocalizedCorrector} and \eqref{EquationLocalizedFluxCorrector} in $\Huloc$.
For any $r>0$, as $T\rightarrow \infty$ the stationary random fields $\nabla \phi_\xi^T$ and $\nabla \sigma_{\xi,jk}^T$ converge strongly in $L^2(\Omega\times B_r)$ to the corrector gradients $\nabla \phi_\xi$ and $\nabla \sigma_{\xi,jk}$, with the correctors $\phi_\xi$ and $\sigma_\xi$ as defined in Definition~\ref{D:corr}.
\end{lemma}
{For the proof we refer to the beginning of the  proof of Lemma~\ref{L:syscorr}. There, a quantitative argument for $\phi_\xi$ is carried out. It extends to $\sigma_\xi$. We remark that the statement (as in the linear case) can also be shown by a purely qualitative argument that only requires stationarity and ergodicity.}

\subsection{Quantitative two-scale approximation}
The estimate on the homogenization error invokes a quantitative two-scale expansion of the homogenized equation. As indicated earlier, for technical reasons, we do not work with the usual expansion $u_{\shom}+\nabla\phi_{\nabla u_{\shom}}$, but consider the  (easier) case of a piecewise affine approximation\footnote{The authors are indebted to Gilles Francfort for this suggestion, which lead to a simplification of the proof.} of the form
\begin{equation*}
{  u_{\shom} + \sum_{k\in K} \eta_k \phi_{\xi_k}},
\end{equation*}
where $\{\eta_k\}_{k\in K}$ denotes a partition of unity subordinate to a cover of $\Rd$ on a scale $\delta\gtrsim\varepsilon$, and $\{\xi_k\}$ denotes the associated piecewise-constant approximation of $\nabla u_{\shom}$. Depending on our application, we will choose either $\delta:=\varepsilon$ or $\delta:=\varepsilon^{1/2}$ (at least for $d\geq 3$).

We first introduce the construction of the piecewise-constant approximation:
\newcommand{\const}{{\bar C}}
\begin{lemma}[Piecewise constant approximation]\label{L:discretization}
  Let $\domain$ either denote $\Rd$, a bounded $C^1$-domain or a bounded, convex Lipschitz domain in $\Rd$. Then there exists a constant $\const\geq1$ that only depends on $\domain$, such that the following holds:
  For  all $0<\delta\leq \frac1{\const}$ there exists a partition of unity
  \begin{equation*}
    \{\eta_k\}_{k\in K}\subset C^\infty_c(\Rd)\qquad\text{with}\qquad
    K\subset\domain\text{ at most countable},
  \end{equation*}
  such that for all $k\in K$,
  \begin{subequations}
  \begin{align}
    \label{pu1}&\sum_{k\in K}\eta_k=1\qquad\text{in }\domain,\qquad 0\leq\eta_k\leq 1,\qquad \delta|\nabla\eta_k|_\infty\leq\const,\\
    \label{pu2}&\supp\eta_k\subset B_{2\delta}(k),\qquad \frac{1}{\const}\delta^d\leq \int_{\domain}\eta_k \,dx\leq \const\delta^d,
  \end{align}
  and such that the partition is locally finite in the sense that for all $k\in K$
  \begin{align}
  \label{pu3}
  \#(K\cap B_{4\delta}(k))\leq\const.
  \end{align}
\end{subequations}
Furthermore, for all $g\in H^1(\domain)$ the local averages
\begin{equation}\label{def:xi}
  \xi_k:=\frac{1}{\int_{\domain}\eta_k\,dx}\int_{\domain}g(x)\eta_k(x)\,dx,\qquad k\in K,
\end{equation}
satisfy for all $1\leq p<\infty$ with a constant $C$ (only depending on $p$ and $\domain$)
\begin{subequations}
  \begin{align}\label{est:xi1}
    |\xi_k|\,\leq& C\left(\delta^{-d}\int_{B_{2\delta}(k)\cap\domain}|g|^p \,dx\right)^\frac1p,\\
    \label{est:xi2}
    \left(\int_{B_{2\delta}(k)\cap\domain}|\xi_k-g|^p\,dx\right)^\frac1p\,\leq& C\delta\left(\int_{B_{2\delta}(k)\cap\domain}|\nabla g|^p \,dx\right)^\frac1p,
  \end{align}
  Moreover, for all $k,{\tilde k}\in K$ that are closed by in the sense that $|k-{\tilde k}|\leq 4\delta$, we have
  \begin{align}
        \label{est:xi3}
    \delta^{-1}|\xi_k-\xi_{{\tilde k}}|\,\leq& C\left(\delta^{-d}\int_{B_{6\delta}({\tilde k})\cap \domain}|\nabla g|^2 \,dx\right)^\frac12.
  \end{align}
\end{subequations}
\end{lemma}

\begin{remark}\label{R:pu_Rd}
  In the case $\domain=\R^d$ the partition can be simply  chosen as
  \begin{equation*}
    K=\delta\mathbb Z^d,\qquad \eta_k:=\eta(\tfrac{\cdot-k}{\delta}),
  \end{equation*}
  for a suitable cut-off function $\eta\in C^\infty_c(\Rd)$ with $\supp \eta \subset [-1,1]^d$ that only depends on the dimension $d$.
\end{remark}

\begin{proposition}[Error representation by the two-scale expansion]
  \label{PropositionTwoScaleExpansion}
  Let $A_\varepsilon:\Rd\times \Rmd\rightarrow \Rmd$ be a monotone operator satisfying uniform ellipticity and boundedness conditions \hyperlink{A1}{(A1)} and \hyperlink{A2}{(A2)}. For any $\xi \in \Rmd$, denote by $\phi_\xi$ a solution to the corrector equation \eqref{EquationCorrector} on $\Rd$ in the sense of Definition~\ref{D:corr} and denote by $\sigma_\xi$ a solution to the equation for the flux corrector \eqref{EquationFluxCorrector} on $\Rd$.

  Let $\domain$ either denote $\Rd$, a bounded $C^1$-domain, or a bounded, convex Lipschitz domain in $\Rd$. Let $\const$, $K$, and $\{\eta_k\}_{k\in K}$ be defined as in Lemma~\ref{L:discretization} for a discretization scale $0<\delta\leq \frac1\const$. Given $\ubar\in H^2(\domain)$ we consider the two-scale expansion,
  \begin{equation}\label{def:twoscaleexpansion}
    \hat u_\e :=\bar u + \sum_{k\in K} \eta_k\phi_k,
  \end{equation}
  with recentered correctors
  \begin{equation*}
    \phi_k:= \big(\phi_{\xi_k}-\fint_{B_\varepsilon(k)}\phi_{\xi_k}\,dx\big),\quad\sigma_k:= \big(\sigma_{\xi_k}-\fint_{B_\varepsilon(k)}\sigma_{\xi_k}\,dx\big)
  \end{equation*}
  associated with the local averages $\xi_k$ (as defined in \eqref{def:xi}) of the function $g:=\nabla\bar u$.  
 
  Then in a distributional sense we have
  \begin{align*}
    -\nabla \cdot (A_\varepsilon(x,\nabla \hat u_\e)) &= -\nabla \cdot (A_\shom(\nabla \bar u)) +\nabla \cdot R\qquad\text{in }\domain,
  \end{align*}
  with a residuum $R\in L^2(\domain;\Rmd)$ satisfying for all $\ell\in K$,
  \begin{align}
    \label{EstimateResidual-2}
    \begin{aligned}
      &\int_{\domain}\eta_\ell |R|^2 \,dx\\
      \, \leq\,&C\left(\int_{B_{2\delta}(\ell)\cap \domain} \delta^2 |\nabla^2 \ubar|^2 \,dx+\frac{1}{\delta^2}\sum_{k\in K\atop|\ell-k|\leq 4\delta}\int_{B_{6\delta}(\ell)}|\phi_{\ell}-\phi_{k}|^2 + |\sigma_{\ell}-\sigma_{k}|^2 \,dx\right).
    \end{aligned}
  \end{align}
  Above the constant $C$ only depends on $d$, $m$, $\lambda$, $\Lambda$, and the domain $\domain$.
\end{proposition}

\subsection{The approximation of the homogenized operator by periodic representative volumes}
In this section we outline the general strategy for the proof of the a~priori estimates for the RVE approximation of the effective material law $A_\shom$ stated in Theorem~\ref{T:RVE}. It is inspired by \cite{GloriaNeukammOttoInventiones}, where the first optimal result for periodic RVEs in a linear discrete setting has been established. In the following we focus on the situation where the small-scale regularity assumption \hyperlink{R}{(R)} is satisfied, since in that case we can obtain better rates. The argument features additional subtleties (even compared to the linear case of \cite{GloriaNeukammOttoInventiones}).  We start with the observation that the total approximation error decomposes into a random and a systematic part
\begin{align*}
  &A^{\RVE,L}(\xi)-\Ahom(\xi)
  \\
  &= \Big(A^{\RVE,L}(\xi)-\mathbb E_L\left[A^{\RVE,L}(\xi)\right]\Big)~+~\Big(\mathbb E_L\left[A^{\RVE,L}(\xi)\right]-\Ahom(\xi)\Big).
\end{align*}
The random error (the first term on the RHS) is a random variable with vanishing expectation and corresponds to the fluctuations of the periodic RVE approximation around its expected value. Theorem~\ref{T:RVE}a asserts that this error decays with the same rate as the fluctuations of a linear average of the random parameter field on scale $L$, i.\,e.\ like $(\frac{L}{\e})^{-\frac d2}$. The second term on the RHS is the systematic error. It captures the error coming from approximating the whole-space law $\mathbb P$ by the $L$-periodic law $\mathbb P_L$. As in \cite{GloriaNeukammOttoInventiones} we decompose the systematic error into different contributions. In particular, we introduce the following notion of localized RVE approximation.
\begin{definition}[Localized RVE approximation]\label{D:RVET}
Let $A$ satisfy \hyperlink{A1}{(A1)}--\hyperlink{A3}{(A3)}.
Let $T\geq 2\varepsilon^2$, let $\eta$ be a non-negative weight $\eta:\R^d\to\R$ with $\int_\Rd \eta\,dx=1$, and let $\widetilde\omega:\Rd\to\Hilbert\cap B_1$ be a parameter field. We then define the localized RVE approximation of size $L$ with localization parameter $T\in [2\varepsilon^2,L^2]$ for the effective operator $A_\shom$ by the expression
\begin{align}
\label{LocalizedRVEApprox}
  A^{\RVE,\eta,T}(\widetilde\omega,\xi)\cdot \Xi:=\int_\Rd \eta\,\bigg(A(\widetilde\omega,\xi+\nabla\phi^T_\xi)\cdot \Xi-\frac1T\phi_\xi^T\phi_{\xi,\Xi}^{*,T}\bigg)\,dx
\end{align}
for any $\xi,\Xi\in\Rmd$,
where $\phi^T_\xi=\phi^T_\xi(\widetilde\omega)$ and $\phi^{*,T}_{\xi,\Xi}=\phi^{*,T}_{\xi,\Xi}(\widetilde\omega)$ denote the localized corrector and the localized linearized adjoint corrector, i.\,e.\ $\smash{\phi_\xi^T}$ is the unique solution in $\Huloc(\Rd;\Rm)$ to the equation \eqref{EquationLocalizedCorrector} and $\smash{\phi_{\xi,\Xi}^{T,*}}$ is the unique solution in $\Huloc(\Rd;\Rm)$ to the equation \eqref{EquationLocalizedCorrectorLinearized} but with $\smash{a_\xi^T}$ replaced by its transpose $\smash{a_\xi^{T,*}}$.
\end{definition}
If we choose in the previous definition a random field $\omega_{\varepsilon}$ with some probability distribution $\mathbb P$ subject to \hyperlink{P1}{(P1)}--\hyperlink{P2}{(P2)}, the localized RVE approximation $A^{\RVE,\eta,T}(\xi)$ converges almost surely for $T\rightarrow \infty$ to {the (still random) material law $A^{\RVE,\eta}(\xi)\cdot \Xi=\int_\Rd \eta\,\big(A(\xi+\nabla\phi_\xi)\cdot \Xi\big)\,dx$, whose expectation is given by the effective material law $A_\shom(\xi)$.
To see this (and in particular that the contribution of the term $\frac1T\phi_\xi^T\phi_{\xi,\Xi}^{*,T}$ vanishes in this limit)}, one may e.\,g.\ use ergodicity and stationarity as well as the sublinear growth of correctors. However, in contrast to the periodic RVE approximation $A^{\RVE,L}$ to the effective material law or the homogenized material law $\Ahom$ itself, the localized RVE can be defined for all parameter fields $\widetilde\omega$, since it only invokes the localized correctors. 
This allows to couple parameter fields sampled with $\mathbb P$ and $\mathbb P_L$, respectively.
More precisely, with the restriction map
\begin{align}
\label{DefPi}
  (\pi_L\widetilde\omega)(x):=
  \begin{cases}
    \widetilde\omega(x)&\text{if }x\in B_{\frac L4},\\
    0&\text{else,}
  \end{cases}
\end{align}
we note that if $\mathbb P_L$ is an $L$-periodic approximation of $\mathbb P$ in the sense of Definition~\ref{D:Lperproxy} and if $\omega_{\e,L}$ and $\omega_\e$ denote random fields distributed according to $\mathbb P_L$ respectively $\mathbb P$, then the equality of laws
\begin{equation*}
  \pi_L\omega_{\e,L}\sim \pi_L\omega_\e
\end{equation*}
holds.
Hence, if the weight $\eta$ of the localized RVE is  supported in $B_{\frac L4}$ and $\mathbb P_L$ is a $L$-periodic approximation of $\mathbb P$, then 
\begin{equation}\label{syserror-midterm}
  \mathbb E_L\left[A^{\RVE,\eta,T}(\pi_L\omega_{\e,L},\xi)\right]=\mathbb E\left[A^{\RVE,\eta,T}(\pi_L\omega_\e,\xi)\right].
\end{equation}
This identity couples localized approximations for the homogenized material law associated with $\mathbb P$ and $\mathbb P_L$.
Motivated by this we decompose the systematic error as
\begin{equation}\label{decompsys}
  \begin{aligned}
    \mathbb E_L\left[A^{\RVE,L}(\xi)\right]-\Ahom(\xi)=
    &\mathbb E_L\left[A^{\RVE,L}(\xi)\right]-\mathbb E_L\left[A^{\RVE,\eta,T}(\xi)\right]\\
    &+\mathbb E_L\left[A^{\RVE,\eta, T}(\xi)\right]  -\mathbb E_L\left[A^{\RVE,\eta, T}(\pi_L\omega_{\e,L},\xi)\right]\\
    &+\mathbb E\left[A^{\RVE,\eta, T}(\pi_L\omega_\e,\xi)\right]  -\mathbb E\left[A^{\RVE,\eta, T}(\xi)\right]\\
    &+\mathbb E\left[A^{\RVE,\eta,T}(\xi)\right]-\Ahom(\xi).
  \end{aligned}
\end{equation}
The differences in the second and third row capture the error coming from replacing $\omega_{\e,L}$ with $\pi_L\omega_{\e}$. As our next lemma shows, this error becomes small upon choosing the localization parameter $T$ suitably. At the heart of the proof of the lemma is the exponential locality of the localized corrector equations \eqref{EquationLocalizedCorrector} and \eqref{EquationLocalizedFluxCorrectorLinearized}, see Lemma~\ref{L:exploc}.
\begin{lemma}[Estimate for the coupling error]\label{L:couperror}
Let $A:\Rd\times \Rmd\rightarrow \Rmd$ be a monotone operator subject to conditions \hyperlink{A1}{(A1)}-\hyperlink{A2}{(A2)} and subject to the Lipschitz estimate for $\partial_\xi A$ as in \hyperlink{R}{(R)}.
  Let $L\geq\sqrt T\geq 2\varepsilon$ and let $\eta_L$ denote a non-negative weight supported in $B_\frac{L}{8}$ with $|\eta_L|\leq C(d) L^{-d}$ and $|\nabla \eta|\leq C(d)L^{-d-1}$. 
  Then there exist $q=q(d,m,\lambda,\Lambda)$, $\gamma=\gamma(d,m,\lambda,\Lambda)$, and $C=C(d,m,\lambda,\Lambda)$ such that for all parameter fields $\widetilde\omega$ and all $\xi,\Xi\in \Rmd$ we have
  \begin{equation*}
    \begin{aligned}
      &\big|A^{\RVE,\eta_L,T}(\widetilde\omega,\xi)\cdot \Xi-A^{\RVE,\eta_L,T}(\pi_L\widetilde\omega,\xi)\cdot \Xi\big|\\
      &\leq C \exp\Big(-\frac{\gamma}{64}\cdot \frac{L}{\sqrt
        T}\Big)
        \Big(|\xi||\Xi|+(1+|\xi|)|\xi|\|\Xi+\nabla\phi_{\xi,\Xi}^{T}\|_{q,L,T}\Big),
    \end{aligned}
  \end{equation*}
  where 
  \begin{equation*}
    \begin{aligned}
      &\|\Xi+\nabla\phi_{\xi,\Xi}^T\|_{q,L,T}\\
      &:=\frac1{\# X_{L,T}}\sum_{x_0\in X_{L,T}}\left(\sqrt
        T^{-d}\int_\Rd|\Xi+\nabla\phi_{\xi,\Xi}^T(\widetilde\omega,x)|^{q}\exp\Big(-\frac\gamma
        2\frac{|x-x_0|}{\sqrt T}\Big)\,dx\right)^\frac{1}{q},
    \end{aligned}
  \end{equation*}
  and where $X_{L,T}\subset B_{\frac L 8}$ denotes an arbitrary finite set with cardinality $\# X_{L,T}\leq C(d) (\frac{L}{\sqrt T})^d$ and $\cup_{x_0\in X_{L,T}}B_{\sqrt T}(x_0)\supset B_{\frac L 8}$.
\end{lemma}
The differences in the first and last row of the right-hand side in \eqref{decompsys} are the \textit{systematic localization errors}, which originate from the localization with parameter $T$. The systematic localization error can be estimated as follows.
\begin{proposition}[Systematic error of localized RVE] \label{P:RVEsystem}
Let $A:\Hilbert\times \Rmd\rightarrow \Rmd$ satisfy \hyperlink{A1}{(A1)}-\hyperlink{A3}{(A3)}.
  Let $\mathbb P$ be stationary in the sense of assumption \hyperlink{P1}{(P1)}. Then the following holds for all $\xi\in\Rmd$:
  \begin{enumerate}[(a)]
  \item $\mathbb E\left[A^{\RVE,\eta,T}(\xi)\right]$ as defined in \eqref{LocalizedRVEApprox} is independent of the weight $\eta$.
  \item Suppose that $\mathbb P$ satisfies a spectral gap estimate in the sense of assumption \hyperlink{P2}{(P2)}. Assume furthermore that the small-scale regularity condition \hyperlink{R}{(R)} holds. Then for all $T\geq 2\varepsilon^2$ and all $\xi,\Xi\in \Rmd$ we have
    \begin{equation*}
      \begin{aligned}
        &~~~~\big|A_{\shom}(\xi)\cdot \Xi-\mathbb E\left[A^{\RVE,\eta,T}(\xi)\cdot \Xi\right]\big|\\
          &~~~~\leq C
          (1+|\xi|)^{C}|\xi||\Xi|\bigg(\frac{\varepsilon}{\sqrt T}\bigg)^{d\wedge 4}\times
          \begin{cases}
            \big|\log (\sqrt T/\varepsilon)\big|^2&\text{for }d=2\text{ and }d=4,\\
            1&\text{for }d=3,\ d=1,\text{ and }d\geq 5
          \end{cases}
        \end{aligned}
      \end{equation*}
      with a constant $C=C(d,m,\lambda,\Lambda,\rho,\nu)$.
  \item If $\mathbb P$ concentrates on $L$-periodic parameter fields and satisfies a periodic spectral gap estimate in the sense of Definition~\ref{SpectralGapInequality}b as well as the regularity condition \hyperlink{R}{(R)}, then we have for all $T\in [2\varepsilon^2,L^2]$ and all $\xi,\Xi\in \Rmd$
    \begin{equation*}
      \begin{aligned}
        &~~~~|\mathbb E\left[A^{\RVE,\eta}(\xi)\cdot \Xi\right]-\mathbb E\left[A^{\RVE,\eta,T}(\xi)\cdot \Xi\right]|\\
          &~~~~\leq C
          (1+|\xi|)^{C}|\xi||\Xi|\bigg(\frac{\varepsilon}{\sqrt T}\bigg)^{d\wedge 4}\times
          \begin{cases}
            \big|\log (\sqrt T/\varepsilon)\big|^2&\text{for }d=2\text{ and }d=4,\\
            1&\text{for }d=3,\ d=1,\text{ and }d\geq 5.
          \end{cases}
        \end{aligned}
      \end{equation*}
  \end{enumerate}
\end{proposition}
In order to prove this result, we need to quantify the systematic error on the level of the correctors. Note that this estimate is slightly pessimistic (by the logarithmic factor) for $d=2$ and $d=4$. Moreover, note that for $d\geq 5$ the estimate saturates, an effect that is also observed in the stochastic homogenization of linear elliptic PDEs, see \cite[Corollary~1]{GloriaOtto2}.
\begin{lemma}[Localization error in the corrector]\label{L:syscorr}
  Let $A$ satisfy the assumptions \hyperlink{A1}{(A1)}--\hyperlink{A3}{(A3)} and let $\mathbb P$ satisfy the assumptions \hyperlink{P1}{(P1)}--\hyperlink{P2}{(P2)}. Then for all $T\geq 2\varepsilon^2$ and all $x_0\in \Rd$ we have
  \begin{align*}
    &\left(\fint_{B_{\sqrt T}(x_0)}\Big(\big|\nabla \phi_\xi^{2T}-\nabla \phi_\xi^T\big|^2 + \frac{1}{T}\big|\phi_\xi^{2T}-\phi_\xi^T\big|^2\Big) \,dx\right)^\frac12\\
    &\qquad\leq\mathcal C|\xi| \bigg(\frac{\varepsilon}{\sqrt{T}}\bigg)^{\frac{d\wedge 4}{2}}
      \begin{cases}
        \big|\log (\sqrt T/\varepsilon)\big|^{1/2}  &\text{for }d\in\{2,4\},\\
        1&\text{for }d=3,\ d=1,\text{ and }d\geq 5.\\
      \end{cases}
  \end{align*}
  Here $\mathcal C$ denotes a random constant as in Proposition~\ref{PropositionLinearizedCorrectorEstimate}.
\end{lemma}
With help of Meyers estimate we may upgrade the previous estimate to an $L^p$ bound.
\begin{corollary}\label{C:syscorr}
Consider the setting of Lemma~\ref{L:syscorr}. Then there exists a Meyers exponent $\bar p=\bar p(d,m,\lambda,\Lambda)>2$ such that for any $p\in [2,\bar p]$ the estimate
\begin{align*}
  &\mathbb E\left[|\nabla \phi_\xi^{2T}-\nabla \phi_\xi^T\big|^{p}\right]^\frac1{p}
  \\
  &\leq C|\xi| \bigg(\frac{\varepsilon}{\sqrt{T}}\bigg)^{\frac{d\wedge 4}{2}}
    \times \begin{cases}
      \big|\log (\sqrt T/\varepsilon)\big| &\text{for }d\in\{2,4\},\\
      1&\text{for }d=3,\ d=1,\text{ and }d\geq 5
    \end{cases}
\end{align*}
holds, where $C=C(d,m,\lambda,\Lambda,p,\rho)$.
\end{corollary}
We also require control of the localization error for the linearized corrector.
\begin{lemma}[Localization error for linearized corrector]\label{L:syslincorr}
  Consider the setting of Lemma~\ref{L:syscorr}. Furthermore, assume that the small-scale regularity condition \hyperlink{R}{(R)} holds. Then 
  \begin{equation*}
    \begin{aligned}
      &\mathbb E\left[|\nabla \phi_{\xi,\Xi}^T-\nabla \phi_{\xi,\Xi}^{2T}|^2+\frac1T|\phi_{\xi,\Xi}^T-\phi_{\xi,\Xi}^{2T}|^2\right]^{1/2}\\
      &\qquad\leq C(1+|\xi|)^{C}|\Xi|  \bigg(\frac{\varepsilon}{\sqrt{T}}\bigg)^{\frac{d\wedge 4}{2}}
    \times \begin{cases}
      \big|\log (\sqrt T/\varepsilon)\big| &\text{for }d\in\{2,4\},\\
      1&\text{for }d=3,\ d=1, \text{ and }d\geq 5,
    \end{cases}
    \end{aligned}
  \end{equation*}
  where $C=C(d,m,\lambda,\Lambda,\rho,\nu)$.
\end{lemma}

\section{Proof of the main results}
\subsection{Ingredients from regularity theory}

Our estimates crucially rely on three basic regularity estimates for elliptic PDEs, the first two being the Caccioppoli inequality and the hole-filling estimate for \emph{nonlinear elliptic equations} (and systems) with monotone nonlinearity and the last one being a weighted Meyers estimate for linear elliptic equations (and systems). We first state the Caccioppoli inequality and the hole-filling estimate in the nonlinear setting. The (standard) proofs are provided in Appendix~\ref{SectionRegularity}.
\begin{lemma}[Caccioppoli inequality and hole-filling estimate for monotone systems]
\label{LemmaCaccioppoliHoleFilling}
Let $A(x,\xi)$ be a monotone operator subject to the assumptions \hyperlink{A1}{(A1)}--\hyperlink{A2}{(A2)}. Let $0<T\leq \infty$ and let $u$ be a solution to the system of PDEs
\begin{align*}
-\nabla \cdot (A(x,\nabla u)) + \frac{1}{T} u=\nabla \cdot g + \frac{1}{T}f
\end{align*}
for some $f\in L^2(\Rd;\Rm)$ and some $g\in L^2(\Rd;\Rmd)$.
Then there exist constants $C>0$ and $\delta>0$ depending only on $d$, $m$, $\lambda$, and $\Lambda$ with the following property: For any $R,r>0$ with $R\geq r$ we have the Caccioppoli inequality
\begin{align}
\label{CaccippoliMonotone}
&\fint_{B_{R/2}(x_0)} |\nabla u|^2 + \frac{1}{T} |u|^2 \,dx
\\&
\nonumber
\leq
\frac{C}{R^2} \fint_{B_R(x_0)} |u-b|^2 \,dx + \frac{C}{T} |b|^2 + C \fint_{B_R(x_0)}  |g|^2 + \frac{1}{T}|f|^2 \,dx
\end{align}
for any $b\in \Rm$
and the hole-filling estimate
\begin{align}
\nonumber
&\int_{B_r(x_0)} |\nabla u|^2 + \frac{1}{T} |u|^2 \,dx
\\&
\label{HoleFillingMonotone}
\leq C \bigg(\frac{r}{R}\bigg)^{\delta} \bigg(\int_{B_R(x_0)} |\nabla u|^2
+\frac{1}{T} |u|^2 \,dx\bigg)
\\&~~~
\nonumber
+ C \int_{B_R(x_0)} \bigg(\frac{r}{r+|x-x_0|}\bigg)^\delta \Big(|g|^2+\frac{1}{T}|f|^2\Big) \,dx.
\end{align}
\end{lemma}

We next state a weighted Meyers-type estimate for linear uniformly elliptic equations and systems. Its proof (which is provided in Appendix~\ref{SectionMeyers}) relies on the usual Meyers estimate, along with a duality argument and a hole-filling estimate for the adjoint operator. The details are provided in \cite{BellaFehrmanFischerOtto} for the case $T=\infty$; however, the proof applies verbatim to the case $T>0$, as the only ingredients are the Meyers estimate for the PDE and the hole-filling estimate for the adjoint PDE.
\begin{lemma}[Weighted Meyers estimate for linear elliptic systems]
\label{LemmaWeightedMeyers}
Let $a:\Rd\rightarrow \Rmd\otimes \Rmd$ be a uniformly elliptic and bounded coefficient field with ellipticity and boundedness constants $\lambda$ and $\Lambda$. Let $r>0$ be arbitrary.
Let $v \in H^1(\Rd;\Rm)$ and $g \in L^2(\Rd;\Rmd)$, $f\in L^2(\Rd;\Rm)$ be functions related through
\begin{align*}
-\nabla \cdot (a \nabla v) + \frac{1}{T} v = \nabla \cdot g + \frac{1}{\sqrt{T}} f.
\end{align*}
There exists a Meyers exponent $\bar p>2$ and a constant $c>0$, which both only depend on $d$, $m$, $\lambda$, and $\Lambda$, such that for all $2\leq p<\bar p$ and all $0<\alpha_0<c$ we have
\begin{align}
\label{WeightedMeyers}
&\left(\int_\Rd \Big(|\nabla v|^{p}+\Big|\frac{1}{\sqrt{T}}v\Big|^p \Big) \bigg(1+\frac{|x|}{r}\bigg)^{\alpha_0} \,dx \right)^{\frac{1}{p}}
\\&\nonumber~~~~~~~~~~
\leq
C \left(\int_\Rd (|g|^{p} + |f|^p) \bigg(1+\frac{|x|}{r}\bigg)^{\alpha_0} \,dx\right)^{\frac{1}{p}},
\end{align}
where the constant $C$ depends only on $d$, $m$, $\lambda$, $\Lambda$, $p$, and $\alpha_0$.
\end{lemma}

The localization ansatz for the correctors relies crucially on the following elementary deterministic energy estimate with exponential localization. As the proof is short and elementary, we directly provide it here.
\begin{lemma}[Exponential localization]
\label{L:exploc}
Suppose that $A:\Rd\times \Rmd\rightarrow \Rmd$ is a monotone operator satisfying \hyperlink{A1}{(A1)} and \hyperlink{A2}{(A2)}. Let $T>0$ and $L\geq\sqrt T$. Consider $u\in H^1_{\loc}(\Rd;\Rm)$ and $f\in L^2_{\loc}(\Rd;\Rm)$, $F\in L^2_{\loc}(\Rd;\Rmd)$ related by
\begin{equation*}
  -\nabla\cdot (A(x,\nabla u))+\frac{1}{T} u=\nabla\cdot F+\frac{1}{T} f
\end{equation*}
in a distributional sense in $\Rd$. Suppose that $u$, $f$, and $F$ have at most polynomial growth in the sense that
\begin{equation*}
  \exists k\in\mathbb N\,:\qquad\limsup\limits_{R\to\infty}R^{-k}\left(\fint_{B_R}(|u|+|\nabla u|+|f|+|F|)^2\right)^\frac{1}{2}=0.
\end{equation*}
Then for $0<\gamma\leq c(d,m,\lambda,\Lambda)$ we have
\begin{align*}
  &\int_\Rd\Big(|\nabla u|^2+\frac1T|u|^2\Big)\exp(-\gamma|x|/L)\,dx
  \\&~~~~~~~~~~~~~~~~~~~~~
  \leq C(d,m,\lambda,\Lambda)
  \int_\Rd\Big(|F|^2+\frac1T|f|^2\Big)\exp(-\gamma|x|/L)\,dx.
\end{align*}
\end{lemma}
\begin{proof}
  Set $\eta(x):=\exp(-\gamma|x|/L)$. We test the equation with $u\eta$ (which can be justified by approximation thanks to the polynomial growth assumption). By an integration by parts, the ellipticity and Lipschitz continuity of $A$, and using $|\nabla\eta|\leq \tfrac{\gamma}{L} \eta \leq \tfrac{\gamma}{\sqrt T}\eta$, we get
  \begin{align*}
    &\lambda\int_\Rd|\nabla u|^2\eta+\frac1T |u|^2\eta\,dx\\
    &\leq \gamma\int_\Rd\Lambda|\nabla u|\frac1{\sqrt T}|u|\eta\,dx+\int_\Rd|F|\Big(|\nabla u|+\gamma\frac1{\sqrt T}|u|\Big)\eta\,dx+\frac1T\int_\Rd |f||u|\eta\,dx
  \end{align*}
The claim now follows for $\gamma\leq c$ by absorbing the terms with $u$ and $\nabla u$ on the right-hand side into the left-hand side with help of Young's inequality.
\end{proof}

\begin{remark}\label{R:exponentialloc}
  We frequently apply the exponential localization in the following form: Suppose that $A:\Rd\times \Rmd\rightarrow \Rmd$ is a monotone operator satisfying \hyperlink{A1}{(A1)} and \hyperlink{A2}{(A2)}. Let $T>0$ and $L\geq\sqrt T$. Consider $u_1,u_2\in H^1_{\loc}(\Rd;\Rm)$ and $f\in L^2_{\loc}(\Rd;\Rm)$, $F\in L^2_{\loc}(\Rd;\Rmd)$, all with at most polynomial growth and related by
  \begin{equation*}
    -\nabla\cdot (A(x,\nabla u_1)-A(x,\nabla u_2))+\frac1T(u_1-u_2)=\nabla\cdot F+\frac1Tf
  \end{equation*}
  in a distributional sense in $\R^d$.
  Then for $0<\gamma\leq c(d,m,\lambda,\Lambda)$ we have
  \begin{align*}
    &\int_\Rd\Big(|\nabla u_1-\nabla u_2|^2+\frac1T|u_1-u_2|^2\Big)\exp(-\gamma|x|/L)\,dx\\
    &\leq C(d,m,\lambda,\Lambda)
      \int_\Rd\Big(|F|^2+\frac1T|f|^2\Big)\exp(-\gamma|x|/L)\,dx.
  \end{align*}
  Indeed, with $a(x):=\int_0^1 \partial_\xi A(x,\xi+(1-s)\nabla u_1(x)+s\nabla u_2(x))\,ds$ and $\delta u:=u_1-u_2$, we have $-\nabla\cdot (a(x)\nabla\delta u) + \frac1T\delta u=\nabla\cdot F+\frac1T f$. Since $A$ is a monotone operator satisfying \hyperlink{A1}{(A1)} and \hyperlink{A2}{(A2)}, the derivative $\partial_\xi A(x,\xi)$ is a uniformly elliptic matrix field; hence, $a(x)$ is a uniformly elliptic coefficient field and the claimed estimate follows from the linear version of Lemma~\ref{L:exploc}.
\end{remark}

\subsection{The convergence rate of the solutions}

We first provide the proof of the error estimate for $||u_\varepsilon-u_\shom||_{L^2}$. It is based on a two-scale expansion with a piecewise constant approximation for the slope of the limiting solution $u_\shom$, whose approximation properties are stated in Lemma~\ref{L:discretization} and Proposition~\ref{PropositionTwoScaleExpansion}.

\begin{proof}[Proof of Lemma~\ref{L:discretization}]
  \textbf{Step 1. Construction of the partition of unity.}
  The construction of the partition of unity in the case $\domain=\R^d$ is elementary, see Remark~\ref{R:pu_Rd}. We thus only discuss the case of a bounded domain. In the following $\const>1$ denotes a constant that may vary from line to line, but that can be chosen only depending on $\domain$ and the dimension $d$. We fix the length scale $\delta$ with $0<\delta\leq\frac1\const$.
  Thanks to the assumptions on $\domain$ ($C^1$-boundary or Lipschitz boundary \& convexity), for $\const>0$ large enough, we have that
  \begin{equation}\label{L:pu:eq3}
    B_{6\delta}(x)\cap\domain\text{ is connected for all }x\in\domain,
  \end{equation}
  and we may cover $\domain$ (enlarged by a layer of thickness $\delta/2$) by balls of radius $B_{\delta}$, whose volume interesected with $\domain$ is comparable with $\delta^d$; more precisely,
  \begin{equation*}
    \domain+B_{\delta/2}(0)\subset \bigcup_{x\in K''}B_{\delta}(k),\qquad K'':=\{\,x\in \domain\,:\,B_{\delta/\const}(x)\subseteq\domain\,\}.
  \end{equation*}
  Since the set on the left is compact, we can find a finite subset $K'\subset K''$ such that the above inclusion holds with $K''$ replaced by $K'$. Moreover, by the Vitali covering lemma we can find a finite subset $K\subseteq K'$ such that the balls $\{B_{\delta/3}(k)\}_{k\in K}$ are disjoint and $\domain+B_{\delta/2}\subseteq \bigcup_{k\in K}B_\delta(k)$. With this covering at hand, we may iteratively define (measurable) functions $\chi_k:\R^d\to\{0,1\}$ such that
  \begin{equation}\label{L:pu:eq1}
    1_{B_{\delta/3}(k)}\leq \chi_k\leq 1_{B_{\delta}(k)}\qquad\text{and}\qquad \sum_{k\in K}\chi_k=1\text{ on }\domain+B_{\delta/2}(0).
  \end{equation}
  Let $\eta_{\delta/2}$ denote the standard mollifier with support in $B_{\delta/2}(0)$. For $k\in K$ set $\eta_k:=\chi_k*\eta_{\delta/2}$. By construction $\{\eta_k\}_{k\in K}$ is a smooth partition of unity for $\domain$ satisfying \eqref{pu1} and \eqref{pu2} (the latter is a consequence of \eqref{L:pu:eq1} and the fact that the balls $B_{\delta/\const}(k)$, $k\in K$, are disjoint and contained in $\domain$) as well as \eqref{pu3}.
  \medskip

  \textbf{Step 2. Estimates.} The arguments for \eqref{est:xi1} and \eqref{est:xi2} are standard. We prove  \eqref{est:xi3}. We first note that for all $k,{\tilde k}\in K$ with $|k-{\tilde k}|\leq 4\delta$ we have $\supp\eta_k\cup\supp\eta_{{\tilde k}}\subset B_{6\delta}({\tilde k})$. Moreover, by \eqref{L:pu:eq3} the set $B_{6\delta}({\tilde k})\cap\domain$ is connected. Denote by $v\in H^1(\domain\cap B_{6\delta}({\tilde k}))$ the unique mean-free weak solution to the Neumann problem
  \begin{eqnarray*}
    -\triangle v&=&\frac{\eta_k}{\int_\domain\eta_k\,dx}-\frac{\eta_{{\tilde k}}}{\int_\domain\eta_{{\tilde k}}\,dx}\qquad\text{in }B_{6\delta}({\tilde k})\cap\domain,\\
    \partial_\nu v&=&0\qquad\text{on }\partial(B_{6\delta}({\tilde k})\cap\domain).
  \end{eqnarray*}
  The standard a priori estimate, Poincar\'e's inequality and \eqref{pu2} yield
  \begin{equation*}
    \int_{B_{6\delta}({\tilde k})\cap\domain}|\nabla v|^2 \,dx\leq C\delta^{2-d},
  \end{equation*}
  for a constant $C$ only depending on $\const$. This implies \eqref{est:xi3}, since
  \begin{equation*}
    \xi_k-\xi_{{\tilde k}}=\int_{\domain}\bigg(\frac{\eta_k}{\int_\domain\eta_k}-\frac{\eta_{{\tilde k}}}{\int_\domain\eta_{{\tilde k}}}\bigg)g\,dx=\int_{\domain\cap B_{6\delta}({\tilde k})}\nabla g\cdot \nabla v\,dx,
  \end{equation*}
  as can be seen by an integration by parts.
\end{proof}

\begin{proof}[Proof of Proposition~\ref{PropositionTwoScaleExpansion}]
  To shorten the notation, we implicitly assume that $k,\tilde k\in K$ and we shall use the shorthand notation
  \begin{equation*}
    k\sim \tilde k\qquad:\Leftrightarrow\qquad \text{$k,\tilde k$ are nearby sites, i.e. $|k-\tilde k|\leq 4\delta$}.
  \end{equation*}
  Note that $\supp\eta_k\cap\supp\eta_{\tilde k}\neq\emptyset$ implies $k\sim\tilde k$.
  We shall also use the notation $A\lesssim B$ if $A\leq CB$ for a constant $C$ that only depends on $d,m,\lambda,\Lambda$, $\const$ and $\domain$.   
  First, we note that in the sense of distribution in $\domain$, we have
  \begin{align*}
    &\nabla \cdot (A_\varepsilon(x,\nabla \hat u_\e))
    \\&
    =\nabla \cdot \Big(A_\varepsilon\Big(x,\nabla \ubar+\sum_{k} \eta_k \nabla \phi_{k}\Big)\Big)
    \\&~~~~
    +\nabla \cdot \Big(A_\varepsilon\Big(x,\nabla \ubar+ \sum_{k} \eta_k \nabla \phi_{k} + \sum_{k\in K} \phi_{k} \nabla \eta_k \Big)
    \\&~~~~~~~~~~~~~~~~
    -A_\varepsilon\Big(x,\nabla \ubar+\sum_{k} \eta_k \nabla \phi_{k}\Big)\Big).
  \end{align*}
  Adding and subtracting intermediate terms and using the fact that the $\eta_k$ form a partition of unity (i.\,e.\ $\sum_{k}\eta_k=1$), we get
\begin{align*}
&\nabla \cdot (A_\varepsilon(x,\nabla \hat u_\e))
\\&
=
\nabla \cdot \big(A_\shom\big(\nabla \ubar \big)\big)
\\&~~~~
+\nabla \cdot \bigg(\sum_{k} \eta_k \big(A_\shom(\xi_k)-A_\shom(\nabla \ubar)\big)\bigg)
\\&~~~~
+\nabla \cdot \bigg(\sum_{k} \eta_k \big(A_\varepsilon\big(x,\xi_k+\nabla \phi_{k}\big)-A_\shom\big(\xi_k\big)\big)\bigg)
\\&~~~~
+\nabla \cdot \bigg(\sum_{k} \eta_k \big(A_\varepsilon\big(x,\nabla \ubar+\nabla \phi_{k}\big)-A_\varepsilon\big(x,\xi_k+\nabla \phi_{k}\big)\big)\bigg)
\\&~~~~
+\nabla \cdot \bigg(A_\varepsilon\Big(x,\nabla \ubar+\sum_{k} \eta_k \nabla \phi_{k}\Big)-\sum_{k} \eta_k A_\varepsilon\big(x,\nabla \ubar+\nabla \phi_{k}\big)\bigg)
\\&~~~~
+\nabla \cdot \bigg(A_\varepsilon\Big(x,\nabla \ubar+ \sum_{k} \eta_k \nabla \phi_{k} + \sum_{k} \phi_{k} \nabla \eta_k \bigg)
\\&~~~~~~~~~~~~~~~~
-A_\varepsilon\Big(x,\nabla \ubar+\sum_{k} \eta_k \nabla \phi_{k}\Big)\Big).
\end{align*}
Using the equation for the flux corrector $\nabla \cdot \sigma_{k} = A_\varepsilon\big(x,\xi_k+\nabla \phi_{k}\big)-A_\shom\big(\xi_k\big)$ (see \eqref{EquationFluxCorrector}) and the skew-symmetry of $\sigma_{k}$ (which implies $\nabla \cdot (\eta\nabla \cdot \sigma_{k})=-\nabla \cdot (\sigma_{k}\nabla \eta)$), we obtain
\begin{align*}
\nabla \cdot (A_\varepsilon(x,\nabla \hat u_\e))\,=\,
\nabla \cdot \big(A_\shom\big(\nabla \ubar \big)\big)
+\nabla \cdot R,
\end{align*}
with a residuum $R:=I+II+III$ where
\begin{align*}
I:=&\sum_{k} \eta_k \big(A_\shom(\xi_k)-A_\shom(\nabla \ubar)\big)
\\&~
+\sum_{k} \eta_k \big(A_\varepsilon\big(x,\nabla \ubar+\nabla \phi_{k}\big)-A_\varepsilon\big(x,\xi_k+\nabla \phi_{k}\big)\big)
\end{align*}
and
\begin{align*}
II:=&
-\sum_{k} \sigma_{k} \nabla \eta_k
\\&~
+A_\varepsilon\Big(x,\nabla \ubar+ \sum_{k} \eta_k \nabla \phi_{k} + \sum_{k} \phi_{k} \nabla \eta_k \Big)
-A_\varepsilon\Big(x,\nabla \ubar+\sum_{k} \eta_k \nabla \phi_{k}\Big)
\end{align*}
as well as
\begin{align*}
III:=A_\varepsilon\Big(x,\nabla \ubar + \sum_{k} \eta_k \nabla \phi_{k}\Big)-\sum_{k} \eta_k A_\varepsilon\big(x,\nabla \ubar+\nabla \phi_{k}\big).
\end{align*}
It is our goal to show that $R$ can be estimated for all $\ell$ as
\begin{align}
\label{ResidualBound1}
  &\int_{\domain}\eta_\ell |R|^2 \,dx\lesssim \int_{\domain}\eta_\ell(|I|^2+|II|^2+|III|^2) \,dx
  \\&
  \nonumber
  \lesssim  \delta^2 \int_{B_{2\delta}(\ell)\cap\domain} |\nabla^2 \ubar|^2\,dx +\frac{1}{\delta^2}\sum_{k\,:\, k\sim \ell}\int_{B_{6\delta}(\ell)}|\phi_{\ell}-\phi_{k}|^2 + |\sigma_{\ell}-\sigma_{k}|^2 \,dx.
\end{align}
For the argument we first argue that $\int\eta_\ell|I|^2$ may be bounded by the first term on the right-hand side of \eqref{ResidualBound1}. Indeed, by  Lipschitz continuity of $A_\varepsilon$ and $A_{\shom}$ in $\xi$ (see \hyperlink{A2}{(A2)} and Theorem~\ref{TheoremStructureProperties}a, respectively), we have the pointwise bound $|I|\lesssim \sum_{k}\eta_k|\nabla\ubar-\xi_k|$, and thus (with help of \eqref{est:xi2})
\begin{eqnarray*}
  \int_{\domain}\eta_\ell|I|^2\,dx
  &\lesssim& \sum_{k\,:\,k\sim \ell}\int_{B_{2\delta}(k)\cap\domain}|\nabla\ubar -\xi_k|^2\,dx\\
  &\lesssim& \delta^2\sum_{k\,:\,k\sim \ell}\int_{B_{2\delta}(k)\cap\domain}|\nabla^2\ubar|^2\,dx\lesssim \delta^2\int_{B_{6\delta}(\ell)\cap\domain}|\nabla^2\ubar|^2\,dx.
\end{eqnarray*}
Next, we show that $\int\eta_\ell |II|^2\,dx$ may be bounded by the second term on the right-hand side of \eqref{ResidualBound1}.
By the Lipschitz-continuity of $A_\varepsilon$ in $\xi$ and the fact that $\sum_k\nabla \eta_k=0$ (as the $\eta_k$ form a partition of unity), we have $|II|\lesssim \big|\,\sum_{k}(\sigma_{\ell}-\sigma_{k})\nabla\eta_k\,\big|+\big|\sum_{k}(\phi_{\ell}-\phi_{k})\nabla\eta_k\big|$. Since $|\nabla \eta_k|\leq \const \delta^{-1}$, we deduce that
\begin{eqnarray*}
  \int_{\domain}\eta_\ell|II|^2\,dx
  &\lesssim& \delta^{-2}\sum_{k:k\sim\ell}  \int_{B_{2\delta}(\ell)}|\phi_{\ell}-\phi_{k}|^2+|\sigma_{\ell}-\sigma_{k}|^2\,dx.
\end{eqnarray*}
Finally, we estimate the third term $III$. Since $\sum_{\tilde k}\eta_{\tilde k}=1$ we have
\begin{align*}
III=\sum_{\tilde k}\eta_{\tilde k}\Big(A_\varepsilon\big(x,\nabla \ubar + \sum_{k} \eta_k \nabla \phi_{k}\big)-A_\varepsilon\big(x,\nabla \ubar+\nabla \phi_{{\tilde k}}\big)\Big).
\end{align*}
Hence, the Lipschitz-continuity of $A_\varepsilon$ in $\xi$ yields $  |III|\lesssim \sum_{\tilde k,k}\eta_{\tilde k}\eta_k |\nabla \phi_{k}-\nabla \phi_{{\tilde k}}|$. By the support property of $\eta_k$ and $\eta_{\tilde k}$ we get
\begin{equation}\label{eq:1230020}  
  \int_{\domain}\eta_\ell|III|^2\,dx\leq C\sum_{k\,:\,k\sim \ell}\int_{B_{2\delta}(\ell)}|\nabla \phi_{\ell}-\nabla \phi_{k}|^2\,dx.
\end{equation}
Thus, we need to bound the difference of the two corrector gradients $\nabla \phi_{\ell}$ and $\nabla \phi_{k}$ for nearby grid points $k$ and $\ell$. To this aim, note that the corrector equation \eqref{EquationCorrector} implies
\begin{align*}
  -\nabla \cdot (A_\varepsilon(x,\xi_\ell+\nabla \phi_{\ell})-A_\varepsilon(x,\xi_{k}+\nabla \phi_{{k}}))&=0.
\end{align*}
An energy estimate based on assumptions \hyperlink{A1}{(A1)} and \hyperlink{A2}{(A2)} thus yield
\begin{align*}
  \fint_{B_{2\delta}(\ell)} |\nabla \phi_{\ell}-\nabla \phi_{k}|^2 \,dx
  \lesssim  |\xi_\ell-\xi_{k}|^2 + \frac{1}{\delta^2} \fint_{B_{6\delta}(\ell)} |\phi_{\ell}-\phi_{k}|^2 \,dx.
\end{align*}
In conjunction with \eqref{est:xi2} and \eqref{eq:1230020}, we see that $\int\eta_\ell|III|^2 \,dx$ may be bounded by the right-hand side of \eqref{ResidualBound1}.
\end{proof}

\begin{proof}[Proof of Theorem~\ref{TheoremErrorEstimate}, Theorem~\ref{TheoremErrorEstimate2d}, and Theorem~\ref{TheoremErrorEstimateDomains}]
  We use the notation $A\lesssim B$ if $A\leq CB$ for a constant $C$ that only depends on $d,m,\lambda,\Lambda$, $\const$ and $\domain$. 

{\bf Step 1: Proof of Theorem~\ref{TheoremErrorEstimateDomains} -- the case of a bounded domain $\domain$.} For the proof we appeal to the two-scale expansion introduced in Proposition~\ref{PropositionTwoScaleExpansion}. Let $\delta$ denote a discretization scale that satisfies $\e \leq \delta \leq \frac{1}{C}$ and that we fix later. According to  Proposition~\ref{PropositionTwoScaleExpansion} we denote by $\{\eta_k\}_{k\in K}$ the  partition of unity of Lemma~\ref{L:discretization} and by $\hat u_\e$ the two-scale expansion defined in~\eqref{def:twoscaleexpansion} associated with $\ubar:=u_{\shom}$. Moreover, we define the correctors $\phi_k,\sigma_k$ as in Proposition~\ref{PropositionTwoScaleExpansion}. Note that by the proposition and the identity $\nabla \cdot (A(\omega_\varepsilon,\nabla u_\varepsilon))=\nabla \cdot (A_{\shom}(\nabla \ubar))$ in $\domain$, we have
\begin{align}
  \label{MainEquationDifferenceDomain}
  \nabla \cdot (A(\omega_\varepsilon,\nabla u_\varepsilon)-A(\omega_\varepsilon,\nabla \hat u_\e)) = \nabla \cdot R\qquad\text{in a distributional sense in $\domain$,}
\end{align}
with a residuum $R\in L^2(\domain;\Rmd)$.

{\it Step 1.1:}
We claim that
\begin{equation}\label{st:001001011}
    \int_\domain |R|^2 \,dx \,\leq\,\mathcal C^2\|\nabla u_{\shom}\|_{H^1(\domain)}^2\,\Big(\delta^2+\Big(\frac{\e}{\delta}\Big)^2\mu_d\Big(\frac{\delta}{\e}\Big)\Big).
  \end{equation}
  Here and below we denote by $\mathcal C$ a random constant that might change from line to line, but that satisfies the stretched exponential moment bound
  \begin{equation}\label{st:stretchedmoment}
    \mathbb E\Big[\exp\Big(\frac{\mathcal C^{\bar\nu}}{C}\Big)\Big]\leq 2,
  \end{equation}
  with $\bar\nu$ only depending on $d,m,\lambda,\Lambda,\rho$ and with $C>0$ depending only on $d,m,\lambda,\Lambda,$ $\rho,$ and $\domain$. Below, we shall tacitly use the calculus rules for random constants with stretched exponential moments, see Lemma~\ref{L:calculusstretched}.

The starting point for the argument is the local residuum estimate \eqref{EstimateResidual-2}, which we post-process by appealing to the triangle inequality:
\begin{equation}
  \begin{aligned}
    \int_\domain |R|^2 \,=\,&\sum_{\ell\in K}\int_{\domain}\eta_\ell|R|^2\\
    \lesssim\, &\delta^2 \int_\domain|\nabla^2u_{\shom}|^2\,dx\,+\delta^{-2}\sum_{\ell\in K}\sum_{k\in K\atop |k-\ell|\leq 4\delta} \int_{B_{6\delta}(\ell)}(|\phi_{k}|^2 + |\sigma_{k}|^2) \,dx
  \end{aligned}
\end{equation}
As we shall show in Step~3, from Corollary~\ref{CorollaryImprovedCorrectorDifferenceBounds} we may obtain the corrector estimate
\begin{eqnarray}
\label{st:correstimat1}
  \sum_{\ell\in K}\sum_{k\in K\atop |k-\ell|\leq 4\delta}\int_{B_{6\delta}(\ell)}|\phi_{k}|^2 + |\sigma_{k}|^2 \,dx
  &\leq&\mathcal C^2\,\int_{\domain}|\nabla u_{\shom}|^2\,\e^2\mu_d\Big(\frac{\delta}{\e}\Big),
\end{eqnarray}
where
\begin{equation*}
  \mu_d(s):=
  \begin{cases}
    s&\text{for }d=1,\\
    |\log(s)|&\text{for }d=2,\\
    1&\text{for }d\geq 3.
  \end{cases}
\end{equation*}
Now, \eqref{st:001001011} follows by combining the previous two estimates.
\smallskip

{\it Step 1.2:} 
To derive a bound on the difference $\nabla u_\e-\nabla\hat u_\e$ from the above estimate on the residuum, we would like to test the equation with $u_\e-\hat u_\e$. Thus, we need to modify $\hat u_\e$ close to the boundary to obtain an admissible test function. Let $\tau \geq \varepsilon$. Set $(\partial\domain)_\tau:=\{x\in\domain\,:\,\dist(x,\partial\domain)<\tau\}$. We denote by $\psi$ a cut-off function with $\psi\equiv 1$ in $\{x\in \domain:\dist(x,\partial\domain)>\tau\}$, $\psi=0$ on $\partial \domain$, and $|\nabla \psi|\lesssim \tau^{-1}$. We claim that 
\begin{equation}\label{st:001001012}
  \begin{aligned}
    &\int_\domain |\nabla u_\varepsilon-\nabla \hat u_\e - \nabla ((1-\psi)(\ubar-\hat u_\e))|^2 \,dx
    \\&
    \leq\mathcal C^2\,\|\nabla u_{\shom}\|_{H^1(\domain)}^2\,\Big(\tau+\delta^2+\Big(\frac{\e}{\delta}\Big)^2\mu_d\Big(\frac{\delta}{\e}\Big) + \frac{\e^2}{\tau}\mu_d\Big(\frac{\delta}{\e}\Big)\Big).
  \end{aligned}
\end{equation}
Note that the terms with the factors $\tau$ and $\frac{\e^2}{\tau}\mu_d(\frac{\delta}{\e})$ are due to the cut-off close to the boundary. For the argument, we test the equation \eqref{MainEquationDifferenceDomain} with the difference $(u_\varepsilon-\hat u_\e+(1-\psi)(\hat u_\e-\ubar))\in H^1_0(\domain)$. By the monotonicity property \hyperlink{A1}{(A1)} and the Lipschitz continuity \hyperlink{A2}{(A2)} we get
\begin{align*}
  \int_\domain \lambda |\nabla u_\varepsilon-\nabla \hat u_\e|^2 \,dx
  \leq& -\int_\domain R \cdot \nabla (u_\varepsilon-\hat u_\e) \,dx
  \\&
  +\Lambda \int_{\domain} |\nabla u_\varepsilon-\nabla \hat u_\e| |\nabla ((1-\psi)(\hat u_\e-\ubar))| \,dx
  \\&
  -\int_{\domain} R \cdot \nabla ((1-\psi)(\hat u_\e-\bar u)) \,dx.
\end{align*}
This entails the estimate
\begin{align*}
  \int_\domain |\nabla u_\varepsilon-\nabla \hat u_\e - \nabla ((1-\psi)(\ubar-\hat u_\e))|^2 \,dx
  \lesssim&
            \int_\domain |R|^2 \,dx
        + \int_{(\partial\domain)_\tau} |\nabla\hat u_\e-\nabla \ubar|^2 \,dx
  \\&
  + \tau^{-2} \int_{(\partial\domain)_\tau} |\hat u_\e-\ubar|^2 \,dx.
\end{align*}
By the definition of $\hat u_\e$ and the properties of the partition of unity (cf.~\eqref{pu1},\eqref{pu2}), this implies
\begin{align*}
  &\int_\domain |\nabla u_\varepsilon-\nabla \hat u_\e - \nabla ((1-\psi)(\ubar-\hat u_\e))|^2 \,dx
  \\&
  \lesssim \int_\domain |R|^2 \,dx + \sum_{k\in K} \int_{(\partial\domain)_\tau}\eta_k(|\nabla \phi_{k}|^2 + \tau^{-2}|\phi_{k}|^2) \,dx.
\end{align*}
We combine this estimate with the following estimate, the proof of which is postponed to Step~3:
\begin{equation}
  \label{st:correstimat2}  \int_{(\partial\domain)_\tau}\sum_{k\in K}\eta_k\big(|\nabla\phi_k|^2+\tau^{-2}|\phi_k|^2) \,dx
  \leq \Big(C \tau +\mathcal C^2\,\frac{\e^2}{\tau}\mu_d\Big(\frac{\delta}{\e}\Big) \Big) \,\|\nabla u_{\shom}\|_{H^1(\domain)}^2.
\end{equation}
Thus, \eqref{st:001001012} follows from the previous two estimates and \eqref{st:001001011}.

{\it Step 1.3:} 
From the definition of the two-scale expansion $\hat u_\e$ and the Poincar\'e inequality we obtain
\begin{align*}
  \|u_\e-&u_{\shom}\|_{L^2(\domain)}=\bigg|\bigg|u_\e-\hat u_\e+(1-\psi)\sum_{k\in K}\eta_k\phi_k\bigg|\bigg|_{L^2(\domain)}+\bigg|\bigg|\sum_{k\in K}\eta_k\phi_k\bigg|\bigg|_{L^2(\domain)}\\
  &\lesssim \|\nabla u_\varepsilon-\nabla \hat u_\e - \nabla ((1-\psi)(\ubar-\hat u_\e))\|_{L^2(\domain)}+\left(\sum_{k\in K}\int_{B_{2\delta}(k)}|\phi_k|^2\,dx\right)^\frac12 .
\end{align*}
Thus, with \eqref{st:001001012} and \eqref{st:correstimat1} we obtain
\begin{equation*}
  \|u_\e-u_{\shom}\|_{L^2(\domain)}^2  \leq\mathcal C^2\,\|\nabla u_{\shom}\|_{H^1(\domain)}^2\,\Big(\tau+\delta^2+\Big(\frac{\e}{\delta}\Big)^2\mu_d\Big(\frac{\delta}{\e}\Big) + \frac{\e^2}{\tau}\mu_d\Big(\frac{\delta}{\e}\Big)\Big).
\end{equation*}
By setting $\tau:=\delta^2$ and
\begin{equation*}
  \delta:=
  \begin{cases}
    \varepsilon^{1/3}&\text{for }d=1,\\
    \varepsilon^{1/2} |\log \varepsilon|^{1/4}&\text{for }d=2,\\
    \varepsilon^{1/2}&\text{for }d\geq 3,
  \end{cases}
\end{equation*}
this completes the proof of Theorem~\ref{TheoremErrorEstimateDomains}.

{\bf Step 2: Proof of Theorem~\ref{TheoremErrorEstimate} and Theorem~\ref{TheoremErrorEstimate2d} -- the cases with $\domain=\Rd$.} As before we appeal to the two-scale expansion $\hat u_\e$ of Proposition~\ref{PropositionTwoScaleExpansion} and thus recall the definitions of $\{\eta_k\}_{k\in K}$, $\phi_k,\sigma_k$ and $\hat u_\e$ from Step~1. To improve the scaling of the error, we will crucially use the improved estimate on the difference of correctors $\phi_{\xi_k}-\phi_{\xi_{\ell}}$ and $\sigma_{\xi_k}-\sigma_{\xi_{\ell}}$ provided by Corollary~\ref{CorollaryImprovedCorrectorDifferenceBounds}; however, as these estimates grow with a factor of $(1+|\xi_k|^C+|\xi_\ell|^C)$, the $L^\infty$-norm for the gradient $\nabla u_\shom$ appears on the right-hand side of the estimate.
As in Step~1 we deduce from Proposition~\ref{PropositionTwoScaleExpansion} that
\begin{align}
\label{MainEquationDifference}
\nabla \cdot (A(\omega_\varepsilon,\nabla u_\varepsilon)-A(\omega_\varepsilon,\nabla \hat u_\e)) = \nabla \cdot R,\qquad\text{for }d\geq 3,
\end{align}
and
\begin{align}
\label{MainEquationDifference2d}
  ~\nabla \cdot (A(\omega_\varepsilon,\nabla u_\varepsilon)-A(\omega_\varepsilon,\nabla \hat u_\e)) - (u_\varepsilon-\hat u_\e) = \nabla \cdot R +\sum_{k\in K} \eta_k \phi_{k}~\text{for }d=1,2.
\end{align}
{\it Step 2.1:}
We claim that
\begin{equation}\label{EstimateR}
  \int_\Rd |R|^2 \,\leq\,\mathcal C^2\widehat C(\nabla u_{\shom})^2\,\Big(\delta^2+\e^2\mu_d\Big(\frac{\delta}{\e}\Big)\Big)
\end{equation}
where
\begin{equation*}
  \widehat C(\nabla u_{\shom})^2:=\|\nabla u_{\shom}\|_{H^1(\Rd)}^2+(1+\sup_{\Rd}|\nabla u_\shom|)^{2C}\int_{\Rd}|\nabla^2u_{\shom}|^2\,dx.
\end{equation*}
Here and below we denote by $\mathcal C$ a random constant that might change from line to line, but that satisfies the stretched exponential moment bound
\begin{equation*}
  \mathbb E\Big[\exp\Big(\frac{\mathcal C^{\bar\nu}}{C}\Big)\Big]\leq 2,
\end{equation*}
with $\bar\nu$ only depending on $d,m,\lambda,\Lambda,\rho,\nu$ and with $C>0$ only depending on $d,m,\lambda,\Lambda,\rho,\nu$ as well as $\domain$.

Indeed, the local residuum estimate \eqref{EstimateResidual-2} yields
\begin{equation}
  \begin{aligned}
    \int_{\R^d} |R|^2 \,dx \,\lesssim\, &\delta^2 \int_{\Rd}|\nabla^2u_{\shom}|^2\,dx\\
    &+\delta^{-2}\sum_{\ell\in K}\sum_{k\in K\atop |k-\ell|\leq 4\delta} \int_{B_{6\delta}(\ell)}(|\phi_{k}-\phi_\ell|^2 + |\sigma_{k}-\sigma_\ell|^2) \,dx.
  \end{aligned}
\end{equation}
We combine it with the improved corrector estimate
\begin{equation}\label{st:12399491}
  \begin{aligned}
  &\delta^{-2}\sum_{\ell\in K}\sum_{k\in K\atop |k-\ell|\leq 4\delta}\int_{B_{6\delta}(\ell)}(|\phi_{k}-\phi_\ell|^2 + |\sigma_{k}-\sigma_\ell|^2) \,dx\\
  \leq\,&\mathcal{C}^2\,(1+\sup_{\Rd}|\nabla u_\shom|)^{2C}\Big(\int_{\Rd}|\nabla^2u_{\shom}|^2\,dx\Big)\,\varepsilon^2\mu_d\Big(\frac{\delta}{\e}\Big).
\end{aligned}
\end{equation}
The proof of the above estimate is postponed to Step~3. It exploits the regularity assumption \hyperlink{R}{(R)}. The combination of the previous two estimates yields \eqref{EstimateR}.
\smallskip

{\it Step 2.2:}
We fix $\delta$ as follows:
\begin{equation*}
  \delta:=
  \begin{cases}
    \e|\log\e|^{\frac12}&d=2,\\
    \e &d\neq 2.    
  \end{cases}
\end{equation*}
We first note that with this choice, we have for $d\geq 3$
\begin{align}
  \|u_\varepsilon-u_\shom\|_{L^{2d/(d-2)}(\Rd)}
  &\leq \bigg|\bigg|\sum_{k\in K}\eta_k\phi_k\bigg|\bigg|_{L^{2d/(d-2)}(\Rd)}+\,\mathcal C\,\widehat C(\nabla u_{\shom})\,\e,
    \intertext{while for $d=1,2$ we have}
    \|u_\varepsilon-u_\shom\|_{L^{2}(\Rd)}
  \leq\,&\,\bigg|\bigg|\sum_{k\in K}\eta_k\phi_k\bigg|\bigg|_{L^{2}(\Rd)}\\\notag
  &\qquad +\,\mathcal C\,\widehat C(\nabla u_{\shom})\,\left\{
    \begin{aligned}
      &\e&&\text{for }d=1,\\
      &\e|\log\e|^\frac12&&\text{for }d=2.
    \end{aligned}\right.
\end{align}
Indeed, for $d=3$ the estimate follows from the energy energy estimate for the PDE \eqref{MainEquationDifference} combined with \eqref{EstimateR} and the Sobolev embedding. For $d=1,2$ the estimate can be directly obtained from the energy estimate for the PDE \eqref{MainEquationDifference2d} in combination with \eqref{EstimateR}.
As we shall prove in Step~3, we have
  \begin{equation}\label{S:eq10010010}
    \bigg\|\sum_{k\in K} \eta_k\phi_{\xi_k}\bigg\|_{L^p(\Rd)}\leq\mathcal C
\,\widehat C(\nabla u_{\shom})\,    \left\{\begin{aligned}
        &\e^\frac12&&\text{for }d=1\text{ and }p=2,\\
      &\e|\log\e|^\frac12&&\text{for }d=2\text{ and }p=2,\\
      &\e&&\text{for }d\geq 3\text{ and }p=\frac{2d}{d-2}.
    \end{aligned}\right.
  \end{equation}
Combining the previous three estimates, we obtain the statements of Theorem~\ref{TheoremErrorEstimate} and Theorem~\ref{TheoremErrorEstimate2d}.
\bigskip

{\bf Step 3: Proofs of the corrector estimates \eqref{st:correstimat1}, \eqref{st:correstimat2}, \eqref{st:12399491}, and \eqref{S:eq10010010}.} 

{\it Step 3.1 - Proof of \eqref{st:correstimat1}.} 
We first claim that there exists a positive constant $C\lesssim 1$ and $\bar\nu>0$ only depending on $d,m,\lambda,\Lambda$ and $\rho$ such that for $k,\ell\in K$ with $|k-\ell|\leq 4\delta$ we have
\begin{equation}\label{st:correstimat1-p1}
  \int_{B_{6\delta}(\ell)}|\phi_k|^2+|\sigma_k|^2 \,dx \leq \mathcal C_k^2\,\Big(\int_{B_{2\delta}(k)\cap\domain}|\nabla u_{\shom}|^2\,dx\Big)\,\e^2\mu_d\Big(\frac{\delta}{\e}\Big),
\end{equation}
with a random constant $\mathcal C_k$ satisfying
\begin{equation}\label{st:correstimat1-p2}
  \mathbb E\Big[\exp\Big(\frac{\mathcal C_k^{\bar\nu}}{C}\Big)\Big]\leq 2.
\end{equation}
For the argument, we first recall that $\phi_k(\omega_\e,x)=\phi_{\xi_k}(\omega_\e,x)-\fint_{B_\e(k)}\phi_{\xi_k}(\omega_\e,y)\,dy$. Thus, by translation of  $\omega_\e$, we may assume without loss of generality that $k=0$ and $|\ell|\leq 4\delta$. The claim then follows from Corollary~\ref{CorollaryImprovedCorrectorDifferenceBounds} (applied with $r=10\delta$ and $x_0=\ell\in B_{4\delta}(0)$) and the estimate of $|\xi_k|$ via \eqref{est:xi1}.

Summation of \eqref{st:correstimat1-p1} thus yields \eqref{st:correstimat1} with a random constant $\mathcal C$ given by the expression
\begin{equation*}
  \mathcal C=\left(\int_\domain|\nabla u_{\shom}|^2\,dx\right)^{-\frac12}\left(\sum_{\ell,k\in K\atop|\ell-k|\leq 4\delta}\mathcal C_k^2\Big(\int_{B_{2\delta}(k)\cap\domain}|\nabla u_{\shom}|^2\,dx\Big)\,\right)^\frac12.
\end{equation*}
Since $\sum_{\ell,k\in K\atop |\ell-k|\leq 4\delta}\int_{B_{2\delta}(k)\cap\domain}|\nabla u_{\shom}|^2\,dx\lesssim \int_{\domain}|\nabla u_{\shom}|^2\,dx$, we deduce with help of the calculus rules for random variables with stretched exponential moments (see Lemma~\ref{L:calculusstretched}) that $\mathcal C$ satisfies the claimed moments bounds.
\smallskip

{\it Step 3.2 - Proof of \eqref{st:correstimat2}.} 
We first claim that there exists a positive constant $C\lesssim 1$ and $\bar\nu>0$ only depending on $d,m,\lambda,\Lambda$ and $\rho$ such that for $k\in K$ we have
\begin{equation}\label{st:correstimat2-p1}
  \int_{(B_{2\delta}(k)\cap \partial\domain)_{\tau}}|\nabla\phi_k|^2+\tau^{-2}|\phi_k|^2\,dx\,\lesssim\, \frac{\tau}{\delta}\bigg(1+\mathcal C_{k}^2\,\Big(\frac{\e}{\tau}\Big)^2\mu_d\Big(\frac{\delta}{\e}\Big)
 \bigg)\int_{\domain\cap B_{3\delta}(k)}|\nabla u_{\shom}|^2 \,dx
\end{equation}
with a random constant $\mathcal C_k$ satisfying \eqref{st:correstimat1-p2}. For the argument, we first note that $\fint_{B_{2\tau}(y)}|\nabla\phi_k|^2\,dx\lesssim |\xi_k|^2+\tau^{-2}\fint_{B_{4\tau}(y)}|\phi_k|^2\,dx$ thanks to the Caccioppoli inequality. The remaining argument is similar to the one in Step 3.1, covering the set $(B_{2\delta}(k)\cap \partial \domain)_\tau$ with balls of the form $B_{4\tau}(y)$.

Thanks to the properties of the partition of unity \eqref{pu2}, we deduce from \eqref{st:correstimat2-p1},
\begin{eqnarray*}
  &&\sum_{k\in K}\int_{(\partial\domain)_\tau}\eta_k(|\nabla\phi_k|^2+\tau^{-2}|\phi_k|^2)\,dx\\
  &\lesssim& \frac{\tau}{\delta} \sum_{k\in K\atop \dist(k,\partial\Omega)\leq 3\delta}\bigg(1+\mathcal C_{k}^2\,\Big(\frac{\e}{\tau}\Big)^2\mu_d\Big(\frac{\delta}{\e}\Big)\bigg)\Big(\int_{\domain\cap B_{2\delta}(k)}|\nabla u_{\shom}|^2\,dx\Big)\\
  &\lesssim& \frac{\tau}{\delta} \bigg(1+\mathcal C^2 \,\Big(\frac{\e}{\tau}\Big)^2\mu_d\Big(\frac{\delta}{\e}\Big)\bigg)\int_{(\partial\domain)_{5\delta}}|\nabla u_{\shom}|^2\,dx,
\end{eqnarray*}
 with a random constant given by 
\begin{equation*}
  \mathcal C=\left(\int_{(\partial\domain)_{5\delta}}|\nabla u_{\shom}|^2\,dx\right)^{-\frac12}\left(\sum_{k\in K\atop \dist(k,\partial\Omega)\leq 3\delta}\mathcal C_k^2\int_{\domain\cap B_{3\delta}(k)}|\nabla u_{\shom}|^2\,dx\right)^\frac12.
\end{equation*}
As in Step 3.1 we deduce that $\mathcal C$ satisfies the claimed moments bounds. To conclude \eqref{st:correstimat2}, it remains to prove the following trace-type estimate:
\begin{equation}
 \label{EstimateDuhombdry}
  \int_{(\partial\domain)_r}|\nabla u_{\shom}|^2\,dx\lesssim r \|\nabla u_{\shom}\|_{H^1(\domain)}^2
\end{equation}
for $r=5\delta$. In fact, the estimate holds for any $v\in H^1(\domain)$ (instead of $\nabla u_{\shom}$). To see this it suffices to consider a smooth $v:\Rd_+\to\R$ with $\Rd_+:=\R^{d-1}\times(0,\infty)$. Then
\begin{align*}
  \int_{\R^{d-1}\times(0,r)}|v|^2\,dx=&\int_{\R^{d-1}}\int_0^r |v(x',s)|^2\,ds\,dx'\\
  \leq &2r \int_{\R^{d-1}}|v(x',0)|^2\,dx'+2\int_{\R^{d-1}}\int_0^r \bigg|\int_0^s\partial_d v(x',t)\,dt \bigg|^2\,ds\,dx'\\
  \leq &2r \int_{\R^{d-1}}|v(x',0)|^2\,dx'+2r^2\int_{\R^{d-1}}\int_0^r|\partial_d v(x',t)|^2\,dt\,dx'\\
  \lesssim &r \|v\|_{H^1(\domain)}^2,
\end{align*}
where the last line holds thanks to the trace estimate. The case of a general Lipschitz or $C^1$-domain can be reduced to $\mathbb{R}^d_+$ by appealing to a partition of unity and a local straightening of the boundary.

{\it Step 3.3 - Proof of \eqref{st:12399491}.} With the regularity assumption \hyperlink{R}{(R)} at hand, Corollary~\ref{CorollaryImprovedCorrectorDifferenceBounds} in combination with \eqref{est:xi3} yields
\begin{align*}
  &\delta^{-2}\sum_{k\in K\atop |k-\ell|\leq 4\delta}\int_{B_{6\delta}(\ell)}(|\phi_{k}-\phi_\ell|^2 + |\sigma_{k}-\sigma_\ell|^2) \,dx\\
  \leq\,&\, \mathcal{C}_{\ell}^2 (1+\sup_{\Rd}|\nabla u_\shom|)^{2C}\int_{B_{6\delta}(\ell)\cap\domain}|\nabla^2u_{\shom}|^2\,dx~ \varepsilon^2\mu_d\Big(\frac{\delta}{\e}\Big)
\end{align*}
where $\mathcal C_\ell$ denotes a random constant satisfying $\mathbb E[\exp(\tfrac{\mathcal C_\ell^{\bar\nu}}{C})]\leq 2$
with $C, \bar\nu$ only depend on $d,m,\lambda,\Lambda,\rho$ and $\nu$. A summation in $\ell\in K$ and the calculus rules for random variable with stretched exponential moments (see Lemma~\ref{L:calculusstretched}) yield \eqref{st:12399491}.

{\it Step 3.4 - Proof of \eqref{S:eq10010010}.}  The estimate follows from Corollary~\ref{CorollaryImprovedCorrectorDifferenceBounds} and \eqref{est:xi1} by an argument similar to the one in Step 3.1.
\end{proof}

\begin{lemma}\label{L:calculusstretched}
  Let $J$ denote a countable index set. For $j\in J$ let $\mathcal C_j$ be a non-negative random variable with stretched exponential moments of the form
  \begin{equation*}
    \mathbb{E}\bigg[\exp\bigg(\frac{\mathcal{C}_j^{\nu_j}}{C_j}\bigg)\bigg]\leq 2,
  \end{equation*}
  with positive constants $C_j$ and exponents $0< \nu_j\leq 2$. Suppose that $(a_j)_{j\in J}$ is a summable sequence of non-negative numbers. Then the random variable
  \begin{equation*}
    \mathcal C:=\frac{\sum_{j\in J}a_j\mathcal C_j}{\sum_{j\in J}a_j}\qquad\text{satisfies}\qquad    \mathbb{E}\bigg[\exp\bigg(\frac{\mathcal{C}^{\hat \nu}}{\hat C}\bigg)\bigg]\leq 2,
  \end{equation*}
  where $\hat \nu:=\inf_{j\in J}\nu_j$ and $\hat C:=C(\nu_1,\ldots,\nu_J)\sup_{j\in J}C_j^{\hat \nu/\nu_j}$.
\end{lemma}
\begin{proof}
  We first note that there exist universal constants $0<c'\leq C'<\infty$ such that for any non-negative random variable $\mathcal Z$ we have the chain of implications
  \begin{equation*}
    \mathbb E\bigg[\exp(C'\mathcal Z)\bigg]\leq 2 \quad\quad \Rightarrow \quad\quad \forall p\geq 1\,:\,    \mathbb E\big[\mathcal Z^p\big]^\frac{1}{p}\leq p \quad\quad \Rightarrow \quad\quad \mathbb E\bigg[\exp(c'\mathcal Z)\bigg]\leq 2.
  \end{equation*}
  We now estimate
  \begin{eqnarray*}
        \mathbb E\bigg[\Big(\frac{\mathcal C^\nu}{\hat C}\Big)^{p}\bigg]^\frac{1}{\hat \nu p}&=&\frac{1}{\hat C^{1/\hat \nu}}\mathbb E\bigg[\Big(\Big(\frac{\sum_{j\in J}a_j\mathcal C_j}{\sum_{j\in J}a_j}\Big)^{\hat\nu}\Big)^{p}\bigg]^\frac{1}{\hat\nu p}
    \leq
        \frac{1}{\hat C^{1/\hat\nu}}\frac{\sum_{j\in J}a_j\mathbb E \Big[\big(\mathcal C_j^{\hat\nu}\big)^p \Big]^{\frac{1}{\hat\nu p}}}{\sum_{j\in J}a_j}\\
    &\leq&
        \frac{1}{\hat C^{1/\hat \nu}} \frac{\sum_{j\in J}a_j\mathbb E \Big[\big(\mathcal C_j^{\nu_j}\big)^p \Big]^{\frac{1}{\nu_j p}}}{\sum_{j\in J}a_j}
        \leq \frac{1}{\hat C^{1/\hat \nu}} \frac{\sum_{j\in J} a_j (p C_j)^{\frac{1}{\nu_j}} }{\sum_{j\in J}a_j}
        \\
        &\leq& \big(c' p\big)^{\frac{1}{\hat \nu}},
  \end{eqnarray*}
  and thus the claimed estimate follows.
\end{proof}

\subsection{Estimates on the random fluctuations of the RVE approximation for the effective material law}

We next establish the estimates on fluctuations of the representative volume approximation for the effective material law stated in Theorem~\ref{T:RVE}a.

\begin{proof}[Proof of Theorem~\ref{T:RVE}a]
  {To ease the notation we drop the indices $\e$ and $L$ and simply write $\omega$ instead of $\omega_{\e,L}$. }
Consider the random variable
\begin{equation*}
  F(\omega):=A^{\RVE,L}(\omega,\xi)\cdot \Xi=\fint_{[0,L]^d}A(\omega(x),\xi+\nabla\phi_\xi(\omega,x)) \cdot \Xi\,dx.
\end{equation*}
Let $\delta \omega$ denote a periodic infinitesimal perturbation in the sense of Definition~\ref{DefinitionSpectralGap}b. Then for all $L$-periodic parameter fields $\omega$ we have
\begin{equation}\label{T:RVEran:eq1}
  \begin{aligned}
    \delta F(\omega):=&\lim\limits_{t\to 0}\frac{F(\omega+t\delta\omega)-F(\omega)}{t}\\
    =&\fint_{[0,L]^d}\partial_\omega A(\omega(x),\xi+\nabla\phi_\xi(\omega,x))\delta\omega(x)\cdot\Xi
    \\&~~~~~~~~~~~~~~~~~~
    +a_\xi(x)\nabla\delta\phi_\xi(\omega,x) \cdot \Xi\,dx,
  \end{aligned}
\end{equation}
where
\begin{align*}
  a_\xi(x)=&\partial_\xi A(\omega(x),\xi+\nabla\phi_\xi(\omega,x)),
\end{align*}
and where $\delta\phi_\xi=\delta\phi_\xi(\omega,\cdot)$ is the unique ($L$-periodic) solution with mean zero to
\begin{align}
\label{EqDeltaphi}
  -\nabla\cdot (a_\xi\nabla\delta\phi_\xi)=\nabla\cdot(\partial_\omega A(\omega(x),\xi+\nabla\phi_\xi(\omega,x))\delta\omega(x)).
\end{align}
Introducing the unique $L$-periodic solution $h$ with vanishing mean to the PDE
\begin{align}
\label{Eqh}
-\nabla \cdot (a_\xi^{*}\nabla h)=\nabla \cdot (a_\xi^*(x)\Xi),
\end{align}
we deduce by testing \eqref{Eqh} with $\delta \phi_{\xi}$ and testing \eqref{EqDeltaphi} with $h$
\begin{align*}
\delta F(\omega)
=&\fint_{[0,L]^d}\partial_\omega A(\omega(x),\xi+\nabla\phi_\xi(\omega,x))\delta\omega(x)\cdot\Xi\,dx
\\&
+\fint_{[0,L]^d}\partial_\omega A(\omega(x),\xi+\nabla\phi_\xi(\omega,x))\delta\omega(x) \cdot \nabla h\,dx.
\end{align*}
This establishes by \hyperlink{A3}{(A3)}
\begin{align*}
\bigg|\frac{\partial F(\omega)}{\partial \omega}\bigg|
\leq C L^{-d} |\xi+\nabla \phi_\xi| (|\Xi|+|\nabla h|)
\end{align*}
which yields by the $q$-th moment version of the spectral gap inequality in Lemma~\ref{LemmaLqSpectralGap}
\begin{align*}
  &\mathbb{E}_L\Big[\big|F-\mathbb{E}_L[F]\big|^{2q}\Big]^{1/2q}\\
  &\leq C q \e^{\frac{d}{2}}L^{-d} \mathbb{E}_L\bigg[\bigg(\int_{[0,L]^d} \bigg(\fint_{B_\varepsilon(x)} |\xi+\nabla \phi_\xi| (|\Xi|+|\nabla h|)\,d\tilde x\bigg)^2 \,d x \bigg)^q\bigg]^{1/2q}.
\end{align*}
By H\"older's inequality, we infer for any $p>2$
\begin{align*}
&\mathbb{E}_L\Big[\big|F-\mathbb{E}_L[F]\big|^{2q}\Big]^{1/2q}
\\
&\leq C q \Big(\frac{\e}{L}\Big)^{\frac{d}{2}} \mathbb{E}_L\bigg[\bigg(\fint_{[0,L]^d} \bigg|\fint_{B_\varepsilon(x)} |\xi+\nabla \phi_\xi|^2 \,d\tilde x\bigg|^{p/(p-2)} \,dx\bigg)^{q(p-2)/p}
\\&~~~~~~~~~~~~~~~~~~~~~~~~~~~~~~~~~~~~~~~\times
\bigg(\fint_{[0,L]^d} (|\Xi|+|\nabla h|)^p \,dx\bigg)^{2q/p}\bigg]^{1/2q}.
\end{align*}
Bounding the last integral by $C|\Xi|$ by Meyers estimate for \eqref{Eqh} (see Lemma~\ref{LemmaWeightedMeyers}) and using the estimate 
\eqref{EstimateGradByrstar}, we obtain
\begin{align*}
&\mathbb{E}_L\Big[\big|F-\mathbb{E}_L[F]\big|^{2q}\Big]^{1/2q}
\\&
\leq C q \Big(\frac{\e}{L}\Big)^{\frac{d}{2}} |\xi| |\Xi| \, \mathbb{E}_L \bigg[\bigg(\fint_{[0,L]^d} \big|r_{*,L,\xi}(x)\big|^{(d-\delta)p/(p-2)} \,dx\bigg)^{q(p-2)/p}\bigg]^{1/2q}.
\end{align*}
Using stationarity of $r_{*,L,\xi}$ and the moment bound of Lemma~\ref{MomentsMinimalRadius} (which we prove explicitly for the probability distribution $\mathbb{P}$, but which may be established for $\mathbb{P}_L$ analogously; furthermore, while the estimates of Lemma~\ref{MomentsMinimalRadius} are stated for finite $T<\infty$, they are uniform in $T\geq \varepsilon^2$ and therefore also hold in the limit $T\rightarrow \infty$), we deduce for $q$ large enough
\begin{align*}
&\mathbb{E}_L\Big[\big|F-\mathbb{E}_L[F]\big|^{2q}\Big]^{1/2q}
\leq C q^C\Big(\frac{\e}{L}\Big)^{\frac{d}{2}} |\xi| |\Xi|.
\end{align*}
This is the assertion of Theorem~\ref{T:RVE}a.
\end{proof}

\subsection{Estimates for the error introduced by localization}
We next establish the estimates from Lemma~\ref{L:syscorr}, Corollary~\ref{C:syscorr}, and Lemma~\ref{L:syslincorr} for the error introduced in the correctors by the exponential localization on scale $\sqrt{T}$ via the massive term.
We then prove Proposition~\ref{P:RVEsystem}, which estimates the systematic error in the approximation for the effective coefficient $\mathbb{E}[A^{\RVE,\eta,T}]$ introduced by the finite localization parameter $T<\infty$.
\begin{proof}[Proof of Lemma~\ref{L:syscorr}]
We will use the exponential weight $\eta(x):=\exp(-\gamma|x|/\sqrt T)$ with $0<\gamma\ll 1$. Note that 
\begin{equation}\label{L:syscorr:eq1}
  |\nabla\eta|\leq\frac\gamma{\sqrt T}\eta.
\end{equation}
By the localized corrector equation \eqref{EquationLocalizedCorrector} we have
\begin{equation*}
  -\nabla \cdot (A(\omega_\varepsilon,\xi+\nabla \phi_\xi^{2T})-A(\omega_\varepsilon,\xi+\nabla \phi_\xi^T)) + \frac{1}{2T}(\phi_\xi^{2T}-\phi_\xi^T) =\frac{1}{2T}\phi_\xi^{T}.
\end{equation*}
Testing with $(\phi_\xi^{2T}-\phi_\xi^{T})\eta$ and using the monotonicity and Lipschitz continuity of $A$ (see \hyperlink{A1}{(A1)}-\hyperlink{A2}{(A2)}) as well as \eqref{L:syscorr:eq1}, we get
\begin{align*}
  &\int_\Rd\Big( \lambda \big|\nabla \phi_\xi^{2T}-\nabla \phi_\xi^T\big|^2 + \frac{1}{2T}\big|\phi_\xi^{2T}-\phi_\xi^T\big|^2\Big)\eta \,dx
  \\
  &\leq \frac{1}{2T} \int_\Rd\phi_\xi^{T} (\phi_\xi^{2T}-\phi_\xi^T)\eta\,dx+\gamma\int_\Rd \Lambda |\nabla\phi_\xi^{2T}-\nabla\phi_\xi^T| \frac1{\sqrt T} |\phi_\xi^{2T}-\phi_\xi^T|\eta\,dx.
\end{align*}
Choosing $\gamma\leq c(d,m,\lambda,\Lambda)$, we may absorb the second term on the RHS into the LHS to obtain
\begin{equation}\label{L:syscorr:eq2}
  \int_\Rd\Big(\big|\nabla \phi_\xi^{2T}-\nabla \phi_\xi^T\big|^2 + \frac{1}{T}\big|\phi_\xi^{2T}-\phi_\xi^T\big|^2\Big)\eta \,dx
  \leq  \frac{C}{T} \int_\Rd \phi_\xi^{T} (\phi_\xi^{2T}-\phi_\xi^T)\eta\,dx.
\end{equation}
In the following, we treat dimensions $d\leq 2$ and $d\geq 3$ separately. 
In the case $d\leq 2$, by Young's inequality and absorption the previous inequality yields
\begin{equation*}
  \int_{\R^2}\Big(\big|\nabla \phi_\xi^{2T}-\nabla \phi_\xi^T\big|^2 + \frac{1}{T}\big|\phi_\xi^{2T}-\phi_\xi^T\big|^2\Big)\eta \,dx
  \leq    \frac{C}{T}\int_{\R^2}|\phi_\xi^T|^2\eta\,dx.
\end{equation*}
Note that by Proposition~\ref{PropositionCorrectorEstimate} (in connection with a dyadic decomposition of $\R^d$ for $d=1,2$ into the ball $B_{\sqrt T}(0)$ and the annuli $\{2^i\sqrt T\leq |x|<2^{i+1}\sqrt T\}$, $i=0,1,2,\ldots$) we obtain
\begin{equation*}
  \frac{1}{T}\int_{\Rd}|\phi_\xi^T|^2\eta\,dx \leq
  \begin{cases}
  \mathcal C |\xi|^2 \varepsilon
  &\text{for }d=1,
  \\
  \mathcal C |\xi|^2 \varepsilon^2 \Big|\log\frac{\sqrt T}{\varepsilon}\Big|
  &\text{for }d=2.
  \end{cases}
\end{equation*}
Since $\eta \geq \exp(-1)$ on $B_{\sqrt T}$, the claimed estimate follows for $d\leq 2$.
\smallskip

In the case $d\geq 3$ we can use in \eqref{L:syscorr:eq2} the representation $\phi^{T}_{\xi}=\nabla\cdot(\theta^{T}_{\xi}-b)$ for any $b\in \Rmd$ (see \eqref{EquationPotentialField}). We obtain by an integration by parts and by the Cauchy-Schwarz inequality  the estimate
\begin{align*}
&\frac{1}{T} \int_\Rd \phi_\xi^{T} (\phi_\xi^{2T}-\phi_\xi^T)\eta\,dx\\
&\leq\frac1T\int_\Rd|\theta^{T}_{\xi}-b|\Big(|\nabla\phi_{\xi}^{2T}-\nabla\phi_{\xi}^{T}|+\frac{\gamma}{\sqrt T}|\phi_{\xi}^{2T}-\phi_{\xi}^{T}|\Big)\eta\,dx
\end{align*}
With Young's inequality we may absorb the second factor into the LHS of \eqref{L:syscorr:eq2}. We thus obtain
\begin{equation*}
  \int_\Rd\Big(\big|\nabla \phi_\xi^{2T}-\nabla \phi_\xi^T\big|^2 + \frac{1}{T}\big|\phi_\xi^{2T}-\phi_\xi^T\big|^2\Big)\eta \,dx\leq \frac{C}{T^2} \int_\Rd |\theta^{T}_{\xi}-b|^2\eta\,dx.
\end{equation*}
By appealing to Proposition~\ref{PropositionCorrectorEstimate} (in connection with a dyadic decomposition of $\Rd$), and the fact that $\eta\geq \exp(-1)$ on $B_{\sqrt T}$, the claimed estimate follows for $d\geq 3$.
\end{proof}
\begin{proof}[Proof of Corollary~\ref{C:syscorr}]
Set $u:=\phi^{2T}_\xi-\phi^T_\xi$ and note that with $\hat a(\xi):=\int_0^1 \partial_\xi A(\omega_\varepsilon,\xi+(1-s)\nabla\phi_\xi^T+s\nabla \phi_\xi^{2T})\,ds$ we have by \eqref{EquationLocalizedCorrector}
\begin{equation*}
-\nabla\cdot(\hat a\nabla u)+\frac1T u=\frac1{2T}\phi_\xi^{2T}.
\end{equation*}
Applying the Meyers estimate of Lemma~\ref{L:linearMeyersLocalized} to this PDE -- upon rewriting the right-hand side using \eqref{EquationPotentialField} in case $d\geq 3$ -- we obtain for $0\leq p-2\ll 1$
\begin{align*}
&\bigg(\fint_{B_{\sqrt{T}}(x_0)}|\nabla (\phi^{2T}_\xi-\phi^T_\xi)|^p \,dx\bigg)^{1/p}
\\&
\leq
C(d,m,\lambda,\Lambda,p) \bigg(\fint_{B_{2\sqrt{T}}(x_0)}|\nabla u|^2 + \Big|\frac{1}{\sqrt{T}}u\Big|^2 \,dx\bigg)^{1/2}
\\&~~~
+
\begin{cases}
C(d,m,\lambda,\Lambda,p) \bigg(\fint_{B_{2\sqrt{T}}(x_0)} \Big|\frac{1}{\sqrt{T}}\phi_\xi^{2T}\Big|^p \,dx\bigg)^{1/p} &\text{in case }d\leq 2,
\\
C(d,m,\lambda,\Lambda,p) \bigg(\fint_{B_{2\sqrt{T}}(x_0)} \Big|\frac{1}{T}(\theta_\xi^{2T}-b)\Big|^p \,dx\bigg)^{1/p} &\text{in case }d\geq 3.
\end{cases}
\end{align*}
This implies by Proposition~\ref{PropositionCorrectorEstimate} and Lemma~\ref{L:syscorr}
\begin{align*}
&\bigg(\fint_{B_{\sqrt{T}}(x_0)}|\nabla (\phi^{2T}_\xi-\phi^T_\xi)|^p \,dx\bigg)^{1/p}
\\&
\leq \mathcal{C}|\xi| \bigg(\frac{\varepsilon}{\sqrt T}\bigg)^{\frac{d\wedge 4}{2}}
    \begin{cases}
      \big|\log (\sqrt T/\varepsilon) \big| &\text{for }d\in\{2,4\},\\
      1&\text{for }d=3,\ d=1,\text{ and }d\geq 5.
    \end{cases}
\end{align*}
Taking the $p$-th stochastic moment and using stationarity, we conclude.
\end{proof}

\begin{proof}[Proof of Lemma~\ref{L:syslincorr}]
Again, we use the weight $\eta(x):=\exp(-\gamma |x|/\sqrt{T})$ for $0<\gamma\ll 1$.
  We subtract the equations for $\phi^{2T}_{\xi,\Xi}$ and $\phi^{T}_{\xi,\Xi}$ (see \eqref{EquationLocalizedCorrectorLinearized}) to obtain
\begin{eqnarray*}
    -\nabla\cdot \Big(a_\xi^{2T}(\Xi+\nabla\phi^{2T}_{\xi,\Xi})-a_\xi^T(\Xi+\nabla\phi^{T}_{\xi,\Xi})\Big)+\frac1 T(\phi^{2T}_{\xi,\Xi}-\phi^{T}_{\xi,\Xi})=\frac1{2T}\phi^{2T}_{\xi,\Xi}.
\end{eqnarray*}
By adding and subtracting $a_\xi^T(\Xi+\phi^{2T}_{\xi,\Xi})$ and by additionally appealing to the representation $\nabla\cdot(\theta^{2T}_{\xi,\Xi}-b)=\phi^{2T}_{\xi,\Xi}$ for any $b\in \Rmd$ (see \eqref{EquationPotentialFieldLinearized}) in the case $d\geq 3$, we get
  \begin{equation}\label{L:syslincorr:eq1}
    \begin{aligned}
      &-\nabla\cdot \big(a_\xi^{T}(\nabla\phi^{2T}_{\xi,\Xi}-\nabla\phi^{T}_{\xi,\Xi})\big)+\frac1 T(\phi^{2T}_{\xi,\Xi}-\phi^{T}_{\xi,\Xi})\\
      &=\,\nabla\cdot\Big((a^{2T}_\xi-a_\xi^T)(\Xi+\nabla\phi^{2T}_{\xi,\Xi})\Big)+
      \begin{cases}
        \frac1{2T}\phi^{2T}_{\xi,\Xi}&\text{for }d\leq 2,\\
        \nabla\cdot(\frac1{2T}(\theta^{2T}_{\xi,\Xi}-b))&\text{for }d\geq 3.
      \end{cases}
    \end{aligned}
  \end{equation}
Testing the equation with $(\phi^{2T}_{\xi,\Xi}-\phi^{T}_{\xi,\Xi})\eta$ (with $0<\gamma\ll 1$) yields the exponentially localized energy estimate
\begin{align*}
  &\int_{\Rd}\Big(\lambda|\nabla\phi^{2T}_{\xi,\Xi}-\nabla\phi^{T}_{\xi,\Xi}|^2+\frac{1}T|\phi^{2T}_{\xi,\Xi}-\phi^{T}_{\xi,\Xi}|^2\Big)\eta\,dx\\
  &\leq C\int_\Rd |a^{2T}_\xi-a_\xi^T|^2|\Xi+\nabla\phi^{2T}_{\xi,\Xi}|^2\eta\,dx
     +C\begin{cases}
      \frac1{T}\int_\Rd|\phi^{2T}_{\xi,\Xi}|^2\eta\,dx&\text{for }d\leq 2,\\
      \frac1{T^2}\int_\Rd|\theta^{2T}_{\xi,\Xi}-b|^2\eta\,dx&\text{for }d\geq 3.
    \end{cases}
\end{align*}
By taking the expectation and exploiting stationarity  of the LHS, we get
\begin{align*}
  &\mathbb E\Big[|\nabla\phi^{2T}_{\xi,\Xi}-\nabla\phi^{T}_{\xi,\Xi}|^2+\frac{1}T|\phi^{2T}_{\xi,\Xi}-\phi^{T}_{\xi,\Xi}|^2\Big]\\
  &\leq C\mathbb E\left[ |a^{2T}_\xi-a_\xi^T|^2|\Xi+\nabla\phi^{2T}_{\xi,\Xi}|^2\right]
+C\left\{\begin{aligned}\displaystyle
      &\frac1{T}\mathbb E\left[|\phi^{2T}_{\xi,\Xi}|^2\right]&\text{for }d\leq 2,\\\displaystyle
      &\frac1{T^2}\mathbb E\left[\int_\Rd|\theta^{2T}_{\xi,\Xi}-b|^2\frac{\eta}{\sqrt T^d}\,dx\right]&\text{for }d\geq 3.
    \end{aligned}\right.
\end{align*}
The second term on the RHS can be estimated with help of Proposition~\ref{PropositionLinearizedCorrectorEstimate}.

We estimate the first term on the RHS: Since $\partial_\xi A$ is Lipschitz by assumption \hyperlink{R}{(R)} and since $a_\xi^T(x)=A(\omega_\varepsilon(x),\xi+\nabla \phi_\xi^T)$, we have $|a_\xi^{2T}-a_\xi^T|\leq C |\nabla\phi_\xi^{2T}-\nabla\phi_\xi^T|$. Hence, with H\"older's inequality with exponents $0<p-1 \leq c(d,m,\lambda,\Lambda)$ and $\frac p{p-1}$, the bound on $\Xi+\nabla\phi_{\xi,\Xi}^T$ from \eqref{EstimateLinearCorrectorLinfty} and Lemma~\ref{MomentsMinimalRadiusLinearized}, and Corollary~\ref{C:syscorr}, we get
\begin{align*}
  &\mathbb E\left[ |a^{2T}_\xi-a_\xi^T|^2|\Xi+\nabla\phi^{2T}_{\xi,\Xi}|^2\right]\\
  &\leq C \mathbb E\left[ |\nabla\phi_\xi^{2T}-\nabla\phi_\xi^{T}|^{2p}\right]^\frac{1}{p}\mathbb E\left[|\Xi+\nabla\phi^{2T}_{\xi,\Xi}|^{2\frac{p}{p-1}}\right]^\frac{p-1}{p}\\
  &\leq C |\Xi|^2(1+|\xi|)^{C}
    |\xi|^2\bigg(\frac{\varepsilon}{\sqrt T}\bigg)^{d\wedge 4}
    \times \begin{cases}
      \big|\log (\sqrt T/\varepsilon)\big|^2 &\text{for }d\in\{2,4\},\\
      1&\text{for }d=3,\ d=1,\text{ and }d\geq 5.
    \end{cases}
\end{align*}
\end{proof}

\begin{proof}[Proof of Proposition~\ref{P:RVEsystem}]
{\bf Step 1: Proof of (a).} This is a direct consequence of stationarity of $\mathbb P$ (see assumption \hyperlink{P1}{(P1)}) and stationarity of the random field $A(\omega_\varepsilon,\xi+\nabla\phi^T_\xi) \cdot \Xi-\frac1T\phi_\xi^T\phi_{\xi,\Xi}^{*,T}$, the latter of which is a consequence of the former.
\smallskip

{\bf Step 2: Proof of (b).} First note that it suffices to prove for any $T\geq 2\varepsilon^2$ the estimate
\begin{equation}\label{P:RVEsystem:mainestimate}
  \begin{aligned}
    &\left|\mathbb E\left[A^{\RVE,\eta,2T}(\xi) \cdot \Xi\right]-\mathbb E\left[A^{\RVE,\eta,T}(\xi) \cdot \Xi\right]\right|\\
    &\leq C
    (1+|\xi|)^{2C}|\xi||\Xi| \bigg(\frac{\varepsilon}{\sqrt T}\bigg)^{d\wedge 4}
    \begin{cases}
      \big|\log (\sqrt T/\varepsilon)\big|^2&\text{for }d=2\text{ and }d=4,\\
      1&\text{for }d=3,\ d=1,\text{ and }d\geq 5.
    \end{cases}
  \end{aligned}
\end{equation}
Indeed, the claimed estimate then follows by rewriting the systematic localization error as a telescopic sum,
\begin{align*}
  &A_{\shom}(\xi) \cdot \Xi \, - \,\mathbb E\left[A^{\RVE,\eta,T}(\xi) \cdot \Xi\right]
 \\&=\sum_{i=0}^\infty\bigg(\mathbb E\left[A^{\RVE,\eta,2^{i+1}T}(\xi) \cdot \Xi\right]-\mathbb E\left[A^{\RVE,\eta,2^iT}(\xi) \cdot \Xi\right]\bigg),
\end{align*}
which holds since
\begin{equation*}
  \lim_{T\rightarrow \infty} \mathbb E\left[A^{\RVE,\eta,T}(\xi) \cdot \Xi\right]= A_{\shom}(\xi) \cdot \Xi\qquad\mathbb P\text{-almost surely}.
\end{equation*}
We present the argument for \eqref{P:RVEsystem:mainestimate}.
In view of (a), we may assume without loss of generality that the weight $\eta$ satisfies
  \begin{equation*}
    \supp\eta\subset B_{\sqrt T},\qquad \int_\Rd\eta\,dx=1,\qquad |\eta|+\sqrt T|\nabla\eta| \leq C(d) \sqrt T^{-d}.
  \end{equation*}
Let $\phi^{*,T}_{\xi,\Xi}$ denote the localized, linearized, adjoint corrector (i.\,e.\ the $T$-localized homogenization corrector associated with the linear elliptic PDE with coefficient field $(a_\xi^{T})^*$). The localized corrector equation \eqref{EquationLocalizedCorrector} yields 
\begin{equation}\label{eq:RVE:1}
  \begin{aligned}
    &-\int_\Rd \eta\Big(A(\omega_\varepsilon,\xi+\nabla\phi^{2T}_\xi)-A(\omega_\varepsilon,\xi+\nabla\phi^{T}_\xi)\Big) \cdot \nabla\phi^{*,T}_{\xi,\Xi}\,dx\\
    &=\int_\Rd \Big(A(\omega_\varepsilon,\xi+\nabla\phi^{2T}_\xi)-A(\omega_\varepsilon,\xi+\nabla\phi^{T}_\xi)\Big) \cdot (\phi^{*,T}_{\xi,\Xi}\nabla\eta)+\eta\Big(\frac{1}{2T}\phi^{2T}_\xi-\frac1T\phi^T_\xi\Big)\phi^{*,T}_{\xi,\Xi}\,dx.
  \end{aligned}
\end{equation}
Combined with the definition of the localized RVE approximation in Definition~\ref{D:RVET}, we get
\begin{align*}
  &\Big(A^{\RVE,\eta,2T}(\xi)-  A^{\RVE,\eta,T}(\xi)\Big) \cdot \Xi\\
  &=\int_\Rd \eta\Bigg(\Big(A(\omega_\varepsilon,\xi+\nabla\phi^{2T}_\xi)-A(\omega_\varepsilon,\xi+\nabla\phi^{T}_\xi)\Big) \cdot \Xi
     -\frac1{2T}\phi_\xi^{2T}\phi_{\xi,\Xi}^{*,2T}
     +\frac1T\phi_\xi^{T}\phi_{\xi,\Xi}^{*,T}\Bigg)\,dx\\
  &=\int_\Rd \eta\Bigg(\Big(A(\omega_\varepsilon,\xi+\nabla\phi^{2T}_\xi)-A(\omega_\varepsilon,\xi+\nabla\phi^{T}_\xi)\Big) \cdot (\Xi+\nabla\phi^{*,T}_{\xi,\Xi})
     -\frac1{2T}\phi_\xi^{2T}\phi_{\xi,\Xi}^{*,2T}
     +\frac1T\phi_\xi^{T}\phi_{\xi,\Xi}^{*,T}\Bigg)\,dx\\
  &~~~+\int_\Rd \Big(A(\omega_\varepsilon,\xi+\nabla\phi^{2T}_\xi)-A(\omega_\varepsilon,\xi+\nabla\phi^{T}_\xi)\Big) \cdot (\phi^{*,T}_{\xi,\Xi}\nabla\eta)+\eta\Big(\frac{1}{2T}\phi^{2T}_\xi-\frac1T\phi^T_\xi\Big)\phi^{*,T}_{\xi,\Xi}\,dx\\
  &=\int_\Rd \eta\Bigg(\Big(A(\omega_\varepsilon,\xi+\nabla\phi^{2T}_\xi)-A(\omega_\varepsilon,\xi+\nabla\phi^{T}_\xi)\Big) \cdot (\Xi+\nabla\phi^{*,T}_{\xi,\Xi})
     -\frac1{2T}\phi_\xi^{2T}(\phi_{\xi,\Xi}^{*,2T}-\phi_{\xi,\Xi}^{*,T})\Bigg)\,dx\\
  &~~~+\int_\Rd \Big(A(\omega_\varepsilon,\xi+\nabla\phi^{2T}_\xi)-A(\omega_\varepsilon,\xi+\nabla\phi^{T}_\xi)\Big) \cdot (\phi^{*,T}_{\xi,\Xi}\nabla\eta)\,dx.
\end{align*}
We subtract the linearized corrector equation \eqref{EquationLocalizedCorrectorLinearized} in its form for the adjoint coefficient field and the corresponding corrector
\begin{equation*}
  \int_\Rd  a_\xi^{T,*}\big(\Xi+\nabla \phi_{\xi,\Xi}^{*,T}\big) \cdot \nabla\big(\eta(\phi^{2T}_\xi-\phi^T_\xi)\big)+\frac{1}{T}\int_\Rd \eta\phi_{\xi,\Xi}^{*,T}(\phi_\xi^{2T}-\phi_\xi^T)\,dx=0,
\end{equation*}
where $a_\xi^{T,*}:=(\partial_\xi A(\xi+\nabla \phi_\xi^T))^*$. We get
\begin{equation}\label{prop:RVEsys:eqS1}
  \begin{aligned}
    &\Big(A^{\RVE,\eta,2T}(\xi)-  A^{\RVE,\eta,T}(\xi)\Big) \cdot \Xi\\
    &=\int_\Rd \eta\Big(A(\omega_\varepsilon,\xi+\nabla\phi^{2T}_\xi)-A(\omega_\varepsilon,\xi+\nabla\phi^{T}_\xi)-a_\xi^{T}(\nabla\phi_\xi^{2T}-\nabla\phi_\xi^T)\Big) \cdot (\Xi+\nabla\phi^{*,T}_{\xi,\Xi})\,dx
    \\
    &~~~~+\int_\Rd \Big(A(\omega_\varepsilon,\xi+\nabla\phi^{2T}_\xi)-A(\omega_\varepsilon,\xi+\nabla\phi^{T}_\xi)\Big)\cdot \phi^{*,T}_{\xi,\Xi}\nabla\eta
    \\&~~~~~~~~~~~~~~~~~~~~~~~~~~~~~~~~~~~~~
    -a_\xi^T(\Xi+\nabla \phi_{\xi,\Xi}^{*,T}) \cdot (\phi_\xi^{2T}-\phi_\xi^T)\nabla \eta \,dx
    \\
    &~~~~-\int_\Rd \eta\Big(\frac1{2T}\phi_\xi^{2T}(\phi_{\xi,\Xi}^{*,2T}-\phi_{\xi,\Xi}^{*,T})+\frac1T\phi_{\xi,\Xi}^{*,T}(\phi_\xi^{2T}-\phi_\xi^T)\Big)\,dx.
  \end{aligned}
\end{equation}
We take the expectation of this identity and note that the expectation of the second integral on the right-hand side vanishes: Indeed, since it is of the form $\mathbb E\big[\int_\Rd B\nabla\eta\big]$ where $B$ is a stationary random field and $\eta$ is compactly supported, we have $\mathbb E\big[\int_\Rd B\nabla\eta\big]=\mathbb E\big[B\big]\int_\Rd \nabla\eta=0$. Moreover, for the first term on the RHS of \eqref{prop:RVEsys:eqS1} we appeal to the uniform bound on $\partial_\xi^2 A$ from assumption \hyperlink{R}{(R)} in form of (recall that $a_\xi^T=\partial_\xi A(\omega_\varepsilon,\xi+\nabla \phi_\xi^T)$)
\begin{equation*}
 \big|A(\omega_\varepsilon,\xi+\nabla \phi_\xi^{2T})-A(\omega_\varepsilon,\xi+\nabla \phi_\xi^T)-a_\xi^T(\nabla \phi_\xi^{2T}-\nabla\phi_\xi^{T})\big|
\leq C \big|\nabla \phi_\xi^{2T}-\nabla\phi_\xi^T\big|^2.
\end{equation*}
We thus get
\begin{equation}\label{prop:RVEsys:eq3}
  \begin{aligned}
    &\Big|\mathbb E\Big[\Big(A^{\RVE,\eta,2T}(\xi)-  A^{\RVE,\eta,T}(\xi)\Big) \cdot \Xi\Big]\Big|\\
    &\qquad\leq C \mathbb E\Big[|\nabla\phi_\xi^{2T}-\nabla\phi_\xi^{T}|^2|\Xi+\nabla\phi^{*,T}_{\xi,\Xi}|\Big]\\
    &\qquad\quad+\ \Bigg|\mathbb E\Bigg[\int_\Rd \eta\Big(\frac1{2T}\phi_\xi^{2T}(\phi_{\xi,\Xi}^{*,2T}-\phi_{\xi,\Xi}^{*,T})+\frac1T\phi_{\xi,\Xi}^{*,T}(\phi_\xi^{2T}-\phi_\xi^T)\Big)\,dx\Bigg]\Bigg|\\
    &\qquad=:I_1+I_2.
  \end{aligned}
\end{equation}
To estimate $I_1$  we first apply H\"older's inequality with exponents $p$ and $\frac{p}{p-1}$ with $0<p-1 \leq c(d,m,\lambda,\Lambda)$ to obtain
\begin{align*}
I_1 \leq C \mathbb{E}\big[\big|\nabla \phi_\xi^{2T} - \nabla \phi_\xi^{T}\big|^{2p}\big]^\frac{1}{p} \mathbb{E}\big[\big|\Xi+\nabla \phi_{\xi,\Xi}^{T}\big|^{\frac{p}{p-1}}\big]^\frac{p-1}{p}.
\end{align*}
We then appeal to Corollary~\ref{C:syscorr} and the moment bound on the linearized corrector from \eqref{EstimateLinearCorrectorLinfty} as well as Lemma~\ref{MomentsMinimalRadiusLinearized} to deduce
\begin{equation*}
  I_1 \leq C
  |\xi|^2(1+|\xi|)^C|\Xi|
  \bigg(\frac{\varepsilon}{\sqrt T}\bigg)^{d\wedge 4}
  \begin{cases}
    \big|\log (\sqrt T/\varepsilon)\big|^2 &\text{for }d\in\{2,4\},\\
    1&\text{for }d=3,\ d=1,\text{ and }d\geq 5.\\
  \end{cases}
\end{equation*}
Regarding $I_2$ we distinguish the cases $d\leq 2$ and $d\geq 3$. In the case $d= 2$,  we apply Lemma~\ref{L:syslincorr}, Lemma~\ref{L:syscorr}, as well as Proposition~\ref{PropositionCorrectorEstimate} and Proposition~\ref{PropositionLinearizedCorrectorEstimate} to obtain
\begin{align*}
  I_2
  &\leq C\mathbb E\left[\frac1{T}|\phi_\xi^{2T}|^2\right]^\frac12\mathbb E\left[\frac1T|\phi_{\xi,\Xi}^{*,2T}-\phi_{\xi,\Xi}^{*,T}|^2\right]^\frac12+
             \mathbb E\left[\frac1T|\phi_{\xi,\Xi}^{*,T}|^2\right]^\frac12\mathbb E\left[\frac1T|\phi_\xi^{2T}-\phi_\xi^T|^2\right]^\frac12\\
  &\leq C(1+|\xi|)^{C}|\xi||\Xi| \varepsilon^2 \sqrt T^{-2}
               \big|\log (\sqrt T/\varepsilon)\big|^2,
\end{align*}
and in case $d=1$ we proceed similarly.

In the case $d\geq 3$ we  appeal to the representation of $\smash{\phi_{\xi}^{2T}}$ and $\smash{\phi_{\xi,\Xi}^{*,T}}$ as $\smash{\nabla\cdot \theta_{\xi}^{2T}}$ and $\smash{\nabla\cdot \theta_{\xi,\Xi}^{*,2T}}$ by \eqref{EquationPotentialField} and \eqref{EquationPotentialFieldLinearized}, respectively. To shorten the notation, in the following we assume without loss of generality that $\smash{\fint_{B_{\sqrt T}}\theta_\xi^{2T}=\fint_{B_{\sqrt T}}\theta_{\xi,\Xi}^{*,T}=0}$. An integration by parts thus yields
\begin{eqnarray*}
  I_2&=&\Bigg|\mathbb E\Bigg[\int_\Rd \eta\Big(\frac1{2T}\theta_\xi^{2T}\cdot(\nabla\phi_{\xi,\Xi}^{*,2T}-\nabla\phi_{\xi,\Xi}^{*,T})+\frac1T\theta_{\xi,\Xi}^{*,T}\cdot(\nabla\phi_\xi^{2T}-\nabla\phi_\xi^T)\Big)\,dx\\
     &&\qquad +\int_\Rd \nabla\eta\cdot \Big(\frac1{2T}\theta_\xi^{2T}(\phi_{\xi,\Xi}^{*,2T}-\phi_{\xi,\Xi}^{*,T})+\frac1T\theta_{\xi,\Xi}^{*,T}(\phi_\xi^{2T}-\phi_\xi^T)\Big)\,dx\Bigg]\Bigg|.
\end{eqnarray*}
With the properties of $\eta$ (in particular, $|\nabla\eta|\lesssim\sqrt T^{-d-1}$ and $\supp \eta \subset B_{\sqrt{T}}$), by the Cauchy-Schwarz inequality, and by stationarity of the localized correctors, we get
\begin{align*}
I_2&
\leq \frac{C}{T}
\mathbb E\Bigg[\fint_{B_{\sqrt T}}|\theta_\xi^{2T}|^2\,dx\Bigg]^\frac12 \mathbb E\bigg[|\nabla\phi_{\xi,\Xi}^{*,2T}-\nabla\phi_{\xi,\Xi}^{*,T}|^2 +\frac{1}{T}|\phi_{\xi,\Xi}^{*,2T}-\phi_{\xi,\Xi}^{*,T}|^2
\bigg]^\frac12
\\&~~~~
+\frac{C}{T} \mathbb E\Bigg[\fint_{B_{\sqrt T}}|\theta_{\xi,\Xi}^{*,2T}|^2\,dx\Bigg]^\frac12\mathbb E\Bigg[|\nabla\phi_{\xi}^{2T}-\nabla\phi_{\xi}^{T}|^2+\frac1T|\phi_{\xi}^{2T}-\phi_{\xi}^{T}|^2\Bigg]^\frac12.
\end{align*}
By appealing to Proposition~\ref{PropositionCorrectorEstimate} and Lemma~\ref{L:syslincorr} for the first term and to Proposition~\ref{PropositionLinearizedCorrectorEstimate} and Lemma~\ref{L:syscorr} for the second term, we obtain
\begin{equation*}
  I_2\leq C
  (1+|\xi|)^{C}|\xi||\Xi|\bigg(\frac{\varepsilon}{\sqrt T}\bigg)^{d\wedge 4}
  \begin{cases}
    1 &\text{for }d=3 \text{ and }d\geq 5,\\
    \big|\log (\sqrt T/\varepsilon)\big|^2 &\text{for }d=4.
  \end{cases}
\end{equation*}

Plugging in the estimates on $I_1$ and $I_2$ into \eqref{prop:RVEsys:eq3}, this establishes the estimate on the localization error for the representative volume element method for $\mathbb{P}$.

{\bf Step 3: Proof of (c).}
For the periodized probability distribution $\mathbb{P}_L$, one may proceed analogously to (b), deriving an error estimate on $\mathbb{E}_L[A^{\RVE,\eta,T}]-\mathbb{E}_L[A^{\RVE,L}]$.
\end{proof}

\subsection{Coupling error for RVEs}

\begin{proof}[Proof of Lemma~\ref{L:couperror}]
  To shorten the presentation we set $\widehat\omega:=\pi_L\widetilde\omega$ and similarly mark quantities that are associated with $\widehat\omega$, i.e.,
  \begin{align*}
    &\widehat\phi_\xi^T(x)=\phi_\xi^T(\widehat\omega,x),\qquad \widehat\phi_{\xi,\Xi}^T(x):=\phi_{\xi,\Xi}^T(\widehat\omega,x),
    \qquad \widehat a_\xi^T(x):=A(\widehat\omega,\xi+\nabla \widehat \phi_\xi^T).
  \end{align*}
  Moreover, we shall use the following notation for the differences
  \begin{equation*}
    \hat\delta\phi_\xi^T:=\phi_\xi^T-\widehat\phi_\xi^T,\qquad     \hat\delta\phi_{\xi,\Xi}^T:=\phi_{\xi,\Xi}^T-\widehat\phi_{\xi,\Xi}^T;
  \end{equation*}
  (we use the symbol $\hat\delta$ to distinguish the quantity from the sensitivities considered in Sections~\ref{SectionCorrectorEstimates} and \ref{SectionLinearizedCorrectorEstimates}).
  In the proof we make use of the exponential test-functions
  \begin{equation*}
    \eta(x):=\exp(-\gamma|x|/\sqrt T)
  \end{equation*}
  where $0<\gamma\ll 1$ is chosen such that the exponential localization estimate Lemma~\ref{L:exploc} applies.
  
  \noindent
  {\bf Step 1. Estimate for $\hat\delta\phi^T_\xi$.}   
  We claim that  
  \begin{equation}\label{L:coup:step1}
    \sup_{x_0\in B_{L/8}}\int_{\Rd}\Big(|\nabla\hat\delta\phi^T_\xi|^2+\frac1T|\hat\delta\phi^T_\xi|^2\Big)\eta(x-x_0)\,dx\leq C \sqrt T^d\exp(-\tfrac{\gamma}{16}L/\sqrt T)|\xi|^2.
  \end{equation}
  Indeed, by subtracting the equations for $\phi_\xi^T$ and $\widehat\phi_\xi^T$, we find that
  \begin{equation}
  \label{L:coup:eq0}
    -\nabla\cdot (a(x)\nabla\hat\delta\phi_\xi^T)+\frac1 T\hat\delta\phi_\xi^T=\nabla\cdot F,
  \end{equation}
  where 
  \begin{align*}
    a(x)&:=\int_0^1 \partial_\xi A(\tilde \omega,\xi+(1-s)\nabla \phi_\xi^T+s\nabla\widehat\phi_\xi^T)\,ds,
    \\
    F&:=A(\tilde \omega,\xi+\nabla\widehat\phi^\xi_T)-A(\widehat \omega,\xi+\nabla\widehat\phi^\xi_T).
  \end{align*}
  By the exponentially localized energy estimate of Lemma~\ref{L:exploc} we have
  \begin{equation*}
    \int_\Rd\Big(|\nabla\hat\delta\phi^T_\xi|^2+\frac1T|\hat\delta\phi^T_\xi|^2\Big)\eta(x-x_0)\,dx\leq C \int_\Rd|F|^2\eta(x-x_0)\,dx.
  \end{equation*}
  Since $|x_0|\leq \frac{L}{8}$, for all $x$ with $|x|\geq \frac{L}{4}$ the estimate $|x-x_0|\geq \frac L{16} +\frac 12|x-x_0|$ applies. This yields for such $x$
  \begin{equation}\label{L:coup:eq1}
    \begin{aligned}
      \eta(x-x_0)\leq
      \exp(-\tfrac{\gamma}{16}L/\sqrt T)\exp(-\tfrac\gamma 2 |x-x_0|/\sqrt
      T)
    \end{aligned}
  \end{equation}
  As $\hat \omega(x)=\tilde \omega(x)$ holds for $|x|\leq \frac{L}{4}$, we see that $F$ vanishes on $B_{L/4}$. We thus obtain
  \begin{align*}
    \int_\Rd|F|^2\eta(x-x_0)\,dx
    &\leq C \exp(-\tfrac{\gamma}{16}L/\sqrt T)\int_\Rd|\xi+\nabla\widehat\phi_T|^2\exp(-\tfrac{\gamma}{2}|x-x_0|/\sqrt T)\,dx\\
    &\leq C \exp(-\tfrac{\gamma}{16}L/\sqrt T)\sqrt T^d|\xi|^2,
  \end{align*}
  where for the last estimate we appealed to the localized energy estimate for $\widehat\phi_\xi^T$, see Lemma~\ref{L:exploc}. We conclude that \eqref{L:coup:step1} holds. For further reference, we note that we may similarly derive
  \begin{equation}\label{L:coup:step1p}
    \sup_{x_0\in B_{L/8}}\int_{\Rd}\Big(|\nabla\hat\delta\phi^T_\xi|^p+\Big|\frac{1}{\sqrt{T}}\hat\delta\phi^T_\xi\Big|^p\Big)\eta(x-x_0)\,dx\leq C \sqrt T^d\exp(-\tfrac{\gamma}{16}L/\sqrt T)|\xi|^p
  \end{equation}
  for some $p=p(d,m,\lambda,\Lambda)>2$ by applying the Meyers estimate of Lemma~\ref{L:linearMeyersLocalized} to \eqref{L:coup:eq0} with the dyadic decomposition $\Rd=B_{\sqrt{T}}\cup \bigcup_{k=1}^\infty (B_{2^k\sqrt{T}}\setminus B_{2^{k-1}\sqrt{T}})$, the estimate \eqref{L:coup:step1}, and the bound
  \begin{align*}
    \int_\Rd|F|^p\eta(x-x_0)\,dx
    &\leq C \exp(-\tfrac{\gamma}{16}L/\sqrt T)\int_\Rd|\xi+\nabla\widehat\phi_T|^p\exp(-\tfrac{\gamma}{2}|x-x_0|/\sqrt T)\,dx\\
    &\leq C \exp(-\tfrac{\gamma}{16}L/\sqrt T)\sqrt T^d|\xi|^p.
  \end{align*}
  Note that in the last step of of the last inequality we have again used the Meyers estimate of Lemma~\ref{L:linearMeyersLocalized} together with the localized energy estimate of Lemma~\ref{L:exploc} and a dyadic decomposition.
  \smallskip

  \noindent
  {\bf Step 2. Estimate for $\hat\delta\phi_{\xi,\Xi}^T$.} 
  We claim that there exists $q=q(d,m,\lambda,\Lambda)$ such that for all $x_0\in B_{\frac L 8}$ we have
  \begin{equation}\label{L:coup:step2}
    \begin{aligned}
      &\int_{\Rd}\Big(|\nabla\hat\delta\phi^T_{\xi,\Xi}|^2+\frac1T|\hat\delta\phi^T_{\xi,\Xi}|^2\Big)\eta(x-x_0)\,dx\\
      &\leq C \sqrt T^d\exp(-\tfrac{\gamma}{32}L/\sqrt
      T)(1+|\xi|^2)\|\Xi+\nabla\phi_{\xi,\Xi}^T\|_{q,T,x_0}^2,
    \end{aligned}
  \end{equation}
  where 
  \begin{equation*}
    \|\Xi+\nabla\phi_{\xi,\Xi}^T\|_{q,T,x_0}:=\left(\sqrt T^{-d}\int_\Rd|\Xi+\nabla\phi_{\xi,\Xi}^T|^{q}\exp(-\tfrac\gamma 2|x-x_0|/\sqrt T)\,dx\right)^\frac{1}{q}.
  \end{equation*}
  Indeed, by subtracting the equations \eqref{EquationLocalizedCorrectorLinearized} for $\phi_{\xi,\Xi}^T$ and $\widehat\phi_{\xi,\Xi}^T$, we get
  \begin{align*}
    -\nabla\cdot (\widehat a_\xi^T\nabla\hat\delta\phi_{\xi,\Xi}^T)+\frac1T \hat\delta\phi_{\xi,\Xi}^T =\nabla\cdot \big((a_\xi^T-\widehat a_\xi^T)(\Xi+\nabla\phi_{\xi,\Xi}^T)\big).
  \end{align*}
  Note that
  \begin{eqnarray*}
    a_\xi^T-\widehat a_\xi^T&=&\partial_\xi A(\tilde\omega,\xi+\nabla\phi_\xi^T)-\partial_\xi  A(\widehat\omega,\xi+\nabla\widehat\phi_\xi^T)\\
    &=&\big(\partial_\xi A(\tilde\omega,\xi+\nabla\phi_\xi^T)-\partial_\xi A(\widehat\omega,\xi+\nabla\phi_\xi^T)\big)
    \\&&+\big(\partial_\xi A(\widehat\omega,\xi+\nabla\phi_\xi^T)-\partial_\xi A(\widehat\omega,\xi+\nabla\widehat\phi_\xi^T)\big)\\
    &=:&F_1+F_2.
  \end{eqnarray*}
  By exponential localization in form of Lemma~\ref{L:exploc}, the Lipschitz continuity of $\partial_\xi A$ (see \hyperlink{R}{(R)}), the fact that $F_1$ vanishes on $B_{\frac L4}$ by $\tilde\omega=\widehat\omega$ on $B_{\frac L4}$, and the uniform bound on $\partial_\xi A$ from \hyperlink{A2}{(A2)}, we get
  \begin{equation}\label{L:coup:eq2}
    \begin{aligned}
      &\int_\Rd\Big(|\nabla\hat\delta\phi^T_{\xi,\Xi}|^2+\frac1T|\hat\delta\phi^T_{\xi,\Xi}|^2\Big)\eta(x-x_0)\,dx\\
      &\leq C \int_\Rd|F_1|^2|\Xi+\nabla\phi_{\xi,\Xi}^T|^2\eta(x-x_0)\,dx+C\int_\Rd|F_2|^2|\Xi+\nabla\phi_{\xi,\Xi}^T|^2\eta(x-x_0)\,dx\\
      &\leq C \int_{\{|x|>\frac
        L4\}} |\Xi+\nabla\phi_{\xi,\Xi}^T|^2\eta(x-x_0)\,dx\\&~~~~+C\int_\Rd|\nabla\hat\delta\phi_{\xi}^T|^2|\Xi+\nabla\phi_{\xi,\Xi}^T|^2\eta(x-x_0)\,dx.
    \end{aligned}
  \end{equation}
  We estimate the first term on the RHS by H\"older's inequality as
  \begin{align*}
    &\int_{\{|x|>\frac L4\}}|\Xi+\nabla\phi_{\xi,\Xi}^T|^2\eta(x-x_0)\,dx\\
    &\leq C(d,m,\lambda,\Lambda,q) \sqrt T^d\exp\big(-\tfrac{\gamma}{16}L/\sqrt T\big) \|\Xi+\nabla\phi_{\xi,\Xi}^T\|^2_{q,T,x_0}.
  \end{align*}
  Next, we estimate the second term on the RHS in \eqref{L:coup:eq2}.
Using H\"older's inequality with exponents $p/2$ and $\frac{p}{p-2}$ (with $0<p-2\ll 1$), setting $q:=\frac{2p}{p-2}$, and recalling \eqref{L:coup:step1p} from Step 1, we get
  \begin{align*}
    &\int_\Rd|\nabla\hat\delta\phi_{\xi}^T|^2|\Xi+\nabla\phi_{\xi,\Xi}^T|^2\eta(x-x_0)\,dx \\
    &\leq C \sqrt T^d\left(\sqrt T^{-d}\int_{\R^d}|\nabla\hat\delta\phi_\xi^T|^{p}\eta(x-x_0)\,dx\right)^\frac 2p    \|\Xi+\nabla\phi_{\xi,\Xi}^T\|^2_{q,T,x_0}\\
    &\leq C \sqrt T^d\exp(-\tfrac{\gamma}{32}L/\sqrt T)|\xi|^2\|\Xi+\nabla\phi_{\xi,\Xi}^T\|^2_{q,T,x_0}.
  \end{align*}
  This completes the argument for \eqref{L:coup:step2}.
  \smallskip
  
  \noindent
  {\bf Step 3. Conclusion.}
  Set
  \begin{align*}
  \zeta(\widetilde\omega,x):=\bigg(A(\widetilde\omega(x),\xi+\nabla\phi^T_\xi(\widetilde\omega,x)) \cdot \Xi-\frac1T\phi_\xi^T(\widetilde\omega,x)\phi_{\xi,\Xi}^{*,T}(\widetilde\omega,x)\bigg).
  \end{align*}
  Since $\eta_L$ is supported in $B_{\frac L 8}$ and $\eta_L \leq L^{-d}$, we have
  \begin{equation*}
    I_1:=\Big|\big(A^{\RVE,\eta_L,T}(\widetilde\omega,\xi)-    A^{\RVE,\eta_L,T}(\widehat\omega,\xi)\big) \cdot \Xi \Big| \leq C L^{-d}\int_{B_{\frac L8}}|\zeta(\widetilde\omega,x)-\zeta(\widehat\omega,x)|\,dx.
  \end{equation*}
  We cover $B_{\frac L 8}$ by balls of radius $\sqrt T\leq L$ and centers in $B_{\frac L 8}$; more precisely, there exists a set $X_{L,T}\subset B_{\frac L 8}$ with $\# X_{L,T}\leq C(d) \big(L/{\sqrt T}\big)^d$ and $\cup_{x_0\in X_{L,T}}B_{\sqrt T}(x_0)\supset B_{\frac L 8}$. Thus,
  \begin{align*}
    I_1&\leq C L^{-d}\sum_{x_0\in X_{L,T}}\int_{B_{\sqrt T}(x_0)\cap B_{L/8}}|\zeta(\widetilde\omega,x)-\zeta(\widehat\omega,x)|\,dx
    \\
    &\leq C \sqrt T^{-d}\Big(\frac1{\# X_{L,T}}\sum_{x_0\in X_{L,T}}\int_{B_{L/8}(x_0)\cap B_{L/8}}|\zeta(\widetilde\omega,x)-\zeta(\widehat\omega,x)|\eta(x-x_0)\,dx\Big).
  \end{align*}
  By the definition of $\zeta$, the estimates in \eqref{L:coup:step1} and \eqref{L:coup:step2} from Step~1 and Step~2, and the deterministic exponentially localized bounds on $\frac{1}{\sqrt T}\phi_\xi^T$ and $\frac{1}{\sqrt T}\phi_{\xi,\Xi}^{*,T}$ (which are a consequence of Lemma~\ref{L:exploc}), and the Lipschitz continuity of $A$ with respect to the second variable (see \hyperlink{A2}{(A2)}), we conclude for $x_0\in X_{L,T}\subset B_{\frac L 8}$ that
  \begin{align*}
    & \int_{B_{L/8}(x_0)\cap B_{L/8}} |\zeta(\widetilde\omega,x)-\zeta(\widehat\omega,x)|\eta(x-x_0)\,dx\\
    &\leq C \int_{\Rd}\big(\big|A(\tilde \omega,\xi+\nabla\phi_\xi^T)-A(\tilde \omega,\xi+\nabla\widehat\phi_\xi^T)\big||\Xi|
    \\&~~~~~~~~~~~~~~
    +\frac1T|\phi_\xi^T||\hat\delta\phi_{\xi,\Xi}^{*,T}|+\frac1T|\hat\delta\phi_\xi^T||\widehat \phi_{\xi,\Xi}^{*,T}|\big)\eta(x-x_0)\,dx
    \\
    &\leq C \left(\int_\Rd|\Xi|^2\eta(x-x_0)\,dx\right)^\frac12\left(\int_{\Rd}|\nabla\hat\delta\phi_\xi^T|^2\eta(x-x_0) \,dx\right)^\frac12
    \\
    &~~~~+C \left(\int_{\Rd}\frac1T|\phi_\xi^T|^2\eta(x-x_0) \,dx\right)^\frac12\left(\int_{\Rd}\frac1T|\hat\delta\phi_{\xi,\Xi}^{*,T}|^2\eta(x-x_0) \,dx\right)^\frac12\\
    &~~~~
      +C \left(\int_{\Rd}\frac1T|\hat\delta\phi_\xi^T|^2\eta(x-x_0) \,dx\right)^\frac12\left(\int_{\Rd}\frac1T|\widehat\phi_{\xi,\Xi}^{*,T}|^2\eta(x-x_0) \,dx\right)^\frac12
      \\
    &\leq C \sqrt T^d\Big(\exp(-\tfrac{\gamma}{32}L/\sqrt T)\Big)^\frac12\Big(|\xi|^2 |\Xi|^2+(1+|\xi|^2)|\xi|^2\|\Xi+\nabla\phi_{\xi,\Xi}^T\|_{q,T,x_0}^2 \Big)^\frac12.
\end{align*}
In total, we have shown the desired deterministic estimate
\begin{align*}
&|A^{\RVE,\eta_L,T}(\widetilde\omega,\xi)-A^{\RVE,\eta_L,T}(\pi_L\widetilde\omega,\xi)|
\\&
\leq C \exp\Big(-\frac{\gamma}{64}\cdot \frac{L}{\sqrt
        T}\Big)
        \Big(|\xi||\Xi|+(1+|\xi|)|\xi|\|\Xi+\nabla\phi_{\xi,\Xi}^{T}\|_{q,L,T}\Big).
\end{align*}
\end{proof}

\subsection{Estimate on the systematic error of the RVE method}
\label{S:P:TRVE1}
We now estimate the systematic error of the RVE approximation for the effective material law. We begin with the case in the presence of the regularity condition \hyperlink{R}{(R)}.
\begin{proof}[{Proof of Theorem~\ref{T:RVE}b} -- the case with \hyperlink{R}{(R)}]
By rescaling we may assume without loss of generality that $\e=1$. In the following $\eta:\R^d\to\R$ denotes a non-negative weight supported in $B_{\frac L 8}$ with $|\eta|\leq C(d) L^{-d}$ and $\int_{\Rd}\eta\,dx=1$. Moreover, we consider a localization parameter $T$ according to 
\begin{equation}\label{choiceofT}
  \sqrt T=\frac{\gamma}{64}\frac{L}{\log ((L/\varepsilon)^{d\wedge 4})}
\end{equation}
for the $\gamma=\gamma(d,m,\lambda,\Lambda)$ from Lemma~\ref{L:couperror}.
Our starting point is the error decomposition \eqref{decompsys}, which yields for any $\Xi\in \Rmd$
\begin{equation}\label{decompsys2}
  \begin{aligned}
    \mathbb E_L&\left[A^{\RVE,L}(\xi)\right]\cdot \Xi-\Ahom(\xi)\cdot \Xi\\
    =&\ \mathbb E_L\left[A^{\RVE,L}(\xi)\cdot \Xi\,\right]-\mathbb E_L\left[A^{\RVE,\eta,T}(\xi)\cdot \Xi\,\right]\\
    &\ +\mathbb E_L\left[A^{\RVE,\eta, T}(\xi)\cdot \Xi\,\right]  -\mathbb E_L\left[A^{\RVE,\eta, T}(\pi_L\omega_{\e,L},\xi)\cdot \Xi\,\right]\\
    &\ +\mathbb E\left[A^{\RVE,\eta, T}(\pi_L\omega_\e,\xi)\cdot \Xi\,\right]  -\mathbb E\left[A^{\RVE,\eta, T}(\xi)\cdot \Xi\,\right]\\
    &\ +\mathbb E\left[A^{\RVE,\eta,T}(\xi)\cdot \Xi\,\right]-\Ahom(\xi)\cdot \Xi\\
    =:&\ I_1+I_2+I_3+I_4.
  \end{aligned}
\end{equation}
Note that in the above decomposition we already used the equality
\begin{align*}
\mathbb E_L\left[A^{\RVE,\eta, T}(\pi_L\omega_{\e,L},\xi)\right]=\mathbb E\left[A^{\RVE,\eta, T}(\pi_L\omega_\e,\xi)\right],
\end{align*}
which is valid since $\mathbb P_L$ is assumed to be a $L$-periodic approximation of $\mathbb P$ in the sense of Definition~\ref{D:Lperproxy} (recall also \eqref{DefPi}). The terms $I_2$ and $I_3$ are coupling errors that can be estimated deterministically with help of Lemma~\ref{L:couperror}. Combined with the choice of $T$ in \eqref{choiceofT} and the bound on high moments of $\nabla\phi_{\xi,\Xi}^T$ obtained by combining \eqref{EstimateLinearCorrectorLinfty} and Lemma~\ref{MomentsMinimalRadiusLinearized}, we arrive at
\begin{align*}
 |I_2|+|I_3|&\leq C (1+|\xi|^C) |\xi||\Xi|
 \bigg(\frac{L}{\varepsilon}\bigg)^{-(d\wedge 4)}.
\end{align*}
The terms $I_1$ and $I_4$ are systematic localization errors that can be estimated with help of Proposition~\ref{P:RVEsystem}b,c. We obtain using again \eqref{choiceofT}
\begin{equation*}
  |I_1|+|I_4|\leq C (1+|\xi|)^{C}|\xi||\Xi| \bigg(\frac{L}{\varepsilon}\bigg)^{-(d\wedge 4)}|\log (L/\varepsilon)|^{\alpha_d}.
\end{equation*}
Having estimated all terms in \eqref{decompsys2}, this establishes the first estimate in Theorem~\ref{T:RVE}b upon taking the supremum with respect to $\Xi$, $|\Xi|\leq 1$.
\end{proof}
We next establish the suboptimal estimate for the systematic error of the RVE method in the case without the small-scale regularity condition \hyperlink{R}{(R)}.
\begin{proof}[{Proof of Theorem~\ref{T:RVE}b} -- the case without \hyperlink{R}{(R)}]
As in the case with small scale regularity \hyperlink{R}{(R)}, we denote by $\eta:\R^d\to\R$ a non-negative weight supported in $B_{\frac L 8}$ with $|\eta|\leq C(d) L^{-d}$ and $\int_{\Rd}\eta=1$. Moreover, we consider a localization parameter $\sqrt T\leq L$, whose relative scaling with respect to $L$ will be specified below in Step~3.
For any parameter field $\widetilde\omega$ we consider the localized RVE-approximation
  \begin{equation*}
  A^{\RVE,\eta,T}(\widetilde\omega,\xi):=\int_\Rd \eta A(\widetilde\omega,\xi+\nabla\phi^T_\xi)\,dx.
\end{equation*}
Note that it has a simpler form compared to the quantity introduced in Definition~\ref{D:RVET}. In particular, the above expression does not invoke the linearized corrector (for which we cannot derive suitable estimates without the small scale regularity condition \hyperlink{R}{(R)}). As in the case with small scale regularity, the starting point is estimate \eqref{decompsys2}, i.\,e.\ the decomposition of the systematic error
\begin{align*}
\mathbb{E}_L[A^{\RVE,L}(\xi)]-A_\shom(\xi) = I_1+I_2+I_3+I_4
\end{align*}
into the two coupling errors
\begin{eqnarray*}
  I_2&:=&\mathbb E_L\left[A^{\RVE,\eta, T}(\xi)\right]  -\mathbb E_L\left[A^{\RVE,\eta, T}(\pi_L\omega_{\e,L},\xi)\right],\\
  I_3&:=&\mathbb E\left[A^{\RVE,\eta, T}(\pi_L\omega_\e,\xi)\right]  -\mathbb E\left[A^{\RVE,\eta, T}(\xi)\right],
\end{eqnarray*}
and the two systematic localization errors
\begin{eqnarray*}
  I_1&:=& \mathbb E_L\left[A^{\RVE,L}(\xi)\right]-\mathbb E_L\left[A^{\RVE,\eta,T}(\xi)\right],\\
  I_4&:=&\mathbb E\left[A^{\RVE,\eta,T}(\xi)\right]-\Ahom(\xi).
\end{eqnarray*}
\smallskip

\noindent
{\bf Step 1. Estimate of the coupling errors.}
We claim that
\begin{equation*}
  |I_2|+|I_3|\leq C \exp\Big(-\frac{\gamma}{32}L/\sqrt T\Big)|\xi|.
\end{equation*}
Indeed, this can be seen by an argument similar to the proof of Lemma~\ref{L:couperror}. In fact the argument is significantly simpler thanks to the absence of the linearized corrector in the definition of the localized RVE-approximation. We only discuss the argument for $I_3$, since the one for $I_2$ is analogous. We first note that thanks to the assumptions on $\eta$ and the Lipschitz continuity of $A$, see \hyperlink{A2}{(A2)}, we have
\begin{equation}\label{P:TRVE2:eq1}
  \Big|A^{\RVE,\eta, T}(\omega_\e,\xi)-A^{\RVE,\eta, T}(\pi_L\omega_\e,\xi)\Big|\leq C \fint_{B_{\frac L 8}}|\nabla\hat\delta\phi_\xi^T| \,dx,
\end{equation}
where $\hat\delta\phi_\xi^T$ is defined by \eqref{L:coup:eq0}. As in Section~\ref{S:P:TRVE1} we consider a minimal cover of $B_{\frac L8}$ by balls of radius $\sqrt T$; more precisely, let $X_{L,T}\subset B_{\frac L8}$ denote a finite set of points such that $\# X_{L,T}\leq C(d) (L/\sqrt T)^d$ and $\cup_{x_0\in X_{L,T}}B_{\sqrt T}(x_0)\supset B_{\frac L 8}$. Then 
\begin{align*}
  [\text{RHS of \eqref{P:TRVE2:eq1}}]
  &\leq C \sqrt T^{-d}\frac{1}{\# X_{L,T}}\sum_{x_0\in X_{L,T}}\int_{B_{\frac{L}{8}}(x_0)}|\nabla\hat\delta\phi_\xi^T|\exp(-\gamma|x-x_0|/\sqrt T)\,dx~~~~~~~~\\
  &\leq C \sqrt T^{-\frac d2}\frac{1}{\# X_{L,T}}\sum_{x_0\in X_{L,T}}\left(\int_{\R^d}|\nabla\hat\delta\phi_\xi^T|^2\exp(-\gamma|x-x_0|/\sqrt T)\,dx\right)^\frac12
\end{align*}
where $0<\gamma\ll1$ is chosen such that the exponential localization estimate of Lemma~\ref{L:exploc} applies. 
Combining the estimate with \eqref{L:coup:step1} thus yields
\begin{align*}
  &\Big|A^{\RVE,\eta, T}(\omega_\e,\xi)-A^{\RVE,\eta, T}(\pi_L\omega_\e,\xi)\Big|\\
  &\leq C
  {\sqrt T}^{-d/2} \sup_{x_0\in B_{\frac L8}}\left(\int_{\R^d}|\nabla\hat\delta\phi_\xi^T|^2\exp(-\gamma|x-x_0|/\sqrt T) \,dx\right)^\frac12\\
  &\leq C
  \exp\big(-\tfrac{\gamma}{32}L/\sqrt T \big)|\xi|.
\end{align*}

\noindent
{\bf Step 2. Estimate of the systematic localization errors.}
We claim that
\begin{equation*}
  |I_1|+|I_4|\leq C |\xi| \bigg(\frac{\varepsilon}{\sqrt T}\bigg)^{\frac{d\wedge 4}{2}}
  \begin{cases}
    \big|\log \sqrt T\big|^\frac12 &\text{for }d\in\{2,4\},\\
    1&\text{for }d=3\text{ and }d\geq 5.\\
  \end{cases}
\end{equation*}
We only discuss $I_4$, since the argument for $I_1$ is similar. The Lipschitz continuity of $A$ (see \hyperlink{A2}{(A2)}) and the localization error estimate for the corrector in form of Lemma~\ref{L:syscorr} yield
\begin{align*}
  &\Big|\mathbb E\left[A^{\RVE,\eta,2T}(\xi)-A^{\RVE,\eta,T}(\xi)\right]\Big|
  \\
  &\leq C
  \mathbb E\left[|\nabla\phi_\xi^{2T}-\nabla\phi_\xi^{T}|^2\right]^\frac12
  \\
  &\leq C|\xi| \bigg(\frac{\varepsilon}{\sqrt T}\bigg)^{\frac{d\wedge 4}{2}}\times
    \begin{cases}
      \big|\log (\varepsilon/\sqrt T)\big|^\frac12 &\text{for }d\in\{2,4\},\\
      1&\text{for }d=3\text{ and }d\geq 5.\\
    \end{cases}
\end{align*}
The claimed estimate now follows by a telescopic sum argument similar to Step~2 of the proof of Proposition~\ref{P:RVEsystem}.
\smallskip

\noindent
{\bf Step 3. Conclusion.} The combination of the previous steps yields
\begin{align*}
  &\frac{|I_1|+|I_2|+|I_3|+|I_4|}{|\xi|}
  \\&
  \leq C
  \exp\Big(-\frac{\gamma}{32} \frac{L}{\sqrt T}\Big)+\bigg(\frac{\varepsilon}{\sqrt{T}}\bigg)^{\frac{d\wedge 4}{2}}\times 
  \begin{cases}
    \big|\log (\sqrt T/\varepsilon)\big|^\frac12 &\text{for }d\in\{2,4\},\\
    1&\text{otherwise}.
  \end{cases}
\end{align*}
With $\sqrt T=\frac{\gamma}{32}L\big(\log{((L/\varepsilon)^{\frac{d\wedge 4}2}})\big)^{-1}$ the RHS turns into
\begin{equation*}
  C\bigg(\frac{\varepsilon}{L}\bigg)^{\frac{d\wedge 4}{2}}+C\bigg(\frac{\varepsilon}{L}\bigg)^{\frac{d\wedge 4}{2}}(\log (L/\varepsilon))^{\frac{d\wedge 4}{2}}
  \times \begin{cases}
    \big|\log L\big|^\frac12 &\text{for }d\in\{2,4\},\\
    1&\text{otherwise}.\\
  \end{cases}
\end{equation*}
This establishes the result.
\end{proof}

\section{Corrector estimates for the nonlinear monotone PDE}

\label{SectionCorrectorEstimates}

\subsection{Estimates on linear functionals of the corrector and the flux corrector}

We now turn to the first major step, the estimate on the localized homogenization corrector $\phi_\xi^T$ and the localized flux corrector $\sigma_\xi^T$. For the proof of the estimate on linear functionals of the corrector and the flux corrector in Lemma~\ref{LemmaEstimateLinearFunctionals}, we need the following auxiliary lemma, which follows by a combination of the Caccioppoli inequality with hole-filling.
\begin{lemma}
Let the assumptions \hyperlink{A1}{(A1)}--\hyperlink{A2}{(A2)} be satisfied, let $\varepsilon>0$ be arbitrary, let $\tilde \omega$ be an arbitrary parameter field, let the correctors $\phi_\xi^T$ and $\phi_{\xi,\Xi}^T$ be defined as the unique solutions to the corrector equations \eqref{EquationLocalizedCorrector} and \eqref{EquationLocalizedCorrectorLinearized} in $\Huloc$, and let $r_{*,T,\xi}$ and $r_{*,T,\xi,\Xi}$ be as in \eqref{Definitionrstar} and \eqref{Definitionrstar2}. Let the parameter $K_{mass}$ in \eqref{Definitionrstar}, \eqref{Definitionrstar2} be chosen as $K_{mass}\geq C(d,m,\lambda,\Lambda)$ for some $C$ determined in the proof below {(independently of $\eps$)}. Then for any $x_0\in \Rd$ the estimates
\begin{align}
\label{EstimateGradByrstar}
\fint_{B_\varepsilon(x_0)} |\xi+\nabla \phi_\xi^T|^2 \,dx
\leq
C |\xi|^2 \bigg(\frac{r_{*,T,\xi}(x_0)}{\varepsilon}\bigg)^{d-\delta}
\end{align}
and
\begin{align}
\label{EstimateGradByrstar2}
\fint_{B_\varepsilon(x_0)} |\Xi+\nabla \phi_{\xi,\Xi}^T|^2 \,dx
\leq
C |\Xi|^2 \bigg(\frac{r_{*,T,\xi,\Xi}(x_0)}{\varepsilon}\bigg)^{d-\delta}
\end{align}
hold true for some $\delta=\delta(d,m,\lambda,\Lambda)$, $0<\delta<\frac{1}{2}$. Furthermore, we have
\begin{align}
\label{UpperBoundrstarxi}
r_{*,T,\xi}(x_0)&\leq C\sqrt{T},
\\
\label{UpperBoundrstarXi}
r_{*,T,\xi,\Xi}(x_0)&\leq C\sqrt{T}.
\end{align}
\end{lemma}
{Note that $r_{*,T,\xi}$ and $r_{*,T,\xi,\Xi}$ as defined in  \eqref{Definitionrstar} and \eqref{Definitionrstar2} are well-defined, as by considerations in the proof below the conditions in the definition of $r_{*,T,\xi}$ are satisfied for all $r\geq C \sqrt{T}$.
}
\begin{proof}
The estimates \eqref{UpperBoundrstarxi} and \eqref{UpperBoundrstarXi} are a simple consequence of the definitions \eqref{Definitionrstar} and \eqref{Definitionrstar2}, the assumed lower bound on $K_{mass}$, and the bounds
\begin{align*}
\fint_{B_R} \frac{1}{T} |\phi_\xi^T|^2 \,dx &\leq C|\xi|^2
\\
\fint_{B_R} \frac{1}{T} |\phi_{\xi,\Xi}^T|^2 \,dx &\leq C|\Xi|^2
\end{align*}
valid for all $R\geq \sqrt{T}$ which are a consequence of Lemma~\ref{L:exploc} applied to the corrector equations \eqref{EquationLocalizedCorrector} {(with $u:=\phi_\xi^T$, $F:=A(x,\xi+\nabla \phi_\xi^T)-A(x,\nabla \phi_\xi^T)$, and $f\equiv 0$)} and \eqref{EquationLocalizedCorrectorLinearized}.

To prove \eqref{EstimateGradByrstar}, we employ the hole-filling estimate \eqref{HoleFillingMonotone} for the function $u(x):=\xi\cdot (x-x_0)+\phi_\xi^T$, which by \eqref{EquationLocalizedCorrector} solves the PDE
\begin{align*}
-\nabla \cdot \big(A\big(\tilde \omega,\nabla u\big)\big) + \frac{1}{T} u = \frac{1}{T}\xi \cdot (x-x_0).
\end{align*}
In combination with the Caccioppoli inequality \eqref{CaccippoliMonotone}, this yields
\begin{align*}
&\fint_{B_\varepsilon(x_0)} |\xi+\nabla \phi_\xi^T|^2 \,dx
\\
&\leq C \bigg(\frac{r_{*,T,\xi}}{\varepsilon}\bigg)^{d-\delta} 
\inf_{b\in \Rm} \bigg(\frac{1}{r_{*,T,\xi}^2} \fint_{B_{r_{*,T,\xi}}(x_0)} |\xi\cdot (x-x_0)+\phi_\xi^T-b|^2 \,dx + \frac{1}{T} |b|^2
\\&~~~~~~~~~~~~~~~~~~~~~~~~~~~~~~~~~
+\fint_{B_{r_{*,T,\xi}}(x_0)} \frac{1}{T} |\xi\cdot (x-x_0)|^2 \,dx
\bigg)
\\&~~~~
+ C\bigg(\frac{r_{*,T,\xi}}{\varepsilon}\bigg)^{d} \fint_{B_{r_{*,T,\xi}}(x_0)} \bigg(\frac{\varepsilon}{\varepsilon+|x-x_0|}\bigg)^\delta \frac{1}{T} |\xi\cdot (x-x_0)|^2 \,dx.
\end{align*}
Splitting the last integral into the integrals $\int_{B_{\eps}(x_0)}$ and $\int_{B_{r_{*,T,\xi}}(x_0)\setminus B_{\eps}(x_0)}$, we see (since $r_{*,T,\xi}\geq \varepsilon$) that the term in the last line is bounded by
\begin{align*}
C \eps^{-d} r_{*,T,\xi}^d \frac{1}{T} |\xi|^2 \eps^2
+C \eps^{-d} \eps^\delta \frac{1}{T} |\xi|^2 r_{*,T,\xi}^{d+2-\delta}
\leq C \frac{1}{T} |\xi|^2 \bigg(\frac{r_{*,T,\xi}}{\varepsilon}\bigg)^{d-\delta} r_{*,T,\xi}^2.
\end{align*}
Choosing $b:=\fint_{B_{r_{*,T,\xi}}(x_0)} \phi_{\xi}^T \,dx$ and using the definition of $r_{*,T,\xi}$ \eqref{Definitionrstar} as well as the fact that $r_{*,T,\xi}\leq C\sqrt{T}$ therefore entails \eqref{EstimateGradByrstar}. To show \eqref{EstimateGradByrstar2}, one argues analogously.
\end{proof}

The following estimate is a consequence of a straightforward application of H\"older's inequality and Jensen's inequality. Nevertheless, we state it to avoid repetition of the same computations, as it is used several times in the proofs of Lemma~\ref{LemmaEstimateLinearFunctionals} and Lemma~\ref{LemmaEstimateLinearFunctionalsLinearized}.
\begin{lemma}
\label{LemmaEstimateSensitivity}
Let $\varepsilon>0$ and let $b$ and $B$ be stationary random fields with
\begin{align*}
\fint_{B_\varepsilon(x)} |b|^2 \,d\tilde x \leq |B(x)|^2.
\end{align*}
Let $h$ be a random field satisfying $\nabla h\in L^p(\Rd)$ almost surely.
Let $0<\alpha_0\leq 1$ and $0<\tau<1$. Let $p>2$ be close enough to $2$ depending only on $d$ and on (a lower bound on) $\alpha_0$, but otherwise arbitrary. Then for any $q\geq C(p,\alpha_0,\tau)$ and for any $r\geq \varepsilon$ the estimate
\begin{align*}
&\mathbb{E}\Bigg[\bigg(\int_\Rd \bigg| \fint_{B_\varepsilon(x)} |b| |\nabla h| \,d\tilde x\bigg|^2 \,dx\bigg)^q\Bigg]^{1/2q}
\\&
\leq
C(d,\alpha_0,p,\tau) r^{d/2}
\mathbb{E}\big[|B|^{2q/(1-\tau)}\big]^{(1-\tau)/2q}
\\&~~~~~~~~~~~~\times
\mathbb{E}\Bigg[\bigg(r^{-d} \int_\Rd |\nabla h|^p \bigg(1+\frac{|x|}{r}\bigg)^{\alpha_0} \,dx\bigg)^{2q/p\tau} \Bigg]^{\tau/2q}
\end{align*}
holds true.
\end{lemma}
\begin{proof}[Proof of Lemma~\ref{LemmaEstimateSensitivity}]
Several applications of H\"older's inequality yield together with the bound on $b$
\begin{align*}
&\mathbb{E}\Bigg[\bigg(\int_\Rd \bigg| \fint_{B_\varepsilon(x)} |b| |\nabla h| \,d\tilde x\bigg|^2 \,dx\bigg)^q\Bigg]^{1/2q}
\\&
\leq
\mathbb{E}\Bigg[\bigg(\int_\Rd \fint_{B_\varepsilon(x)} |b|^2 \,d\tilde x ~ \fint_{B_\varepsilon(x)}  |\nabla h|^2 \,d\tilde x ~ \,dx\bigg)^q\Bigg]^{1/2q}
\\&
\leq
\mathbb{E}\Bigg[\bigg(\int_\Rd |B(x)|^2 ~ \fint_{B_\varepsilon(x)}  |\nabla h|^2 \,d\tilde x ~ \,dx\bigg)^q\Bigg]^{1/2q}
\\&
\leq
r^{d/2} \mathbb{E}\Bigg[\bigg(r^{-d} \int_\Rd |B(x)|^{2p/(p-2)} \bigg(1+\frac{|x|}{r}\bigg)^{-2\alpha_0/(p-2)} \,dx\bigg)^{q(p-2)/p}
\\&~~~~~~~~~~~~~
\times
 \bigg(r^{-d} \int_\Rd  \bigg(\fint_{B_\varepsilon(x)}  |\nabla h|^2 \,d\tilde x\bigg)^{p/2} \bigg(1+\frac{|x|}{r}\bigg)^{\alpha_0} \,dx\bigg)^{2q/p} \Bigg]^{1/2q}.
\end{align*}
This implies using H\"older's inequality with exponents $\frac{1}{1-\tau}$ and $\frac{1}{\tau}$ and Jensen's inequality
\begin{align*}
&\mathbb{E}\Bigg[\bigg(\int_\Rd \bigg| \fint_{B_\varepsilon(x)} |b| |\nabla h| \,d\tilde x\bigg|^2 \,dx\bigg)^q\Bigg]^{1/2q}
\\&
\leq
r^{d/2} \mathbb{E}\Bigg[\bigg(r^{-d} \int_\Rd |B(x)|^{2p/(p-2)} \bigg(1+\frac{|x|}{r}\bigg)^{-2\alpha_0/(p-2)} \,dx\bigg)^{q(p-2)/p(1-\tau)}\Bigg]^{(1-\tau)/2q}
\\&~~~~~~~~~~~~~~
\times
\mathbb{E}\Bigg[
\bigg(r^{-d} \int_\Rd \fint_{B_\varepsilon(x)}  |\nabla h|^p \,d\tilde x \bigg(1+\frac{|x|}{r}\bigg)^{\alpha_0} \,dx\bigg)^{2q/p\tau} \Bigg]^{\tau/2q}.
\end{align*}
Assuming that $q\geq p/(p-2)$ and $2\alpha_0/(p-2)\geq 2d$, we obtain using Jensen's inequality in the first term (note that the integral $\smash{r^{-d} \int_\Rd \big(1+\frac{|x|}{r}\big)^{\alpha_0/(p-2)} \,dx}$ is bounded by $C(d,\alpha_0,p)$) as well as the fact that the supremum and the infimum of the function $1+\smash{\frac{|x|}{r}}$ on an $\varepsilon$-ball differ by at most a factor of $2$ in the second term
\begin{align*}
&\mathbb{E}\Bigg[\bigg(\int_\Rd \bigg| \fint_{B_\varepsilon(x)} |b| |\nabla h| \,d\tilde x\bigg|^2 \,dx\bigg)^q\Bigg]^{1/2q}
\\&
\leq C(d,\alpha_0,p,\tau) r^{d/2}
\mathbb{E}\Bigg[r^{-d} \int_\Rd |B(x)|^{2q/(1-\tau)} \bigg(1+\frac{|x|}{r}\bigg)^{-2\alpha_0/(p-2)} \,dx\Bigg]^{(1-\tau)/2q}
\\&~~~~~~~~~~~~~~~~~~~~~~~
\times
\mathbb{E}\Bigg[\bigg(r^{-d} \int_\Rd |\nabla h|^p \bigg(1+\frac{|x|}{r}\bigg)^{\alpha_0} \,dx\bigg)^{2q/p\tau} \Bigg]^{\tau/2q}.
\end{align*}
By the stationarity of $B$, this yields the assertion of the lemma.
\end{proof}

We now prove the estimate in Lemma~\ref{LemmaEstimateLinearFunctionals} for functionals of the localized corrector $\phi_\xi^T$  for the nonlinear elliptic PDE and the corresponding localized flux corrector $\sigma_\xi^T$. As the proof of Lemma~\ref{LemmaCorrectorAverages} is very similar, we combine their proofs.

\begin{proof}[Proof of Lemma~\ref{LemmaEstimateLinearFunctionals} and Lemma~\ref{LemmaCorrectorAverages}] {\bf Part a: Estimates for linear functionals of the homogenization corrector $\phi_\xi^T$.}
Taking the derivative of the corrector equation \eqref{EquationLocalizedCorrector} with respect to a perturbation $\delta \omega_\varepsilon$ of the random field $\omega_\varepsilon$, we see that the change $\delta \phi_{\xi}^T$ of the corrector under such an infinitesimal perturbation satisfies the linear elliptic PDE
\begin{align}
\label{EquationSensitivityCorrector}
&-\nabla \cdot \big(\partial_\xi A\big(\omega_\varepsilon,\xi+\nabla \phi_\xi^T\big)\nabla \delta \phi_\xi^T\big)
+\frac{1}{T}\delta \phi_\xi^T
=\nabla \cdot \big(\partial_\omega A\big(\omega_\varepsilon,\xi+\nabla \phi_\xi^T\big) \delta\omega_\varepsilon \big).
\end{align}
Define the coefficient field $a_\xi^T(x):=\partial_\xi A(\omega_\varepsilon(x),\xi+\nabla \phi_\xi^T(x))$. By our assumptions \hyperlink{A1}{(A1)} and \hyperlink{A2}{(A2)}, the coefficient field $\smash{a_\xi^T:\Rd\rightarrow \Rmd\otimes \Rmd}$ exists and is uniformly elliptic and bounded in the sense $\smash{a_\xi^T v \cdot v\geq \lambda |v|^2}$ and $\smash{|a_\xi^T v|\leq \Lambda |v|}$ for any $v\in \Rmd$.

Now, consider a functional of the form $F:=\int_\Rd g \cdot \nabla \phi_\xi^T \,dx$ for a deterministic compactly supported function $g$. Denoting by $h\in H^1(\Rd;\Rm)$ the (unique) solution to the dual equation
\begin{align}
\label{EquationDual}
-\nabla \cdot (a_\xi^{T,*} \nabla h)+\frac{1}{T} h=\nabla \cdot g,
\end{align}
we deduce
\begin{align*}
\delta F
&= \int_\Rd g \cdot \nabla \delta \phi_\xi^T \,dx
\stackrel{\eqref{EquationDual}}{=}
-\int_\Rd a_\xi^T \nabla \delta \phi_\xi^T \cdot \nabla h + \frac{1}{T} \delta \phi_\xi^T h\,dx
\\&
\stackrel{\eqref{EquationSensitivityCorrector}}{=}
\int_\Rd \partial_\omega A\big(\omega_\varepsilon,\xi+\nabla \phi_\xi^T\big) \delta\omega_\varepsilon \cdot \nabla h \,dx,
\end{align*}
i.\,e.\ we have
\begin{align}
\label{SensitivityCorrector}
\frac{\partial F}{\partial \omega_\varepsilon} = \partial_\omega A\big(\omega_\varepsilon,\xi+\nabla \phi_\xi^T\big) \cdot \nabla h.
\end{align}
By the assumption \hyperlink{A3}{(A3)}, this implies
\begin{align*}
&\int_\Rd \bigg|\fint_{B_\varepsilon(x)} \bigg|\frac{\partial F}{\partial \omega_\varepsilon}\bigg| \,d\tilde x \bigg|^2 \,dx
\\&
\leq C \int_\Rd \bigg|\fint_{B_\varepsilon(x)} |\xi+\nabla \phi_\xi^T| |\nabla h| \,d\tilde x \bigg|^2 \,dx.
\end{align*}
Using the version of the
spectral gap inequality for the $q$-th moment (see Lemma~\ref{LemmaLqSpectralGap}), we deduce for any $q\geq 1$
\begin{align*}
\mathbb{E}\Big[\big|F-\mathbb{E}[F]\big|^{2q}\Big]^{1/2q}
\leq C q \varepsilon^{d/2} \mathbb{E}\Bigg[\bigg(\int_\Rd \bigg|\fint_{B_\varepsilon(x)} |\xi+\nabla \phi_\xi^T| |\nabla h| \,d\tilde x \bigg|^2 \,dx\bigg)^{q}\Bigg]^{1/2q}.
\end{align*}
As $\phi_\xi^T$ is a stationary vector field, we have $\mathbb{E}[F]=\smash{\int_\Rd g\cdot \nabla \mathbb{E}[\phi_\xi^T] \,dx}=0$.
By \eqref{EstimateGradByrstar}, we have
\begin{align*}
\fint_{B_\varepsilon(x)} |\xi+\nabla \phi_\xi^T|^2 \,d\tilde x
\leq C|\xi|^2
\bigg(\frac{r_{*,T,\xi}(x_0)}{\varepsilon}\bigg)^{d-\delta}=:B(x)^2.
\end{align*}
An application of Lemma~\ref{LemmaEstimateSensitivity} therefore yields for any $0<\tau<1$, any $0<\alpha_0<c$, any $p\in (2,2+c(d,m,\lambda,\Lambda,\alpha_0))$, and any $q\geq C(d,m,\lambda,\Lambda,\tau,\alpha_0,p)$ (note that $\nabla h \in L^p$ holds by Lemma~\ref{LemmaWeightedMeyers} and our assumptions on $g$)
\begin{align}
\label{SensitivitySigmaXiAlmostFinal}
\mathbb{E}\big[|F|^{2q}\big]^{1/2q}
&\leq C |\xi| q \varepsilon^{d/2} r^{d/2}
\mathbb{E}\bigg[\bigg(\frac{r_{*,T,\xi}}{\varepsilon}\bigg)^{(d-\delta)q/(1-\tau)}\bigg]^{(1-\tau)/2q}
\\&~~~~~~~~~~~~\times
\nonumber
\mathbb{E}\Bigg[\bigg(r^{-d} \int_\Rd |\nabla h|^p \bigg(1+\frac{|x|}{r}\bigg)^{\alpha_0} \,dx\bigg)^{2q/p\tau} \Bigg]^{\tau/2q}.
\end{align}
Applying the weighted Meyers estimate of Lemma~\ref{LemmaWeightedMeyers} to the equation \eqref{EquationDual} and using our assumption $(\fint_{B_r(x_0)} |g|^p \,dx)^{1/p} \leq r^{-d}$, the integral in the last term may be estimated to yield
\begin{align*}
\mathbb{E}\big[|F|^{2q}\big]^{1/2q}
&\leq C |\xi| q \bigg(\frac{\varepsilon}{r} \bigg)^{d/2}
\mathbb{E}\bigg[\bigg(\frac{r_{*,T,\xi}}{\varepsilon}\bigg)^{(d-\delta)q/(1-\tau)}\bigg]^{(1-\tau)/2q}
\end{align*}
for any $2<p<2+c$, any $0<\tau<1$, and any $q\geq C(d,m,\lambda,\Lambda,p,\tau)$. This is the desired estimate for functionals of $\phi_\xi^T$ in Lemma~\ref{LemmaEstimateLinearFunctionals}.

Next, we establish \eqref{EstimateCorrectorAverage}. To this end, we consider the random variable $F:=\fint_{B_R(x_0)} \phi_\xi^T \,dx$ for $R\geq \sqrt{T}$.
By $\mathbb{E}[\phi_\xi^T]=0$ -- which follows {since $\phi_\xi^T$ is the divergence of a stationary random field}
-- we see that we have $\mathbb{E}[F]=0$.
We may therefore repeat the previous computation to estimate the stochastic moments of $F$, up to the following changes: The equation \eqref{EquationDual} is replaced by the equation with non-divergence form right-hand side $-\nabla \cdot (a_\xi^{T,*} \nabla h)+\frac{1}{T} h=\frac{1}{|B_R|} \chi_{B_R(x_0)}$, and the estimate for $\nabla h$ deduced from Lemma~\ref{LemmaWeightedMeyers} now reads
\begin{align*}
&\bigg(R^{-d} \int_\Rd |\nabla h|^p \bigg(1+\frac{|x|}{R}\bigg)^{\alpha_0} \,dx\bigg)^{1/p}
\\&
\leq 
C \sqrt{T} \bigg(R^{-d} \int_\Rd \bigg|\frac{1}{|B_R|} \chi_{B_R(x_0)}\bigg|^p \,dx\bigg)^{1/p}.
\end{align*}
In total, we deduce
\begin{align*}
\mathbb{E}\bigg[\bigg|\fint_{B_R(x_0)} \phi_\xi^T \,dx\bigg|^q \bigg]^{1/q}
\leq C |\xi| q  \sqrt{T} \bigg(\frac{\varepsilon}{R} \bigg)^{d/2}
\mathbb{E}\bigg[\bigg(\frac{r_{*,T,\xi}}{\varepsilon}\bigg)^{(d-\delta)q/(1-\tau)}\bigg]^{(1-\tau)/2q}.
\end{align*}
This is precisely the desired estimate on the average of $\phi_\xi^T$ from \eqref{EstimateCorrectorAverage}.

\noindent
{\bf Part b: Estimates for linear functionals of the flux corrector $\sigma_\xi^T$.}
Taking the derivative of the flux corrector equation \eqref{EquationLocalizedFluxCorrector} with respect to a perturbation $\delta \omega_\varepsilon$ of the random field $\omega_\varepsilon$, we see that the change $\delta \sigma_{\xi}^T$ of the corrector under such an infinitesimal perturbation satisfies the linear elliptic PDE
\begin{align}
\label{EquationSensitivitySigma}
-\Delta \delta \sigma_{\xi,jk}^T + \frac{1}{T} \delta \sigma_{\xi,jk}^T
=&
\nabla \cdot (\partial_\omega A(\omega_\varepsilon,\xi+\nabla \phi_\xi^T)\delta \omega_\varepsilon \cdot (e_k\otimes e_j-e_j\otimes e_k))
\\&~
\nonumber
+\nabla \cdot (\partial_\xi A(\omega_\varepsilon,\xi+\nabla \phi_\xi^T)\nabla \delta \phi_\xi^T \cdot (e_k\otimes e_j-e_j\otimes e_k)).
\end{align}
For the sensitivity of the functional
\begin{align*}
F:=\int_\Rd g \cdot \nabla \sigma_{\xi,jk}^T \,dx,
\end{align*}
we obtain by defining $\bar h\in H^1(\Rd;\Rmd)$ as the unique decaying solution to the (system of) Poisson equation(s) with massive term
\begin{align}
\label{EquationSensitivitySigmabarh}
-\Delta \bar h + \frac{1}{T} \bar h = \nabla \cdot g
\end{align}
and defining $\hat h\in H^1(\Rd;\Rmd)$ as the unique decaying solution to the uniformly elliptic PDE (with $a_\xi^T(x):=\partial_\xi A(\omega_\varepsilon(x),\xi+\nabla \phi_\xi^T)$)
\begin{align}
\label{EquationSensitivitySigmahath}
-\nabla \cdot (a_\xi^{T,*} \nabla \hat h) +  \frac{1}{T} \hat h
=
\nabla \cdot \big((\partial_j \bar h e_k - \partial_k \bar h e_j)\cdot \partial_\xi A(\omega_\varepsilon,\xi+\nabla \phi_\xi^T)\big)
\end{align}
that
\begin{align*}
\delta F &= \int_\Rd g \cdot \nabla \delta \sigma_{\xi,jk}^T  \,dx
\stackrel{\eqref{EquationSensitivitySigmabarh}}{=} - \int_\Rd \nabla \bar h \cdot \nabla \delta \sigma_{\xi,jk}^T + \frac{1}{T} \bar h \, \delta \sigma_{\xi,jk}^T \,dx
\\&
\stackrel{\eqref{EquationSensitivitySigma}}{=} \int_\Rd \nabla \bar h \cdot  \big(\partial_\omega A(\omega_\varepsilon,\xi+\nabla \phi_\xi^T)\delta\omega_\varepsilon \cdot (e_k\otimes e_j-e_j\otimes e_k)\big) \,dx
\\&~~~~~
+\int_\Rd \nabla \bar h \cdot  \big(\partial_\xi A(\omega_\varepsilon,\xi+\nabla \phi_\xi^T)\nabla \delta \phi_\xi^T \cdot (e_k\otimes e_j-e_j\otimes e_k)\big) \,dx
\\&
\stackrel{\eqref{EquationSensitivitySigmahath}}{=} \int_\Rd \nabla \bar h \cdot  \big(\partial_\omega A(\omega_\varepsilon,\xi+\nabla \phi_\xi^T)\delta\omega_\varepsilon \cdot (e_k\otimes e_j-e_j\otimes e_k)\big) \,dx
\\&~~~~~
-\int_\Rd \partial_\xi A(\omega_\varepsilon,\xi+\nabla \phi_\xi^T)\nabla \delta \phi_\xi^T \cdot \nabla \hat h + \frac{1}{T} \delta \phi_\xi^T \, \hat h \,dx
\\&
\stackrel{\eqref{EquationSensitivityCorrector}}{=} \int_\Rd \nabla \bar h \cdot  \big(\partial_\omega A(\omega_\varepsilon,\xi+\nabla \phi_\xi^T)\delta\omega_\varepsilon \cdot (e_k\otimes e_j-e_j\otimes e_k)\big) \,dx
\\&~~~~
+\int_\Rd \partial_\omega A(\omega_\varepsilon,\xi+\nabla \phi_\xi^T) \delta \omega_\varepsilon \cdot \nabla \hat h \,dx.
\end{align*}
i.\,e.\ we have
\begin{align*}
\frac{\partial F}{\partial \omega_\varepsilon} = (\partial_j \bar h e_k - \partial_k \bar h e_j) \cdot \partial_\omega A(\omega_\varepsilon,\xi+\nabla \phi_\xi^T)
+\nabla \hat h \cdot \partial_\omega A(\omega_\varepsilon,\xi+\nabla \phi_\xi^T).
\end{align*}
Using again the version of the spectral gap inequality for the $q$-th moment from Lemma~\ref{LemmaLqSpectralGap}, the fact that $\mathbb{E}[F]=\smash{\int_\Rd g \cdot \nabla \mathbb{E}[\sigma_{\xi,jk}^T] \,dx}=0$ by stationarity, and the bound $|\partial_\omega A(\omega,\xi)|\leq \Lambda |\xi|$ (recall \hyperlink{A1}{(A1)}-\hyperlink{A2}{(A2)}), we obtain
\begin{align*}
\mathbb{E}\big[|F|^{2q}\big]^{1/2q}
\leq C q \varepsilon^{d/2} \mathbb{E}\Bigg[\bigg(\int_\Rd \bigg|\fint_{B_\varepsilon(x)} |\xi+\nabla \phi_\xi^T| (|\nabla \bar h|+|\nabla \hat h|) \,d\tilde x \bigg|^2 \,dx\bigg)^{q}\Bigg]^{1/2q}.
\end{align*}
By \eqref{EstimateGradByrstar}, we have
\begin{align*}
\fint_{B_\varepsilon(x)} |\xi+\nabla \phi_\xi^T|^2 \,d\tilde x
\leq C|\xi|^2
\bigg(\frac{r_{*,T,\xi}(x_0)}{\varepsilon}\bigg)^{d-\delta}=:B(x)^2.
\end{align*}
Using Lemma~\ref{LemmaEstimateSensitivity} we deduce that for any $\alpha_0\in (0,c(d,m,\lambda,\Lambda))$, any $p\in (2,2+c(d,m,\lambda,\Lambda,\alpha_0))$, any $\tau\in (0,1)$, and any $q\geq q(d,m,\lambda,\Lambda,\alpha_0,p)$
\begin{align*}
\mathbb{E}\big[|F|^{2q}\big]^{1/2q}
&\leq C |\xi| q \varepsilon^{d/2} r^{d/2}
\mathbb{E}\bigg[\bigg(\frac{r_{*,T,\xi}}{\varepsilon}\bigg)^{(d-\delta)q/(1-\tau)}\bigg]^{(1-\tau)/2q}
\\&~~~~~~~~~~~~\times
\mathbb{E}\Bigg[\bigg(r^{-d} \int_\Rd (|\nabla \bar h|^p + |\nabla \hat h|^p) \bigg(1+\frac{|x|}{r}\bigg)^{\alpha_0} \,dx\bigg)^{2q/p\tau} \Bigg]^{\tau/2q}.
\end{align*}
By the weighted Meyers estimate of Lemma~\ref{LemmaWeightedMeyers} and the defining equation \eqref{EquationSensitivitySigmahath} for $\hat h$ as well as the uniform bound on $\partial_\xi A$ inferred from \hyperlink{A2}{(A2)}, we deduce for any $\alpha_{1/2}\in (\alpha_0,c)$
\begin{align*}
\mathbb{E}\big[|F|^{2q}\big]^{1/2q}
&\leq C |\xi| q \varepsilon^{d/2} r^{d/2}
\mathbb{E}\bigg[\bigg(\frac{r_{*,T,\xi}}{\varepsilon}\bigg)^{(d-\delta)q/(1-\tau)}\bigg]^{(1-\tau)/2q}
\\&~~~~~~~~~~\times
\mathbb{E}\Bigg[\bigg(r^{-d} \int_\Rd |\nabla \bar h|^p \bigg(1+\frac{|x|}{r}\bigg)^{\alpha_{1/2}} \,dx\bigg)^{2q/p\tau} \Bigg]^{\tau/2q}.
\end{align*}
Applying Lemma~\ref{LemmaWeightedMeyers} to the equation \eqref{EquationSensitivitySigmabarh} and using $(\fint_{B_r(x_0)} |g|^p \,dx)^{1/p}\leq r^{-d}$, the integral in the last term may be estimated to yield
\begin{align*}
\mathbb{E}\big[|F|^{2q}\big]^{1/2q}
&\leq C |\xi| q \bigg(\frac{\varepsilon}{r} \bigg)^{d/2}
\mathbb{E}\bigg[\bigg(\frac{r_{*,T,\xi}}{\varepsilon}\bigg)^{(d-\delta)q/(1-\tau)}\bigg]^{(1-\tau)/2q}.
\end{align*}
This is the desired estimate on functionals of $\sigma_\xi^T$ in Lemma~\ref{LemmaEstimateLinearFunctionals}.

As in Part a, the estimate on the average of $\sigma_{\xi}^T$ from \eqref{EstimateCorrectorAverage} follows upon replacing the equation for $\bar h$ by the PDE $-\Delta \bar h+\frac{1}{T} \bar h=\frac{1}{|B_R|} \chi_{B_R(x_0)}$ in the previous computation.

{\bf Part c: Estimates on linear functionals of the potential $\theta_{\xi,i}^T$.} {We consider the case
  \begin{equation*}
    F=\int_{\Rd}g(x)\cdot \frac{1}{r}\nabla\theta^T_{\xi,i}(x)\,dx.
  \end{equation*}
In order to remain entirely rigorous, for the purpose of the present argument we would in principle need to approximate $\smash{\theta_{\xi,\Xi,i}^T}$ by quantities like $\theta_{\xi,\Xi,i}^{T,S}$ defined as the solution to the PDE $-\Delta \theta_{\xi,\Xi,i}^{T,S}+\tfrac{1}{S}\theta_{\xi,\Xi,i}^{T,S}=-\partial_i \phi_{\xi,\Xi}^T$. After obtaining our estimates, we would let $S\rightarrow \infty$ to conclude. However, as this step is analogous to our massive regularization of the corrector equation which we perform in full detail, we omit this additional technicality.}

Taking the derivative of the equation \eqref{EquationPotentialFieldGauge} with respect to a perturbation $\delta \omega_\varepsilon$ of the random field $\omega_\varepsilon$, we get
\begin{align}
\label{SensitivityPotential}
\Delta \delta \theta_{\xi,i}^T = \partial_i \delta \phi_\xi^T.
\end{align}
Defining $\bar g\in H^1(\Rd;\Rm)$ as the unique decaying solution to the PDE
\begin{align}
\label{Definitionbarg}
-\Delta \bar g = \nabla \cdot g,
\end{align}
we obtain
\begin{align*}
  \delta F= \frac{1}{r}\int_\Rd g \cdot \nabla \delta\theta_{\xi,i}^T \,dx
\stackrel{\eqref{Definitionbarg}}{=} -\frac{1}{r}\int_\Rd \nabla \bar g \cdot \nabla \delta \theta_{\xi,i}^T \,dx
\stackrel{\eqref{SensitivityPotential}}{=} \frac{1}{r}\int_\Rd \bar g (e_i \cdot \nabla) \delta \phi_\xi^T \,dx.
\end{align*}
In other words, we are back in the situation of Part a, the only difference being that the function $g$ has been replaced by the function $\frac{1}{r} \bar g e_i$. The function $\bar g$ now no longer has compact support; rather, being the solution to the Poisson equation \eqref{Definitionbarg} in $d\geq 3$ dimensions, for any $0<\alpha_1<1$ and any $p\geq 2$ it satisfies the weighted Calderon-Zygmund estimate
\begin{align*}
  \bigg(\dashint_{\Rd} |\bar g|^p \Big(1+\frac{|x-x_0|}{r}\Big)^{\alpha_1} \,dx\bigg)^{1/p} \leq C r \bigg(\dashint_{B_r(x_0)} |g|^p \,dx\bigg)^{1/p}.
\end{align*}
Now, we may follow the estimates from Part a leading to \eqref{SensitivitySigmaXiAlmostFinal} line by line. In the next step, we again employ Lemma~\ref{LemmaWeightedMeyers}, which by the previous estimate on $\bar g$  gives
\begin{align*}
\mathbb{E}\big[|F|^{2q}\big]^{1/2q}
&\leq C |\xi| q \bigg(\frac{\varepsilon}{r} \bigg)^{d/2}
\mathbb{E}\bigg[\bigg(\frac{r_{*,T,\xi}}{\varepsilon}\bigg)^{(d-\delta)q/(1-\tau)}\bigg]^{(1-\tau)/2q}.
\end{align*}
This is our desired estimate.

{\bf Part d: Proof of the estimate \eqref{EstimateCorrectorAverageDifference}.}
Let $w\in H^1(\Rd)$ be arbitrary.
Assuming for the moment $R=2^N r$ for some $N\in \mathbb{N}$, we may write
\begin{align*}
\bigg|\fint_{B_R(x_0)} w\,dx-\fint_{B_r(x_0)} w\,dx\bigg|^2
=\sum_{n=1}^N \big|\langle f_n,w\rangle_{L^2}\big|^2
\end{align*}
with $f_n:=\frac{1}{|B_{2^n r}|}\chi_{B_{2^n r}(x_0)}-\frac{1}{|B_{2^{n-1} r}|}\chi_{B_{2^{n-1} r}(x_0)}$, where we have used the orthogonality of the functions $f_n$. Solving the PDEs $\Delta v=f_n$ on $B_{2^n r}(x_0)$ with Neumann boundary conditions, we see that $g_n:=\nabla v \chi_{B_{2^n r}(x_0)}$ solves $\nabla \cdot g_n=f_n$ on $\Rd$ and satisfies for any $p\geq 2$ the estimate $(\fint_{B_{2^n r}(x_0)} |g_n|^p \,dx)^{1/p}\leq C (2^n r)^{1-d}$. This implies
\begin{align*}
\mathbb{E}\bigg[\bigg|\fint_{B_R(x_0)} w\,dx-\fint_{B_r(x_0)} w\,dx\bigg|^{2q}\bigg]^{1/q}
\leq 
\sum_{n=1}^N
\mathbb{E}\bigg[\bigg|\int_\Rd g_n\cdot\nabla w\bigg|^{2q}\bigg]^{1/q}.
\end{align*}
Combining this estimate with the bounds on $g_n$ and the estimates from Lemma~\ref{LemmaEstimateLinearFunctionals},
we deduce the estimate \eqref{EstimateCorrectorAverageDifference} if $R$ is of the form $R=2^N r$. If $R$ is not of this form, we need one additional step to estimate the difference of the averages on the radius $2^{\lfloor\log_2 \frac{R}{r}\rfloor}r$ and $R$.

\end{proof}

\subsection{Estimates on the corrector for the nonlinear PDE}
We will need the following technical lemma.
\begin{lemma}
\label{LemmaWavelets}
Let $R>0$ and $K\in \mathbb{N}$. For any function $v\in H^1([-R,R]^d;\mathbb{R}^m)$ the estimate
\begin{align*}
&\fint_{[-R,R]^d} \bigg|v-\fint_{[-R,R]^d} v \,d\tilde x\bigg|^2 \,dx
\\&
\leq C \sum_{k\in \{0,\ldots,K\}^d} R^2 \bigg|\fint_{[-R,R]^d}  \prod_{i=1}^d \cos\Big(\frac{\pi k_i(x_i+R)}{2R}\Big) \,\nabla \phi_\xi^T \,dx\bigg|^2
\\&~~~
+\frac{C}{K^2} R^2 \fint_{[-R,R]^d} |\nabla v|^2 \,dx
\end{align*}
holds.
\end{lemma}
\begin{proof}
The proof is an elementary consequence of the Fourier series representation of $v$.
\end{proof}

Combining the estimates on linear functionals of the corrector from Lemma~\ref{LemmaEstimateLinearFunctionals} with the estimate on the corrector gradient \eqref{EstimateGradByrstar} and the technical Lemma~\ref{LemmaWavelets}, we now derive stochastic moment bounds on the minimal radius $r_{*,T,\xi}$.

\begin{proof}[Proof of Lemma~\ref{MomentsMinimalRadius}]
In order to obtain a bound for the minimal radius $r_{*,T,\xi}$, we derive an estimate on the probability of the event $r_{*,T,\xi}(x_0)=R=2^\ell \varepsilon$ for a fixed $x_0\in \mathbb{R}^d$ and any $R=2^\ell \varepsilon > \varepsilon$. In the case of this event, we have by the Caccioppoli inequality \eqref{CaccippoliMonotone} applied to the function $\xi\cdot (x-x_0)+\phi_\xi^T$, which solves the PDE
\begin{align*}
-\nabla \cdot (A(\omega_\varepsilon,\nabla (\xi\cdot (x-x_0)+\phi_\xi^T))+\frac{1}{T}(\xi\cdot (x-x_0)+\phi_\xi^T)= \frac{1}{T} \xi \cdot (x-x_0),
\end{align*}
the definition of $r_{*,T,\xi}(x_0)$ in \eqref{Definitionrstar}, and the fact that $r_{*,T,\xi}\leq C \sqrt{T}$ (see \eqref{UpperBoundrstarxi})
\begin{align}
\label{AverageUpperBound}
\fint_{B_{dR}(x_0)} |\nabla \phi_\xi^T|^2 \,dx \leq C |\xi|^2.
\end{align}
Furthermore, in the event $r_{*,T,\xi}(x_0)=R>\varepsilon$ we also know by the definition \eqref{Definitionrstar} that at least one of the inequalities
\begin{align}
\label{AverageLowerBound}
&\frac{1}{R^2} \fint_{x_0+[-R,R]^d} \bigg|\phi_\xi^T-\fint_{x_0+[-R,R]^d} \phi_\xi^T \,d\tilde x\bigg|^2 \,dx
\\&
\nonumber
=\inf_{b\in \Rm} \frac{1}{R^2} \fint_{x_0+[-R,R]^d} |\phi_\xi^T-b|^2 \,dx
\\&
\nonumber
\geq c(d) \inf_{b\in \Rm} \frac{1}{(R/2)^2} \fint_{B_{R/2}(x_0)} |\phi_\xi^T-b|^2 \,dx
\\&
\nonumber
> c(d) |\xi|^2
\end{align}
or
\begin{align}
\label{AverageTLowerBound}
\frac{1}{\sqrt{T}} \bigg|\fint_{B_{R/2}(x_0)} \phi_\xi^T \,dx \bigg| > K_{mass} |\xi|
\end{align}
holds. We now distinguish these two cases.

\emph{Case 1: The estimate \eqref{AverageLowerBound} holds.}
By Lemma~\ref{LemmaWavelets}, we have for any $\delta>0$ for a sufficiently large $K=K(d,\delta)$
\begin{align}
\nonumber
&\frac{1}{R^2} \fint_{x_0+[-R,R]^d} \bigg|\phi_\xi^T-\fint_{x_0+[-R,R]^d} \phi_\xi^T \,d\tilde x\bigg|^2 \,dx
\\&
\label{BoundsRstarCase1}
\leq \delta \fint_{x_0+[-R,R]^d} |\nabla \phi_\xi^T|^2 \,dx
\\&~~~~
\nonumber
+C\sum_{k\in \{0,\ldots,K\}^d} \bigg|\fint_{x_0+[-R,R]^d}  \prod_{i=1}^d \cos\Big(\frac{\pi k_i(x_i+R)}{2R}\Big) \nabla \phi_\xi^T \,dx\bigg|^2.
\end{align}
Now let $g_{n,R}$ be the family of all functions
\begin{align*}
\mathbf{1}_{x_0+[-R,R]^d}(x) (2R)^{-d} \prod_{i=1}^d \cos\left(\frac{\pi k_i((x_0)_i+R)}{2R}\right) e_l\otimes e_j
\end{align*}
with $1\leq l\leq m$, $1\leq j\leq d$, and $k\in \{0,\ldots,K\}^d$.
Note that all $g_{n,R}$ are supported in  $B_{dR}(x_0)$ and satisfy for any $p\in [1,\infty]$
\begin{align}
\label{VerifiedConditionForPhiEstimate}
\bigg(\fint_{B_{dR}(x_0)} |g_{n,R}|^{p} \,dx\bigg)^{1/p}
\leq C(d,p) (dR)^{-d}.
\end{align}
Inserting both \eqref{AverageLowerBound} and \eqref{AverageUpperBound} in \eqref{BoundsRstarCase1} and choosing $\delta>0$ small enough (depending only on $d$ and the constants $c$ and $C$ from \eqref{AverageLowerBound} and \eqref{AverageUpperBound}), we obtain
\begin{align*}
\frac{1}{2} c|\xi|^2
\leq& C\sum_{k\in \{0,\ldots,K\}^d} \bigg|\fint_{x_0+[-R,R]^d}  \prod_{i=1}^d \cos\Big(\frac{\pi k_i(x_i+R)}{2R}\Big) \nabla \phi_\xi^T \,dx\bigg|^2.
\end{align*}
This implies that for at least one of the $N=N(d,K)$ functionals $\int_\Rd g_{n,R} \cdot \nabla \phi_\xi^T \,dx$ we have
\begin{align*}
\bigg|\int_\Rd g_{n,R} \cdot \nabla \phi_\xi^T \,dx\bigg|\geq c(d,m,\lambda,\Lambda,K) |\xi|.
\end{align*}
Fixing $\delta$ and $K$ depending only on $d$, $m$, $\lambda$, and $\Lambda$, this entails for any $q\geq 1$ by Chebyshev's inequality
\begin{align*}
\mathbb{P}[r_{*,T,\xi}(x_0)=R\text{ and }\eqref{AverageLowerBound}\text{ holds}] \leq \sum_{n=1}^N \frac{\mathbb{E}\big[\big|\int_\Rd g_{n,R} \cdot \nabla \phi_\xi^T \,dx\big|^q\big]}{c^q |\xi|^q}
\end{align*}
with $c=c(d,m,\lambda,\Lambda)$.
Estimating the weighted averages of $\nabla \phi_\xi^T$ by means of Lemma~\ref{LemmaEstimateLinearFunctionals} -- note that by \eqref{VerifiedConditionForPhiEstimate} and the support property of the $g_{n,R}$ the lemma is indeed applicable to the function $\frac{g_{n,R}}{C(d,p)}$ -- we obtain for any $0<\tau<1$ and any $q\geq C(d,m,\lambda,\Lambda,\tau)$
\begin{align}
\label{EstimateEventrstar1}
&\mathbb{P}[r_{*,T,\xi}(x_0)=R\text{ and }\eqref{AverageLowerBound}\text{ holds}]
\\&\nonumber~~~~~~~~
\leq C^q q^q \bigg(\frac{\varepsilon}{R} \bigg)^{q d/2}
\mathbb{E}\bigg[\bigg(\frac{r_{*,T,\xi}}{\varepsilon}\bigg)^{(d-\delta)q/2(1-\tau)}\bigg]^{(1-\tau)}.
\end{align}

\emph{Case 2: The estimate \eqref{AverageTLowerBound} holds.}
We first observe that as a consequence of Lemma~\ref{L:exploc} and \eqref{EquationLocalizedCorrector} -- rewritten in form of $-\nabla \cdot (A(\tilde \omega,\nabla \phi_\xi^T))+\frac{1}{T}\phi_\xi^T=\nabla \cdot \hat g$ with $\hat g:=A(\tilde \omega,\xi+\nabla \phi_\xi^T)-A(\tilde \omega,\nabla \phi_\xi^T)$ -- we have
\begin{align*}
\bigg|\fint_{B_{\sqrt{T}}(x_0)} \phi_\xi^T \,dx \bigg| \leq \bigg(\fint_{B_{\sqrt{T}}(x_0)} |\phi_\xi^T|^2 \,dx\bigg)^{1/2} \leq C |\xi| \sqrt{T}.
\end{align*}
Upon choosing $K_{mass}\geq C+1$, \eqref{AverageTLowerBound} is seen to imply
\begin{align*}
\bigg|\fint_{B_{R/2}(x_0)} \phi_\xi^T \,dx-\fint_{B_{\sqrt{T}}(x_0)} \phi_\xi^T \,dx\bigg| \geq |\xi| \sqrt{T}.
\end{align*}
In particular, we have for any $q\geq 1$
\begin{align*}
&\mathbb{P}[r_{*,T,\xi}(x_0)=R\text{ and }\eqref{AverageTLowerBound}\text{ holds}]
\\&
\leq
\mathbb{E}\bigg[\bigg|\frac{1}{|\xi|\sqrt{T}}\fint_{B_{R/2}(x_0)} \phi_\xi^T \,dx-\frac{1}{|\xi| \sqrt{T}}\fint_{B_{\sqrt{T}}(x_0)} \phi_\xi^T \,dx\bigg|^q\bigg].
\end{align*}
Inserting the estimate \eqref{EstimateCorrectorAverageDifference}, we deduce
\begin{align}
\nonumber
&\mathbb{P}[r_{*,T,\xi}(x_0)=R\text{ and }\eqref{AverageTLowerBound}\text{ holds}]
\\&
\nonumber
\leq C^q q^q \mathbb{E}\bigg[\bigg(\frac{r_{*,T,\xi}}{\varepsilon}\bigg)^{(d-\delta)q/2(1-\tau)}\bigg]^{(1-\tau)} \bigg(\sum_{l=0}^{\log_2 \frac{\sqrt{T}}{R}} \bigg(\frac{2^l R}{\sqrt{T}}\bigg)^2 \bigg(\frac{\varepsilon}{2^l R}\bigg)^d \bigg)^{q/2}
\\&
\label{EstimateEventrstar2}
\leq C^q q^q  \bigg(\frac{\varepsilon}{R}\bigg)^{qd/2} \mathbb{E}\bigg[\bigg(\frac{r_{*,T,\xi}}{\varepsilon}\bigg)^{(d-\delta)q/2(1-\tau)}\bigg]^{(1-\tau)}.
\end{align}

\emph{Conclusion: Estimates on the minimal radius.}
Taking the sum of \eqref{EstimateEventrstar1} and \eqref{EstimateEventrstar2} over all dyadic $R=2^k \varepsilon>\varepsilon$ and using the fact that $r_{*,T,\xi}\geq \varepsilon$ by its definition \eqref{Definitionrstar}, we deduce
\begin{align*}
&\mathbb{E}\bigg[\Big(\frac{r_{*,T,\xi}(x_0)}{\varepsilon}\Big)^{q(d-\delta/2)/2}\bigg]
\\&
~~~\leq 1 +\sum_{k=1}^\infty (2^k)^{q(d-\delta/2)/2} \mathbb{P}\big[r_{*,T,\xi}(x_0)= 2^{k}\varepsilon\big]
\\&
\stackrel{\eqref{EstimateEventrstar1},\eqref{EstimateEventrstar2}}{\leq}
1 + C^q q^q \sum_{k=1}^\infty (2^k)^{q(d-\delta/2)/2}
\bigg(\frac{\varepsilon}{\varepsilon 2^k} \bigg)^{q d/2}
\mathbb{E}\bigg[\bigg(\frac{r_{*,T,\xi}}{\varepsilon}\bigg)^{(d-\delta)q/2(1-\tau)}\bigg]^{(1-\tau)}
\\&~~~
\leq
1 + C^q q^q
\mathbb{E}\bigg[\bigg(\frac{r_{*,T,\xi}}{\varepsilon}\bigg)^{(d-\delta)q/2(1-\tau)}\bigg]^{(1-\tau)},
\end{align*}
where $C=C(d,m,\lambda,\Lambda,\rho,\delta,\tau)$ and where $q\geq C(d,m,\lambda,\Lambda,\delta)$.
Now, fixing $\delta>0$ small enough (depending only on $d$, $m$, $\lambda$, and $\Lambda$) and choosing $\tau\in (0,1)$ such that $1-\tau=(d-\delta)/(d-\delta/2)$, the estimate buckles and yields a bound of the form
\begin{align*}
\mathbb{E}\bigg[\Big(\frac{r_{*,T,\xi}(x_0)}{\varepsilon}\Big)^{(d-\delta/2)q/2}\Big]^\tau \leq C(d,m,\lambda,\Lambda,\rho)^q q^q
\end{align*}
for all $q\geq C(d,m,\lambda,\Lambda)$. This may be rewritten as
\begin{align*}
\mathbb{E}\bigg[\Big(\frac{r_{*,T,\xi}(x_0)}{\varepsilon}\Big)^{(d-\delta/2)q/2}\Big]^{1/q} \leq C(d,m,\lambda,\Lambda,\rho) q^C,
\end{align*}
which establishes Lemma~\ref{MomentsMinimalRadius}.
\end{proof}

We now derive the estimates on the corrector $\phi_\xi^T$, the flux corrector $\smash{\sigma_\xi^T}$, and the potential field $\smash{\theta_\xi^T}$. Note that the only required ingredients for the proof are the estimate for functionals of the corrector from Lemma~\ref{LemmaEstimateLinearFunctionals}, the estimate on the corrector gradient given by \eqref{EstimateGradByrstar}, the moment bounds for the minimal radius from Lemma~\ref{MomentsMinimalRadius}, as well as the general technical results of Lemma~\ref{LemmaWavelets} and Lemma~\ref{LemmaMultiscaleDecomposition}.

\begin{proof}[Proof of Proposition~\ref{PropositionCorrectorEstimate}]
We insert the estimate on $r_{*,T,\xi}$ from Lemma~\ref{MomentsMinimalRadius} into Lemma~\ref{LemmaEstimateLinearFunctionals}.
For any $x_0\in \Rd$, any $r\geq \varepsilon$, and any $g$ with $\supp g\subset B_r(x_0)$ and
\begin{align*}
\bigg(\fint_{B_r(x_0)} |g|^{p} \,dx\bigg)^{1/p}\leq r^{-d},
\end{align*}
this provides bounds of the form
\begin{align}
\label{BoundLinearFunctionalPhi}
&\mathbb{E}
\Bigg[\Bigg|\int_\Rd g \cdot \nabla \phi_\xi^T  \,dx \bigg|^{q}\Bigg]^{1/q}
\leq
C |\xi| q^{C(d,m,\lambda,\Lambda)} \bigg(\frac{\varepsilon}{r}\bigg)^{d/2}
\end{align}
and
\begin{align}
\label{BoundLinearFunctionalSigma}
&\mathbb{E}
\Bigg[\Bigg|\int_\Rd g \cdot \nabla \sigma_{\xi,jk}^T  \,dx \bigg|^{q}\Bigg]^{1/q}
\leq
C |\xi| q^{C(d,m,\lambda,\Lambda)} \bigg(\frac{\varepsilon}{r}\bigg)^{d/2}
\end{align}
as well as
\begin{align}
\label{BoundLinearFunctionalTheta}
&\mathbb{E}
\Bigg[\Bigg|\int_\Rd g \cdot \nabla \theta_\xi^T  \,dx \bigg|^{q}\Bigg]^{1/q}
\leq
C r |\xi| q^{C(d,m,\lambda,\Lambda)} \bigg(\frac{\varepsilon}{r}\bigg)^{d/2}
\end{align}
for any $q\geq C(d,m,\lambda,\Lambda)$.

Plugging in the estimates on $r_{*,T,\xi}$ from Lemma~\ref{MomentsMinimalRadius} into \eqref{EstimateGradByrstar}, we deduce for any $x_0\in \Rd$
\begin{align}
\label{MomentBoundNablaPhi}
&\mathbb{E}
\Bigg[\Bigg|\fint_{B_\varepsilon(x_0)} |\nabla \phi_\xi^T|^2  \,dx \bigg|^{q/2}\Bigg]^{1/q}
\leq
C |\xi| q^{C(d,m,\lambda,\Lambda)}
\end{align}
for any $q\geq C(d,m,\lambda,\Lambda)$. Plugging in this bound and \eqref{BoundLinearFunctionalPhi} into the (spatially rescaled) multiscale estimate for the $L^2$ norm from Lemma~\ref{LemmaMultiscaleDecomposition},
{we obtain
\begin{align*}
\bigg(\fint_{B_r(x_0)} \bigg|\phi_\xi^T-\fint_{B_r(x_0)} \phi_\xi^T(\tilde x) \,d\tilde x\bigg|^2 \,dx\bigg)^{1/2}
\leq
\mathcal{C}(x_0) |\xi| \varepsilon
\begin{cases}
(r/\varepsilon)^{1/2} &\text{for }d=1,
\\
 \big|\log \frac{r}{\varepsilon}\big|^{1/2} &\text{for }d=2,
\\
1 &\text{for }d\geq 3,
\end{cases}
\end{align*}
as well as the corresponding estimate for the $L^p$ norm (with an additional factor $|\log \tfrac{r}{\varepsilon}|^{1/2}$ in case $d=2$).
Combining these bounds with Lemma~\ref{LemmaCorrectorAverages} (and Lemma~\ref{MomentsMinimalRadius}), this establishes the estimate on $\phi_\xi^T$ stated in \eqref{CorrectorEstimate} as well as the corresponding $L^p$ norm bound.
}

The estimates on $\phi_\xi^T$ in \eqref{CorrectorEstimateWithoutAverage} and \eqref{CorrectorEstimateAveragesLowd} are shown by combining \eqref{CorrectorEstimate} with the estimates from Lemma~\ref{LemmaCorrectorAverages} and \eqref{MomentBoundNablaPhi} as well as using in case $d\geq 3$ the relation
\begin{align*}
\lim_{R\rightarrow \infty} \fint_{B_R(x_0)} \phi_\xi^T \,dx =0.
\end{align*}

In view of this use of Lemma~\ref{LemmaMultiscaleDecomposition}, in order to obtain the estimates on $\sigma_{\xi,jk}^T$ and $\theta_\xi^T$ stated in \eqref{CorrectorEstimate} and \eqref{PotentialFieldEstimate}, we only need to establish estimates on $\fint_{B_\varepsilon(x_0)} |\nabla \sigma_{\xi ij}^T|^2  \,dx$ and $\fint_{B_\varepsilon(x_0)} |\nabla \theta_{\xi}^T|^2 \,dx$.
The Caccioppoli inequality \eqref{CaccippoliMonotone} for the defining equation of the flux corrector \eqref{EquationLocalizedFluxCorrector} yields in conjunction with \hyperlink{A2}{(A2)}
\begin{align*}
\fint_{x_0+[-\varepsilon,\varepsilon]^d} |\nabla \sigma_{\xi,jk}^T|^2 \,dx
\leq &\frac{C}{\varepsilon^2} \fint_{x_0+[-2\varepsilon,2\varepsilon]^d} \bigg|\sigma_{\xi,jk}^T-\fint_{x_0+[-2\varepsilon,2\varepsilon]^d} \sigma_{\xi,jk}^T \,d\tilde x\bigg|^2 \,dx
\\&
+C\fint_{x_0+[-2\varepsilon,2\varepsilon]^d} |\xi+\nabla \phi_\xi|^2 \,dx
\\&
+\frac{C}{T}\bigg| \fint_{x_0+[-2\varepsilon,2\varepsilon]^d} \sigma_{\xi,jk}^T \,d\tilde x \bigg|^2.
\end{align*}
Estimating the first term on the right-hand side by Lemma~\ref{LemmaWavelets} and taking stochastic moments, we deduce
\begin{align*}
&\mathbb{E}\bigg[\bigg|\fint_{x_0+[-2\varepsilon,2\varepsilon]^d} |\nabla \sigma_{\xi,jk}^T|^2 \,dx\bigg|^{q/2}\bigg]^{1/q}
\\
&\leq \frac{C}{K} \mathbb{E}\bigg[\bigg|\fint_{x_0+[-2\varepsilon,2\varepsilon]^d} |\nabla \sigma_{\xi,jk}^T|^2 \,dx\bigg|^{q/2}\bigg]^{1/q}
\\&~~~
+\sum_{n=1}^{N(d,K)} C \mathbb{E}\bigg[\bigg|\int_\Rd g_n \cdot\nabla \sigma_{\xi,jk}^T \,dx\bigg|^{q}\bigg]^{1/q}
\\&~~~
+C \mathbb{E}\bigg[\bigg|\fint_{x_0+[-2\varepsilon,2\varepsilon]^d} |\xi+\nabla \phi_\xi^T|^2 \,dx\bigg|^{q/2}\bigg]^{1/q}
\\&~~~
+\frac{C}{\sqrt{T}} \mathbb{E}\bigg[\bigg| \fint_{x_0+[-2\varepsilon,2\varepsilon]^d} \sigma_{\xi,jk}^T \,d\tilde x \bigg|^{q}\bigg]^{1/q}.
\end{align*}
Choosing $K$ large enough and using stationarity, we may absorb the first term on the right-hand side in the left-hand side. Estimating the linear functionals of $\nabla \sigma_{\xi,jk}^T$ by \eqref{BoundLinearFunctionalSigma}, bounding the third term on the right-hand side by \eqref{MomentBoundNablaPhi}, and estimating the last term by \eqref{EstimateCorrectorAverageDifference} and \eqref{EstimateCorrectorAverage} (where we may replace the average over the box $x_0+[-2\varepsilon,2\varepsilon]^d$ by the average over the ball $B_\varepsilon(x_0)$ using \eqref{BoundLinearFunctionalSigma}), we deduce
\begin{align*}
&\mathbb{E}\bigg[\bigg|\fint_{x_0+[-\varepsilon,\varepsilon]^d} |\nabla \sigma_{\xi,jk}^T|^2 \,dx\bigg|^{q/2}\bigg]^{1/q}
\leq
C|\xi| q^C.
\end{align*}
Together with \eqref{BoundLinearFunctionalSigma} and Lemma~\ref{LemmaMultiscaleDecomposition}, we deduce the bound for $\sigma_{\xi,jk}^T$ in \eqref{CorrectorEstimate}.

The estimates on $\sigma_\xi^T$ in \eqref{CorrectorEstimateWithoutAverage} and \eqref{CorrectorEstimateAveragesLowd} are again shown by combining \eqref{CorrectorEstimate} with the estimate from Lemma~\ref{LemmaCorrectorAverages} and \eqref{MomentBoundNablaPhi} as well as using in case $d\geq 3$ the relation
\begin{align*}
\lim_{R\rightarrow \infty} \fint_{B_R(x_0)} \sigma_\xi^T \,dx =0.
\end{align*}

The estimate for the gradient of the potential field $\nabla \theta_\xi^T$ is analogous but even simpler (due to the lack of the massive regularization in \eqref{EquationPotentialFieldGauge}). We obtain the bound
\begin{align*}
\mathbb{E}\bigg[\bigg|\fint_{x_0+[-\varepsilon,\varepsilon]^d} |\nabla \theta_{\xi}^T|^2 \,dx\bigg|^{q/2}\bigg]^{1/q}
\leq C q^C \varepsilon |\xi|.
\end{align*}
Using this estimate and \eqref{BoundLinearFunctionalSigma} in Lemma~\ref{LemmaMultiscaleDecomposition}, we obtain \eqref{PotentialFieldEstimate}.
\end{proof}

\section{Corrector estimates for the linearized PDE}

\label{SectionLinearizedCorrectorEstimates}

\subsection{Estimates on linear functionals of the corrector and the flux corrector for the linearized PDE}

\begin{proof}[Proof of Lemma~\ref{LemmaEstimateLinearFunctionalsLinearized} and Lemma~\ref{LemmaCorrectorAveragesLinearized}]{\bf Part a: Estimates for linear functionals of the homogenization corrector $\phi_{\xi,\Xi}^T$.}
Without loss of generality we may assume in the following argument that $x_0=0$, i.\,e.\ $g$ is supported in $B_r(0)$, and that $(\fint_{B_r}|g|^{p^2/2}\,dx)^{2/p^2}\leq 1$ -- otherwise replace in the following argument $p$ by $\sqrt{2p}$.

The argument is similar to the case of the corrector $\phi_\xi^T$. We first observe that the expectation $\mathbb{E}[F]$ vanishes. Indeed, $\phi_{\xi,\Xi}^T$ is easily seen to be a stationary random field, which entails $\mathbb{E}[F]=\int_\Rd g\cdot \nabla \mathbb{E}[\phi_{\xi,\Xi}^T]\,dx=0$. By Lemma~\ref{LemmaLqSpectralGap}, to obtain stochastic moment bounds for $F$ it suffices to estimate the sensitivity of $F$ with respect to changes in the random field $\omega_\varepsilon$.
Taking the derivative with respect to $\omega_\varepsilon$ in \eqref{EquationLocalizedCorrectorLinearized}, we obtain
\begin{align}
\nonumber
&-\nabla \cdot \big(\partial_\xi A(\omega_\varepsilon(x),\xi+\nabla \phi_\xi^T)\nabla \delta \phi_{\xi,\Xi}^T\big)
+\frac{1}{T}\delta \phi_{\xi,\Xi}^T
\\&
\label{EquationSensitivityLinearizedCorrector}
=\nabla \cdot \big(\partial_\omega \partial_\xi A(\omega_\varepsilon(x),\xi+\nabla \phi_\xi^T)\delta\omega_\varepsilon \big(\Xi+\nabla \phi_{\xi,\Xi}^T\big)\big)
\\&~~~~
\nonumber
+\nabla \cdot \big(\partial_\xi^2 A(\omega_\varepsilon(x),\xi+\nabla \phi_\xi^T)\big(\Xi+\nabla \phi_{\xi,\Xi}^T\big) \nabla \delta \phi_\xi^T \big).
\end{align}
Denoting by $h$ the unique solution in $H^1(\Rd;\Rm)$ to the auxiliary PDE
\begin{align}
\label{EquationDualLinearized}
-\nabla \cdot (a_\xi^{T,*}\nabla h)+\frac{1}{T} h=\nabla \cdot g
\end{align}
(where we again used the abbreviation $a_\xi^{T,*}:=(\partial_\xi A(\omega_\varepsilon(x),\xi+\nabla \phi_\xi^T))^*$) and denoting by $\hat h\in H^1(\Rd;\Rm)$ the unique solution to the auxiliary PDE
\begin{align}
\label{EquationDual2Linearized}
&-\nabla \cdot (a_\xi^{T,*}\nabla \hat h)+\frac{1}{T} \hat h
\\&
\nonumber
=\sum_{j=1}^d \sum_{l=1}^m \partial_j \big(\partial_\xi^2 A (\omega_\varepsilon(x),\xi+\nabla \phi_\xi^T) (\Xi+\nabla \phi_{\xi,\Xi}^T)(e_l\otimes e_j) \cdot \nabla h\big) e_l,
\end{align}
we deduce
\begin{align*}
\delta F &= \int_\Rd g\cdot \nabla \delta \phi_{\xi,\Xi}^T \,dx
\stackrel{\eqref{EquationDualLinearized}}{=}
-\int_\Rd a_\xi^T \nabla \delta \phi_{\xi,\Xi}^T \cdot \nabla h + \frac{1}{T} \delta \phi_{\xi,\Xi}^T \, h\,dx
\\&
\stackrel{\eqref{EquationSensitivityLinearizedCorrector}}{=}
\int_\Rd \partial_\omega \partial_\xi A (\omega_\varepsilon(x),\xi+\nabla \phi_\xi^T) \delta\omega_\varepsilon (\Xi+\nabla \phi_{\xi,\Xi}^T) \cdot \nabla h \,dx
\\&~~~~
+\int_\Rd \partial_\xi^2 A (\omega_\varepsilon(x),\xi+\nabla \phi_\xi^T) (\Xi+\nabla \phi_{\xi,\Xi}^T) \nabla \delta \phi_\xi^T  \cdot \nabla h \,dx
\\&
\stackrel{\eqref{EquationDual2Linearized}}{=}
\int_\Rd \partial_\omega \partial_\xi A (\omega_\varepsilon(x),\xi+\nabla \phi_\xi^T) \delta\omega_\varepsilon (\Xi+\nabla \phi_{\xi,\Xi}^T) \cdot \nabla h \,dx
\\&~~~~~~~~
-\int_\Rd a_\xi^{T,*}\nabla \hat h\cdot \nabla \delta \phi_\xi^T +\frac{1}{T} \hat h \,\delta \phi_\xi^T \,dx
\\&
\stackrel{\eqref{EquationSensitivityCorrector}}{=}
\int_\Rd \partial_\omega \partial_\xi A (\omega_\varepsilon(x),\xi+\nabla \phi_\xi^T) \delta\omega_\varepsilon (\Xi+\nabla \phi_{\xi,\Xi}^T) \cdot \nabla h \,dx
\\&~~~~~~~~
+\int_\Rd \partial_\omega A(\omega_\varepsilon(x),\xi+\nabla \phi_\xi^T) \delta \omega_\varepsilon \cdot \nabla \hat h \,dx.
\end{align*}
In other words, we have the representation
\begin{align}
\label{FunctionalSensitivityEquation}
\frac{\partial F}{\partial \omega_\varepsilon} = \partial_\omega \partial_\xi A(\omega_\varepsilon(x),\xi+\nabla \phi_\xi^T)(\Xi+\nabla \phi_{\xi,\Xi}^T) \cdot \nabla h + \partial_\omega A(\omega_\varepsilon(x),\xi+\nabla \phi_\xi^T) \cdot \nabla \hat h.
\end{align}
By \hyperlink{A3}{(A3)}, this implies the sensitivity estimate
\begin{align*}
&\int_\Rd \bigg|\dashint_{B_\varepsilon(x)} \bigg|\frac{\partial F}{\partial \omega_\varepsilon}\bigg| \,d\tilde x \bigg|^2 \,dx
\\&
\leq
C \int_\Rd \bigg|\dashint_{B_\varepsilon(x)} |\Xi + \nabla \phi_{\xi,\Xi}^T| |\nabla h| \,d\tilde x \bigg|^2 \,dx
+C \int_\Rd \bigg|\dashint_{B_\varepsilon(x)} |\xi+\nabla \phi_\xi^T| |\nabla \hat h| \,d\tilde x \bigg|^2 \,dx.
\end{align*}
Plugging this bound into the version of the spectral gap inequality for the $q$-th moment (see Lemma~\ref{LemmaLqSpectralGap}) and using $\mathbb{E}[F]=0$, we deduce for any $q\geq 1$
\begin{align*}
\mathbb{E}\Big[\big|F\big|^{2q}\Big]^{1/2q}
&\leq C q \varepsilon^{d/2} \mathbb{E}\Bigg[\bigg(\int_\Rd \bigg|\dashint_{B_\varepsilon(x)} |\Xi + \nabla \phi_{\xi,\Xi}^T| |\nabla h| \,d\tilde x \bigg|^2 \,dx\bigg)^{q}\Bigg]^{1/2q}
\\&~~~~
+C q \varepsilon^{d/2} \mathbb{E}\Bigg[\bigg(\int_\Rd \bigg|\dashint_{B_\varepsilon(x)} |\xi + \nabla \phi_{\xi}^T| |\nabla \hat h| \,d\tilde x \bigg|^2 \,dx\bigg)^{q}\Bigg]^{1/2q}.
\end{align*}
By Lemma~\ref{LemmaEstimateSensitivity} and \eqref{EstimateGradByrstar2} as well as \eqref{EstimateGradByrstar}, this entails for any $\tau,\bar \tau \in (0,1)$
\begin{align}
\label{EstimateSensitivityLinearizedFirstIntermediate}
&\mathbb{E}\Big[\big|F\big|^{2q}\Big]^{1/2q}
\\&
\nonumber
\leq
C |\Xi| q \varepsilon^{d/2} r^{d/2}
\mathbb{E}\bigg[\bigg(\frac{r_{*,T,\xi,\Xi}}{\varepsilon}\bigg)^{(d-\delta)q/(1-\tau)}\bigg]^{(1-\tau)/2q}
\\&~~~~~~~~~~~~\times
\nonumber
\mathbb{E}\Bigg[\bigg(r^{-d} \int_\Rd |\nabla h|^p \bigg(1+\frac{|x|}{r}\bigg)^{\alpha_0} \,dx\bigg)^{2q/p\tau} \Bigg]^{\tau/2q}
\\&~~~~
\nonumber
+C |\xi| q \varepsilon^{d/2} r^{d/2}
\mathbb{E}\bigg[\bigg(\frac{r_{*,T,\xi}}{\varepsilon}\bigg)^{(d-\delta)q/(1-\bar \tau)}\bigg]^{(1-\bar \tau)/2q}
\\&~~~~~~~~~~~~\times
\nonumber
\mathbb{E}\Bigg[\bigg(r^{-d} \int_\Rd |\nabla \hat h|^p \bigg(1+\frac{|x|}{r}\bigg)^{\alpha_0} \,dx\bigg)^{2q/p\bar \tau} \Bigg]^{\bar \tau/2q}
\end{align}
with $C=C(d,m,\lambda,\Lambda,\rho,\alpha_0,p,\tau,\bar \tau)$.
By the weighted Meyers estimate in Lemma~\ref{LemmaWeightedMeyers} -- applied to \eqref{EquationDual2Linearized} -- and the uniform bound $|\partial_\xi^2 A|\leq \Lambda$ from \hyperlink{R}{(R)}, we infer
\begin{align}
\label{EstimateSensitivityLinearizedIntermediate}
&\mathbb{E}\Bigg[\bigg(r^{-d} \int_\Rd |\nabla \hat h|^p \bigg(1+\frac{|x|}{r}\bigg)^{\alpha_0} \,dx\bigg)^{2q/p\bar \tau} \Bigg]^{\bar \tau/2q}
\\&~~~~
\nonumber
\leq C
\mathbb{E}\Bigg[\bigg(r^{-d} \int_\Rd |\nabla h|^p |\Xi+\nabla \phi_{\xi,\Xi}^T|^p \bigg(1+\frac{|x|}{r}\bigg)^{\alpha_0} \,dx\bigg)^{2q/p\bar \tau} \Bigg]^{\bar \tau/2q}
\end{align}
with $C=C(d,m,\lambda,\Lambda,\rho,\alpha_0,p,\tau,\bar\tau)$.
Inserting the bound \eqref{EstimateLinearCorrectorLinfty} into \eqref{EstimateSensitivityLinearizedIntermediate}, we get
\begin{align*}
&\mathbb{E}\Bigg[\bigg(r^{-d} \int_\Rd |\nabla \hat h|^p \bigg(1+\frac{|x|}{r}\bigg)^{\alpha_0} \,dx\bigg)^{2q/p\bar \tau} \Bigg]^{\bar \tau/2q}
\\&
\leq
C q^C (1+|\xi|)^C|\Xi|
\\&~~\times
\mathbb{E}\Bigg[\bigg(\int_\Rd r^{-d} \mathcal{C}_{reg,\xi}^p(x)\bigg(\frac{r_{*,T,\xi,\Xi}(x)}{\varepsilon}\bigg)^{(d-\delta)p/2} |\nabla h|^p 
 \bigg(1+\frac{|x|}{r}\bigg)^{\alpha_0} \,dx\bigg)^{2q/p\bar \tau} \Bigg]^{\bar \tau/2q}
\end{align*}
where $C$ now additionally depends on $\nu$.
To proceed further, we write the second factor on the right-hand side in the form 
\begin{equation*}
\mathbb{E}\Bigg[\bigg(\int_\Rd v_1v_2 \,dx\bigg)^{2q/p\bar \tau} \Bigg]^{\bar \tau/2q}
\end{equation*}
with
\begin{eqnarray*}
  v_1= \mathcal{C}_{reg,\xi}^p \bigg(\frac{r_{*,T,\xi,\Xi}}{\varepsilon}\bigg)^{(d-\delta)p/2},\qquad v_2=r^{-d}|\nabla h|^p 
 \bigg(1+\frac{|x|}{r}\bigg)^{\alpha_0}.
\end{eqnarray*}
By first smuggling in the weight
\begin{equation*}
\varphi(x)=r^{-d}\bigg(1+\frac{|x|}{r}\bigg)^{-(d+1)},
\end{equation*}
via H\"older's inequality with exponents $\frac p{p-2}$ and $\frac p 2$ in space, and next by H\"older's inequality with exponents $\frac{\bar\tau}{\bar\tau-\tau/2}$ and $\frac{2\bar\tau}{\tau}$ in probability (here we need to assume $\bar \tau>\tfrac{\tau}{2}$), we get
\begin{align*}
 &\mathbb{E}\Bigg[\bigg(\int_\Rd v_1v_2 \,dx\bigg)^{2q/p\bar \tau} \Bigg]^{\bar \tau/2q}\\
&\leq\,
  \mathbb{E}\Bigg[\bigg(\int_\Rd v_1^{\frac p{p-2}}\varphi\,dx\bigg)^{\frac{p-2}{p}\frac{2q}{p(\bar\tau-\tau/2)}}\Bigg]^{\frac{\bar \tau-\tau/2}{2q}}\,      \mathbb{E}\Bigg[\bigg(\int_\Rd v_2^{\frac{p}{2}}\varphi^{-\frac{p-2}{2}} \,dx\bigg)^{\frac{8q}{p^2 \tau}} \Bigg]^{\frac{\tau}{4q}}.
\end{align*}
By Jensen's inequality for the integral $\int_{\Rd}f \varphi \,dx$  -- using that $\varphi$ has mass of order unity and assuming also that $q\geq C(p,\tau,\bar \tau)$ --  and by exploiting the fact that $v_1$ is a stationary random field, we deduce that the right-hand side is bounded from above by
\begin{align*}
C  \mathbb{E}\Big[v_1^{\frac{2q}{p(\bar\tau-\tau/2)}}\Big]^{\frac{\bar \tau-\tau/2}{2q}}\,      \mathbb{E}\Bigg[\bigg(\int_\Rd v_2^{\frac{p}{2}}\varphi^{-\frac{p-2}{2}} \,dx\bigg)^{\frac{8q}{p^2 \tau}} \Bigg]^{\frac{\tau}{4q}}.
\end{align*}
By combining the previous estimates and plugging in the definitions of $v_1,v_2$ and $\varphi$, we thus arrive at
\begin{align*}
&\mathbb{E}\Bigg[\bigg(r^{-d} \int_\Rd |\nabla \hat h|^p \bigg(1+\frac{|x|}{r}\bigg)^{\alpha_0} \,dx\bigg)^{2q/p\bar \tau} \Bigg]^{\bar \tau/2q}
\\&
\leq
C q^C (1+|\xi|)^C |\Xi| \,
\mathbb{E}\Bigg[|\mathcal{C}_{reg,\xi}|^{2q/(\bar\tau-\tau/2)} \bigg(\frac{r_{*,T,\xi,\Xi}}{\varepsilon}\bigg)^{(d-\delta)q/(\bar \tau-\tau/2)}\Bigg]^{(\bar \tau-\tau/2)/2q}
\\&~~~~~
\times
\mathbb{E}\Bigg[\bigg(\int_\Rd r^{-d} |\nabla h|^{p^2/2} 
 \bigg(1+\frac{|x|}{r}\bigg)^{\alpha_{0}p/2+(d+1)(p-2)/2} \,dx\bigg)^{8q/p^2 \tau} \Bigg]^{\tau/4q}.
\end{align*}
Applying once more H\"older's inequality to the first expected value on the right-hand side, choosing $\bar \tau\in (1-\tfrac{\tau}{2},1)$, and using the stretched exponential moment bounds for $\mathcal{C}_{reg,\xi}$, we deduce
\begin{align}
\label{MomentBoundsLpHath}
&\mathbb{E}\Bigg[\bigg(r^{-d} \int_\Rd |\nabla \hat h|^p \bigg(1+\frac{|x|}{r}\bigg)^{\alpha_0} \,dx\bigg)^{2q/p\bar \tau} \Bigg]^{\bar \tau/2q}
\\&
\nonumber
\leq
C q^C (1+|\xi|)^C |\Xi| \,
\mathbb{E}\bigg[\bigg(\frac{r_{*,T,\xi,\Xi}}{\varepsilon}\bigg)^{(d-\delta)q/(1-\tau)}\bigg]^{(1-\tau)/2q}
\\&~~~~~
\nonumber
\times
\mathbb{E}\Bigg[\bigg(\int_\Rd r^{-d} |\nabla h|^{p^2/2} 
 \bigg(1+\frac{|x|}{r}\bigg)^{\alpha_{0}p/2+(d+1)(p-2)/2} \,dx\bigg)^{8q/p^2 \tau} \Bigg]^{\tau/4q}.
\end{align}
Plugging in this estimate into \eqref{EstimateSensitivityLinearizedFirstIntermediate} and estimating  $r_{*,T,\xi}$ in \eqref{EstimateSensitivityLinearizedFirstIntermediate} via Lemma~\ref{MomentsMinimalRadius} yields
\begin{align*}
&\mathbb{E}\Big[\big|F\big|^{2q}\Big]^{1/2q}
\\&
\nonumber
\leq
C  q^C (1+|\xi|)^C |\Xi| \varepsilon^{d/2} r^{d/2}
\mathbb{E}\bigg[\bigg(\frac{r_{*,T,\xi,\Xi}}{\varepsilon}\bigg)^{(d-\delta)q/(1-\tau)}\bigg]^{(1-\tau)/2q}
\\&~~~~~~~~~~~~\times
\nonumber
\mathbb{E}\Bigg[\bigg(\int_\Rd r^{-d} |\nabla h|^{p^2/2} 
 \bigg(1+\frac{|x|}{r}\bigg)^{\alpha_{0}p/2+(d+1)(p-2)/2} \,dx\bigg)^{8q/p^2 \tau} \Bigg]^{\tau/4q}.
\end{align*}
Choosing $p$ close enough to $2$ and choosing $\alpha_0$ small enough (all depending only on $d$, $m$, $\lambda$, and $\Lambda$), we may estimate the last factor by applying the weighted Meyers estimate from Lemma~\ref{LemmaWeightedMeyers} to the PDE \eqref{EquationDualLinearized}. This yields by our assumed bound on $g$
\begin{align}
\label{EstimateSensitivityLinearized}
&\mathbb{E}\Big[\big|F\big|^{2q}\Big]^{1/2q}
\leq
C  q^C (1+|\xi|)^C |\Xi| \bigg(\frac{\varepsilon}{r}\bigg)^{d/2} \mathbb{E}\bigg[\bigg(\frac{r_{*,T,\xi,\Xi}}{\varepsilon}\bigg)^{(d-\delta)q/(1-\tau)}\bigg]^{(1-\tau)/2q}
\end{align}
for any $q\geq C$ with $C=C(d,m,\lambda,\Lambda,\rho,\nu,\tau)$ and any $0<\tau<1$, which is the desired bound for functionals of $\phi_{\xi,\Xi}^T$ in Lemma~\ref{LemmaEstimateLinearFunctionalsLinearized}.

\noindent
{\bf Part b: Estimates for linear functionals of the flux corrector $\sigma_{\xi,\Xi}^T$.}
Differentiating the equation \eqref{EquationLocalizedFluxCorrectorLinearized}, we see that the infinitesimal perturbation $\delta \sigma_{\xi,\Xi}^T$ caused by an infinitesimal perturbation $\delta \omega_\varepsilon$ in the random field $\omega_\varepsilon$ satisfies the PDE
\begin{align}
\nonumber
&-\Delta \delta \sigma_{\xi,\Xi,jk}^T + \frac{1}{T} \delta \sigma_{\xi,\Xi,jk}^T
\\&~~
\label{EquationSensitivityLinearizedFluxCorrector}
=
\nabla \cdot ( \partial_\xi A(\omega_\varepsilon,\xi+\nabla \phi_\xi^T)
\nabla \delta \phi_{\xi,\Xi}^T \cdot (e_k \otimes e_j-e_j \otimes e_k))
\\&~~~~~~
\nonumber
+\nabla \cdot( \partial_\xi^2 A(\omega_\varepsilon,\xi+\nabla \phi_\xi^T)
(\Xi+\nabla \phi_{\xi,\Xi}^T) \nabla \delta \phi_\xi^T \cdot (e_k \otimes e_j-e_j \otimes e_k))
\\&~~~~~~
\nonumber
+\nabla \cdot( \partial_\omega \partial_\xi A(\omega_\varepsilon,\xi+\nabla \phi_\xi)\delta \omega_\varepsilon
(\Xi+\nabla \phi_{\xi,\Xi}) \cdot (e_k \otimes e_j-e_j \otimes e_k)).
\end{align}
Introducing the solution $\bar h\in H^1(\Rd;\Rm)$ to the equation
\begin{align}
\label{Defbarhsigma}
-\Delta \bar h + \frac{1}{T} \bar h = \nabla \cdot g,
\end{align}
the solution $h_2\in H^1(\Rd;\Rm)$ to the equation
\begin{align}
\label{Defh2}
&-\nabla \cdot (a_\xi^{T,*}\nabla h_{2}) + \frac{1}{T} h_{2}
\\&~~~
\nonumber
= \sum_{i=1}^d \sum_{l=1}^m \partial_i (\partial_\xi A(\omega_\varepsilon,\xi+\nabla \phi_\xi^T) (e_l\otimes e_i) \cdot (e_k \otimes e_j-e_j \otimes e_k) \cdot \nabla \bar h) e_l
\end{align}
(where $a_\xi^{T}(x):=\partial_\xi A(\omega_\varepsilon(x),\xi+\nabla \phi_\xi^T(x))$),
and the solution $h_{3}\in H^1(\Rd;\Rm)$ to the equation
\begin{align}
\label{Defh3}
&-\nabla \cdot (a_\xi^{T,*}\nabla h_{3}) + \frac{1}{T} h_{3}
\\&
\nonumber
= \sum_{i=1}^d \sum_{l=1}^m \partial_i (\partial_\xi^2 A(\omega_\varepsilon,\xi+\nabla \phi_\xi^T) (\Xi+\nabla \phi_{\xi,\Xi}^T) (e_l\otimes e_i) \cdot (e_k \otimes e_j-e_j \otimes e_k) \cdot \nabla \bar h) e_l
\\&~~~~\nonumber
+\sum_{i=1}^d \sum_{l=1}^m \partial_i (\partial_\xi^2 A(\omega_\varepsilon(x),\xi+\nabla \phi_\xi^T)\big(\Xi+\nabla \phi_{\xi,\Xi}^T\big) (e_l\otimes e_i) \cdot \nabla h_2) e_l
,
\end{align}
we may compute the sensitivity of linear functionals of the form
\begin{equation*}
    F:=\int_{\Rd}g\cdot \nabla\sigma^T_{\xi,\Xi,jk}\,dx  
  \end{equation*}
  Indeed, we have
\begin{align*}
\delta F &=\int_\Rd g \cdot \nabla \delta \sigma_{\xi,\Xi,jk}^T \,dx
\stackrel{\eqref{Defbarhsigma}}{=}
-\int_\Rd \nabla \bar h \cdot \nabla \delta \sigma_{\xi,\Xi,jk}^T \,dx-\frac{1}{T}\int_\Rd \bar h \, \delta \sigma_{\xi,\Xi,jk}^T \,dx
\\&
\stackrel{\eqref{EquationSensitivityLinearizedFluxCorrector}}{=}
\int_\Rd \partial_\xi A(\omega_\varepsilon,\xi+\nabla \phi_\xi^T)
\nabla \delta \phi_{\xi,\Xi}^T \cdot (e_k \otimes e_j-e_j \otimes e_k) \cdot \nabla \bar h \,dx
\\&~~~~~
+\int_\Rd \partial_\xi^2 A(\omega_\varepsilon,\xi+\nabla \phi_\xi^T)(\Xi+\nabla \phi_{\xi,\Xi}^T)
\nabla \delta \phi_{\xi}^T \cdot (e_k \otimes e_j-e_j \otimes e_k) \cdot \nabla \bar h \,dx
\\&~~~~~
+\int_\Rd 
\partial_\omega \partial_\xi A(\omega_\varepsilon,\xi+\nabla \phi_\xi^T)\delta \omega_\varepsilon
(\Xi+\nabla \phi_{\xi,\Xi}) \cdot (e_k \otimes e_j-e_j \otimes e_k) \cdot \nabla \bar h \,dx
\\&
\stackrel{\eqref{Defh2}}{=}
-\int_\Rd \partial_\xi A(\omega_\varepsilon,\xi+\nabla \phi_\xi^T) 
\nabla \delta \phi_{\xi,\Xi}^T \cdot \nabla h_2 + \frac{1}{T} \delta \phi_{\xi,\Xi}^T \, h_2 \,dx
\\&~~~~~~
+\int_\Rd \partial_\xi^2 A(\omega_\varepsilon,\xi+\nabla \phi_\xi^T)(\Xi+\nabla \phi_{\xi,\Xi}^T)
\nabla \delta \phi_{\xi}^T \cdot (e_k \otimes e_j-e_j \otimes e_k) \cdot \nabla \bar h \,dx
\\&~~~~~~
+\int_\Rd 
\partial_\omega \partial_\xi A(\omega_\varepsilon,\xi+\nabla \phi_\xi^T)\delta \omega_\varepsilon
(\Xi+\nabla \phi_{\xi,\Xi}) \cdot (e_k \otimes e_j-e_j \otimes e_k) \cdot \nabla \bar h \,dx
\\&
\stackrel{\eqref{EquationSensitivityLinearizedCorrector}}{=}
\int_\Rd \partial_\omega \partial_\xi A(\omega_\varepsilon(x),\xi+\nabla \phi_\xi^T)\delta\omega_\varepsilon \big(\Xi+\nabla \phi_{\xi,\Xi}^T\big) \cdot \nabla h_2 \,dx
\\&~~~~~~
+\int_\Rd \partial_\xi^2 A(\omega_\varepsilon(x),\xi+\nabla \phi_\xi^T)\big(\Xi+\nabla \phi_{\xi,\Xi}^T\big) \nabla \delta \phi_\xi^T \cdot \nabla h_2 \,dx
\\&~~~~~~
+\int_\Rd \partial_\xi^2 A(\omega_\varepsilon,\xi+\nabla \phi_\xi^T)(\Xi+\nabla \phi_{\xi,\Xi}^T)
\nabla \delta \phi_{\xi}^T \cdot (e_k \otimes e_j-e_j \otimes e_k) \cdot \nabla \bar h \,dx
\\&~~~~~~
+\int_\Rd 
\partial_\omega \partial_\xi A(\omega_\varepsilon,\xi+\nabla \phi_\xi^T)\delta \omega_\varepsilon
(\Xi+\nabla \phi_{\xi,\Xi}) \cdot (e_k \otimes e_j-e_j \otimes e_k) \cdot \nabla \bar h \,dx.
\end{align*}
Inserting \eqref{Defh3} and \eqref{EquationSensitivityCorrector}, we obtain
\begin{align*}
\delta F &
=
\int_\Rd \partial_\omega \partial_\xi A(\omega_\varepsilon(x),\xi+\nabla \phi_\xi^T)\delta \omega_\varepsilon \big(\Xi+\nabla \phi_{\xi,\Xi}^T\big) \cdot \nabla h_2 \,dx
\\&~~~~
+\int_\Rd \partial_\omega A(\omega_\varepsilon,\xi+\nabla \phi_\xi^T) \delta \omega_\varepsilon \cdot \nabla h_{3} \,dx
\\&~~~~
+\int_\Rd 
\partial_\omega \partial_\xi A(\omega_\varepsilon,\xi+\nabla \phi_\xi^T)
\,\delta \omega_\varepsilon
\,
(\Xi+\nabla \phi_{\xi,\Xi}) \cdot (e_k \otimes e_j-e_j \otimes e_k) \cdot \nabla \bar h \,dx.
\end{align*}
This identification of $\frac{\partial F}{\partial \omega_\varepsilon}$ together with assumption \hyperlink{A3}{(A3)} gives rise to the sensitivity estimate
\begin{align*}
&\mathbb{E}\bigg[\bigg(\int_\Rd \bigg|\fint_{B_\varepsilon(x)} \bigg|\frac{\partial F}{\partial \omega_\varepsilon}\bigg| \,d\tilde x \bigg|^2 \,dx\bigg)^q \bigg]^{1/2q}
\\&
\leq C \mathbb{E}\Bigg[\bigg(\int_\Rd \bigg|\fint_{B_\varepsilon(x)} |\Xi+\nabla \phi_{\xi,\Xi}^T| (|\nabla \bar h| + |\nabla h_2|) \,d\tilde x \bigg|^2 \,dx\bigg)^{q}\Bigg]^{1/2q}
\\&~~~~
+C
\mathbb{E}\Bigg[\bigg(\int_\Rd \bigg|\fint_{B_\varepsilon(x)} |\xi+\nabla \phi_{\xi}^T| |\nabla h_3| \,d\tilde x \bigg|^2 \,dx\bigg)^{q}\Bigg]^{1/2q}.
\end{align*}
Plugging in this bound into Lemma~\ref{LemmaLqSpectralGap} and using $\mathbb{E}[F]=\int_\Rd g\cdot \nabla \mathbb{E}[\sigma_{\xi,\Xi,jk}^T]\,dx=0$ by stationarity of $\sigma_{\xi,\Xi,jk}^T$, we infer
\begin{align*}
\mathbb{E}\Big[\big|F\big|^{2q}\Big]^{1/2q}
&
\leq C q \varepsilon^{d/2}
\mathbb{E}\Bigg[\bigg(\int_\Rd \bigg|\fint_{B_\varepsilon(x)} |\Xi+\nabla \phi_{\xi,\Xi}^T| (|\nabla \bar h| + |\nabla h_2|) \,d\tilde x \bigg|^2 \,dx\bigg)^{q}\Bigg]^{1/2q}
\\&~~~~
+C q \varepsilon^{d/2}
\mathbb{E}\Bigg[\bigg(\int_\Rd \bigg|\fint_{B_\varepsilon(x)} |\xi+\nabla \phi_{\xi}^T| |\nabla h_3| \,d\tilde x \bigg|^2 \,dx\bigg)^{q}\Bigg]^{1/2q}.
\end{align*}
By Lemma~\ref{LemmaEstimateSensitivity} and \eqref{EstimateGradByrstar2} as well as \eqref{EstimateGradByrstar}, we obtain for $\tau,\bar\tau \in (0,1)$ and for $p$ close enough to $2$
\begin{align*}
&\mathbb{E}\Big[\big|F\big|^{2q}\Big]^{1/2q}
\\&
\leq  C |\Xi| q \varepsilon^{d/2}  r^{d/2}
\mathbb{E}\bigg[\bigg(\frac{r_{*,T,\xi,\Xi}}{\varepsilon}\bigg)^{(d-\delta)q/(1-\tau)}\bigg]^{(1-\tau)/2q}
\\&~~~~~~~~~~~~\times
\nonumber
\mathbb{E}\Bigg[\bigg(r^{-d} \int_\Rd (|\nabla \bar h|^p+|\nabla h_2|^p) \bigg(1+\frac{|x|}{r}\bigg)^{\alpha_0} \,dx\bigg)^{2q/p \tau} \Bigg]^{\tau/2q}
\\&~~~~
+ C |\xi| q \varepsilon^{d/2} r^{d/2}
\mathbb{E}\bigg[\bigg(\frac{r_{*,T,\xi}}{\varepsilon}\bigg)^{(d-\delta)q/(1-\bar \tau)}\bigg]^{(1-\bar \tau)/2q}
\\&~~~~~~~~~\times
\nonumber
\mathbb{E}\Bigg[\bigg(r^{-d} \int_\Rd |\nabla h_3|^p \bigg(1+\frac{|x|}{r}\bigg)^{\alpha_0} \,dx\bigg)^{2q/p \bar \tau} \Bigg]^{\bar \tau/2q}.
\end{align*}
By the weighted Meyers estimate in Lemma~\ref{LemmaWeightedMeyers} applied to \eqref{Defh3} and the uniform bound on $\partial_\xi^2 A$ from \hyperlink{R}{(R)}, we infer for any $\alpha_0$
\begin{align*}
&\mathbb{E}\Big[\big|F\big|^{2q}\Big]^{1/2q}
\\&
\leq C |\Xi| q \varepsilon^{d/2} r^{d/2}
\mathbb{E}\bigg[\bigg(\frac{r_{*,T,\xi,\Xi}}{\varepsilon}\bigg)^{(d-\delta)q/(1-\tau)}\bigg]^{(1-\tau)/2q}
\\&~~~~~~~~~~~\times
\nonumber
\mathbb{E}\Bigg[\bigg(r^{-d} \int_\Rd (|\nabla \bar h|^p +|\nabla h_2|^p) \bigg(1+\frac{|x|}{r}\bigg)^{\alpha_0} \,dx\bigg)^{2q/p \tau} \Bigg]^{\tau/2q}
\\&~~~
+C |\xi| q \varepsilon^{d/2} r^{d/2}
\mathbb{E}\bigg[\bigg(\frac{r_{*,T,\xi}}{\varepsilon}\bigg)^{(d-\delta)q/(1-\bar \tau)}\bigg]^{(1-\bar \tau)/2q}
\\&~~~~~~~~~\times
\nonumber
\mathbb{E}\Bigg[\bigg(r^{-d} \int_\Rd (|\nabla \bar h|^p+|\nabla h_2|^p) |\Xi+\nabla \phi_{\xi,\Xi}^T|^p \bigg(1+\frac{|x|}{r}\bigg)^{\alpha_0} \,dx\bigg)^{2q/p\bar \tau} \Bigg]^{\bar \tau/2q}.
\end{align*}
Arguing analogously to the derivation of \eqref{EstimateSensitivityLinearized} from \eqref{EstimateSensitivityLinearizedFirstIntermediate} but using also the estimate
\begin{align*}
&\int_\Rd r^{-d} |\nabla h_2|^{p^2/2}  \bigg(1+\frac{|x|}{r}\bigg)^{\alpha_{0}p/2+(d+1)(p-2)/2} \,dx
\\&
\leq C \int_\Rd r^{-d} |\nabla \bar h|^{p^2/2} 
 \bigg(1+\frac{|x|}{r}\bigg)^{\alpha_{1}} \,dx
\end{align*}
for $\alpha_{1}:=\alpha_{0}p/2+(d+1)(p-2)/2$
(which follows from \eqref{Defh2}, \hyperlink{A2}{(A2)}, and Lemma~\ref{LemmaWeightedMeyers}),
we deduce
\begin{align*}
&\mathbb{E}\Big[\big|F\big|^{2q}\Big]^{1/2q}
\leq 
C (1+|\xi|)^C |\Xi| q^C \bigg(\frac{\varepsilon}{r}\bigg)^{d/2}
\mathbb{E}\bigg[\bigg(\frac{r_{*,T,\xi,\Xi}}{\varepsilon}\bigg)^{(d-\delta)q/(1-\tau)}\bigg]^{(1-\tau)/2q}
\end{align*}
for any $0<\tau<1$ and any $q\geq C$. This establishes the estimate on functionals of $\sigma_{\xi,\Xi}^T$ in Lemma~\ref{LemmaEstimateLinearFunctionalsLinearized}.

\noindent
{\bf Part c: Estimates for linear functionals of the potential field $\theta_{\xi,\Xi}^T$.} {We consider
\begin{equation*}
  F:=\int_{\Rd}g\cdot\frac{1}{r}\nabla\theta^T_{\xi,\Xi,i}\,dx.
\end{equation*}}%
The argument is exactly the same as in part c of the proof of Lemma~\ref{LemmaEstimateLinearFunctionals}, as the relation between $\theta_{\xi,\Xi}^T$ and $\phi_{\xi,\Xi}^T$ is exactly the same as the one between $\theta_\xi^T$ and $\phi_{\xi,\Xi}^T$. As in the proof of Lemma~\ref{LemmaEstimateLinearFunctionals}, we introduce $\bar g$ as the solution to the PDE
\begin{align*}
-\Delta \bar g = \nabla \cdot g.
\end{align*}
We then obtain an estimate for $F$ of the form \eqref{EstimateSensitivityLinearizedFirstIntermediate}, but with an additional factor $\frac{1}{r}$ on the right-hand side, with $h$ solving the PDE $\smash{-\nabla \cdot (a_\xi^{T,*}\nabla h)+\frac{1}{T} h=\nabla \cdot (\bar g e_i)}$, and with $\smash{\hat h}$ solving the same PDE but with the new $h$. Inserting the Calderon-Zygmund bounds on $\bar g$ and the modified bound on $\nabla \hat h$ in the steps leading to \eqref{EstimateSensitivityLinearizedIntermediate} and \eqref{EstimateSensitivityLinearized}, we deduce the desired estimate
\begin{align*}
&\mathbb{E}\Big[\big|F\big|^{2q}\Big]^{1/2q}
\leq 
C (1+|\xi|)^C |\Xi| q^C \bigg(\frac{\varepsilon}{r}\bigg)^{d/2}
\mathbb{E}\bigg[\bigg(\frac{r_{*,T,\xi,\Xi}}{\varepsilon}\bigg)^{(d-\delta)q/(1-\tau)}\bigg]^{(1-\tau)/2q}
\end{align*}
for any $0<\tau<1$ and any $q\geq C$.

{\bf Part d: Proof of Lemma~\ref{LemmaCorrectorAveragesLinearized}.}
The proof of the estimates \eqref{EstimateCorrectorAverageDifferenceLinearized} and \eqref{EstimateCorrectorAverageLinearized} is analogous to the proof of \eqref{EstimateCorrectorAverageDifference} and \eqref{EstimateCorrectorAverage}.
\end{proof}

\subsection{Estimates on the corrector}

We first establish the moment bounds on $r_{*,T,\xi,\Xi}$.
\begin{proof}[Proof of Lemma~\ref{MomentsMinimalRadiusLinearized}]
The proof is entirely analogous to the proof of Lemma~\ref{MomentsMinimalRadius}:
Recall that the only required ingredients for the proof of moment bounds for $r_{*,T,\xi}$ in Lemma~\ref{MomentsMinimalRadius} were the estimate for functionals of the corrector $\phi_\xi^T$ from Lemma~\ref{LemmaEstimateLinearFunctionals}, the estimate on corrector averages from Lemma~\ref{LemmaCorrectorAverages}, the estimate on the corrector gradient given by \eqref{EstimateGradByrstar}, as well as the general technical results of Lemma~\ref{LemmaWavelets} and Lemma~\ref{LemmaMultiscaleDecomposition}. Lemma~\ref{LemmaEstimateLinearFunctionalsLinearized} provides bounds on the stochastic moments of linear functionals of the corrector $\phi_{\xi,\Xi}^T$ for the linearized PDE that are essentially analogous (up to the prefactor $q^C (1+|\xi|)^C$) to the bounds for functionals of $\phi_\xi^T$ in Lemma~\ref{LemmaEstimateLinearFunctionals}. Similarly, Lemma~\ref{LemmaCorrectorAveragesLinearized} provides estimates on averages of the linearized corrector $\phi_{\xi,\Xi}^T$ that are (again up to the prefactor $q^C (1+|\xi|)^C$) analogous to the bounds on averages of $\phi_\xi^T$ from Lemma~\ref{LemmaCorrectorAverages}.
Furthermore, the estimate \eqref{EstimateGradByrstar2} for the gradient of $\phi_{\xi,\Xi}^T$ is completely analogous to the estimate \eqref{EstimateGradByrstar} for the gradient of $\phi_\xi^T$. In conclusion, by the same arguments as the proof of Lemma~\ref{MomentsMinimalRadius} (up to replacing $\phi_{\xi}^T$ by $\phi_{\xi,\Xi}^T$, $\xi$ by $\Xi$, $r_{*,T,\xi}$ by $r_{*,T,\xi,\Xi}$, and including an additional prefactor $q^C (1+|\xi|)^C$ in the bounds), we obtain
\begin{align*}
\mathbb{E}\bigg[\Big(\frac{r_{*,T,\xi}(x_0)}{\varepsilon}\Big)^{(d-\delta/2)q/2}\Big]^{1/q} \leq C(d,m,\lambda,\Lambda,\rho,\nu) q^C (1+|\xi|)^C
.
\end{align*}
\end{proof}

\begin{proof}[Proof of Proposition~\ref{PropositionLinearizedCorrectorEstimate}]
By the same arguments as in the proof of Proposition~\ref{PropositionCorrectorEstimate}, we obtain the desired bounds on $\phi_{\xi,\Xi}^T$, $\sigma_{\xi,\Xi}^T$, and $\theta_{\xi,\Xi}^T$. Note that we simply need to replace the use of Lemma~\ref{LemmaEstimateLinearFunctionals} by Lemma~\ref{LemmaEstimateLinearFunctionalsLinearized}, the use of Lemma~\ref{LemmaCorrectorAverages} by Lemma~\ref{LemmaCorrectorAveragesLinearized}, the use of Lemma~\ref{MomentsMinimalRadius} by Lemma~\ref{MomentsMinimalRadiusLinearized}, and the use of \eqref{EstimateGradByrstar} by \eqref{EstimateGradByrstar2}.
\end{proof}

\begin{proof}[Proof of Corollary~\ref{CorollaryImprovedCorrectorDifferenceBounds}]
To establish the estimate \eqref{CorrectorBound}, we simply pass to the limit $T\rightarrow \infty$ in the $L^p$ version of the estimate \eqref{CorrectorEstimateWithoutAverage} for $\phi_\xi^T$ and $\sigma_\xi^T$.
By combining \eqref{CorrectorEstimate} with Lemma~\ref{LemmaCorrectorAverages} and Lemma~\ref{MomentsMinimalRadius}, we deduce
\begin{align*}
&\bigg(\fint_{B_r(x_0)} \bigg|\phi_\xi^T-\fint_{B_\varepsilon(0)} \phi_\xi^T(\tilde x) \,d\tilde x\bigg|^2 + \bigg|\sigma_\xi^T-\fint_{B_\varepsilon(0)} \sigma_\xi^T(\tilde x) \,d\tilde x\bigg|^2 \,dx\bigg)^{1/2}
\\&
\leq \mathcal{C} |\xi| \, \varepsilon \mu((|x_0|+r)/\varepsilon)
\end{align*}
with $\mu(s):=1$ for $d\geq 3$, $\mu(s):=(\log(2+s))^{1/2}$ for $d=2$, and $\mu(s):=(1+s)^{1/2}$ for $d=1$.
Passing to the limit $T\rightarrow \infty$, the quantities $\phi_\xi^T-\fint_{B_\varepsilon(0)} \phi_\xi^T(\tilde x) \,d\tilde x$ converge to a solution $\phi_\xi$ to the corrector equation with vanishing average in $B_\varepsilon(0)$ subject to the bound \eqref{CorrectorBound2d} (and similarly for $\sigma_\xi$).

It remains to show that the limit $\sigma_\xi$ satisfies the PDE \eqref{EquationFluxCorrector}, as passing to the limit in the equation \eqref{EquationLocalizedFluxCorrector} yields only
\begin{align*}
-\Delta \sigma_{\xi,jk} = \partial_j (A(\omega_\varepsilon,\xi+\nabla \phi_\xi)\cdot e_k)-\partial_k (A(\omega_\varepsilon,\xi+\nabla \phi_\xi)\cdot e_j).
\end{align*}
However, as we have $\nabla \cdot (A(\omega_\varepsilon,\xi+\nabla \phi_\xi))=0$ by the equation for $\phi_\xi$, we deduce
\begin{align*}
-\Delta \sum_k \partial_k \sigma_{\xi,jk} = -\Delta (A(\omega_\varepsilon,\xi+\nabla \phi_\xi)\cdot e_j).
\end{align*}
By the sublinear growth of $\sigma_{\xi,jk}$, it follows that
\begin{align*}
\nabla \cdot \sigma_\xi=A(\omega_\varepsilon,\xi+\nabla \phi_\xi)-\mathbb{E}[A(\omega_\varepsilon,\xi+\nabla \phi_\xi)]
=A(\omega_\varepsilon,\xi+\nabla \phi_\xi)-A_\shom(\xi),
\end{align*}
which is the flux corrector equation \eqref{EquationFluxCorrector}.

In order to derive \eqref{DifferenceBound}, we first write
\begin{align*}
\phi_{\xi_2}^T-\phi_{\xi_1}^T = \int_0^1 \phi_{(1-s)\xi_1+s\xi_2,\Xi}^T \,ds
\end{align*}
with $\Xi:=\xi_2-\xi_1$. This entails in case $d\geq 3$ using Proposition~\ref{PropositionLinearizedCorrectorEstimate} 
\begin{align*}
&\mathbb{E}\bigg[
\bigg(\fint_{B_r(x_0)} \big|\phi_{\xi_2}^T-\phi_{\xi_1}^T\big|^2\,dx\bigg)^{q/2}\bigg]^{\tfrac 1q} \leq C q^C (1+|\xi_1|+|\xi_2|)^C |\xi_2-\xi_1| \varepsilon
\end{align*}
respectively in case $d\leq 2$ using Proposition~\ref{PropositionLinearizedCorrectorEstimate}, Lemma~\ref{LemmaCorrectorAveragesLinearized}, and Lemma~\ref{MomentsMinimalRadiusLinearized}
\begin{align*}
&\mathbb{E}\bigg[
\bigg(\fint_{B_r(x_0)} \bigg|\phi_{\xi_2}^T-\phi_{\xi_1}^T-\fint_{B_\varepsilon(0)}\phi_{\xi_2}^T-\phi_{\xi_1}^T\,dx\bigg|^2\,dx\bigg)^{q/2}\bigg]^{\tfrac 1q}
\\&
\leq C q^C (1+|\xi_1|+|\xi_2|)^C |\xi_2-\xi_1| \varepsilon \mu((r+|x_0|)/\varepsilon).
\end{align*}
We may then pass to the limit $T\rightarrow \infty$ in these estimates to deduce \eqref{DifferenceBound}. The corresponding bound for $\sigma_{\xi_2}-\sigma_{\xi_2}$ is derived analogously.
\end{proof}

\section{Structural properties of the effective equation}

We finally establish the structural properties of the effective (homogenized) equation, as stated in Theorem~\ref{TheoremStructureProperties}.
\begin{proof}[Proof of Theorem~\ref{TheoremStructureProperties}]
{\bf Part a.}
The fact that the homogenized material law $A_\shom$ inherits the monotonicity properties of the law of the random material $A(\omega,\cdot)$ has already been established in \cite{DalMasoModica2,DalMasoModica} for Euler-Lagrange equations associated with convex integral functionals. Nevertheless, we shall briefly recall the (standard) proof of this result. The result $A_\shom(0)=0$ is immediate from the definition $A_\shom(\xi):=\mathbb{E}[A(\omega,\xi+\nabla \phi_\xi)]$ and \hyperlink{A2}{(A2)}, as $\phi_0\equiv 0$ is the (up to additive constants unique) solution to the corrector equation for $\xi=0$. We next have
\begin{align}
\label{EstimateXi}
|A_\shom(\xi_2)-A_\shom(\xi_1)|
&=\big|\mathbb{E}[A(\omega_\varepsilon,\xi_2+\nabla \phi_{\xi_2})]
-\mathbb{E}[A(\omega_\varepsilon,\xi_1+\nabla \phi_{\xi_1})]\big|
\\&
\nonumber
\stackrel{\hyperlink{A2}{(A2)}}{\leq} \Lambda \mathbb{E}[|\xi_2+\nabla \phi_{\xi_2}-\xi_1-\nabla \phi_{\xi_1}|].
\end{align}
Subtracting the corrector equations \eqref{EquationCorrector} for $\xi_1$ and $\xi_2$ from each other and testing the resulting equation with $(\phi_{\xi_2}-\phi_{\xi_1})\eta^2(x/r)$ for some cutoff $\eta$ with $\supp \eta \subset B_2$ and $\eta\equiv 1$ in $B_1$, we deduce
\begin{align*}
&\int_\Rd \big(A(\omega_\varepsilon,\xi_2+\nabla \phi_{\xi_2})-A(\omega_\varepsilon,\xi_1+\nabla \phi_{\xi_1})\big)\cdot\big(\nabla \phi_{\xi_2}-\nabla \phi_{\xi_1}\big) \eta^2\Big(\frac{x}{r}\Big) \,dx
\\&
\leq
\int_\Rd \big|A(\omega_\varepsilon,\xi_2+\nabla \phi_{\xi_2})-A(\omega_\varepsilon,\xi_1+\nabla \phi_{\xi_1})\big| \frac{1}{r}|\phi_{\xi_2}-\phi_{\xi_1}|
|\nabla \eta^2|\Big(\frac{x}{r}\Big) \,dx.
\end{align*}
This implies by \hyperlink{A1}{(A1)}, \hyperlink{A2}{(A2)}, and Young's inequality
\begin{align*}
&\int_\Rd |\nabla \phi_{\xi_2}-\nabla \phi_{\xi_1}|^2 \eta^2\Big(\frac{x}{r}\Big) \,dx
\\&
\leq \frac{4\Lambda^2}{\lambda^2} |\xi_1-\xi_2|^2 \int_\Rd \eta^2\Big(\frac{x}{r}\Big) \,dx
\\&~~~~
+\frac{C}{r^2} \int_\Rd |\phi_{\xi_2}-\phi_{\xi_1}|^2
|\nabla\eta|^2\Big(\frac{x}{r}\Big) \,dx.
\end{align*}
Dividing by $\int_\Rd \eta^2(\frac{x}{r}) \,dx$, taking the expectation, using stationarity, and passing to the limit $r\rightarrow \infty$ (using the fact that the correctors $\phi_\xi$ grow sublinearly), we deduce
\begin{align}
\label{LipschitzDependenceCorrector}
\mathbb{E}[|\nabla \phi_{\xi_1}-\nabla \phi_{\xi_2}|^2]\leq C(d) \frac{\Lambda^2}{\lambda^2} |\xi_1-\xi_2|^2
\end{align}
Inserting this into \eqref{EstimateXi}, we deduce the Lipschitz estimate in Theorem~\ref{TheoremStructureProperties}a.

Concerning the monotonicity property, we deduce by testing the corrector equations \eqref{EquationCorrector} for $\xi_1$ and $\xi_2$ with $(\phi_{\xi_1}-\phi_{\xi_2})\eta(x/r)$
\begin{align*}
&\int_\Rd \eta\Big(\frac{x}{r}\Big) \big(A(\omega_\varepsilon,\xi_2+\nabla \phi_{\xi_2})-A(\omega_\varepsilon,\xi_1+\nabla \phi_{\xi_1})\big)\cdot(\xi_2-\xi_1) \,dx
\\&
=\int_\Rd \eta\Big(\frac{x}{r}\Big) \big(A(\omega_\varepsilon,\xi_2+\nabla \phi_{\xi_2})-A(\omega_\varepsilon,\xi_1+\nabla \phi_{\xi_1})\big)\cdot(\xi_2+\nabla \phi_{\xi_2}-\xi_1-\nabla \phi_{\xi_1}) \,dx
\\&~~~~
+\frac{1}{r} \int_\Rd (\phi_{\xi_2}-\phi_{\xi_1}) (\nabla \eta)\Big(\frac{x}{r}\Big) \cdot \big(A(\omega_\varepsilon,\xi_2+\nabla \phi_{\xi_2})-A(\omega_\varepsilon,\xi_1+\nabla \phi_{\xi_1})\big) \,dx
\\&
\geq \lambda \int_\Rd \eta\Big(\frac{x}{r}\Big) \big|\xi_2+\nabla \phi_{\xi_2}-\xi_1-\nabla \phi_{\xi_1}\big|^2 \,dx
\\&~~~~
+\frac{1}{r} \int_\Rd (\phi_{\xi_2}-\phi_{\xi_1}) (\nabla \eta)\Big(\frac{x}{r}\Big) \cdot \big(A(\omega_\varepsilon,\xi_2+\nabla \phi_{\xi_2})-A(\omega_\varepsilon,\xi_1+\nabla \phi_{\xi_1})\big) \,dx
\end{align*}
which after dividing by $\int_\Rd \eta(\frac{x}{r}) \,dx$, taking the expectation and using stationarity, and passing to the limit $r\rightarrow \infty$ yields by the sublinear growth of the $\phi_{\xi}$
\begin{align*}
&\big(A_\shom(\xi_2)-A_\shom(\xi_1)\big)\cdot(\xi_2-\xi_1)
\\&
=\mathbb{E}\big[\big(A(\omega,\xi_2+\nabla \phi_{\xi_2})-A(\omega,\xi_1+\nabla \phi_{\xi_1})\big)\cdot(\xi_2-\xi_1)\big]
\\&
\geq \lambda |\xi_1-\xi_2|^2.
\end{align*}
{It only remains to show that $A_\shom\in C^1$, provided that additionally the condition \hyperlink{R}{(R)} is satisfied by the coefficient field $A(\omega_\varepsilon,\xi)$.
Taking the derivative with respect to $\xi$ in the relation $A_\shom^T(\xi):=\mathbb{E}[A(\omega_\varepsilon,\xi+\nabla \phi_\xi^T)]$, we obtain
\begin{align*}
\partial_\xi A_\shom^T(\xi) \Xi
=
\mathbb{E}\big[\partial_\xi A(\omega_\varepsilon,\xi+\nabla \phi_\xi^T)(\Xi+\nabla \phi_{\xi,\Xi}^T)\big].
\end{align*}
Passing to the limit $T\rightarrow \infty$ using Lemma~\ref{L:syscorr} and Lemma~\ref{L:syslincorr}, we deduce
\begin{align*}
\partial_\xi A_\shom(\xi) \Xi
=
\mathbb{E}\big[\partial_\xi A(\omega_\varepsilon,\xi+\nabla \phi_\xi)(\Xi+\nabla \phi_{\xi,\Xi})\big].
\end{align*}
The assertion $A_\shom\in C^1$ is then a consequence of \eqref{LipschitzDependenceCorrector} and the estimate $\mathbb{E}[|\nabla \phi_{\xi_1,\Xi}-\nabla \phi_{\xi_2,\Xi}|^2]\leq C|\xi_1-\xi_2|^\theta |\Xi|$ for some $\theta>0$,
which is derived analogously to \eqref{LipschitzDependenceCorrector} using the equation
\begin{align*}
&
-\nabla \cdot \big(\partial_\xi A(\omega_\varepsilon,\xi_1+\nabla \phi_{\xi_1})(\nabla \phi_{\xi_1,\Xi}-\nabla \phi_{\xi_2,\Xi})\big)
\\&
=\nabla \cdot \Big(
\big(\partial_\xi A(\omega_\varepsilon,\xi_1+\nabla \phi_{\xi_1})-\partial_\xi A(\omega_\varepsilon,\xi_2+\nabla \phi_{\xi_2})\big)
(\Xi+\nabla \phi_{\xi_2,\Xi})
\Big)
\end{align*}
as well as the higher integrability result $\mathbb{E}[|\Xi+\nabla \phi_{\xi_2,\Xi}|^p]\leq C |\Xi|^p$ for some $p>2$ (the latter being a consequence of Meyers).
}

\noindent
{\bf Part b.}
We next show that frame-indifference of the material law -- in the sense that $A(\omega_\varepsilon,O\xi)=O A(\omega_\varepsilon,\xi)$ for all $x\in \Rd$, almost every $\omega_\varepsilon$, all $\xi\in \Rmd$, and all $O\in \SO(m)$ -- is preserved under homogenization. Indeed, if $\phi_\xi$ solves the corrector equation $\nabla \cdot (A(\omega_\varepsilon,\xi+\nabla \phi_\xi))=0$ and if $O\in SO(m)$, then $O\phi_\xi$ solves the corrector equation $\nabla \cdot (A(\omega_\varepsilon,O\xi+\nabla O\phi_\xi))=0$. Using the uniqueness of the corrector up to additive constants, we obtain $\nabla O\phi_\xi=\nabla \phi_{O\xi}$. This entails
\begin{align*}
A_\shom(O\xi)
&= \mathbb{E}[A(\omega_\varepsilon,O\xi+\nabla \phi_{O\xi})]
\\&
= \mathbb{E}[A(\omega_\varepsilon,O\xi+\nabla O\phi_{\xi})]
= \mathbb{E}[O A(\omega_\varepsilon,\xi+\nabla \phi_{\xi})]
\\&
= O A_\shom(\xi).
\end{align*}

\noindent
{\bf Part c.}
We next show that isotropy of the probability distribution of the material law implies isotropy of the homogenized material law.
Let $V\in \SO(d)$. In this case, if $\phi_\xi$ solves the corrector equation $\nabla \cdot (A(\omega_\varepsilon(x),\xi+\nabla \phi_\xi(x)))=0$, then the rotated function $\phi_\xi(V\cdot)$ solves the corrector equation $\nabla \cdot \big(A\big(\omega_\varepsilon(V x),\xi V V^{-1}+\nabla(\phi_\xi(Vx))V^{-1}\big)V \big)=0$ for a rotated monotone operator, i.\,e.\ $\tilde \phi_{\xi V}(x):=\phi_\xi(Vx)$ is the corrector associated with the rotated operator 
\begin{align*}
\smash{A_\varepsilon^V}(x,\smash{\tilde\xi}):=A(\omega_\varepsilon(V x),\smash{\tilde \xi} \smash{V^{-1}})V
\end{align*} 
and slope $\xi V$ and we have $\nabla \tilde \phi_{\xi V}(x) = \nabla \phi_\xi (Vx) V$. This entails by the assumed equality of the laws of $A_\varepsilon(x,\smash{\tilde \xi}):=A(\omega_\varepsilon(x),\smash{\tilde \xi})$ and $\smash{A_\varepsilon^V}$
\begin{align*}
A_\shom(\xi V)
&= \mathbb{E}[A_\varepsilon(x,\xi V+\nabla \phi_{\xi V}(x))]
\\&
= \mathbb{E}[A_\varepsilon^{V}(x,\xi V+\nabla \tilde \phi_{\xi V}(x))]
\\&
= \mathbb{E}[A_\varepsilon^{V}(x,\xi V V^{-1}+\nabla (\phi_{\xi}(Vx))V^{-1})V]
\\&
= \mathbb{E}[A_\varepsilon(\omega_\varepsilon(V x),\xi +\nabla \phi_{\xi}(V x) )V]
\\&
= A_\shom(\xi)V.
\end{align*}
\end{proof}

\appendix

\section{Auxiliary results from regularity theory}
\label{SectionRegularity}

We now provide the (standard) proof of the Caccioppoli inequality and the hole-filling estimate for nonlinear elliptic PDEs with monotone nonlinearity from Lemma~\ref{LemmaCaccioppoliHoleFilling}.
\begin{proof}[Proof of Lemma~\ref{LemmaCaccioppoliHoleFilling}]
Let $R>0$ and let $\eta$ be a standard cutoff with $\eta\equiv 0$ outside of $B_R$, $\eta\equiv 1$ in $B_{R/2}$, and $|\nabla \eta|\leq C R^{-1}$. Testing the equation with $\eta^2 (u-b)$ for some $b\in \Rm$ to be chosen, we obtain by \hyperlink{A1}{(A1)}--\hyperlink{A2}{(A2)}
\begin{align*}
&\lambda \int_\Rd \eta^2 |\nabla u|^2 \,dx + \frac{1}{T} \int_\Rd \eta^2 |u|^2 \,dx
\\&
\stackrel{\hyperlink{A1}{(A1)}}{\leq} \int_\Rd  A(x,\nabla u)\cdot\eta^2 \nabla u \,dx + \frac{1}{T} \int_\Rd \eta^2 |u|^2 \,dx
\\&
\leq -\int_\Rd 2\eta (A(x,\nabla u)+g)\cdot (u-b) \otimes \nabla \eta \,dx + \frac{1}{T} \int_\Rd \eta^2 u b \,dx
\\&~~~~
+ \int_\Rd \eta^2 \Big(-g\cdot \nabla u + \frac{1}{T} f (u-b)\Big) \,dx 
\\&
\stackrel{\hyperlink{A2}{(A2)}}{\leq}
\frac{C \Lambda}{R} \int_{B_R(x_0)\setminus B_{R/2}(x_0)} \eta (|\nabla u| + |g|) |u-b| \,dx
+ \frac{1}{T} \int_\Rd \eta^2 u b \,dx
\\&~~~~
+ \int_\Rd \eta^2 \Big(-g\cdot \nabla u + \frac{1}{T} f (u-b)\Big) \,dx.
\end{align*}
Young's inequality and an absorption argument yields the Caccioppoli-type inequality
\begin{align*}
&\int_{B_{R/2}(x_0)} |\nabla u|^2 + \frac{1}{T} |u|^2 \,dx
\\&
\leq
\frac{C}{R^2} \int_{B_R(x_0)\setminus B_{R/2}(x_0)} |u-b|^2 \,dx
+\frac{C}{T} \int_{B_R(x_0)} |b|^2 \,dx
+C \int_{B_R(x_0)} |g|^2 + \frac{1}{T}|f|^2 \,dx
\end{align*}
which directly implies \eqref{CaccippoliMonotone}.

An application of the Poincar\'e inequality on the annulus $B_R(x_0)\setminus B_{R/2}(x_0)$ in the previous estimate gives upon choosing $b:=\fint_{B_R(x_0)\setminus B_{R/2}(x_0)} u \,dx$, using also Jensen's inequality and the fact that $|B_R(x_0)|\sim |B_R(x_0)\setminus B_{R/2}(x_0)|$ to estimate the second term on the right-hand side,
\begin{align*}
&\int_{B_{R/2}(x_0)} |\nabla u|^2 \,dx + \frac{1}{T} \int_{B_{R/2}(x_0)} |u|^2 \,dx
\\&
\leq
C \int_{B_R(x_0)\setminus B_{R/2}(x_0)} |\nabla u|^2 \,dx
+ \frac{C}{T} \int_{B_R(x_0)\setminus B_{R/2}(x_0)} |u|^2 \,dx
\\&~~~~~
+C \int_{B_R(x_0)} |g|^2 + \frac{1}{T}|f|^2 \,dx.
\end{align*}
This in turn yields by the hole-filling argument
\begin{align*}
&\int_{B_{R/2}(x_0)} |\nabla u|^2 \,dx + \frac{1}{T} \int_{B_{R/2}(x_0)} |u|^2 \,dx
\\&
\leq
\theta \bigg(\int_{B_R(x_0)} |\nabla u|^2 \,dx
+ \frac{1}{T} \int_{B_R(x_0)} |u|^2 \,dx\bigg)
+C \int_{B_R(x_0)} |g|^2 + \frac{1}{T}|f|^2 \,dx
\end{align*}
for $\theta=\frac{C}{C+1}$. Iterating this estimate with $R$ replaced by $2^{-k} R$, we deduce our desired estimate \eqref{HoleFillingMonotone} if $r$ is of the form $r=2^{-K}R$. For other values of $r$, we may simply use the already-established inequality for the next bigger radius of the form $r=2^{-K}R$ and increase the constant $C$ if necessary.
\end{proof}

We next provide a small-scale H\"older regularity result for the linearized corrector $\phi_{\xi,\Xi}^T$.
\begin{proposition}
\label{PropositionRegularityLinearizedCorrector}
Let the assumptions \hyperlink{A1}{(A1)}--\hyperlink{A3}{(A3)} and \hyperlink{P1}{(P1)}--\hyperlink{P2}{(P2)} as well as \hyperlink{R}{(R)} be in place. Then
there exists $\alpha>0$ such that for all $\xi,\Xi \in \Rmd$ and all $T\geq \varepsilon^2$ the gradient of the linearized corrector $\phi_{\xi,\Xi}^T$ is subject to a H\"older regularity estimate of the form
\begin{align*}
&|(\Xi+\nabla \phi_{\xi,\Xi}^T)(x)-(\Xi+\nabla \phi_{\xi,\Xi}^T)(y)|
\\&
\leq \mathcal{C}_{\xi,T}(x_0) (1+|\xi|)^C \bigg(\frac{|x-y|}{\varepsilon}\bigg)^\alpha
\bigg(\dashint_{B_{\varepsilon}(x_0)} |\Xi+\nabla \phi_{\xi,\Xi}^T|^2 \,dx
+ |\Xi|^2\bigg)^{1/2}
\end{align*}
for all $x,y\in B_{\varepsilon/2}(x_0)$,
where $\mathcal{C}_{\xi,T}=\mathcal C_{\xi,T}(\omega_\e,x_0)$ denotes a stationary random field with stretched exponential moment bounds $\mathbb{E}[\exp(\tfrac{1}{C}\mathcal{C}^{1/C}_{\xi,T})]\leq 2$ for some constant $C=C(d,m,\lambda,\Lambda,\rho,\nu)$.
\end{proposition}
\begin{proof}
The result is a straightforward consequence of classical Schauder theory (see e.\,g.\ the proof of \cite[Theorem~5.19]{GiaquintaMartinazzi})
applied to the equation
\begin{align*}
-\nabla \cdot (\partial_\xi A(\omega_\varepsilon,\xi+\nabla \phi_\xi^T)(\Xi+\nabla \phi_{\xi,\Xi}^T)) + \frac{1}{T} (\Xi\cdot(x-x_0)+\phi_{\xi,\Xi}^T) =\frac{1}{T} \Xi\cdot(x-x_0),
\end{align*}
which is possible by the H\"older continuity of the coefficient $\partial_\xi A(\omega_\varepsilon,\xi+\nabla \phi_\xi^T)$ on $B_\varepsilon(x_0)$, which in turn may be deduced from Proposition~\ref{C1alphaRegularity}, our regularity assumptions on $\omega_\varepsilon$ (see \hyperlink{R}{(R)}), and the Lipschitz dependence of $\partial_\xi A$ on both variables (see \hyperlink{A3}{(A3)} and \hyperlink{R}{(R)}).
\end{proof}

\begin{proposition}
\label{C1alphaRegularity}
Let the assumptions \hyperlink{A1}{(A1)}--\hyperlink{A3}{(A3)} and \hyperlink{P1}{(P1)}--\hyperlink{P2}{(P2)} as well as \hyperlink{R}{(R)} be in place. Then
there exists $\alpha>0$ such that for all $\xi \in \Rmd$ and all $T\geq \varepsilon^2$ the gradient of the corrector $\phi_\xi^T$ is subject to a H\"older regularity estimate of the form
\begin{align*}
|(\xi+\nabla \phi_\xi^T)(x)-(\xi+\nabla \phi_\xi^T)(y)|
\leq \mathcal{C}_{\xi,T}(x_0) (1+|\xi|) \bigg(\frac{|x-y|}{\varepsilon}\bigg)^\alpha
\end{align*}
for all $x,y\in B_{\varepsilon/2}(x_0)$,
where $\mathcal{C}_{\xi,T}= \mathcal{C}_{\xi,T}(x_0)$ denotes a stationary random field with stretched exponential moment bounds $\mathbb{E}[\exp(\tfrac{1}{C}\mathcal{C}^{1/C}_{\xi,T})]\leq 2$ for some constant $C=C(d,m,\lambda,\Lambda,\rho,\nu)$.
\end{proposition}
\begin{proof}
We differentiate the equation $-\nabla \cdot (A(\omega_\varepsilon,\xi+\nabla \phi_\xi^T))+\frac{1}{T}\phi_\xi^T=0$. This yields
\begin{align}
\label{DifferentiatedEquation}
-\nabla \cdot (\partial_\xi A(\omega_\varepsilon,\xi+\nabla \phi_\xi)\nabla \partial_i \phi_\xi^T)+\frac{1}{T}\partial_i \phi_\xi^T=\nabla \cdot (\partial_\omega \partial_\xi A(\omega_\varepsilon,\xi+\nabla \phi_\xi^T)\partial_i \omega_\varepsilon).
\end{align}
By \hyperlink{A1}{(A1)}--\hyperlink{A3}{(A3)}, this is a linear elliptic system for the derivative $\partial_i \phi_\xi^T$ with uniformly elliptic and bounded coefficient field.

\emph{Case a: Two-dimensional systems with smooth coefficients.}
Meyers estimate for the PDE \eqref{DifferentiatedEquation} in the form of Lemma~\ref{L:linearMeyersLocalized} with $T:=\infty$ and the condition \hyperlink{A3}{(A3)} imply for some $p=p(d,m,\lambda,\Lambda)>2$ and any $b\in \Rm$
\begin{align*}
&\bigg(\fint_{B_{\varepsilon/2}(x_0)} |\nabla \partial_i \phi_\xi^T|^p \,dx \bigg)^{1/p}
\\&
\leq C \bigg(\fint_{B_\varepsilon(x_0)} |\nabla \partial_i \phi_\xi^T|^2 \,dx\bigg)^{1/2} + C \bigg(\fint_{B_\varepsilon(x_0)} |\nabla \omega_\varepsilon|^p + \Big|\frac{1}{T}(\phi_\xi^T-b)\Big|^p \,dx\bigg)^{1/p}
\end{align*}
Using the Caccioppoli inequality \eqref{CaccippoliMonotone} with $T=\infty$ for the PDE \eqref{DifferentiatedEquation}, choosing $p-2>0$ small enough, and using the Poincar\'e-Sobolev inequality, we get by $T\geq \varepsilon^2$
\begin{align*}
&\bigg(\fint_{B_{\varepsilon/2}(x_0)} |\nabla \partial_i \phi_\xi^T|^p \,dx \bigg)^{1/p}
\\&
\leq \frac{C}{\varepsilon} \bigg(\fint_{B_{2\varepsilon}(x_0)} |\nabla \phi_\xi^T|^2 \,dx\bigg)^{1/2} + C \bigg(\fint_{B_{2\varepsilon}(x_0)} |\nabla \omega_\varepsilon|^p \,dx\bigg)^{1/p}.
\end{align*}
By our estimate \eqref{EstimateGradByrstar}, Lemma~\ref{MomentsMinimalRadius}, and the bound on $\nabla \omega_\varepsilon$ in \hyperlink{R}{(R)}, the right-hand side may be bounded by $\mathcal{C} \varepsilon^{-1} (|\xi|+1)$ for some random constant $\mathcal{C}$ with stretched exponential moments.
By Morrey's embedding, we obtain the desired estimate.

\emph{Case b: Scalar equations and systems with Uhlenbeck structure with smooth coefficients.}

In the case of a scalar equation, we infer the desired H\"older continuity of $\partial_i \phi_\xi^T$ from De~Giorgi-Nash-Moser theory: Applying \cite[Theorem~8.24]{GT} to the equation \eqref{DifferentiatedEquation}, we deduce
\begin{align*}
&\varepsilon^{\alpha-1} \sup_{x_1,x_2\in B_{\varepsilon/2}(x_0)} \frac{|\partial_i \phi_\xi^T(x_1)-\partial_i \phi_\xi^T(x_2)|}{|x_1-x_2|^\alpha}
\\&
\leq C \varepsilon^{-1} \bigg(\fint_{B_\varepsilon(x_0)} |\partial_i \phi_\xi^T|^2 \,dx\bigg)^{1/2}
+C \bigg(\fint_{B_\varepsilon(x_0)} |\nabla \omega_\varepsilon|^{2d} \,dx\bigg)^{1/2d}.
\end{align*}
Using our regularity assumption on $\omega_\varepsilon$ from \hyperlink{R}{(R)} and again  \eqref{EstimateGradByrstar} and Lemma~\ref{MomentsMinimalRadius}, we conclude.

In the systems' case, one replaces the De~Giorgi-Nash-Moser theory by Uhlenbeck's regularity result \cite{Uhlenbeck}.
\end{proof}

We present the arguments for Remark~\ref{R:T2} and Remark~\ref{R:T3}.
\begin{proof}[Proof of Remark~\ref{R:T2}.]
  The existence of $u_{\shom}$ is guaranteed (and only requires $g\in H^1(\Rd;\Rmd)$), since in view of Theorem~\ref{TheoremStructureProperties}, the effective material law $A_\shom$ inherits the monotone structure from the heterogeneous material law $A(\omega,\cdot)$. A standard energy estimate yields
  \begin{equation*}
    \|\nabla u_{\shom}\|_{L^2(\Rd)}\leq C\|g\|_{L^2(\Rd)},
  \end{equation*}
  where here and below $C$ only depends on $d,m,\lambda$ and $\Lambda$. By appealing to the difference quotient technique of L.\ Nirenberg we may differentiate this equation with respect to the spatial coordinate $x_i$ and we get the linear system
  \begin{align}
    \label{DifferentiatedEquationUhom}
    &-\nabla \cdot (a_{\xi}\nabla \partial_iu_\shom) =\nabla \cdot (\partial_ig),
  \end{align}
  where $a_\xi:=\partial_\xi A_\shom(\nabla u_\shom)$ is a uniformly elliptic and bounded coefficient field by the structure properties of $A_{\hom}$ stated in Theorem~\ref{TheoremStructureProperties}a. An energy estimate thus yields
  \begin{equation*}
    \|\nabla\partial_i u_{\shom}\|_{H^1(\Rd)}\leq C\|g\|_{H^1(\Rd)}.
  \end{equation*}
  We claim that with help of the small-scale regularity condition \hyperlink{R}{(R)} we get the Lipschitz-estimate,
  \begin{equation}\label{UniformGradBoundUhom}
    ||\nabla u_\shom||_{L^\infty(\Rd)} \leq C (||\nabla g||_{L^p(\Rd)}+||\nabla u_{\shom}||_{L^2(\Rd)}).
  \end{equation}
  For the argument note that condition \hyperlink{R}{(R)} for $d\geq 3$ either assumes that we are in the scalar case, i.e.~$m=1$, or that $A_{\shom}$ has Uhlenbeck structure. In the scalar case, \eqref{UniformGradBoundUhom} follows from a Moser iteration for \eqref{DifferentiatedEquationUhom} (cf.~\cite[Theorem~8.15]{GT}), while in the vector-valued case we have to appeal to the Uhlenbeck structure of the limiting equation (and hence, in essence, use the fact that $|\nabla u_\shom|^2$ is a subsolution to a suitable elliptic PDE, which allows one to apply a Moser iteration again). In summary we get the claimed estimate on $\widehat C(\nabla u_{\shom})$.
\end{proof}

\begin{proof}[Proof of Remark~\ref{R:T3}.]
  The existence of $u_{\shom}$ is guaranteed, since in view of Theorem~\ref{TheoremStructureProperties}, the effective material law $A_\shom$ inherits the monotone structure from the heterogeneous material law $A(\omega,\cdot)$. The standard energy estimate yields
  \begin{equation*}
    \|u_{\shom}\|_{H^1(\Rd)}\leq C\|g\|_{L^2(\Rd)},
  \end{equation*}
  where here and below $C$ only depends on $d,m,\lambda$ and $\Lambda$. Differentiation with respect to $x_i$ yields the linear system
  \begin{align}
    \label{DifferentiatedEquationUhom2}
    &\partial_iu_{\shom}-\nabla \cdot (a_{\xi}\nabla \partial_iu_\shom) =\nabla \cdot (\partial_ig),
  \end{align}
  where $a_\xi:=\partial_\xi A_\shom(\nabla u_\shom)$. Meyers' estimate implies that
$\|\nabla\partial_i u_{\shom}\|_{L^{p'}(\Rd)}\leq C'\|\nabla g\|_{L^{p'}(\R^d)}$
  for some $2<p'\leq p$ only depending on $d,m,\lambda$ and $\Lambda$. By Sobolev embedding (and an interpolation inequality used to estimate $\|\nabla g\|_{L^{p'}(\Rd)}$ in terms of $\|\nabla g\|_{L^2(\Rd)}$ and $\|\nabla g\|_{L^p(\Rd)}$), we obtain
  \begin{equation*}
    ||\nabla u_\shom||_{L^\infty(\Rd)} \leq C\big(\|g\|_{H^1(\Rd)}+\|\nabla g\|_{L^p(\Rd)}\big).
  \end{equation*}
Our claim
$\widehat C(\nabla u_{\shom})\leq C\big(1+\|g\|_{H^1(\Rd)}+\|\nabla g\|_{L^p(\Rd)}\big)^C\|g\|_{H^1(\Rd)}$ thus follows.
\end{proof}

\begin{proof}[Proof of Remark~\ref{R:T4}]
The difference quotient technique of L.\ Nirenberg yields the desired bound
$\|u_{\shom}\|_{H^2(\domain)}\leq C\big(\|f\|_{L^2(\domain)}+\|g\|_{H^1(\domain)}+\|u_{\rm Dir}\|_{H^2(\domain)}\big)$.
Note that in this argument in fact no $C^1$ differentiability of $A_\shom$ is required: As the difference quotient technique works with finite differences and not with the formally differentiated PDE $-\nabla \cdot (a_\xi \nabla \partial_i u_\shom)=\partial_i f + \nabla \cdot (\partial_i g)$, there is no need to rigorously justify the formally differentiated PDE.
\end{proof}

\section{Qualitative differentiability of correctors}

We now provide the proof of the qualitative differentiability results for the homogenization correctors that we have used throughout the present work.
\begin{proof}[Proof of Lemma~\ref{LemmaCorrectorDifferentiability}]
Let $h>0$. Subtracting the equations for $\phi_{\xi+h\Xi}^T$ and $\phi_\xi^T$, we obtain
\begin{align}
\label{DifferenceTwoXi}
-\nabla \cdot \big((A(\tilde\omega,\xi+h\Xi+\nabla \phi_{\xi+h\Xi}^T)-A(\tilde\omega,\xi+\nabla \phi_\xi^T)\big) + \frac{1}{T} (\phi_{\xi+h\Xi}^T-\phi_\xi^T) =0
\end{align}
which we may rewrite as
\begin{align*}
&-\nabla \cdot \big((A(\tilde\omega,\xi+\nabla \phi_{\xi+h\Xi}^T)-A(\tilde\omega,\xi+\nabla \phi_\xi^T)\big) + \frac{1}{T} (\phi_{\xi+h\Xi}^T-\phi_\xi^T)
\\&~~~~~~~~~~~~~~~~~~
=\nabla \cdot (A(\tilde\omega,\xi+h\Xi+\nabla \phi_{\xi+h\Xi}^T)-A(\tilde\omega,\xi+\nabla \phi_{\xi+h\Xi}^T)).
\end{align*}
Note that by \hyperlink{A2}{(A2)} we have $|A(\tilde\omega,\xi+h\Xi+\nabla \phi_{\xi+h\Xi}^T)-A(\tilde\omega,\xi+\nabla \phi_{\xi+h\Xi}^T)|\leq C h |\Xi|$. Lemma~\ref{L:exploc} yields
\begin{align*}
&\int_\Rd \Big(|\nabla \phi_{\xi+h\Xi}^T-\nabla \phi_{\xi}^T|^2 +\frac{1}{T}|\phi_{\xi+h\Xi}^T-\phi_{\xi}^T|^2 \Big) \exp(-c|x|/\sqrt{T}) \,dx
\\&
\leq C \int_\Rd h^2 |\Xi|^2 \exp(-c|x|/\sqrt{T}) \,dx \leq C h^{2} |\Xi|^{2} \sqrt{T}^d
\end{align*}
and Meyers estimate (see Lemma~\ref{L:linearMeyersLocalized}) with a dyadic decomposition of $\Rd$ into $B_{\sqrt{T}}$ and $B_{2^k\sqrt{T}}\setminus B_{2^{k-1}\sqrt{T}}$ for $k\in \mathbb{N}$ upgrades this to
\begin{align*}
\int_\Rd |\nabla \phi_{\xi+h\Xi}^T-\nabla \phi_{\xi}^T|^{2p} \exp(-c|x|/\sqrt{T}) \,dx \leq C h^{2p} |\Xi|^{2p} \sqrt{T}^d
\end{align*}
for some $p>1$.

Subtracting a multiple of the PDE for the linearized corrector $\phi_{\xi,\Xi}^T$ from \eqref{DifferenceTwoXi}, we deduce
\begin{align*}
&-\nabla \cdot \big((A(\tilde\omega,\xi+h\Xi+\nabla \phi_{\xi+h\Xi}^T)-A(\tilde\omega,\xi+\nabla \phi_\xi^T)-h\partial_\xi A(\tilde \omega,\xi+\nabla \phi_{\xi}^T)(\Xi+\nabla \phi_{\xi,\Xi}^T)\big)
\\&~~~~~~~~~~~~~~~~~~~~~~~~~~~~~~~~~~~~~~~~~~~~~~~~~~~~~~~~~~~~~~~~~
+ \frac{1}{T} (\phi_{\xi+h\Xi}^T-\phi_\xi^T-h \phi_{\xi,\Xi}^T) =0.
\end{align*}
Using Taylor expansion, the uniform bound $|\partial_\xi^2 A|\leq \Lambda$ from \hyperlink{R}{(R)}, and the Lipschitz estimate for $A$ from \hyperlink{A2}{(A2)}, we obtain for any $\delta\in (0,1]$
\begin{align*}
-\nabla \cdot \big(\partial_\xi A(\tilde\omega,\xi+\nabla \phi_\xi^T)(\nabla \phi_{\xi+h\Xi}^T-\nabla \phi_\xi^T - h \nabla \phi_{\xi,\Xi}^T )\big) + \frac{1}{T} (\phi_{\xi+h\Xi}^T-\phi_\xi^T-&h\phi_{\xi,\Xi}^T)
\\&= \nabla \cdot R
\end{align*}
with $|R|\leq C h^{1+\delta} |\Xi|^{1+\delta} + C |\nabla \phi_{\xi+h\Xi}^T-\nabla \phi_\xi^T|^{1+\delta}$. Choosing $\delta$ such that $1+\delta<p$ and applying Lemma~\ref{L:exploc}, this finally yields the estimate
\begin{align*}
\int_\Rd |\nabla \phi_{\xi+h\Xi}^T-\nabla \phi_\xi^T-h\nabla \phi_{\xi,\Xi}^T|^2 \exp(-c|x|/\sqrt{T}) \,dx \leq C h^{2+2\delta} |\Xi|^{2+2\delta} \sqrt{T}^{d}.
\end{align*}

The proof of the corresponding result for $\sigma_\xi$ is even easier, as the equation for $\sigma_\xi$ is linear in $q_\xi^T=A(\tilde \omega,\xi+\nabla \phi_\xi^T)$.
\end{proof}

\section{Meyers estimate for elliptic equations with massive term}

\label{SectionMeyers}

We recall Gehring's lemma in the following form.
\begin{lemma}[{see e.\,g.\ \cite[Theorem 6.38]{GiaquintaMartinazzi}}]
\label{L:gehring}
Let $K>0$, $m\in(0,1)$, $s>1$ and $B=B_R(x_0)$ for some $x_0\in\R^d$ and $R>0$ be given. Suppose that $f\in L^1(B)$ and $g\in L^s(B)$ are such that for every $z\in\R^d$ and $r>0$ with $B_r(z)\subset B$ it holds that
\begin{align*}
\fint_{B_{r/2}(z)} |f| \,dx
\leq K\bigg(\fint_{B_r(z)} |f|^m \,dx\bigg)^{1/m}
+\fint_{B_{r}(z)} |g| \,dx.
\end{align*}
Then there exist $q=q(K,m)\in(1,s]$ and $C=C(K,m)\in[1,\infty)$ such that $f\in L^q(B_{R/2}(x_0))$ and the estimate
\begin{equation*}
\bigg(\fint_{B_{R/2}(z)} |f|^q \,dx\bigg)^{1/q}
\leq C\fint_{B_R(z)} |f| \,dx
+C\bigg(\fint_{B_{R}(z)} |g|^q \,dx\bigg)^{1/q}.
\end{equation*}
holds.
\end{lemma}

\begin{lemma}[Meyers estimate for PDEs with massive term]\label{L:linearMeyersLocalized}
Let $d,m\in \mathbb{N}$, $T>0$, and let $a:\Rd\rightarrow (\Rmd\otimes \Rmd)$ be a uniformly elliptic and bounded coefficient field with ellipticity and boundedness constants $0<\lambda\leq \Lambda<\infty$. Let $f\in L^2(\Rd;\Rm)\cap L^p(\Rd;\Rm)$, $g\in L^2(\Rd;\Rmd)\cap L^p(\Rd;\Rmd)$, and let $v\in H^1(\Rd;\Rm)$ be the (unique) weak solution to the linear system
  \begin{equation*}
    -\nabla\cdot(a \nabla v)+\frac1T v=\nabla \cdot g + \frac{1}{\sqrt{T}} f
    \qquad \text{on }\Rd.
  \end{equation*}
  Then there exists $p_0=p_0(d,m,\lambda,\Lambda)>2$ such that for all $2\leq p\leq p_0$, any $x_0\in \Rd$, and any $R>0$ we have
\begin{align*}
&\fint_{B_R(x_0)}|\nabla v|^p + \Big|\frac{1}{\sqrt{T}}v\Big|^p \,dx
\\&
\leq C(d,m,\lambda,\Lambda,p) \fint_{B_{2R}(x_0)} |g|^p + |f|^p \,dx
\\&~~~~
+ C(d,m,\lambda,\Lambda,p) \bigg(\fint_{B_{2R}(x_0)}|\nabla v|^2 + \Big|\frac{1}{\sqrt{T}}v\Big|^2 \,dx\bigg)^{p/2}.
\end{align*}
\end{lemma}
\begin{proof}
Since $a$ is uniformly elliptic and bounded, we deduce by the Caccioppoli inequality \eqref{CaccippoliMonotone} for any $b\in \Rm$, and $r>0$, and any $z\in \Rd$
\begin{align*}
\int_{B_{\frac{r}{2}}(z)} |\nabla v|^2 + \frac{1}{T} |v|^2 \,dx
  &\leq C(d,m,\lambda,\Lambda) \int_{B_r(z)} r^{-2} |v-b|^2 +\frac{1}{T} |b|^2 + |g|^2 + |f|^2 \,dx.
\end{align*}
Choosing $b:=\fint_{B_r}v \,dx$ and using the Poincar\'e-Sobolev inequality as well as Jensen's inequality, we deduce
\begin{align*}
&\fint_{B_{\frac{r}{2}}(z)} |\nabla v|^2 + \frac{1}{T} |v|^2 \,dx
\\
&\leq C(d,m,\lambda,\Lambda) \bigg(\fint_{B_r(z)} |\nabla v|^{2d/(d+2)} + \Big|\frac{1}{\sqrt{T}}v\Big|^{2d/(d+2)} \,dx \bigg)^{(d+2)/d}
\\&~~~~
+ C(d,m,\lambda,\Lambda) \fint_{B_r(z)}  |g|^2 + |f|^2 \,dx.
\end{align*}
Lemma~\ref{L:gehring} now yields the desired estimate.
\end{proof}

\begin{proof}[Proof of Lemma~\ref{LemmaWeightedMeyers}]
Let $R>0$. We split $v$ as $v=v_{out}+\sum_{k=1}^\infty v_{in,k}$, where $v_{in,k}\in H^1(\Rd;\Rm)$ is the unique weak solution to the PDE
\begin{subequations}
\begin{align}
\label{EquationVin}
-\nabla \cdot (a\nabla v_{in,k})+\frac{1}{T} v_{in,k} =& \nabla \cdot (g\chi_{B_{2^{-k}R}(x_0)\setminus B_{2^{-k-1}R}(x_0)})
\\&
\nonumber
+\frac{1}{\sqrt{T}} f\chi_{B_{2^{-k}R}(x_0)\setminus B_{2^{-k-1}R}(x_0)}
\end{align}
and where $v_{out}\in H^1(\Rd;\Rm)$ is the unique weak solution to the PDE
\begin{align}
\label{EquationVout}
-\nabla \cdot (a\nabla v_{out})+\frac{1}{T} v_{out} = \nabla \cdot (g\chi_{\Rd\setminus B_{R/2}(x_0)})+\frac{1}{\sqrt{T}} f\chi_{\Rd \setminus B_{R/2}(x_0)}.
\end{align}
\end{subequations}
Passing to the limit $R\rightarrow \infty$ in the hole-filling estimate \eqref{HoleFillingMonotone} for $v_{out}$ and inserting $4R$ in place of $r$ in the resulting bound, we deduce for any $\delta>0$ with $\delta<c$
\begin{align}
\nonumber
&\int_{B_{4R}(x_0)} |\nabla v_{out}|^2 + \frac{1}{T} |v_{out}|^2 \,dx
\\&
\label{HoleFillingLimit}
\leq C \int_{\Rd\setminus B_{R/2}(x_0)} \bigg(\frac{R}{R+|x-x_0|}\bigg)^{\delta} \big(|g|^2+|f|^2\big) \,dx.
\end{align}
This yields by dividing by $c(d) R^{d}$ and applying H\"older's inequality
\begin{align*}
&\fint_{B_{4R}(x_0)} |\nabla v_{out}|^2 + \frac{1}{T} |v_{out}|^2 \,dx
\\&
\leq C(d,m,\lambda,\Lambda,p,\gamma) \bigg(\int_{\Rd\setminus B_{R/2}(x_0)} R^{-d} \bigg(\frac{R}{R+|x-x_0|}\bigg)^{\delta p/2-\gamma} \big(|g|^p+|f|^p\big) \,dx\bigg)^{2/p}
\end{align*}
for any $\gamma$ with $d(p-2)/2<\gamma<\delta$.
Plugging this bound into Lemma~\ref{L:linearMeyersLocalized}, we obtain
\begin{align}
&\int_{B_{2R}(x_0)} |\nabla v_{out}|^p + \Big|\frac{1}{\sqrt{T}}v_{out}\Big|^p \,dx
\nonumber
\\&
\nonumber
\leq
C(d,m,\lambda,\Lambda,p) \int_{B_{4R}(x_0)\setminus B_{R/2}(x_0)} |g|^p + |f|^p \,dx
\\&~~~~
\nonumber
+ C(d,m,\lambda,\Lambda,p,\gamma) \int_{\Rd\setminus B_{R/2}(x_0)} \bigg(\frac{R}{R+|x-x_0|}\bigg)^{\delta p/2-\gamma} \big(|g|^p+|f|^p\big) \,dx
\\&
\leq
C(d,m,\lambda,\Lambda,p,\gamma) \int_{\Rd\setminus B_{R/2}(x_0)} \bigg(\frac{R}{R+|x-x_0|}\bigg)^{\delta p/2-\gamma} \big(|g|^p+|f|^p\big) \,dx.
\label{BoundVout}
\end{align}
We next estimate the contributions of $v_{in,k}$. We have
\begin{align*}
\int_{B_{4R}(x_0)\setminus B_{R/2}(x_0)} |\nabla v_{in,k}|^2 + \frac{1}{T} |v_{in,k}|^2 \,dx
\leq 2 \bigg|\sup_{\tilde g,\tilde f}\int_\Rd \tilde g\cdot \nabla v_{in,k} + \frac{1}{\sqrt{T}} v_{in,k} \tilde f \,dx\bigg|^2
\end{align*}
where the supremum runs over all functions $\tilde g$ and $\tilde f$ with $\supp \tilde g\cup \supp \tilde f \subset B_{4R}(x_0)\setminus B_{R/2}(x_0)$ and $\int_\Rd |\tilde g|^2 + |\tilde f|^2 \,dx\leq 1$. Denoting by $w\in H^1(\Rd;\Rm)$ the unique solution to the dual PDE
\begin{align*}
-\nabla \cdot (a^*\nabla w)+\frac{1}{T} w = -\nabla \cdot \tilde g + \frac{1}{\sqrt{T}} \tilde f,
\end{align*}
we obtain
\begin{align*}
&\int_{B_{4R}(x_0)\setminus B_{R/2}(x_0)} |\nabla v_{in,k}|^2 + \frac{1}{T} |v_{in,k}|^2 \,dx
\\
&\leq 2 \bigg|\sup_{\tilde g,\tilde f} \int_\Rd a\nabla v_{in,k} \cdot \nabla w + \frac{1}{T} v_{in,k} \, w \,dx \bigg|^2
\\&
\leq 2 \sup_{\tilde g,\tilde f} \bigg|\int_{B_{2^{-k}R}(x_0)\setminus B_{2^{-k-1}R}(x_0)} g \cdot \nabla w - \frac{1}{\sqrt{T}} f w \,dx \bigg|^2
\\&
\leq 2 \int_{B_{2^{-k}R}(x_0)\setminus B_{2^{-k-1}R}(x_0)} |g|^2 + |f|^2 \,dx
\\&~~~~\times \sup_{\tilde g,\tilde f} \int_{B_{2^{-k}R}(x_0)\setminus B_{2^{-k-1}R}(x_0)} |\nabla w|^2 + \frac{1}{T} |w|^2 \,dx.
\end{align*}
By the hole-filling estimate for $w$ in the form analogous to \eqref{HoleFillingLimit}, we deduce
\begin{align*}
&\int_{B_{2^{-k}R}(x_0)\setminus B_{2^{-k-1}R}(x_0)} |\nabla w|^2 + \frac{1}{T} |w|^2 \,dx
\\&
\leq C \int_{\Rd} \bigg(\frac{2^{-k} R}{2^{-k} R+|x-x_0|}\bigg)^\delta \big(|\tilde g|^2+|\tilde f|^2\big) \,dx
\end{align*}
which entails by the properties of $\tilde g$ and $\tilde f$
\begin{align*}
&\int_{B_{4R}(x_0)\setminus B_{R/2}(x_0)} |\nabla v_{in,k}|^2 + \frac{1}{T} |v_{in,k}|^2 \,dx
\\&
\leq C (2^{-k})^\delta \int_{B_{2^{-k}R}(x_0)\setminus B_{2^{-k-1}R}(x_0)} |g|^2 + |f|^2 \,dx.
\end{align*}
An application of Lemma~\ref{L:linearMeyersLocalized} to \eqref{EquationVin} (with a covering argument for the annulus $B_{2R}(x_0)\setminus B_{R}(x_0)$) yields by the preceding estimate and Jensen's inequality
\begin{align*}
&\int_{B_{2R}(x_0)\setminus B_{R}(x_0)} |\nabla v_{in,k}|^p + \frac{1}{T} |v_{in,k}|^p \,dx
\\&
\leq C (2^{-k})^{\delta p/2+d(p-2)/2} \int_{B_{2^{-k}R}(x_0)\setminus B_{2^{-k-1}R}(x_0)} |g|^p + |f|^p \,dx.
\end{align*}
Taking the sum in $k$ and adding \eqref{BoundVout}, we deduce by requiring $p$ to be close enough to $2$ and then choosing $\gamma>0$ small enough
\begin{align*}
&\int_{B_{2R}(x_0)\setminus B_{R}(x_0)} |\nabla v|^p + \Big|\frac{1}{\sqrt{T}} v\Big|^p \,dx
\\&
\leq C \int_{B_R(x_0)} \bigg(\frac{|x-x_0|}{R}\bigg)^{p\delta/3} \big(|g|^p+|f|^p\big) \,dx
\\&~~~~
+ C \int_{\Rd \setminus B_{R/2}(x_0)} \bigg(\frac{R}{|x-x_0|}\bigg)^{p \delta/3} \big(|g|^p+|f|^p\big) \,dx.
\end{align*}
Multiplying both sides by $(\frac{R}{r})^{\alpha_0}$, taking the sum over all dyadic $R=2^l r$, $l\in \mathbb{N}$, and using a standard Meyers estimate on the ball $B_r(x_0)$, we obtain
\begin{align*}
&\int_{\Rd} \Big(|\nabla v|^p + \Big|\frac{1}{\sqrt{T}}v\Big|^p\Big) \Big(1+\frac{|x-x_0|}{r}\Big)^{\alpha_0} \,dx
\\&
\leq C \sum_{l=1}^\infty \int_{B_{2^l r}(x_0)} \bigg(\frac{|x-x_0|}{2^l r}\bigg)^{p\delta/3} (2^l)^{\alpha_0} \big(|g|^p+|f|^p\big) \,dx
\\&~~~~
+ C \sum_{l=1}^\infty \int_{\Rd \setminus B_{2^l r}(x_0)} \bigg(\frac{2^l r}{|x-x_0|}\bigg)^{p \delta/3} (2^l)^{\alpha_0} \big(|g|^p+|f|^p\big) \,dx.
\end{align*}
We may estimate the last sum as
\begin{align*}
&\sum_{l=1}^\infty \int_{\Rd \setminus B_{2^l r}(x_0)} \bigg(\frac{2^l r}{|x-x_0|}\bigg)^{p \delta/3} (2^l)^{\alpha_0} \big(|g|^p+|f|^p\big) \,dx
\\&
\leq
C(\delta) \sum_{l=1}^\infty \sum_{n=l}^\infty \int_{B_{2^{n+1}r}(x_0) \setminus B_{2^n r}(x_0)} \big(2^{l-n}\big)^{p\delta/3} (2^l)^{\alpha_0} \big(|g|^p+|f|^p\big) \,dx
\\&
\leq
C(\delta,\alpha_0,\alpha_1) \sum_{n=1}^\infty \int_{B_{2^{n+1}r}(x_0) \setminus B_{2^n r}(x_0)} (2^n)^{\alpha_0} \big(|g|^p+|f|^p\big) \,dx.
\end{align*}
If $\alpha_0>0$ is chosen small enough, the previous two estimates imply \eqref{WeightedMeyers}.
\end{proof}

\bibliographystyle{abbrv}
\bibliography{stochastic_homogenization}

\end{document}